\documentclass{memo-l}
\usepackage{graphicx, amsfonts, amssymb}
\usepackage{mathrsfs, amsmath, amsthm}
\usepackage{hyperref}

\usepackage{graphicx}
\usepackage{subfigure}
\usepackage{lpic}

\newcommand{\R}{\mathcal{R}}
\newcommand{\D}{\mathbb{D}}
\newcommand{\T}{\mathbb{T}}
\newcommand{\N}{\mathbb{N}}
\newcommand{\BN}{\textit{BN}}
\renewcommand{\H}{\mathcal{H}}
\newcommand{\B}{\mathcal{B}}
\newcommand{\CC}{\mathcal{C}}
\newcommand{\BMOA}{\mathord{\rm BMOA}}
\newcommand{\VMOA}{\mathord{\rm VMOA}}
\newcommand{\I}{\mathcal{I}}
\newcommand{\SSS}{\mathcal{S}}
\newcommand{\op}{\mathrm{o}}

\newcommand{\e}{\varepsilon}
\newcommand{\vp}{\varphi}
\newcommand{\z}{\zeta}
\newcommand{\De}{\Delta}
\newcommand{\om}{\omega}
\newcommand{\p}{\psi}

\def\t{\theta}
\def\a{\alpha}
\def\b{\beta}

\newtheorem{theorem}{Theorem}[chapter]
\newtheorem{lemma}[theorem]{Lemma}
\newtheorem{proposition}[theorem]{Proposition}
\newtheorem{corollary}[theorem]{Corollary}

\newtheorem{lettertheorem}{Theorem}

\newtheorem{letterlemma}[lettertheorem]{Lemma}
\newtheorem{letterproposition}[lettertheorem]{Proposition}

\theoremstyle{definition}

\numberwithin{section}{chapter}

\numberwithin{equation}{chapter}

\makeindex

\begin{document}

\frontmatter

\title[Weighted Bergman Spaces induced by rapidly increasing weights]{Weighted Bergman Spaces induced by rapidly increasing weights}

\author{Jos\'e \'Angel Pel\'aez}
\address{Departamento de An´alisis Matem´atico, Universidad de M´alaga, Campus de
Teatinos, 29071 M´alaga, Spain} \email{japelaez@uma.es}

\thanks{This research was supported in part by the Ram\'on y Cajal program
of MICINN (Spain), Ministerio de Edu\-ca\-ci\'on y Ciencia, Spain,
(MTM2007-60854), from La Junta de Andaluc{\'i}a, (FQM210) and
(P09-FQM-4468), Academy of Finland 121281, MICINN- Spain ref.
MTM2008-05891, European Networking Programme HCAA of the European
Science Foundation.}

\author{Jouni R\"atty\"a}
\address{University of Eastern Finland, P.O.Box 111, 80101 Joensuu, Finland}
\email{jouni.rattya@uef.fi}

\date{\today}



\begin{abstract}
This monograph is devoted to the study of the weighted Bergman
space $A^p_\om$ of the unit disc $\D$ that is induced by a radial
continuous weight $\om$ satisfying
    \begin{equation}\label{absteq}
    \lim_{r\to
    1^-}\frac{\int_r^1\om(s)\,ds}{\om(r)(1-r)}=\infty.\tag{\dag}
    \end{equation}
Every such $A^p_\om$ lies between the Hardy space $H^p$ and every
classical weighted Bergman space~$A^p_\a$. Even if it is well
known that $H^p$ is the limit of~$A^p_\a$, as $\a\to-1$, in many
respects, it is shown that $A^p_\om$ lies ``closer'' to $H^p$ than
any $A^p_\a$, and that several finer function-theoretic properties
of $A^p_\a$ do not carry over to $A^p_\om$.

As to concrete objects to be studied, positive Borel measures
$\mu$ on $\D$ such that $A^p_\om\subset L^q(\mu)$, $0<p\le
q<\infty $, are characterized in terms of a neat geometric
condition involving Carleson squares. In this characterization
Carleson squares can not be replaced by (pseudo)hyperbolic discs.
These measures are shown to coincide with those for which a
H\"ormander-type maximal function from $L^p_\om$ to $L^q(\mu)$ is
bounded. It is also proved that each $f\in A^p_\om$ can be
represented in the form $f=f_1\cdot f_2$, where $f_1\in
A^{p_1}_\om$, $f_2\in A^{p_2}_\om$ and $\frac{1}{p_1}+
\frac{1}{p_2}=\frac{1}{p}$. Because of the tricky nature of
$A^p_\om$ several new concepts are introduced. In particular, the
use of a certain equivalent norm involving a square area function
and a non-tangential maximal function related to lens type regions
with vertexes at points in $\D$, gives raise to a some what new
approach to the study of the integral operator
    $$
    T_g(f)(z)=\int_{0}^{z}f(\zeta)\,g'(\zeta)\,d\zeta.
    $$
This study reveals the fact that $T_g:A^p_\om\to A^p_\om$ is
bounded if and only if $g$ belongs to a certain space of analytic
functions that is not conformally invariant. The lack of this
invariance is one of the things that cause difficulties in the
proof leading the above-mentioned new concepts, and thus further
illustrates the significant difference between $A^p_\om$ and the
standard weighted Bergman space $A^p_\a$. The symbols $g$ for
which $T_g$ belongs to the Schatten $p$-class $\SSS_p(A^2_\om)$
are also described. Furthermore, techniques developed are applied
to the study of the growth and the oscillation of analytic
solutions of (linear) differential equations.
\end{abstract}

\maketitle

\setcounter{page}{4}

\tableofcontents

\mainmatter

\chapter*{Preface}

This work concerns the weighted Bergman space $A^p_\om$ of the
unit disc $\D$ that is induced by a radial continuous weight
$\om:[0,1)\to(0,\infty)$ such that
    \begin{equation}\label{100}
    \lim_{r\to 1^-}\frac{\int_r^1\om(s)\,ds}{\om(r)(1-r)}=\infty.\tag{\ddag}
    \end{equation}
A radial continuous weight $\om$ with the property \eqref{100} is
called \emph{rapidly increasing} and the class of all such weights
is denoted by $\I$. Each $A^p_\om$ induced by $\om\in\I$ lies
between the Hardy space $H^p$ and every classical weighted Bergman
space $A^p_\a$ of~$\D$.

In many respects the Hardy space $H^p$ is the limit of $A^p_\a$,
as $\a\to-1$, but it is well known to specialists that this is a
very rough estimate since none of the finer function-theoretic
properties of the classical weighted Bergman space $A^p_\a$ is
carried over to the Hardy space $H^p$ whose (harmonic) analysis is
much more delicate. One of the main motivations for us to write
this monograph is to study the spaces $A^p_\om$, induced by
rapidly increasing weights, that indeed lie ``closer'' to $H^p$
than any $A^p_\a$ in the above sense and explore the change of
these finer properties related to harmonic analysis. We will see
that the ``transition'' phenomena from $H^p$ to~$A^p_\a$ does
appear in the context of~$A^p_\om$ with rapidly increasing weights
and raises a number of new interesting problems, some of which are
addressed in this monograph.

Chapter~\ref{S2} is devoted to proving several basic properties of
the rapidly increasing weights that will be used frequently in the
monograph. Also examples are provided to show that, despite of
their name, rapidly increasing weights are by no means necessarily
increasing and may admit a strong oscillatory behavior. Most of
the presented results remain true or have analogues for those
weights $\om$ for which the quotient in the left hand side
of~\eqref{100} is bounded and bounded away from zero. These
weights are called \emph{regular} and the class of all such
weights is denoted by~$\R$. Each standard weight
$\om(r)=(1-r^2)^\alpha$ is regular for all $-1<\a<\infty$.
Chapter~\ref{S2} is instrumental for the rest of the monograph.

Many conventional tools used in the theory of the classical
Bergman spaces fail to work in $A^p_\om$ that is induced by a
rapidly increasing weight $\om$. For example, one can not find a
weight $\om'=\om'(p)$ such that
$\|f\|_{A^p_\om}\asymp\|f'\|_{A^p_{\om'}}$ for all analytic
functions $f$ in $\D$ with $f(0)=0$, because such a
Littlewood-Paley type formula\index{Littlewood-Paley} does not
exist unless $p=2$. Moreover, neither $q$-Carleson measures for
$A^p_\om$ can be characterized by a simple condition on
pseudohyperbolic discs, nor the Riesz projection is necessarily
bounded on $A^p_\om$ if $0<p\le1$. However, it is shown in
Chapter~\ref{Sec:Carleson} that the embedding $A^p_\om\subset
L^q(\mu)$ can be characterized by a geometric condition on
Carleson squares $S(I)$ when $\om$ is rapidly increasing and
$0<p\le q<\infty$. In these considerations we will see that the
weighted maximal function
    $$
    M_{\om}(\vp)(z)=\sup_{I:\,z\in S(I)}\frac{1}{\om\left(S(I)
    \right)}\int_{S(I)}|\vp(\xi)|\om(\xi)\,dA(\xi),\quad
    z\in\D,
    $$
introduced by H\"ormander~\cite{HormanderL67}, plays a role on
$A^p_\om$ similar to that of the Hardy-Littlewood maximal function
on the Hardy space~$H^p$. Analogously, the conventional norm in
$A^p_\om$ is equivalent to a norm expressed in terms of certain
square area functions. These results illustrate in a very concrete
manner the significant difference between the function-theoretic
properties of the classical weighted Bergman space~$A^p_\a$ and
those of $A^p_\om$ induced by a rapidly increasing weight $\om$.

We will put an important part of our attention to the integral
operator
    \begin{displaymath}
    T_g(f)(z)=\int_{0}^{z}f(\zeta)\,g'(\zeta)\,d\zeta,\quad
    z\in\D,
    \end{displaymath}
induced by an analytic function $g$ on $\D$. This operator will
allow us to further underscore the transition phenomena from
$A^p_\a$ to $H^p$ through $A^p_\om$ with $\om\in\I$. The choice
$g(z)=z$ gives the usual Volterra operator and the Ces\`{a}ro
operator is obtained when $g(z)=-\log(1-z)$. The study of integral
operators on spaces of analytic functions merges successfully with
other areas of mathematics, such as the theory of univalent
functions, factorization theorems, harmonic analysis and
differential equations. Pommerenke was probably one of the first
authors to consider the operator $T_g$~\cite{Pom}. However, an
extensive study of this operator was initiated by the seminal
works by Aleman, Cima and Siskakis~\cite{AC,AS0,AS}. In all these
works classical function spaces such as $\BMOA$ and the Bloch
space $\B$ arise in a natural way. It is known that embedding-type
theorems and equivalent norms in terms of the first derivative
have been key tools in the study of the integral operator.
Therefore the study of $T_g$ has also lead developments that
evidently have many applications in other branches of the operator
theory on spaces of analytic functions. Recently, the spectrum
of~$T_g$ on the Hardy space $H^p$~\cite{AlPe} and the classical
weighted Bergman space $A^p_\alpha$~\cite{AlCo} has been studied.
The approach used in these works reveals, in particular, a strong
connection between a description of the resolvent set of $T_g$ on
$H^p$ and the classical theory of the Muckenhoupt weights. In the
case of the classical weighted Bergman space $A^p_\a$, the
Bekoll\'e-Bonami weights take the role of Muckenhoupt weights.

The approach we take to the study of the boundedness of the
integral operator~$T_g$ requires, among other things, a
factorization of $A^p_\om$-functions. In
Chapter~\ref{sec:factorizacion} we establish the required
factorization by using a probabilistic method introduced by
Horowitz~\cite{HorFacto}. We prove that if $\om$ is a weight (not
necessarily radial) such that
    \begin{equation}\label{222}
    \om(z)\asymp\om(\zeta),\quad z\in\Delta(\zeta,r),\quad \z\in\D,\tag{\pounds}
    \end{equation}
where $\Delta(\zeta,r)$ denotes a pseudohyperbolic disc, and
polynomials are dense in $A^p_\om$, then each $f\in A^p_\omega$
can be represented in the form $f=f_1\cdot f_2$, where $f_1\in
A^{p_1}_\omega$, $f_2\in A^{p_2}_\omega$ and $\frac{1}{p_1}+
\frac{1}{p_2}=\frac{1}{p}$, and the following norm estimates hold
    \begin{equation}\label{P1}
    \|f_1\|_{A^{p_1}_\omega}^p\cdot\|f_2\|_{A^{p_2}_\omega}^p\le\frac{p}{p_1}\|f_1\|_{A^{p_1}_\omega}^{p_1}+\frac{p}{p_2}\|f_2\|_{A^{p_2}_\omega}^{p_2}\le
    C(p_1,p_2,\omega)\|f\|_{A^p_\omega}^p.\tag{\S}
    \end{equation}
These estimates achieve particular importance when we recognize
that under certain additional hypothesis on the parameters
$p,\,p_1,$ and $p_2$, the constant in \eqref{P1} only depends on
$p_1$. This allows us to describe those analytic symbols $g$ such
that $T_g:A^p_\om\to A^q_\om$ is bounded, provided $0<q<p<\infty$
and $\om\in\I$ satisfies \eqref{222}. By doing this we avoid the
use of interpolation theorems and arguments based on Kinchine's
inequality, that are often employed when solving this kind of
problems on classical spaces of analytic functions on $\D$. The
techniques used to establish the above-mentioned factorization
result permit us to show that each subset of an $A^p_\om$-zero set
is also an $A^p_\om$-zero set. Moreover, we will show that the
$A^p_\om$-zero sets depend on $p$ whenever $\om\in\I\cup\R$. This
will be done by estimating the growth of the maximum modulus of
certain infinite products whose zero distribution depends on both
$p$ and $\om$. We will also briefly discuss the zero distribution
of functions in the Bergman-Nevanlinna class $\BN_\om$ that
consists of those analytic functions in $\D$ for which
    $$
    \int_\D\log^+|f(z)|\om(z)\,dA(z)<\infty.
    $$
Results related to this discussion will be used in
Chapter~\ref{difequ} when the oscillation of solutions of linear
differential equations in the unit disc is studied.

In Chapter~\ref{SecVolterra} we first equip $A^p_\om$ with several
equivalent norms inherited from different $H^p$-norms through
integration. Those ones that are obtained via the classical
Fefferman-Stein estimate or in terms of a non-tangential maximal
function related to lens type regions with vertexes at points in
$\D$, appear to be the most useful for our purposes. Here, we also
prove that it is not possible to establish a
Littlewood-Paley\index{Littlewood-Paley} type formula if
$\om\in\I$ unless $p=2$. In Section~\ref{sec:Volterra} we
characterize those analytic symbols $g$ on $\D$ such that
$T_g:A^p_\om\to A^q_\om$, $0<p, q<\infty$, is bounded or compact.
The case $q>p$ does not give big surprises because essentially
standard techniques work yielding a condition on the maximum
modulus of~$g'$. This is no longer true if $q=p$. Indeed, we will
see that $T_g:A^p_\om\to A^p_\om$ is bounded exactly when $g$
belongs to the space $\CC^{1}(\om^\star)$ that consists of those
analytic functions on $\D$ for which
   \begin{equation*}
    \|g\|^2_{\CC^{1}(\om^\star)}=|g(0)|^2+\sup_{I\subset\T}\frac{\int_{S(I)}|g'(z)|^2\om^\star(z)\,dA(z)}
    {\om\left(S(I)\right)}<\infty,
    \end{equation*}
where $$
    \omega^\star(z)=\int_{|z|}^1\omega(s)\log\frac{s}{|z|}s\,ds,\quad z\in\D\setminus\{0\}.
    $$
Therefore, as in the case of $H^p$ and $A^p_\a$, the boundedness
(and the compactness) is independent of $p$. It is also worth
noticing that the above $\CC^{1}(\om^\star)$-norm has all the
flavor of the known Carleson measure characterizations of $\BMOA$
and $\B$. Both of these spaces admit the important and very
powerful property of conformal invariance. In fact, this
invariance plays a fundamental role in the proofs of the
descriptions of when $T_g$ is bounded on either $H^p$ or~$A^p_\a$.
In contrast to $\BMOA$ and~$\B$, the space $\CC^1(\om^\star)$ is
not necessarily conformally invariant if $\om$ is rapidly
increasing, and therefore we will employ different techniques. In
Section~\ref{Hardy} we will show in passing that the methods used
are adaptable to the Hardy spaces, and since they also work for
$A^p_\om$ when $\om$ is regular, we consequently will obtain as a
by-product a unified proof for the classical results on the
boundedness and compactness of $T_g$ on $H^p$ and $A^p_\a$.
Chapter~\ref{Sec:SpaceCCpp} is devoted to the study of the space
$\CC^1(\om^\star)$ and its \lq\lq little oh\rq\rq counter\-part
$\CC_0^1(\om^\star)$. In particular, here we will prove that if
$\om\in\I$ admits certain regularity, then $\CC^1(\om^\star)$ is
not conformally invariant, and further, the strict inclusions
    $$
    \BMOA\subsetneq\CC^1(\om^\star)\subsetneq\B
    $$
are valid. Moreover, we will show, among other things, that
$C_0^1(\om^\star)$ is the closure of polynomials in
$\CC^1(\om^\star)$.

Chapter~\ref{Sec:Schatten} offers a complete description of those
analytic symbols $g$ in $\D$ for which the integral operator $T_g$
belongs to the Schatten $p$-class $\SSS_p(A^2_\om)$, where
$\omega\in\I\cup\R$. If $p>1$, then $T_g\in\SSS_p(A^2_\om)$ if and
only if $g$ belongs to the analytic Besov space $B_p$, and if
$0<p\le 1$, then $T_g\in \SSS_p(A^2_\om)$ if and only if $g$ is
constant. It is appropriate to mention that these results are by
no means unexpected. This is due to the fact that the operators
$T_g$ in both $\SSS_p(H^2)$ and $\SSS_p(A^2_\a)$ are also
characterized by the condition $g\in B_p$, provided $p>1$. What
makes this chapter interesting is the proofs which are carried
over in a much more general setting. Namely, we will study the
Toeplitz operator, induced by a complex Borel measure and a
reproducing kernel, in certain Dirichlet type spaces that are
induced by $\om^\star$, in the spirit of Luecking~\cite{Lu87}. Our
principal findings on this operator are gathered in a single
theorem at the end of Chapter~\ref{Sec:Schatten}.

In Chapter~\ref{difequ} we will study linear differential
equations with solutions in either the weighted Bergman space
$A^p_\om$ or the Bergman-Nevalinna class $\BN_\omega$. Our primary
interest is to relate the growth of coefficients to the growth and
the zero distribution of solutions. In
Section~\ref{Sec:SolutionsBergman} we will show how results and
techniques developed in the presiding chapters can be used to find
a set of sufficient conditions for the analytic coefficients of a
linear differential equation of order $k$ forcing all solutions to
the weighted Bergman space $A^p_\omega$. Since the zero
distribution of functions in $A^p_\om$ is studied in
Chapter~\ref{sec:factorizacion}, we will also obtain new
information on the oscillation of solutions. In
Section~\ref{Sec:SolutionsBergmanNevanlinna} we will see that it
is natural to measure the growth of the coefficients by the
containment in the weighted Bergman spaces depending on $\om$,
when all solutions belong to the Bergman-Nevalinna class
$\BN_\omega$. In particular, we will establish a one-to-one
correspondence between the growth of coefficients, the growth of
solutions and the zero distribution of solutions whenever $\om$ is
regular. In this discussion results from the Nevanlinna value
distribution theory are explicitly or implicitly present in many
instances. Apart from tools commonly used in the theory of complex
differential equations in the unit disc, Chapter~\ref{difequ} also
relies strongly on results and techniques from
Chapters~\ref{S2}--\ref{Sec:SpaceCCpp}, and is therefore
unfortunately a bit hard to read independently.

Chapter~\ref{sec:furtherdiscussion} is devoted to further
discussion on topics that this monograph does not cover. We will
briefly discuss $q$-Carleson measures for $A^p_\om$ when $q<p$,
generalized area operators as well as questions related to
differential equations and the zero distribution of functions in
$A^p_\om$. We include few open problems that are particularly
related to the special features of the weighted Bergman spaces
$A^p_\om$ induced by rapidly increasing weights.

\vspace{2em}

\begin{equation*}
\text{\em{Jos\'e \'Angel Pel\'aez}\quad \rm{(M\'alaga)}}
\end{equation*}

\begin{equation*}
\text{\em{Jouni R\"atty\"a}\quad \rm{(Joensuu)}}
\end{equation*}

\chapter{Basic Notation and Introduction to Weights}\label{S2}

In this chapter we first define the weighted Bergman spaces and
the classical Hardy spaces of the unit disc and fix the basic
notation. Then we introduce the classes of radial and non-radial
weights that are considered in the monograph, show relations
between them, and prove several lemmas on weights that are
instrumental for the rest of the monograph.

\section{Basic notation}

Let $\H(\D)$\index{$\H(\D)$} denote the algebra of all analytic
functions in the unit disc $\D=\{z:|z|<1\}$ of the complex plane
$\mathbb{C}$.\index{$\D$}\index{$\mathbb{C}$} Let $\T$\index{$\T$}
be the boundary of $\D$, and let $D(a,r)=\left\{z:
|z-a|<r\right\}$\index{$D(a,r)$} denote the Euclidean disc of
center $a\in\mathbb{C}$ and radius $r\in(0,\infty)$. A function
$\omega:\D\to (0,\infty)$, integrable over $\D$, is called a
\emph{weight function}\index{weight function} or simply a
\emph{weight}\index{weight}. It is \emph{radial}\index{radial
weight} if $\omega(z)=\omega(|z|)$ for all $z\in\D$. For
$0<p<\infty$ and a weight $\omega$, the \emph{weighted Bergman
space} \index{weighted Bergman space}
$A^p_\omega$\index{$A^p_\om$} consists of those $f\in\H(\D)$ for
which
    $$
    \|f\|_{A^p_\omega}^p=\int_\D|f(z)|^p\omega(z)\,dA(z)<\infty,\index{$\Vert\cdot\Vert_{A^p_\omega}$}
    $$
where $dA(z)=\frac{dx\,dy}{\pi}$ \index{$dA(z)$}is the normalized
Lebesgue area measure on $\D$. As usual, we write
$A^p_\alpha$\index{$A^p_\alpha$} for the \emph{classical weighted
Bergman space}\index{classical weighted Bergman space} induced by
the standard radial weight $\omega(z)=(1-|z|^2)^\alpha$,
$-1<\alpha<\infty$.\index{standard weight} For $0<p\le\infty$, the
\emph{Hardy space} \index{Hardy space} $H^p$\index{$H^p$} consists
of those $f\in\H(\D)$ for which
    \begin{equation*}\label{normi}
    \|f\|_{H^p}=\lim_{r\to1^-}M_p(r,f)<\infty,\index{$\Vert\cdot\Vert_{H^p}$}
    \end{equation*}
where
    $$
    M_p(r,f)=\left (\frac{1}{2\pi }\int_0^{2\pi}
    |f(re^{i\theta})|^p\,d\theta\right )^{\frac{1}{p}},\quad 0<p<\infty,
    $$
 and \index{$M_p(r,f)$}
    $$
    M_\infty(r,f)=\max_{0\le\theta\le2\pi}|f(re^{i\theta})|.
    $$
\index{$M_\infty(r,f)$}For the theory of the Hardy and the
classical weighted Bergman spaces,
see~\cite{Duren1970,DurSchus,Garnett1981,HKZ,Zhu}.

Throughout the monograph, the letter $C=C(\cdot)$
\index{$C(\cdot)$} will denote an absolute constant whose value
depends on the parameters indicated in the parenthesis, and may
change from one occurrence to another. We will use the notation
$a\lesssim b$ \index{$\lesssim$} if there exists a constant
$C=C(\cdot)>0$ such that $a\le Cb$, and $a\gtrsim b$
\index{$\gtrsim$} is understood in an analogous manner. In
particular, if $a\lesssim b$ and $a\gtrsim b$, then we will write
$a\asymp b$\index{$\asymp$}.

\section{Regular and rapidly increasing
weights}\label{Sec:DefIUR}

The \emph{distortion function}\index{distortion function} of a
radial weight $\om:[0,1)\to(0,\infty)$ is defined by
    $$
    \psi_{\om}(r)=\frac{1}{\om(r)}\int_{r}^1\om(s)\,ds,\quad
    0\le r<1,\index{$\psi_{\om}(r)$}
    $$
and was introduced by Siskakis in~\cite{Si}. A radial weight $\om$
is called {\em{regular}}\index{regular weight}, if $\om$ is
continuous and its distortion function satisfies
    \begin{equation}\label{eq:r0}
    \psi_\om(r)\asymp(1-r),\quad 0\le r<1.
    \end{equation}
The class of all regular weights is denoted by $\R$.\index{$\R$}
We will show in Section~\ref{Sec:LemmasOnWeights} that if
$\om\in\R$, then for each $s\in[0,1)$ there exists a constant
$C=C(s,\omega)>1$ such that
    \begin{equation}\label{eq:r2}
    C^{-1}\om(t)\le \om(r)\le C\om(t),\quad 0\le r\le t\le
    r+s(1-r)<1.
    \end{equation}
It is easy to see that \eqref{eq:r2} implies
    \begin{equation}\label{eq:r3}
    \psi_{\om}(r)\ge C(1-r),\quad 0\le r<1,
    \end{equation}
for some constant $C=C(\omega)>0$. However, \eqref{eq:r2} does not
imply the existence of $C=C(\om)>0$ such that
    \begin{equation}\label{eq:r1}
    \psi_{\om}(r)\le C(1-r),\quad0\le r<1,
    \end{equation}
as is seen by considering the weights
    \begin{equation}\label{eq:def-of-v_alpha}
    v_\a(r)=\left((1-r)\left(\log\frac{e}{1-r}\right)^\a\right)^{-1},\quad 1<\a<\infty.\index{$v_\a(r)$}
    \end{equation}
Putting the above observations together, we deduce that the
regularity of a radial continuous weight $\om$ is equivalently
characterized by the conditions \eqref{eq:r2} and \eqref{eq:r1}.
As to concrete examples, we mention that every $\om(r)=(1-r)^\a$,
$-1<\a<\infty$, as well as all the weights in
\cite[(4.4)--(4.6)]{AS} are regular.

The good behavior of regular weights can be broken in different
ways. On one hand, the condition~\eqref{eq:r2} implies, in
particular, that $\om$ can not decrease very fast. For example,
the exponential type weights \index{exponential weight}
    \begin{equation}\label{Eq:ExponentialWeights}
    \om_{\gamma,\alpha}(r)=(1-r)^{\gamma}\exp
    \left(\frac{-c}{(1-r)^\alpha}\right), \quad\gamma\ge0,\quad
    \alpha>0,\quad c>0,
    \end{equation}
satisfy neither~\eqref{eq:r2} nor \eqref{eq:r3}. Meanwhile the
weights $\om_{\gamma,\alpha}$ are monotone near $1$, the condition
\eqref{eq:r2} clearly also requires local smoothness and therefore
the regular weights can not oscillate too much. We will soon come
back to such oscillatory weights. On the other hand, we will say
that a radial weight $\om$ is \emph{rapidly
increasing}\index{rapidly increasing weight}, denoted by
$\om\in\I$,\index{$\I$} if it is continuous and
    \begin{equation}\label{eq:I}
    \lim_{r\to 1^-}\frac{\psi_{\om}(r)}{1-r}=\infty.
    \end{equation}
It is easy to see that if $\om$ is a rapidly increasing weight,
then $A^p_\om\subset A^p_\beta$ for any $\beta>-1$, see Section
\ref{Sec:LemmasOnWeights}. Typical examples of rapidly increasing
weights are $v_\a$\index{$v_\a(r)$}, defined in
\eqref{eq:def-of-v_alpha}, and
    \begin{equation}\label{Eq:PesosEnV-alpha}
    \om(r)=\left((1-r)\prod_{n=1}^{N}\log_n\frac{\exp_{n}0}{1-r}\left(\log_{N+1}\frac{\exp_{N+1}0}{1-r}\right)^\a\right)^{-1}
    \end{equation}
for all $1<\a<\infty$ and
$N\in\N=\{1,2,\ldots\}$.\index{$\mathbb{N}$} Here, as usual,
$\log_nx=\log(\log_{n-1}x)$\index{$\log_n$}, $\log_1x=\log x$,
$\exp_n x=\exp(\exp_{n-1}x)$\index{$\exp_n$} and $\exp_1x=e^x$. It
is worth noticing that if $\om\in\I$, then $\sup_{r\le
t}\psi_\om(r)/(1-r)$ can grow arbitrarily fast as $t\to1^-$. See
the weight defined in \eqref{pesomalo2} and
Lemma~\ref{Lemma:Distortion} in Section~\ref{Sec:LemmasOnWeights}.

In this study we are particularly interested in the weighted
Bergman space $A^p_\om$ induced by a rapidly increasing weight
$\om$, although most of the results are obtained under the
hypotheses \lq\lq $\om\in\I\cup\R$\rq\rq. However, there are
proofs in which more regularity is required for $\om\in\I$. This
is due to the fact that rapidly increasing weights may admit a
strong oscillatory behavior. Indeed, consider the weight
    \begin{equation}\label{pesomalo1}\index{$v_\a(r)$}\index{oscillatory weight}
    \omega(r)=\left|\sin\left(\log\frac{1}{1-r}\right)\right|v_\alpha(r)+1,\quad
    1<\alpha<\infty.
    \end{equation}
It is clear that $\omega$ is continuous and $1\le\omega(r)\le
v_\alpha(r)+1$\index{$v_\a(r)$} for all $r\in[0,1)$. Moreover, if
    $$
    1-e^{-(\frac{\pi}{4}+n\pi)}\le r \le
    1-e^{-(\frac{3\pi}{4}+n\pi)},\quad n\in\N\cup\{0\},
    $$
then $-\log(1-r)\in[\frac{\pi}{4}+n\pi,\frac{3\pi}{4}+n\pi]$, and
thus $\omega(r)\asymp v_\alpha(r)$\index{$v_\a(r)$} in there. Let
now $r\in(0,1)$, and fix $N=N(r)$ such that
$1-e^{-(N-1)\pi}<r\le1-e^{-N\pi}$. Then
    \begin{equation*}
    \begin{split}
    \int_r^1\omega(s)\,ds
    &\gtrsim\sum_{n=N}^\infty\int_{1-e^{-(\frac{\pi}{4}+n\pi)}}^{1-e^{-(\frac{3\pi}{4}+n\pi)}}v_\alpha(s)\,ds
    =\sum_{n=N}^\infty\frac{1}{\left(\frac{\pi}{4}+n\pi\right)^{\alpha-1}}-\frac{1}{(\frac{3\pi}{4}+n\pi)^{\alpha-1}}\\
    &\asymp\sum_{n=N}^\infty\frac{1}{n^{\alpha}}\asymp\int_N^\infty\frac{dx}{x^\alpha}\asymp\frac{1}{N^{\alpha-1}}
    \asymp\frac{1}{\left(\log\frac{1}{1-r}\right)^{\alpha-1}}\asymp\int_r^1v_\alpha(s)\,ds,
    \end{split}
    \end{equation*}
and it follows that $\omega\in\I$. However,
    $$
    \frac{\omega(1-e^{-n\pi-\frac{\pi}{2}})}{\omega(1-e^{-n\pi})}=v_\alpha(1-e^{-n\pi-\frac{\pi}{2}})+1\to\infty,\quad
    n\to\infty,
    $$
yet
    $$
    \frac{1-(1-e^{-n\pi-\frac{\pi}{2}})}{1-(1-e^{-n\pi})}=e^{-\frac{\pi}{2}}\in(0,1)
    $$
for all $n\in\N$. Therefore $\omega$ does not satisfy
\eqref{eq:r2}. Another bad-looking example in the sense of
oscillation is
    \begin{equation}\label{pesomalo2}\index{$v_\a(r)$}\index{oscillatory weight}
    \omega(r)=\left|\sin\left(\log\frac{1}{1-r}\right)\right|v_\alpha(r)+\frac{1}{e^{e^{\frac1{1-r}}}},\quad
    1<\alpha<\infty,
    \end{equation}
which  belongs to $\I$, but does not satisfy \eqref{eq:r2} by the
reasoning above. Moreover, by passing through the zeros of the
$\sin$ function, we see that
    $$
    \liminf_{r\to1^-}\omega(r)e^{e^\frac{1}{1-r}}=1.
    $$
Our last example on oscillatory weights is
    \begin{equation}\label{65}
    \om(r)=\left|\sin\left(\log\frac{1}{1-r}\right)\right|(1-r)^\a+(1-r)^\b,\index{oscillatory weight}
    \end{equation}
where $-1<\a<\b<\infty$. Obviously,
$(1-r)^\b\lesssim\om(r)\lesssim(1-r)^\a$, so $A^p_\a\subset
A^p_\om\subset A^p_\b$. However, $\om\not\in\I$ because the limit
in \eqref{eq:I} does not exist, and $\om\not\in\R$ because $\om$
neither satisfies \eqref{eq:r2} nor \eqref{eq:r1}, yet $\om$ obeys
\eqref{eq:r3}.

In the case when more local regularity is required for $\om\in\I$
in the proof, we will consider the class $\widetilde{\mathcal I}$
of those $\om\in\I$ that
satisfy~\eqref{eq:r2}.\index{$\widetilde{\mathcal I}$}

\section{Bekoll\'e-Bonami and invariant
weights}\label{sec:BBINV}

The \emph{Carleson square}\index{Carleson square}
$S(I)$\index{$S(I)$} associated with an interval $I\subset\T$ is
the set $S(I)=\{re^{it}\in\D:\,e^{it}\in I,\, 1-|I|\le r<1\}$,
where $|E|$\index{$\vert E\vert$} denotes the Lebesgue measure of
the measurable set $E\subset\T$. For our purposes it is also
convenient to define for each $a\in\D\setminus\{0\}$ the interval
$I_a=\{e^{i\t}:|\arg(a
e^{-i\t})|\le\frac{1-|a|}{2}\}$,\index{$I_a$}\index{$S(I_a)$}\index{$S(a)$}
and denote $S(a)=S(I_a)$.

Let $1<p_0,p_0'<\infty$ such that
$\frac{1}{p_0}+\frac{1}{p'_0}=1$, and let $\eta>-1$. A weight
$\om:\D\to(0,\infty)$ satisfies the \emph{Bekoll\'e-Bonami
$B_{p_0}(\eta)$-condition}\index{Bekoll\'e-Bonami
weight},\index{$B_{p_0}(\eta)$} denoted by $\om\in B_{p_0}(\eta)$,
if there exists a constant $C=C(p_0,\eta,\omega)>0$ such that
    \begin{equation}\label{eq:BB}
    \begin{split}
    &\left(\int_{S(I)}\om(z)(1-|z|)^{\eta}\,dA(z)\right)
    \left(\int_{S(I)}\om(z)^{\frac{-p'_0}{p_0}}(1-|z|)^{\eta}\,dA(z)\right)^{\frac{p_0}{p'_0}}\\
    &\le C|I|^{(2+\eta)p_0}
    \end{split}
    \end{equation}
for every interval $I\subset \T$. Bekoll\'e and Bonami introduced
these weights in~\cite{Bek,BB}, and showed that $\om\in
B_{p_0}(\eta)$\index{Bekoll\'e-Bonami
weight}\index{$B_{p_0}(\eta)$} if and only if the Bergman
projection\index{Bergman projection}\index{$P_\eta(f)$}
    $$
    P_\eta(f)(z)=(\eta+1)\int_\D\frac{f(\xi)}{(1-\overline{\xi}z)^{2+\eta}}(1-|\xi|^2)^\eta\,dA(\xi)
    $$
is bounded from $L^{p_0}_\om$ to $A^{p_0}_\om$~\cite{BB}. This
equivalence allows us to identify the dual space of $A^{p_0}_\om$
with $A^{p_0'}_\om$. In the next section we will see that if
$\om\in\R$, then for each $p_0>1$ there exists
$\eta=\eta(p_0,\omega)>-1$ such that $\frac{\om(z)}{(1-|z|)^\eta}$
belongs to $B_{p_0}(\eta)$.\index{Bekoll\'e-Bonami
weight}\index{$B_{p_0}(\eta)$} However, this is no longer true if
$\om\in\I$.

There is one more class of weights that we will consider. To give
the definition, we need to recall several standard concepts. For
$a\in\D$, define $\vp_a(z)=(a-z)/(1-\overline{a}z)$
\index{$\varphi_a$}. The automorphism $\vp_a$ of $\D$ is its own
inverse and interchanges the origin and the point $a\in\D$. The
\emph{pseudohyperbolic} and \emph{hyperbolic
distances}\index{pseudohyperbolic distance}\index{hyperbolic
distance} from $z$ to $w$ are defined as
$\varrho(z,w)=|\vp_z(w)|$\index{$\varrho(z,w)$} and
    $$
    \varrho_h(z,w)=\frac12\log\frac{1+\varrho(z,w)}{1-\varrho(z,w)},\quad
    z,w\in\D,
    $$
respectively.\index{$\varrho_h(z,w)$} The \emph{pseudohyperbolic
disc}\index{hyperbolic disc}\index{pseudohyperbolic disc} of
center $a\in\D$ and radius $r\in(0,1)$ is denoted by
$\Delta(a,r)=\{z:\varrho(a,z)<r\}$\index{$\Delta(a,r)$}. It is
clear that $\Delta(a,r)$ coincides with the \emph{hyperbolic disc}
$\Delta_h(a,R)=\{z:\varrho_h(a,z)<R\}$, where
$R=\frac12\log\frac{1+r}{1-r}\in(0,\infty)$.\index{$\Delta_h(a,r)$}

The class ${\mathcal Inv}$\index{${\mathcal Inv}$}\index{invariant
weight} of \emph{invariant weights}\index{invariant weight}
consists of those weights $\om$ (that are not necessarily radial
neither continuous) such that for some (equivalently for all)
$r\in(0,1)$ there exists a constant $C=C(r)\ge1$ such that
    $$
    C^{-1}\om(a)\le\om(z)\le C\om(a)
    $$
for all $z\in\De(a,r)$. In other words, $\om\in{\mathcal Inv}$ if
$\om(z)\asymp\om(a)$ in $\De(a,r)$. It is immediate that radial
invariant weights are neatly characterized by the condition
\eqref{eq:r2}, and thus $\I\cap{\mathcal Inv}=\widetilde{\I}$ and
$\R\cap{\mathcal Inv}=\R$. To see an example of a radial weight
that just fails to satisfy \eqref{eq:r2}, consider
    $$
    \om(r)=(1-r)^{\log_n\left(\frac{\exp_n0}{1-r}\right)}=\exp\left(-\log\left(\frac{1}{1-r}\right)\cdot\log_n\left(\frac{\exp_n0}{1-r}\right)\right),\quad
    n\in\N.
    $$
It is easy to see that
    $$
    \frac{\psi_\om(r)}{1-r}\asymp\frac{1}{\log_n\left(\frac{1}{1-r}\right)},\quad
    r\to1^-,
    $$
and hence $\om\not\in\I\cup\R$.\index{${\mathcal
Inv}$}\index{invariant weight} For an example of a weight which is
not rapidly decreasing neither belongs to $\I\cup\R$,
see~\eqref{112}. In Chapter~\ref{sec:factorizacion} we will see
that the class of those invariant weights $\om$\index{${\mathcal
Inv}$}\index{invariant weight} for which polynomials are dense in
$A^p_\om$ form a natural setting for the study of factorization of
functions in~$A^p_\om$.

\section{Lemmas on weights}\label{Sec:LemmasOnWeights}

In this section we will prove several lemmas on different weights.
These results do not only explain the relationships between
different classes of weights, but are also instrumental for the
rest of the monograph. The first three lemmas deal with the
integrability of regular and rapidly increasing weights.

\begin{lemma}\label{le:condinte}
\begin{itemize}
\item[\rm(i)] Let $\om\in\R$. Then there exist constants
$\a=\a(\om)>0$ and $\b=\b(\omega)\ge\a$ such that
    \begin{equation}\label{eq:Integral1}
    \left(\frac{1-r}{1-t}\right)^\a\int_t^1\om(s)\,ds\le\int_r^1\om(s)\,ds\le
    \left(\frac{1-r}{1-t}\right)^\b\int_t^1\om(s)\,ds
    \end{equation}
for all $0\le r\le t<1$. \item[\rm(ii)] Let $\om\in\I$. Then for
each $\b>0$ there exists a constant $C=C(\b,\omega)>0$ such that
    \begin{equation}\label{eq:Integral2}
    \int_r^1\om(s)\,ds\le C
    \left(\frac{1-r}{1-t}\right)^\b\int_t^1\om(s)\,ds,\quad 0\le r\le
    t<1.
    \end{equation}
\end{itemize}
\end{lemma}

\begin{proof} (i) Let $\om\in\R$, and let $C_1=C_1(\om)>0$ and $C_2=C_2(\om)\ge C_1$ such that
    \begin{equation}\label{eq:r4}
    C_1(1-r)\le\psi_\om(r)\le C_2(1-r),\quad 0\le r<1.
    \end{equation}
Then a direct calculation based on the first inequality in
\eqref{eq:r4} shows that the differentiable function
$h_{C_1}(r)=\frac{\int_{r}^1\om(s)\,ds}{(1-r)^{1/{C_1}}}$ is
increasing on $[0,1)$. The second inequality in
\eqref{eq:Integral1} with $\b=1/C_1$ follows from this fact. The
first inequality in \eqref{eq:Integral1} with $\a=1/C_2$ can be
proved in an analogous manner by showing that $h_{C_2}(r)$ is
decreasing on $[0,1)$.

(ii) Let now $\om\in\I$ and $\beta>0$. By \eqref{eq:I}, there
exists $r_0\in (0,1)$ such that $\psi_{\om}(r)/(1-r)\ge\b^{-1}$
for all $r\in [r_0,1)$. As above we deduce that $h_{1/\beta}(r)$
is increasing on $[r_0,1)$, and \eqref{eq:Integral2} follows.
\end{proof}

We make several observations on Lemma~\ref{le:condinte} and its
proof.
\begin{itemize}
\item[\rm(i)] If $\om\in\R$ and $0\le r\le t\le
    r+s(1-r)<1$, then Part (i) implies
    \begin{equation*}
    \frac{\om(r)}{\om(t)}\asymp\frac{1-t}{1-r}\cdot\frac{\int_r^1\om(v)\,dv}{\int_t^1\om(v)\,dv}
    \asymp\frac{\int_r^1\om(v)\,dv}{\int_t^1\om(v)\,dv}\asymp1,
    \end{equation*}
where the constants of comparison depend only on $s\in[0,1)$ and
$\b=\b(\om)>0$. This proves the important local
smoothness~\eqref{eq:r2} of regular weights.

\item[\rm(ii)] If $\om\in\R$, then $h_{C_1}(r)$ is increasing and
$h_{C_2}(r)$ is decreasing by the proof of Part~(i). Therefore
    \begin{equation}\label{64}
    (1-r)^{1/C_1}\lesssim\int_r^1\om(s)\,ds=\om(r)\psi_\om(r)\lesssim(1-r)^{1/C_2},
    \end{equation}
and since $\psi_\om(r)\asymp(1-r)$, we deduce $A^p_\a\subset
A^p_\om\subset A^p_\b$ for $\a=C_2^{-1}-1$ and $\b=C_1^{-1}-1$.
This means that each weighted Bergman space~$A^p_\om$, induced by
a regular weight $\om$, lies between two classical weighted
Bergman spaces. Of course, if $\om$ is a radial continuous weight
such that the chain of inclusions above is satisfied for some
$-1<\a<\b<\infty$, then $\om$ does not need to be regular as is
seen by considering the oscillatory weight given in~\eqref{65}. It
is also worth noticing that $\om(r)(1-r)^{-\gamma}$ is a weight
for each $\gamma<1/C_2$ by \eqref{64}. Moreover, an integration by
parts gives
    \begin{equation}\label{77}
    \begin{split}
    \int_r^1\om(s)(1-s)^{-\gamma}\,ds&=(1-r)^{-\gamma}\int_r^1\om(s)\,ds\\
    &\quad+\gamma\int_r^1\left(\int_s^1\om(t)\,dt\right)(1-s)^{-\gamma-1}\,ds\\
    &\asymp\om(r)(1-r)^{1-\gamma}+\gamma\int_r^1\om(s)(1-s)^{-\gamma}\,ds,
    \end{split}
    \end{equation}
so, by choosing $\gamma>0$ sufficiently small and reorganizing
terms, we obtain \eqref{eq:r0} for the weight
$\om(r)(1-r)^{-\gamma}$, that is, $\om(r)(1-r)^{-\gamma}\in\R$.

\item[\rm(iii)] Let now $\omega\in\I$, and let $\b>-1$ be fixed.
The proof of Part~(ii) shows that there exists
$r_0=r_0(\a)\in(\frac12,1)$ such that $h_{2/(1+\b)}$ is increasing
on $[r_0,1)$. Therefore $\int_r^1\om(s)\,ds\gtrsim
(1-r)^{\frac{1+\b}{2}}$ on $[r_0,1)$. If $f\in A^p_{\om}$, then
    \begin{equation*}
    \|f\|_{A^p_\omega}^p\ge\int_{\D\setminus
    D(0,r)}|f(z)|^p\omega(z)\,dA(z)\gtrsim
    M_p^p(r,f)\int_r^1\omega(s)\,ds,\quad r\ge\frac12,
    \end{equation*}
and hence
    \begin{equation*}
    \begin{split}
    \|f\|_{A^p_\b}^p&\lesssim\|f\|_{A^p_\omega}^p\int_0^1\frac{(1-r)^\b}{\int_r^1\omega(s)\,ds}\,dr\\
    &\lesssim\int_0^{r_0}\frac{(1-r)^\b}{\int_r^1\omega(s)\,ds}\,dr
    +\int_{r_0}^1\frac{dr}{(1-r)^{\frac{1-\b}{2}}}<\infty.
    \end{split}
    \end{equation*}
Therefore $f\in A^p_\b$, and we obtain the inclusion
$A^p_\om\subset A^p_\b$ for all $\b>-1$. Moreover, by combining
the equality in \eqref{77} and the estimate
$\int_r^1\om(s)\,ds\gtrsim (1-r)^{\frac{1+\b}{2}}$, established
above, for $\b=2\gamma-1$, we see that $\om(r)(1-r)^{-\gamma}$ is
not a weight for any $\gamma>0$.

\item[\rm(iv)] If $\om\in\I$ is differentiable and
    \begin{equation}\label{PP}
    \lim_{r\to 1^-}\frac{\om'(r)}{\om^2(r)}\int_r^1\om(s)\,ds
    \end{equation}
exists, then this limit is equal to $\infty$ by
Bernouilli-l'H\^{o}pital theorem. This in turn implies that $\om$
is \emph{essentially increasing} \index{essentially increasing} on
$[0,1)$, that is, there exists a constant $C\ge1$ such that
$\om(r)\le C\om (s)$ for all $0\le r\le s<1$. It is also known
that if the limit (superior) in \eqref{PP} is finite, then a
Littlewood-Paley type formula exists for all
$0<p<\infty$~\cite{PavP,Si}\index{Littlewood-Paley}, that is,
$\|f\|_{A^p_\om}\asymp\|f'\psi_\om\|_{L^p_{\om}}$ for all
$f\in\H(\D)$ with $f(0)=0$.

\item[\rm(v)] It is
easy to see that a radial continuous weight $\om$ is regular if and only if there
exist $-1<a<b<\infty$ and $r_0\in(0,1)$ such that
    \begin{equation}\label{75}
    \frac{\om(r)}{(1-r)^b}\nearrow\infty,\quad r\ge r_0,
    \quad\textrm{and}\quad
    \frac{\om(r)}{(1-r)^a}\searrow0,\quad r\ge r_0.
    \end{equation}
These weights without the continuity assumption were first studied
by Shields and Williams~\cite{ShieldsWilliams} in the range
$0<a<b<\infty$, and they are known as \emph{normal
weights}.\index{normal weight} Each weight
$v_\a\in\I$\index{$v_\a(r)$} satisfies \eqref{75} if we allow the
number $a$ to attain the value $-1$, which is usually excluded in
the definition.
\end{itemize}

\begin{lemma}\label{le:nec1}
\begin{itemize}
\item[\rm(i)] Let $\om\in\R$. Then there exist constants
$\gamma=\gamma(\om)>0$, $C_1=C_1(\gamma,\omega)>0$ and
$C_2=C_2(\gamma,\omega)>0$ such that
    \begin{equation}\label{Eq-Condition-Extra}
    C_1\int_t^1\om(s)\,ds\le\int_0^t\left(\frac{1-t}{1-s}\right)^{\gamma}\om(s)\,ds\le
    C_2\int_t^1\om(s)\,ds,\quad0<t<1.
    \end{equation}
More precisely, the first inequality is valid for all $\gamma>0$,
and the second one for all $\gamma>\b$, where $\b$ is from
\eqref{eq:Integral1}.

\item[\rm(ii)] Let $\om\in\I$. Then for each $\gamma>0$ there
exists a constant $C=C(\gamma,\omega)>0$ such that
    \begin{equation*}
    \int_0^t\left(\frac{1-t}{1-s}\right)^{\gamma}\om(s)\,ds\le C
    \int_t^1\om(s)\,ds,\quad0<t<1.
    \end{equation*}
\end{itemize}
\end{lemma}

\begin{proof} Let $t\in (0,1)$, and let $N$ be the smallest natural number such that
$(1-t)2^N>1$. Set $t_N=0$ and $t_n=1-2^n(1-t)$ for $n=0,\dots,
N-1$. Also, set $\om(s)=0$ for $s<0$.

(i) Let $\om\in\R$ and $\gamma>0$. Then, by
Lemma~\ref{le:condinte}(i), there exists $\a=\a(\om)>0$ such that
    \begin{equation*}
    \begin{split}
    \int_0^t\left(\frac{1-t}{1-s}\right)^{\gamma}\om(s)\,ds
    &=\sum_{n=0}^{N-1}\int_{t_{n+1}}^{t_n}\left(\frac{1-t}{1-s}\right)^{\gamma}\om(s)\,ds\\
    &\ge\sum_{n=0}^{N-1}2^{-(n+1)\gamma}\left(\int_{t_{n+1}}^{1}\om(s)\,ds-\int_{t_n}^1\om(s)\,ds\right)\\
    &\ge\sum_{n=0}^{N-1}2^{-(n+1)\gamma}(2^{\a}-1)\int_{t_{n}}^{1}\om(s)\,ds\\
    &\ge\frac{2^{\a}-1}{2^\gamma}\int_t^1\om(s)\,ds,
    \end{split}
    \end{equation*}
and thus the first inequality in \eqref{Eq-Condition-Extra} with
$C_1=\frac{2^{\a}-1}{2^\gamma}$ is proved. To see the second one,
we may argue as above to obtain
    \begin{equation}\label{eq:l11}
    \begin{split}
    \int_0^t\left(\frac{1-t}{1-s}\right)^{\gamma}\om(s)\,ds
    &\le(2^{\b}-1)\sum_{n=0}^{N-1}2^{-n\gamma}\int_{t_{n}}^{1}\om(s)\,ds\\
    &\le(2^{\b}-1)\sum_{n=0}^{\infty}2^{n(\b-\gamma)}\int_{t}^{1}\om(s)\,ds.
    \end{split}
    \end{equation}
This gives the second inequality in~\eqref{Eq-Condition-Extra} for
all $\gamma>\b$.

(ii) This can be proved by arguing as above and by using
Lemma~\ref{le:condinte}(ii).
\end{proof}

\begin{lemma}\label{le:momentos}
If $\om\in\I\cup\R$, then
    \begin{equation*}
    \int_0^1
    s^{x}\om(s)\,ds\asymp\int_{1-\frac{1}{x}}^1\om(s)\,ds,\quad
    x\in[1,\infty).
    \end{equation*}
\end{lemma}

\begin{proof}
If $x=1$, then the assertion follows by the inequalities
    $$
    \int_0^1s\om(s)\,ds\le\int_0^1\om(s)\,ds\le\left(2+\frac{\int_0^\frac12\om(s)\,ds}{\int_0^\frac12s\om(s)\,ds}\right)
    \int_0^1s\om(s)\,ds.
    $$
For $x>1$ it suffices to prove
    \begin{equation*}
    \int_0^{1-\frac{1}{x}} s^{x}\om(s)\lesssim
    \int_{1-\frac{1}{x}}^1\om(s)\,ds.
    \end{equation*}
To see this, let $\gamma=\gamma(\om)>0$ be the constant in
Lemma~\ref{le:nec1}. A simple calculation shows that
    $$
    s^{x-1}(1-s)^\gamma\le\left(\frac{x-1}{x-1+\gamma}\right)^{x-1}\left(\frac{\gamma}{x-1+\gamma}\right)^\gamma
    \le\left(\frac{\gamma}{x-1+\gamma}\right)^\gamma
    $$
for all $s\in[0,1]$. Therefore Lemma~\ref{le:nec1}, with
$t=1-\frac{1}{x}$, yields
    $$
    \int_0^{1-\frac{1}{x}} s^{x}\om(s)\,ds
    \le\left(\frac{\gamma x}{x-1+\gamma}\right)^\gamma\int_0^{1-\frac{1}{x}}\frac{\om(s)}{x^\gamma(1-s)^\gamma}s\,ds\lesssim
    \int_{1-\frac{1}{x}}^1\om(s)\,ds.
    $$
This finishes the proof.
\end{proof}

The next lemma shows that a continuous radial weight $\om$ that
satisfies \eqref{eq:r2} is regular if and only if it is a
Bekoll\'e-Bonami weight. Moreover, Part (iii) quantifies in a
certain sense the self-improving integrability of radial weights.

\begin{lemma}\label{le:RAp}
\begin{itemize}
\item[{\rm(i)}] If $\om\in\R$, then for each $p_0>1$ there exists
$\eta=\eta(p_0,\omega)>-1$ such that $\frac{\om(z)}{(1-|z|)^\eta}$
belongs to $B_{p_0}(\eta)$.\index{Bekoll\'e-Bonami
weight}\index{$B_{p_0}(\eta)$}

\item[{\rm(ii)}] If $\om$ is a continuous radial weight such that
\eqref{eq:r2} is satisfied and $\frac{\om(z)}{(1-|z|)^\eta}$
belongs to $B_{p_0}(\eta)$ for some $p_0>0$ and $\eta>-1$, then
$\om\in\R$.\index{Bekoll\'e-Bonami weight}\index{$B_{p_0}(\eta)$}

\item[{\rm(iii)}] For each radial weight $\om$ and $0<\a<1$,
define\index{$\widetilde{\om}(r)$}
    $$
    \widetilde{\om}(r)=\left(\int_r^1\om(s)\,ds\right)^{-\a}\om(r),\quad
    0\le r<1.
    $$
Then $\widetilde{\om}$ is also a weight and
$\psi_{\widetilde{\om}}(r)=\frac1{1-\a}\psi_{{\om}}(r)$ for all
$0\le r<1$.
\end{itemize}
\end{lemma}

\begin{proof}
(i) Since each regular weight is radial, it suffices to show that
there exists a constant $C=C(p,\eta,\omega)>0$ such that
    \begin{equation}\label{eq:rb1}
    \left (\int_{1-|I|}^1\om(t)\,dt\right )\left
    (\int_{1-|I|}^1\om(t)^{\frac{-p'_0}{p_0}}(1-t)^{p'_0\eta}\,dt\right
    )^{\frac{p_0}{p'_0}}\le C|I|^{(1+\eta)p_0}
    \end{equation}
for every interval $I\subset \T$. To prove~\eqref{eq:rb1}, set
$s_0=1-|I|$ and $s_{n+1}=s_n+s(1-s_n)$, where $s\in (0,1)$ is
fixed. Take $p_0$ and $\eta$ such that $\eta>\frac{\log
C}{p_0\log\frac{1}{1-s}}>0$, where the constant $C=C(s,\omega)>1$
is from~\eqref{eq:r2}. Then \eqref{eq:r2} yields
    \begin{equation*}
    \begin{split}
    \int_{1-|I|}^1\om(t)^{\frac{-p'_0}{p_0}}(1-t)^{p'_0\eta}\,dt
    &\le\sum_{n=0}^\infty(1-s_n)^{p'_0\eta}\int_{s_n}^{s_{n+1}}\om(t)^{\frac{-p'_0}{p_0}}\,dt\\
    &\le C^{\frac{p'_0}{p_0}}\sum_{n=0}^\infty (1-s_n)^{p'_0\eta+1}\om(s_{n})^{\frac{-p'_0}{p_0}}\\
    &\le|I|^{p'_0\eta+1}\om(1-|I|)^{\frac{-p'_0}{p_0}}\\
    &\quad\cdot\sum_{n=0}^\infty (1-s)^{n(p'_0\eta+1)}C^{(n+1)\frac{p'_0}{p_0}}\\
    &=C(p_0,\eta,s,\omega)|I|^{p'_0\eta+1}\om(1-|I|)^{\frac{-p'_0}{p_0}},
    \end{split}
    \end{equation*}
which together with \eqref{eq:r1} gives \eqref{eq:rb1}.

(ii) The asymptotic inequality $\p_\om(r)\lesssim(1-r)$ follows by
\eqref{eq:rb1} and further appropriately modifying the argument in
the proof of (i). Since the assumption \eqref{eq:r2} gives
$\p_\om(r)\gtrsim(1-r)$, we deduce $\om\in\R$.

(iii) If $0\le r<t<1$, then an integration by parts yields
    \begin{equation*}
    \begin{split}
    \int_r^t\frac{\om(s)}{\left(\int_{s}^1\om(v)\,dv\right)^{\a}}\,ds
    &=\left(\int_{r}^1\om(v)\,dv\right)^{1-\a}-\left(\int_{t}^1\om(v)\,dv\right)^{1-\a}\\
    &\quad+\a\int_r^t
    \frac{\om(s)}{\left(\int_{s}^1\om(v)\,dv\right)^{\a}}\,ds,
    \end{split}
    \end{equation*}
from which the assertion follows by letting $t\to1^-$.
\end{proof}

We note that a routine calculation based on \eqref{eq:rb1} shows
that $\frac{v_\a(z)}{(1-|z|)^\eta}$,\index{$v_\a(r)$}
$1<\a<\infty$, is not a $B_{p_0}(\eta)$ weight for any $p_0>1$ and
$\eta>-1$. Later on we will see that this is actually true for
each $\om\in\I$.\index{Bekoll\'e-Bonami
weight}\index{$B_{p_0}(\eta)$}

The next lemma shows that we can find a radial weight $\om$ such
that the growth of the quotient $\psi_\om(r)/(1-r)$ is of any
given admissible rate. Here one should observe that the
assumptions on the auxiliary function $h$ are not restrictions
because they are necessary conditions for \eqref{57} to hold, as
$\om_\lambda$ is a weight and thus satisfies~\eqref{63}.

\begin{lemma}\label{Lemma:Distortion}
Let $\om$ be a radial weight. Then
    \begin{equation}\label{63}
    \int_0^1\frac{dr}{\psi_\om(r)}=\infty.
    \end{equation}
Moreover, for each function $\lambda:[0,1)\to(0,\infty)$ such that
    $$
    h(r)=\int_0^r\frac{ds}{\lambda(s)(1-s)}
    $$
exists for all $0<r<1$ and $\lim_{r\to1^-}h(r)=\infty$, there
exists a radial weight $\om_\lambda$ such that
    \begin{equation}\label{57}
    \frac{\psi_{\om_\lambda}(r)}{(1-r)}=\lambda(r),\quad 0<r<1.
    \end{equation}
\end{lemma}

\begin{proof}
Clearly,
    \begin{equation*}
    \begin{split}
    \int_0^1\frac{dr}{\psi_\om(r)}&=\int_0^1\frac{\om(r)}{\int_{r}^1\om(s)\,ds}\,dr\\
    &=\lim_{r\to1^-}\log\frac{1}{\int_r^1\om(s)\,ds}
    -\log\frac{1}{\int_0^1\om(s)\,ds}=\infty
    \end{split}
    \end{equation*}
for all radial weights $\om$. Moreover, if $\lambda$ satisfies the
hypothesis, then each radial weight
    $$
    \omega_\lambda(r)=C\frac{\exp\left(-\int_0^r\frac{ds}{\lambda(s)(1-s)}\right)}{\lambda(r)(1-r)},\quad
    C>0,
    $$
solves the integral equation~\eqref{57}.
\end{proof}

Another way to interpret the last assertion in
Lemma~\ref{Lemma:Distortion} is to notice that
    $$
    \om(r)=\left(\int_0^1\om(s)\,ds\right)\cdot\frac{\exp\left(-\int_0^r\frac{ds}{\psi_\om(s)}\right)}{\psi_\om(r)},
    $$
provided $\om$ is a radial weight.

For each radial weight $\om$, we define its \emph{associated
weight}\index{associated weight}
$\omega^\star$\index{$\omega^\star(z)$} by
    $$
    \omega^\star(z)=\int_{|z|}^1\omega(s)\log\frac{s}{|z|}s\,ds,\quad z\in\D\setminus\{0\}.
    $$
Even though $\om\in\I$ might have a bad oscillatory behavior, as
we saw in Section~\ref{Sec:DefIUR}, its associated weight
$\om^\star$ is regular by Lemma~\ref{le:sc1} below. We will also
consider the non-tangential regions\index{non-tangential region}
    \begin{equation}\label{eq:gammadeu}\index{$\Gamma(u)$}
    \Gamma(u)=\left\{z\in \D:\,|\t-\arg
    z|<\frac12\left(1-\frac{|z|}{r}\right)\right\},\quad
    u=re^{i\theta}\in\overline{\D}\setminus\{0\},
    \end{equation}
and the tents\index{tent}
    \begin{equation}\label{eq:Tdez}
    \begin{split}
    T(z)&=\left\{u\in\D:\,z\in\Gamma(u)\right\},\quad
    z\in\D,\index{$T(z)$}
    \end{split}
    \end{equation}
which are closely interrelated. Figure~\ref{fig:2} illustrates how
the sets $\Gamma(z)$ and $T(z)$ change when $z$ varies.

\begin{figure}[htp]
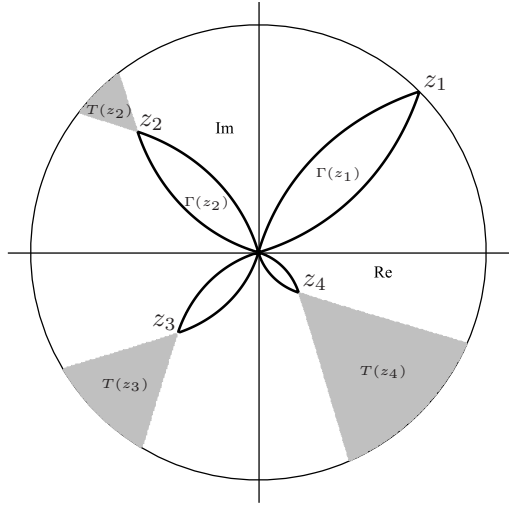

\begin{center}
  \begin{lpic}[]{j3(0.35,0.35)}
    \lbl[t]{165,165;$z_1$}
    \lbl[t]{57,150;$z_2$}
    \lbl[t]{62,74;$z_3$}
    \lbl[t]{119,89;$z_4$}
   \lbl[t]{128,130;${\scriptscriptstyle \Gamma(z_1)}$}
    \lbl[t]{78,119;${\scriptscriptstyle \Gamma(z_2)}$}
    \lbl[t]{41,154;${\scriptscriptstyle T(z_2)}$}
    \lbl[t]{47,50;${\scriptscriptstyle T(z_3)}$}
    \lbl[t]{145,55;${\scriptscriptstyle T(z_4)}$}
  \end{lpic}
  \caption{Let $z_j=(1-\frac{j-1}{4})e^{i(\frac\pi4+(j-1)\frac\pi2)}$,
$j=1,2,3,4$. Then the point $z_j$ as well as the sets
$\Gamma(z_j)$ and $T(z_j)$ belong to the $j$:th quadrant of $\D$.}
  \label{fig:2}
\end{center}
\end{figure}

Lemma~\ref{le:cuadrado-tienda} shows that
$\omega(S(z))\asymp\omega(T(z))\asymp\omega^\star(z)$, as
$|z|\to1^-$, provided $\om\in\I\cup\R$. Here, as usual,
$\om(E)=\int_E\om(z)\,dA(z)$\index{$\om(E)$} for each measurable
set $E\subset\D$.

\begin{lemma}\label{le:cuadrado-tienda}\index{$\omega^\star(z)$}
\begin{itemize}
\item[{\rm(i)}] If $\om$ is a radial weight, then
    \begin{equation}\label{2}
    \om(T(z))\asymp\omega^\star(z),\quad |z|\ge\frac12.
    \end{equation}
\item[{\rm(ii)}] If $\om\in\I\cup\R$, then
    \begin{equation}\label{3}
    \omega\left(T(z)\right)\asymp\omega\left(S(z)\right),\quad
    z\in\D.
    \end{equation}
\end{itemize}
\end{lemma}

\begin{proof}
The first assertion \eqref{2} is clear by the definitions of the
region~$T(z)$ and the associated weight $\om^\star$. Moreover, a
routine application of Bernouilli-l'H\^{o}pital theorem,
exploiting the relation $\psi_{\om}(r)\asymp (1-r)$ for $\om\in\R$
and the definition \eqref{eq:I} for $\om\in\I$, shows \eqref{3}.
\end{proof}

For $\alpha\in\mathbb{R}$ and a radial weight $\omega$, set
$\omega_\alpha(z)=(1-|z|)^\alpha\om(z)$\index{$\omega_\alpha(z)$}.
We will see next that $\om^\star_\a(z)=(1-|z|)^\a\om^\star(z)$ is
regular whenever $\a>-2$.

\begin{lemma}\label{le:sc1}\index{$\omega^\star(z)$}
If $0<\alpha<\infty$ and $\om\in\I\cup\R$, then
$\omega^\star_{\alpha-2}\in\R$ and
$\left(\omega^\star_{\alpha-2}\right)^\star(z)\asymp\omega^\star_{\alpha}(z)$
for all $|z|\ge\frac12$.
\end{lemma}

\begin{proof}
It suffices to show that $\omega_{\a-2}^\star\in\R$, since for
each $\omega\in\R$ we have
    \begin{equation}\label{22}
    \omega^\star(r)\asymp\omega(S(r))=\frac{(1-r)}{\pi}\int_r^1\omega(s)s\,ds\asymp(1-r)^2\omega(r),\quad
    r\ge\frac12,
    \end{equation}
by Lemma~\ref{le:cuadrado-tienda} and $\psi_\omega(r)\asymp(1-r)$.

To prove $\omega_{\a-2}^\star\in\R$, note first that the relation
$\omega^\star(r)\asymp(1-r)\int_r^1\omega(s)s\,ds$ and
Lemma~\ref{le:condinte} with standard arguments give \eqref{eq:r2}
for $\omega_{\alpha-2}^\star$. Further, since $\alpha>0$, we
deduce
    \begin{equation*}\begin{split}
    \frac{\int_r^1\omega^\star_{\alpha-2}(s)\,ds}{\omega^\star_{\alpha-2}(r)}&
    \asymp \frac{\int_r^1(1-s)^{\alpha-1}\left(\int_s^1\om(t)t\,dt\right)ds}{(1-r)^{\alpha-1}\int_r^1\om(t)t\,dt}\le \frac{\int_r^1
    (1-s)^{\alpha-1}\,ds}{(1-r)^{\alpha-1}}=\frac{(1-r)}{\alpha},
    \end{split}\end{equation*}
and so $\omega^\star_{\alpha-2}$ satisfies also \eqref{eq:r1}.
Thus $\omega_{\a-2}^\star\in\R$.
\end{proof}

The last lemma is related to invariant weights.\index{${\mathcal
Inv}$}\index{invariant weight}

\begin{lemma}\label{Lemma:InvariantWeights}
If $\om\in{\mathcal Inv}$, then there exists a function
$C:\D\to[1,\infty)$ such that
    \begin{equation}\label{Eq:InvariantWeightsOldDefinition}
    \om(u)\le C(z)\om(\vp_u(z)),\quad u,z\in\D,
    \end{equation}
and
    \begin{equation}\label{Eq:InvariantWeightsOldDefinitionIntegral}
    \int_\D\log C(z)\,dA(z)<\infty.
    \end{equation}

Conversely, if $\om$ is a weight satisfying
\eqref{Eq:InvariantWeightsOldDefinition} and the function $C$ is
uniformly bounded in compact subsets of $\D$, then
$\om\in{\mathcal Inv}$.
\end{lemma}

\begin{proof}
Let first $\om\in{\mathcal Inv}$. Then there exists a constant
$C\ge1$ such that
    \begin{equation}\label{111}
    C^{-1}\om(a)\le\om(z)\le C\om(a)
    \end{equation}
for all $z$ in the hyperbolic disc $\Delta_h(a,1)$. For each
$z,u\in\D$, the hyperbolic distance between $u$ and $\vp_u(z)$ is
    $$
    \varrho_h(u,\vp_u(z))=\frac12\log\frac{1+|z|}{1-|z|}.
    $$
By the additivity of the hyperbolic distance on the geodesic
joining $u$ and $\vp_u(z)$, and \eqref{111} we deduce
    $$
    \om(u)\le C^{E(\varrho_h(u,\vp_u(z)))+1}\om(\vp_u(z))\le C\left(\frac{1+|z|}{1-|z|}\right)^{\frac{\log
    C}{2}}\om(\vp_u(z)),
    $$
where $E(x)$ is the integer such that $E(x)\le x<E(x)+1$. It
follows that \eqref{Eq:InvariantWeightsOldDefinition} and
\eqref{Eq:InvariantWeightsOldDefinitionIntegral} are satisfied.

Conversely, let $\om$ be a weight satisfying
\eqref{Eq:InvariantWeightsOldDefinition} such that the function
$C$ is uniformly bounded in compact subsets of $\D$. Then, for
each $r\in(0,1)$, there exists a constant $C=C(r)>0$ such that
$\om(u)\le C(r)\om(z)$ whenever $|\vp_u(z)|<r$. Thus
$\om\in{\mathcal Inv}$.
\end{proof}

Recently, Aleman and Constantin~\cite{AlCo} studied the weighted
Bergman spaces $A^p_\om$ induced by the differentiable weights
$\om$ that satisfy
    \begin{equation}\label{Pesos:AlemanConstantin}
    |\nabla\om(z)|(1-|z|^2)\le C\om(z),\quad z\in\D,
    \end{equation}
for some $C=C(\om)>0$. By using \cite[Proposition~3.1]{AlCo} and
Lemma~\ref{Lemma:InvariantWeights}, we see that every such weight
is invariant. Moreover, $v_\a$\index{$v_\a(r)$} gives an example
of an invariant rapidly increasing weight that does not satisfy
\eqref{Pesos:AlemanConstantin}.

\section{Density of polynomials in $A^p_\om$}\label{Sec:DensityOfPolynomials}

It is known that the polynomials are not necessarily dense in the
weighted Bergman space $A^p_\om$ induced by a weight
$\om:\D\to(0,\infty)$, see \cite[p.~120]{Hedberg},
\cite[p.~134]{Mergelyan} or \cite[p.~138]{DurSchus}). Questions
related to polynomial approximation have attracted (and still do)
a considerable amount of attention during the last decades, and as
far as we know, the problem of characterizing the weights $\om$
for which the polynomials are dense in $A^p_\om$ remains unsolved.
Our interest in this problem comes from the study of factorization
of functions in~$A^p_\om$.\index{factorization} The main result of
this monograph in that direction is
Theorem~\ref{Thm:FactorizationBergman} in which the density of the
polynomials in $A^p_\om$ is taken as an assumption. This because
the proof relies on the existence of a dense set of~$A^p_\om$
whose elements have finitely many zeros only. In this section we
discuss known results on approximation by the polynomials in
$A^p_\om$ and their relations to the classes of weights considered
in this monograph.

We begin with recalling that the polynomials are dense in
$A^p_\om$, whenever $\om$ is a radial weight,
see~\cite[p.~134]{Mergelyan} or \cite[p.~138]{DurSchus}. This is
due to the fact that in this case the dilatation $f_r(z)=f(rz)$ of
$f$ satisfies\index{dilatation}\index{$f_r$}
    \begin{equation}\label{densdilat}
    \lim_{r\to 1^-}\|f_r-f\|_{A^p_\om}=0,\quad f\in A^p_\om.
    \end{equation}
Another neat condition which ensures the approximation in
$A^p_\om$ by dilatations, and therefore by polynomials, is
    \begin{equation}\label{conditionrz}
    \om(z)\le C(\om)\om(rz),\quad r_0\le r<1,\quad r_0\in (0,1),
    \end{equation}
see~\cite[p.~134]{Mergelyan}.

Recall that $M\subset A^p$ is called
\emph{invariant}\index{invariant subspace} if $zM\subset M$, and
that $f\in A^p$ is a \emph{cyclic element}\index{cyclic element}
if the closure of the set of polynomial multiples of $f$ is the
whole space~$A^p$. If now $Q$ is a cyclic element of~$A^p$ and
$\om$ is defined by $\om(z)=|Q(z)|^p$, then the polynomials are
dense in~$A^p_\om$ by the cyclicity. A natural way to construct
cyclic elements of $A^p$ is to choose a function $\phi\in
L^q(\T)$, where $q\ge\max\{p/2,1\}$, such that $\log\phi\in
L^1(\T)$, and consider the associated $H^q$-outer function
   \begin{equation}\label{outer}
   Q(z)=\exp\left(\frac1{2\pi}\int_{\T}\,\frac{\z+z}{\z-z}\log\phi(\z)\,d\z\right).
   \end{equation}
Then Beurling's theorem~\cite{BeurlingActa} ensures that $Q$ is
cyclic in $H^q$, and hence in $A^p$ as well. However, despite of
the Hardy space case there are singular inner functions (in the
classical sense) that are cyclic in $A^p$. For a detailed
discussion involving Beurling's theorem and cyclic elements in
$A^p$ the reader is invited to see \cite[p.~245--260]{DurSchus}
and the references therein.

The results of Hedberg~\cite{Hedberg} can also be used to
construct explicit examples of non-radial weights $\om$ such that
the polynomials are dense in $A^p_\om$. We show two direct
applications that are stated as lemmas. As usual, an analytic
function $f$ is called \emph{univalent}\index{univalent function}
if it is injective.

\begin{lemma}\label{le:univalentweights}
Let $f$ be a non-vanishing univalent function in $\D$,
$0<\gamma<1$ and $\om=|f|^\gamma$. Then the polynomials are dense
in $A^p_\om$ for all $p\ge1$.
\end{lemma}

\begin{proof}
Since $f$ is univalent and zero-free, so is $1/f$, and hence both
$f$ and $1/f$ belong to $A^p$ for all $0<p<1$. By choosing
$\delta>0$ such that $\gamma(1+\delta)<1$ we deduce that both
$\om$ and $\frac{1}{\om}$ belong to $L^{1+\delta}$. Therefore the
polynomials are dense in~$A^p_\om$ by \cite[Theorem~2]{Hedberg}.
\end{proof}

\begin{lemma}\label{le:univalentweights2}
Let $f$ be a univalent function in $\D$ such that $f(0)=0$, and
$0<\gamma<\infty$. Let the weight $\om$ be defined by
$\om(z)=\left|\frac{z}{f(z)}\right|^\gamma$ for all $z\in\D$. Then
the polynomials are dense in $A^p_\om$ for all $p\ge1$.
\end{lemma}

\begin{proof}
The function $\int_0^{2\pi}|f(re^{it})|^{-p}\,dt$ is a decreasing
function of $r$ on $(0,1)$ for all $0<p<\infty$ by
\cite[Theorem~4.3.1]{Pabook}. It follows that $\om\in L^{\alpha}$
for all $0<\alpha<\infty$. Moreover, $f(z)/z$ is univalent, and so
$\frac{1}{\om}\in L^{\alpha}$ for all $\alpha<1/\gamma$. Therefore
the polynomials are dense in~$A^p_\om$ by
\cite[Theorem~2]{Hedberg}.
\end{proof}

Motivated by factorization of functions in~$A^p_\om$, we are
mainly interested in setting up invariant weights $\om$ such that
the polynomial approximation is possible in~$A^p_\om$. Obviously,
the radial weights $\om$ in this class are precisely those that
satisfy \eqref{eq:r2}. Further, it is easy to see that, for each
$0<\a<2$ and $\xi\in\T$, the non-radial invariant weight
$\om_\xi(z)=\left|\frac{\xi-z}{\xi+z}\right|^\a$ has also this
property by Lemma~\ref{le:univalentweights}. More generally, let
$G$ be non-vanishing such that $\log G$ is a Bloch function. It is
well-known that then there exists a constant $C>0$ such that
$|\log G(z)-\log G(w)|\le C\varrho_h(z,w)$ for all $z,w\in\D$. It
follows that the weight $\om$ defined by $\om(z)=|G(z)|^\gamma$ is
invariant for all $\gamma\in\mathbb{R}$. In particular, if $G$ is
a non-vanishing univalent function and $0<\gamma<1$, then $\log G$
is a Bloch function by the Koebe $\frac14$-theorem~\cite{Hayman},
and hence $\om(z)=|G(z)|^\gamma$ is invariant and further, the
polynomials are dense in $A^p_\om$ by
Lemma~\ref{le:univalentweights}.

Finally, let us consider the class of weights that appears in a
paper by Abkar~\cite{Abkar1}. We will see that also these weights
are invariant. To do this, we need to introduce the following
concepts. A function $u$ defined on $\D$ is said to be
\emph{superbiharmonic}\index{superbiharmonic} if $\triangle^2u\ge
0$, where $\triangle$ stands for the Laplace
operator\index{Laplace operator}\index{$\triangle_z$}
    $$
    \triangle=\triangle_z=\frac{\partial^2}{\partial z\partial \overline{z}}=\frac{1}{4} \left(\frac{\partial^2}{\partial^2 x}
         +\frac{\partial^2}{\partial^2 y}
         \right)
    $$
in the complex plane $\mathbb{C}$. The superbihamonic weights play
an essential role in the study of invariant subspaces of the
Bergman space $A^p$. In particular,~\cite[Theorem~2.6]{Abkar1},
that is sated as Theorem~\ref{th:Abkar} below, is a key ingredient
in the proof of the fact that any invariant subspace $M$ of $A^2$,
that is induced by the zero-set of a function in $A^2$, is
generated by its extremal function $\varphi_M$, see
\cite[Theorem~3.1]{Abkar1}~and~\cite{AleRichSundActa}.

\begin{lettertheorem}\label{th:Abkar}
Let $\om$ be a superbiharmonic weight such that
    \begin{equation}\label{Abkarii}
    \lim_{r\to 1^-}\int_{\T}\om(r\zeta)\,dm(\zeta)=0.
    \end{equation}
Then \eqref{densdilat} is satisfied for all $f\in A^p_\om$. In
particular, the polynomials are dense in~$A^p_\om$.
\end{lettertheorem}

The proof of Theorem~\ref{th:Abkar} relies on showing that these
type of weights $\om$ satisfy \eqref{conditionrz}. For a similar
result, see \cite[Theorem~3.5]{Abkar2}. Next,
let\index{$\Gamma(u,\zeta)$}
    \begin{equation}\label{eq:biGreen}
    \Gamma(u,\zeta)=|u-\zeta|^2\log\left|\frac{u-\zeta}{1-u\overline{\zeta}}\right|+(1-|u|^2)(1-|z|^2),\quad
    (u,\zeta)\in \D\times\D,
    \end{equation}
be the \emph{biharmonic Green function}\index{biharmonic Green
function} for the operator $\Delta^2$ in $\D$, and
let\index{$H(u,\zeta)$}
    \begin{equation}\label{eq:harcom}
    H(u,\zeta)=\frac{(1-|\zeta|^2)^2}{|1-\overline{\zeta}u|^2},\quad
    (u,\zeta)\in \T\times\D,
    \end{equation}
be the \emph{harmonic compensator}.\index{harmonic compensator}

\begin{lemma}
Every superbihamonic weight that satisfies \eqref{Abkarii} is
invariant.
\end{lemma}

\begin{proof}
The proof relies on the representation formula obtained in
\cite[Corollary~3.7]{AbkHed}. It states that every superbihamonic
weight $\om$ that satisfies \eqref{Abkarii} can be represented in
the form
    \begin{equation*}
    \om(\zeta)=\int_{\D}\Gamma(u,\zeta)\triangle^2(u)\,dA(u)+\int_\T
    H(u,\zeta)\,d\mu(u),\quad \zeta\in\D,
    \end{equation*}
where $\mu$ is a uniquely determined positive Borel measure on
$\T$. Therefore it suffices to prove the relations
    \begin{equation}
    \begin{split}\label{eq:equivG}
    \Gamma(u,\zeta)\asymp \Gamma(u,z), \quad
    z\in\Delta(\zeta,r),\quad u\in\D,
    \end{split}
    \end{equation}
and
    \begin{equation}\begin{split}\label{eq:equivH}
    H(u,\zeta)\asymp H(u,z),\quad z\in\Delta(\zeta,r),\quad u\in\T,
    \end{split}\end{equation}
where the constants of comparison may depend only on $r$. The
definition \eqref{eq:harcom} and standard estimates yield
\eqref{eq:equivH}. Moreover, by \cite[Lemma~2.2(d)]{Abkar1} or
\cite[Lemma~2.3(a)]{AleRichSundActa}, we have
    $$
    \frac{1}{2}\frac{(1-|\zeta|^2)^2(1-|u|^2)^2}{|1-\overline{\zeta}u|^2}\le
    \Gamma(u,\zeta)\le\frac{(1-|\zeta|^2)^2(1-|u|^2)^2}{|1-\overline{\zeta}u|^2},\quad
    (u,\z)\in\D\times\D,
    $$
which easily gives \eqref{eq:equivG}, and finishes the proof.
\end{proof}

\chapter{Description of $q$-Carleson Measures for
$A^p_\om$}\label{Sec:Carleson}\index{Carleson measure}

In this chapter we give a complete description of $q$-Carleson
measures for $A^p_\om$ when $0<p\le q<\infty$ and
$\om\in\I\cup\R$. For a given Banach space (or a complete metric
space) $X$ of analytic functions on $\D$, a positive Borel measure
$\mu$ on $\D$ is called a \emph{$q$-Carleson measure for $X$}
\index{Carleson measure} if the identity operator\index{identity
operator} $I_d:\, X\to L^q(\mu)$\index{$I_d$} is bounded. It is
known that a characterization of $q$-Carleson measures for a space
$X\subset\H(\D)$ can be an effective tool, for example, in the
study of several questions related to different operators acting
on $X$. In this study we are particularly interested in the
integral operator\index{integral operator}
    \begin{displaymath}
    T_g(f)(z)=\int_{0}^{z}f(\zeta)\,g'(\zeta)\,d\zeta,\quad
    g\in\H(\D). \index{$T_g(f)$}
    \end{displaymath}

We will explain how the decay or growth of a radial weight $\om$
affects to the study of $q$-Carleson measures for the weighted
Bergman space $A^p_{\om}$. The influence of $\omega$ appears
naturally in the characterizations as well as in the techniques
used to obtain them. We will draw special attention to rapidly
increasing radial weights, because, to the best of our knowledge,
this situation is not well understood yet.

Hastings, Luecking, Oleinik and Pavlov~\cite{Ha,Luf,Lu93,OP},
among others, have characterized $q$-Carleson
measures\index{Carleson measure} for $A^p_\alpha$. In a recent
paper Constantin~\cite{OC} gives an extension of these classical
results to the case when $\frac{\omega(z)}{(1-|z|)^\eta}$ belongs
to the class $B_{p_0}(\eta)$ of Bekoll\'e-Bonami
weights.\index{Bekoll\'e-Bonami weight}\index{$B_{p_0}(\eta)$}
Recall that these weights are not necessarily radial. A
characterization of $q$-Carleson measures for $A^p_{\om}$, $\om\in
\R$, follows by \cite[Theorems~3.1 and~3.2]{OC} because every
regular weight $\om$ satisfies $\frac{\om(z)}{(1-|z|)^\eta}\in
B_{p_0}(\eta)$\index{Bekoll\'e-Bonami
weight}\index{$B_{p_0}(\eta)$} for some $p_0>1$ and $\eta>-1$ by
Lemma~\ref{le:RAp}(i). If $\om$ is a rapidly decreasing radial
weight\index{rapidly decreasing weight}, such as those in
\eqref{Eq:ExponentialWeights}, then different techniques from
those that work for $\om\in\R$ must be used to describe
$q$-Carleson measures for $A^p_{\om}$~\cite{PP}.

If $\om\in\I$, then the weight $\frac{\om(z)}{(1-|z|)^\eta}$ is
not a Bekoll\'e-Bonami weight \index{Bekoll\'e-Bonami weight} for
any $p_0>1$ and $\eta>-1$. This is a consequence of
Theorem~\ref{Thm-integration-operator-1},
Proposition~\ref{pr:blochcpp}(D) and \cite[Theorem~4.1]{AlCo}.
Therefore the above-mentioned results by Constantin~\cite{OC} do
not yield a characterization of $q$-Carleson measures for
$A^p_\om$ when $\omega\in\I$. As far as we know these measures
have not been characterized yet.

From now on and throughout we will write
$\|T\|_{(X,Y)}$\index{$\Vert\cdot\Vert_{(X,Y)}$} for the norm of
an operator $T:X\to Y$, and if no confusion arises with regards to
$X$ and $Y$, we will simply write $\|T\|$.\index{$\|T\|$}

The following theorem is the main result of this chapter.

\begin{theorem}\label{th:cm}\index{Carleson measure}
Let $0<p\le q<\infty$ and $\om\in\I\cup\R$, and let $\mu$ be a
positive Borel measure on $\D$. Then the following assertions
hold:
\begin{itemize}
\item[\rm(i)] $\mu$ is a $q$-Carleson measure for $A^p_\omega$ if
and only if
    \begin{equation}\label{eq:s1}
    \sup_{I\subset\T}\frac{\mu\left(S(I) \right)}{\left(\om\left(S(I)
    \right)\right)^\frac{q}p}<\infty.
    \end{equation}
Moreover, if $\mu$ is a $q$-Carleson measure for $A^p_\omega$,
then the identity operator $I_d:A^p_{\om}\to L^q(\mu)$ satisfies
    $$
    \|I_d\|^q_{\left(A^p_\om, L^q(\mu)\right)}\asymp\sup_{I\subset\T}\frac{\mu\left(S(I) \right)}{\left(\om\left(S(I)
    \right)\right)^\frac{q}p}.
    $$

\item[\rm(ii)] The identity operator $I_d:A^p_{\om}\to L^q(\mu)$
is compact if and only if
\begin{equation}\label{eq:s1compact}
    \lim_{|I|\to 0}\frac{\mu\left(S(I) \right)}{\left(\om\left(S(I)
    \right)\right)^\frac{q}p}=0.
    \end{equation}
    \end{itemize}
\end{theorem}

Recall that $I_d:A^p_{\om}\to L^q(\mu)$ is compact if it
transforms bounded sets of $A^p_\om$ to relatively compact sets of
$L^q(\mu)$. Equivalently, $I_d$ is compact if and only if for
every bounded sequence $\{f_n\}$ in $A^p_\om$ some subsequence of
$\{f_n\}$ converges in $L^q(\mu)$.

One important feature of the proof of Theorem~\ref{th:cm} is that
it relies completely on Carleson squares and do not make any use
of techniques related to (pseudo)hyperbolic discs as it is
habitual when classical weighted Bergman spaces are considered.
Therefore the reasoning gives an alternative proof of the known
characterization of $q$-Carleson measures for the classical
weighted Bergman space~$A^p_\a$ \cite{DurSchus,Zhu}. Indeed, it is
known that, for any $\om\in\R$ and $\beta\in(0,1)$, there exists a
constant $C=C(\beta,\om)>0$ such that
    $$
    C^{-1}\om\left(D(a,\beta(1-|a|))\right)\le\om\left(S(a)\right)\le C\om\left(D(a,\beta(1-|a|))\right)
    $$
for all $a\in \D$, and so $q$-Carleson measures for $\om\in\R$ can
be characterized either in terms of Carleson squares or
pseudohyperbolic discs. However, this is no longer true when
$\om\in\I$. The proof of Theorem~\ref{th:cm} relies heavily on the
properties of the maximal function
    $$\index{weighted maximal function}\index{$M_\om(\vp)$}
    M_{\om}(\vp)(z)=\sup_{I:\,z\in S(I)}\frac{1}{\om\left(S(I)
    \right)}\int_{S(I)}|\vp(\xi)|\om(\xi)\,dA(\xi),\quad
    z\in\D,
    $$
introduced by H\"ormander~\cite{HormanderL67}. Here we must
require $\vp\in L^1_\om$ and that $\vp(re^{i\t})$ is
$2\pi$-periodic with respect to $\t$ for all $r\in(0,1)$. It
appears that $M_{\om}(\vp)$\index{weighted maximal
function}\index{$M_\om(\vp)$} plays a similar role on $A^p_\om$
than the Hardy-Littlewood maximal function\index{Hardy-Littlewood
maximal function} on the Hardy space $H^p$. As a by-product of the
proof of Theorem~\ref{th:cm}, we obtain the following result which
is of independent interest.

\begin{corollary}\label{co:maxbou}\index{weighted maximal function}\index{$M_\om(\vp)$}
Let $0<p\le q<\infty$ and $0<\alpha<\infty$ such that $p\alpha>1$.
Let $\om\in\I\cup\R$, and let $\mu$ be a positive Borel measure on
$\D$. Then
$[M_{\om}((\cdot)^{\frac{1}{\alpha}})]^{\alpha}:L^p_\omega\to
L^q(\mu)$ is bounded if and only if $\mu$ satisfies \eqref{eq:s1}.
Moreover,
    $$
    \|[M_{\om}((\cdot)^{\frac{1}{\alpha}})]^{\alpha}\|^q_{\left(L^p_\om,L^q(\mu)\right)}\asymp\sup_{I\subset\T}\frac{\mu\left(S(I) \right)}{\left(\om\left(S(I)
    \right)\right)^\frac{q}p}.
    $$
\end{corollary}

The proof of Theorem~\ref{th:cm} is presented in the next two
sections. In Section~\ref{SubSection:counterexample} we give a
counterexample that shows that Theorem~\ref{th:cm}(i) does not
remain valid for $\om\in\widetilde{\I}$ if Carleson squares are
replaced by pseudohyperbolic discs.

\section{Weighted maximal function $M_\om$}\index{weighted maximal
function}\index{$M_\om(\vp)$}

We begin with constructing a family of test functions that allows
us to show the necessity of conditions \eqref{eq:s1} and
\eqref{eq:s1compact}. To do so, we will need the following lemma.

\begin{lemma}\label{Lemma:Zhu-type}
\begin{itemize}
\item[\rm(i)] If $\om\in\R$, then there exists
$\gamma_0=\gamma_0(\om)$ such that
    \begin{equation}\label{84}
    \int_\D\frac{\om(z)}{|1-\overline{a}z|^{\gamma+1}}\,dA(z)\asymp\frac{\int_{|a|}^1\om(r)\,dr}{(1-|a|)^{\gamma}}\asymp\frac{\om(a)}{(1-|a|)^{\gamma-1}},\quad
    a\in\D,
    \end{equation}
for all $\gamma>\gamma_0$. \item[\rm(ii)] If $\om\in\I$, then
    \begin{equation}
    \int_\D\frac{\om(z)}{|1-\overline{a}z|^{\gamma+1}}\,dA(z)\asymp\frac{\int_{|a|}^1\om(r)\,dr}{(1-|a|)^{\gamma}},\quad
    a\in\D,
    \end{equation}
for all $\gamma>0$.
\end{itemize}
\end{lemma}

\begin{proof} (i) Let $\om\in\R$ and let $\gamma_0=\b(\omega)>0$ be the constant in
Lemma~\ref{le:condinte}(i). Then
    \begin{equation*}
    \begin{split}\label{eq:tf3}
    \int_\D\frac{\om(z)}{|1-\overline{a}z|^{\gamma+1}}\,dA(z)&\asymp\int_0^1\frac{\om(r)r\,dr}{\left(1-|a|r\right)^{\gamma}}
    \gtrsim\frac{\om(a)\psi_\om(|a|)}{\left(1-|a|\right)^{\gamma}}\asymp\frac{\om(a)}{(1-|a|)^{\gamma-1}}
    \end{split}
    \end{equation*}
and
    \begin{equation*}
    \begin{split}
    \int_\D\frac{\om(z)}{|1-\overline{a}z|^{\gamma+1}}\,dA(z)
    &\asymp
    \left(\int_0^{|a|}+\int_{|a|}^1\right)\frac{\om(r)r\,dr}{\left(1-|a|r\right)^{\gamma}}\lesssim\frac{\om(a)}{(1-|a|)^{\gamma-1}}
    \end{split}
    \end{equation*}
by Lemma~\ref{le:nec1}. These inequalities give~\eqref{84}.
Part~(ii) can be proved in a similar manner.
\end{proof}

Let now $\gamma=\gamma(\omega)>0$ be the constant in
Lemma~\ref{Lemma:Zhu-type}, and define
$F_{a,p}(z)=\left(\frac{1-|a|^2}{1-\overline{a}z}\right)^{\frac{\gamma+1}{p}}$.\index{$F_{a,p}$}
Lemma~\ref{Lemma:Zhu-type} immediately yields the following
result.

\begin{lemma}\label{testfunctions1}
If $\omega\in\I\cup\R$, then for each $a\in \D$ and $0<p<\infty$
there exists a function $F_{a,p}\in\H(\D)$ such that
    \begin{equation}\label{eq:tf1}
    |F_{a,p}(z)|\asymp 1,\quad z\in S(a),\index{$F_{a,p}$}
    \end{equation}
and
    \begin{equation}\label{eq:tf2}
    \|F_{a,p}\|_{A^p_\om}^p\asymp\om\left(S(a)\right).
    \end{equation}
\end{lemma}

To prove the sufficiency of conditions \eqref{eq:s1} and
\eqref{eq:s1compact} in Theorem~\ref{th:cm} ideas
from~\cite{CarlesonL58,Duren1970,Garnett1981,HormanderL67} are
used. The following lemma plays an important role in the proof,
but it is also of independent interest. It will be also used in
Chapter~\ref{Sec:Schatten} when the norms of reproducing kernels
in the Hilbert space $A^2_\omega$ are estimated, see
Lemma~\ref{kernels}.

\begin{lemma}\label{le:suf1}
Let $0<\alpha<\infty$ and $\om\in\I\cup\R$. Then there exists a
constant $C=C(\a,\omega)>0$ such that
    \begin{equation}\label{eq:s3}\index{weighted maximal function}\index{$M_\om(\vp)$}
    |f(z)|^\alpha\le CM_{\om}(|f|^\alpha)(z),\quad z\in\D,
    \end{equation}
for all $f\in\H(\D)$.
\end{lemma}

\begin{proof} Let $0<\alpha<\infty$ and $f\in\H(\D)$. Let
first $\om\in\R$. Let $a\in\D$, and assume without loss of
generality that $|a|>1/2$. Set $a^\star=\frac{3|a|-1}{2}e^{i\arg
a}$ so that $D(a,\frac12(1-|a|))\subset S(a^\star)$. This
inclusion together with the subharmonicity property of $|f|^\a$,
\eqref{eq:r1} and \eqref{eq:r2} give
    \begin{equation}\label{le:suf11}\index{weighted maximal
function}\index{$M_\om(\vp)$}
    \begin{split}
    |f(a)|^\alpha &\lesssim\frac{1}{\om(a)(1-|a|)^2}\int_{D(a,\frac12(1-|a|))}|f(z)|^\alpha\om(z)\,dA(z)\nonumber\\
    &\lesssim\frac{1}{\om(a^\star)(1-|a^\star|)^2}\int_{S(a^\star)}|f(z)|^\alpha\om(z)\,dA(z)\nonumber\\
    &\asymp\frac{1}{\om(S(a^\star))}\int_{S(a^\star)}|f(z)|^\alpha\om(z)\,dA(z)\le
    M_\omega(|f|^\alpha)(a),
    \end{split}
    \end{equation} which is the desired inequality for
$\omega\in\R$.

Let now $\om\in\I$. It again suffices to prove the assertion for
the points $re^{i\t}\in\D$ with $r>\frac12$. If $r<\rho<1$, then
    \begin{equation*}
    \begin{split}
    |f(re^{i\t})|^\alpha
    &\le\frac{1}{2\pi}\int_{-\pi}^{\pi}\frac{1-(\frac{r}{\rho})^2}{|1-\frac{r}{\rho}e^{it}|^2}|f(\rho
    e^{i(t+\t)})|^\alpha\,dt\\
    &=\int_{-\pi}^{\pi}P\left(\frac{r}{\rho},t\right)|f(\rho e^{i(t+\t)})|^\alpha\,dt,
    \end{split}
    \end{equation*}
where
    $$
    P(r,t)=\frac{1}{2\pi}\frac{1-r^2}{|1-re^{it}|^2},
    \quad0<r<1,\index{$P(r,t)$}
    $$
is the Poisson kernel.\index{Poisson kernel} Set
$t_n=2^{n-1}(1-r)$ and $J_n=[-t_n,t_n]$ for $n=0,1,\ldots,N+1$,
where $N$ is the largest natural number such that $t_N<\frac12$.
Further, set $G_0=J_0$, $G_n=J_n\setminus J_{n-1}$ for
$n=1,\ldots,N$, and $G_{N+1}=[-\pi,\pi]\setminus J_N$. Then
    \begin{equation*}
    \begin{split}
    |f(re^{i\t})|^\alpha
    &\le\sum_{n=0}^{N+1}\int_{G_n}P\left(\frac{r}{\rho},t\right)|f(\rho
    e^{i(t+\t)})|^\alpha dt\\
    &\le\sum_{n=0}^{N+1}P\left(\frac{r}{\rho},t_{n-1}\right)\int_{G_n}|f(\rho
    e^{i(t+\t)})|^\alpha dt\\
    &\lesssim\frac{1}{1-\frac{r}{\rho}}\sum_{n=0}^{N+1}4^{-n}\int_{G_n}|f(\rho
    e^{i(t+\t)})|^\alpha dt,
    \end{split}
    \end{equation*}
and therefore
    \begin{eqnarray*}
    &&|f(re^{i\t})|^\alpha
    (1-r)\int_{(1+r)/2}^1\omega(\rho)\rho\,d\rho
    \le2\int_r^1|f(re^{i\t})|^\alpha(\rho-r)\omega(\rho)\rho\,d\rho\\
    &&\lesssim\sum_{n=0}^{N+1}4^{-n}\int_{r}^1\int_{G_n}\left|f\left(\rho
    e^{i(t+\t)}\right)\right|^\alpha dt\,\omega(\rho)\rho^2\,d\rho.
    \end{eqnarray*}
It follows that
    \begin{equation*}
    \begin{split}
    |f(re^{i\t})|^\alpha
    &\lesssim\sum_{n=0}^{N}2^{-n}
    \frac{\int_{r}^1\int_{-t_n}^{t_n}\left|f\left(\rho
    e^{i(t+\t)}\right)\right|^\alpha dt\,\omega(\rho)\rho\,d\rho}{\int_{-t_n}^{t_n}\int_{(1+r)/2}^1\omega(\rho)\rho\,
    d\rho\,dt}\\
    &\quad+2^{-N}\frac{\int_{r}^1\int_{-\pi}^{\pi}\left|f\left(\rho
    e^{it}\right)\right|^\alpha
    dt\omega(\rho)\rho\,d\rho}{\int_{-\pi}^{\pi}\int_{(1+r)/2}^1\omega(\rho)\rho
    d\rho\,dt}\\
    &\lesssim\sum_{n=0}^{N}2^{-n}
    \frac{\int_{1-t_{n+1}}^1\int_{-t_n}^{t_n}\left|f\left(\rho
    e^{i(t+\t)}\right)\right|^\alpha dt\,\omega(\rho)\rho\,d\rho}{\int_{-t_n}^{t_n}\int_{(1+r)/2}^1\omega(\rho)\rho\,
    d\rho\,dt}\\
    &\quad+2^{-N}\frac{\int_{0}^1\int_{-\pi}^{\pi}\left|f\left(\rho
    e^{it}\right)\right|^\alpha
    dt\omega(\rho)\rho\,d\rho}{\int_{-\pi}^{\pi}\int_{(1+r)/2}^1\omega(\rho)\rho
    d\rho\,dt},
    \end{split}
    \end{equation*}
where the last step is a consequence of the inequalities
$0<1-t_{n+1}\le r$. Denoting the interval centered at $e^{i\t}$
and of the same length as $J_n$ by $J_n(\t)$, and applying
Lemma~\ref{le:condinte}, with $\b\in(0,1)$ fixed, to the
denominators, we obtain
    \begin{equation*}
    \begin{split}
    |f(re^{i\t})|^\alpha&\lesssim\sum_{n=0}^{N}2^{-n(1-\b)}
    \frac{\int_{S(J_n(\t))}\left|f(z)\right|^\alpha\omega(z)\,dA(z)}{\omega(S(J_n(\t)))}\\
    &\quad+2^{-N(1-\b)}\frac{\int_{\D}|f(z)|^\alpha\omega(z)\,dA(z)}{\omega(\D)}\\
    &\lesssim\left(\sum_{n=0}^\infty2^{-n(1-\beta)}\right)M_{\om}(|f|^\alpha)(re^{i\t})
    \lesssim M_{\om}(|f|^\alpha)(re^{i\t}),
    \end{split}
    \end{equation*}
which finishes the proof for $\om\in\I$.
\end{proof}

\section{Proof of the main result}

\subsection{Boundedness}

Let $0<p\le q<\infty$ and $\om\in\I\cup\R$, and assume first that
$\mu$ is a $q$-Carleson measure for $A^p_\om$. Consider the test
functions $F_{a,p}$\index{$F_{a,p}$} defined just before
Lemma~\ref{testfunctions1}. Then the assumption together with
relations \eqref{eq:tf1} and \eqref{eq:tf2} yield
    \begin{equation*}\index{$F_{a,p}$}
    \begin{split}
    \mu(S(a))&\lesssim\int_{S(a)}|F_{a,p}(z)|^q\,d\mu(z)\le\int_{\D}|F_{a,p}(z)|^q\,d\mu(z)\lesssim\|F_{a,p}\|_{A^p_\om}^q\lesssim\om\left(S(a)\right)^\frac{q}{p}
    \end{split}
    \end{equation*}
for all $a\in\D$, and thus $\mu$ satisfies \eqref{eq:s1}.

Conversely, let $\mu$ be a positive Borel measure on $\D$ such
that \eqref{eq:s1} is satisfied. We begin with proving that there
exists a constant $K=K(p,q,\omega)>0$ such that the $L^1_\om$-weak
type inequality\index{$L^1_\om$-weak type inequality}
    \begin{equation}\label{eq:s4}\index{weighted maximal function}\index{$M_\om(\vp)$}
    \mu\left(E_{s}\right)\le
    Ks^{-\frac{q}{p}}\|\vp\|_{L^1_\om}^\frac{q}{p},\quad E_s=\left\{z\in\D: M_{\om}(\vp)(z)>s
    \right\},
    \end{equation}
is valid for all $\varphi\in L^1_\om$ and $0<s<\infty$.

If $E_s=\emptyset$, then \eqref{eq:s4} is clearly satisfied. If
$E_s\not=\emptyset$, then recall that
$I_z=\{e^{i\t}:|\arg(ze^{-i\t})|<(1-|z|)/2\}$ and $S(z)=S(I_z)$,
and define for each $\e>0$ the sets
    $$
    A_s^{\e}=\left\{z\in\D:\, \int_{S(I_z)}|\vp(\xi)|\om(\xi)\,dA(\xi)>s
    \left(\e+\om(S(z)) \right) \right\}
    $$
and
    $$
    B_s^{\e}=\left\{z\in\D:\, I_z\subset I_u\,\text{for some $u\in A_s^{\e}$}\right\}.
    $$
The sets $B_s^{\e}$ expand as $\e\to 0^+$, and
    $$\index{weighted maximal function}\index{$M_\om(\vp)$}
    E_s=\left\{z\in\D: M_{\om}(\vp)(z)>s \right\}=\bigcup_{\e>0}B_s^{\e},
    $$
so
    \begin{equation}\label{eq:s5}
    \mu(E_s)=\lim_{\e\to0^+}\mu(B_s^{\e}).
    \end{equation}
We notice that for each $\e>0$ and $s>0$ there are finitely many
points $z_n\in A_s^{\e}$ such that the arcs $I_{z_n}$ are
disjoint. Namely, if there were infinitely many points $z_n\in
A_s^{\e}$ with this property, then the definition of $A_s^{\e}$
would yield
    \begin{equation}\label{eq:s6}
    s\sum_n[\e+\om(S(z))]\le\sum_n\int_{S(I_{z_n})}|\vp(\xi)|\om(\xi)\,dA(\xi)
    \le\|\vp\|_{L^1_\om},
    \end{equation}
and therefore
    $$
    \infty=s\sum_n\e\le\|\vp\|_{L^1_\om},
    $$
which is impossible because $\vp\in L^1_\om$.

We now use Covering lemma~\cite[p.~161]{Duren1970} to find
$z_1,\ldots,z_m\in A_s^{\e}$ such that the arcs $I_{z_n}$ are
disjoint and
    $$
    A_s^{\e}\subset\bigcup_{n=1}^m\left\{z:I_z\subset J_{z_n}\right\},
    $$
where $J_z$ is the arc centered at the same point as $I_z$ and of
length $5|I_z|$. It follows easily that
    \begin{equation}\label{eq:s7}
    B_s^{\e}\subset\bigcup_{n=1}^m\left\{z:I_z\subset J_{z_n}\right\}.
    \end{equation}
But now the assumption \eqref{eq:s1} and Lemma~\ref{le:condinte}
give
    \begin{equation*}\begin{split}
    \mu\left(\left\{z:I_z\subset J_{z_n}\right\}\right)
    &=\mu\left(\left\{z:S(z)\subset S(J_{z_n}) \right\}\right)\le\mu\left( S(J_{z_n})\right)\\
    &\lesssim\left(\om\left( S(J_{z_n})\right)\right)^\frac{q}p\lesssim\left(\om\left(S(z_n)\right)\right)^\frac{q}{p},\quad
n=1,\ldots,m.
    \end{split}
    \end{equation*}
This combined with \eqref{eq:s7} and \eqref{eq:s6} yield
    $$
    \mu(B_s^{\e})\lesssim\sum_{n=1}^m\left(\om\left(S(z_n)\right)\right)^\frac{q}{p}
    \le\left(\sum_{n=1}^m\om\left(S(z_n)\right)\right)^\frac{q}{p}\le s^{-\frac{q}{p}}
    \|\vp\|_{L^1_\om}^\frac{q}{p},
    $$
which together with \eqref{eq:s5} gives \eqref{eq:s4} for some
$K=K(p,q,\omega)$.

We will now use Lemma~\ref{le:suf1} and \eqref{eq:s4} to show that
$\mu$ is a $q$-Carleson measure for $A^p_\om$. To do this, fix
$\alpha>\frac{1}{p}$ and let $f\in A^p_\om$. For $s>0$, let
$|f|^{\frac{1}{\alpha}}=\psi_{\frac{1}{\alpha},s}+\chi_{\frac{1}{\alpha},s}$,
where
    \begin{equation*}
    \psi_{\frac{1}{\alpha},s}(z)=\left\{
        \begin{array}{rl}
        |f(z)|^{\frac{1}{\alpha}},&\quad\text{if}\;|f(z)|^{\frac{1}{\alpha}}>s/(2K)\\
        0,&\quad\textrm{otherwise}
        \end{array}\right.
    \end{equation*}
and $K$ is the constant in \eqref{eq:s4}, chosen such that
$K\ge1$. Since $p>\frac{1}{\alpha}$, the function
$\psi_{\frac{1}{\alpha},s}$ belongs to $L^1_\om$ for all $s>0$.
Moreover,
    $$
    M_{\om}(|f|^{\frac{1}{\alpha}})\le M_{\om}(\psi_{\frac{1}{\alpha},s})+M_{\om}(\chi_{\frac{1}{\alpha},s})\le
    M_{\om}(\psi_{\frac{1}{\alpha},s})+\frac{s}{2K},
    $$
and therefore
    \begin{equation}\label{eq:inclusion}
    \left\{z\in\D: M_{\om}(|f|^{\frac{1}{\alpha}})(z)>s\right\}
   \subset
   \left\{z\in\D: M_{\om}(\psi_{\frac{1}{\alpha},s})(z)>s/2\right\}.
    \end{equation}
Using Lemma~\ref{le:suf1}, the inclusion \eqref{eq:inclusion},
\eqref{eq:s4} and Minkowski's inequality in continuous form
(Fubini in the case $q=p$), we finally deduce
    \begin{equation*}
    \begin{split}
    \int_{\D}|f(z)|^q\,d\mu(z)&\lesssim
    \int_{\D}\left(M_{\om}(|f|^{\frac{1}{\alpha}})(z)\right)^{q\alpha}\,d\mu(z)\\
    &=q\alpha\int_0^\infty s^{q\alpha-1}
    \mu\left(\left\{z\in\D: M_{\om}(|f|^{\frac{1}{\alpha}})(z)>s \right\} \right)\,ds\\
    &\le q\alpha\int_0^\infty s^{q\alpha-1}
    \mu\left(\left\{z\in\D: M_{\om}(\psi_{\frac{1}{\alpha},s})(z)>s/2\right\}\right)\,ds\\
    &\lesssim
    \int_0^\infty s^{q\alpha-1-\frac{q}{p}}
    \|\psi_{\frac{1}{\alpha},s}\|_{L^1_\om}^\frac{q}p\,ds\\
    &=\int_0^\infty s^{q\alpha-1-\frac{q}{p}}
    \left(\int_{\left\{z:\,|f(z)|^{\frac{1}{\alpha}}>\frac{s}{2K}\right\}}
    |f(z)|^{\frac{1}{\alpha}}\om(z)\,dA(z)\right)^\frac{q}p\,ds\\
    &\le\left(\int_{\D}|f(z)|^{\frac{1}{\alpha}}\om(z)\left(\int_0^{2K|f(z)|^{\frac{1}{\alpha}}}s^{q\alpha-1-\frac{q}{p}}
    \,ds\right)^\frac{p}{q}\,dA(z)\right)^\frac{q}{p}\\
    &\lesssim\left(\int_{\D}|f(z)|^p\om(z)\,dA(z)\right)^\frac{q}{p}.
    \end{split}
    \end{equation*}
Therefore $\mu$ is a $q$-Carleson measure for $A^p_\om$, and the
proof of Theorem~\ref{th:cm}(i) is complete.\hfill$\Box$

We note that the proof of Theorem~\ref{th:cm}(i) also yields
Corollary~\ref{co:maxbou}.

\subsection{Compactness}

The proof is based on a well-known method that has been used in
several times in the existing literature, see, for example,
\cite{PP}. Let $0<p\le q<\infty$ and $\om\in\I\cup\R$, and assume
first that $I_d:\, A^p_\om\to L^q(\mu)$ is compact. For each
$a\in\D$, consider the function
    \begin{equation}\label{testfunctions}\index{$f_{a,p}$}
    f_{a,p}(z)=\frac{(1-|a|)^{\frac{\gamma+1}{p}}}{(1-\overline{a}z)^{\frac{\gamma+1}{p}}\om\left(S(a)\right)^{\frac1p}},\index{$f_{a,p}$}
    \end{equation}
where $\gamma>0$ is chosen large enough such that
$\sup_{a\in\D}\|f_{a,p}\|_{A^p_\om}<\infty$\index{$f_{a,p}$} by
the proof of Lemma~\ref{testfunctions1}. Therefore the closure of
set $\{f_{a,p}:a\in\D\}$ is compact in $L^q(\mu)$ by the
assumption, and hence standard arguments yield
    \begin{equation}\label{101}
    \lim_{r\to1^-}\int_{\D\setminus
    D(0,r)}|f_{a,p}(z)|^q\,d\mu(z)=0
    \end{equation}
uniformly in $a$. Moreover, the proof of Lemma~\ref{le:condinte}
gives
    $$
    \lim_{|a|\to1^-}\frac{(1-|a|)^{\gamma}}{\int_{|a|}^1\omega(s)s\,ds}=0,
    $$
if $\gamma>0$ is again large enough. So, if $\gamma$ is fixed
appropriately, then
    $$
    \lim_{|a|\to1^-}f_{a,p}(z)=\lim_{|a|\to1^-}\left(\frac{(1-|a|)^\gamma \pi}{(1-\overline{a}z)^{\gamma+1}\int_{|a|}^1\omega(s)s\,ds}\right)^\frac1p=0
    $$
uniformly on compact subsets of $\D$. By combining this with
\eqref{101} we obtain
$\lim_{|a|\to1^-}\|f_{a,p}\|^q_{L^q(\mu)}=0$, and consequently,
    $$
    0=\lim_{|a|\to1^-}\|f_{a,p}\|^q_{L^q(\mu)}\ge\lim_{|a|\to1^-}\int_{S(a)}|f_{a,p}(z)|^q\,d\mu(z)
    \gtrsim\lim_{|a|\to1^-}\frac{\mu\left(S(a) \right)}{\left(\om\left(S(a)
    \right)\right)^\frac{q}p},
    $$
which proves \eqref{eq:s1compact}.

Conversely, assume that $\mu$ satisfies \eqref{eq:s1compact}, and
set
    $$
    d\mu_r(z)=\chi_{\left\{r\le |z|<1\right\}}(z)\,d\mu(z).
    $$
We claim that Theorem~\ref{th:cm}(i) implies
    $$
    \|h\|_{L^q(\mu_r)}\le K_{\mu_r}\|h\|_{A^p_\om},\quad h\in A^p_\om,
    $$
where
    \begin{equation}\label{kmur}
    \lim_{r\to 1^-}K_{\mu_r}
    =\lim_{r\to 1^-}\left(\sup_{I\subset\T}\frac{\mu_r(S(I))}{\left(\om\left(S(I)\right)\right)^{\frac{q}{p}}}\right)=0,
    \end{equation}
the proof of which is postponed for a moment. Let $\{f_n\}$ be a
bounded sequence in~$A^p_\om$. Then $\{f_n\}$ is uniformly bounded
on compact sets of $\D$ by Lemma~\ref{le:suf1}, and hence
$\{f_n\}$ constitutes a normal family by Montel's theorem.
Therefore we may extract a subsequence $\{f_{n_k}\}$ that
converges uniformly on compact sets of $\D$ to some $f\in\H(\D)$.
Fatou's lemma shows that $f\in A^p_\om$. Let now $\varepsilon>0$.
By \eqref{kmur} there exists $r_0\in(0,1)$ such that
$K_{\mu_r}\le\e$ for all $r\ge r_0$. Moreover, by the uniform
convergence on compact sets, we may choose $n_0\in\N$ such that
$|f_{n_k}(z)-f(z)|^q<\e$ for all $n_k\ge n_0$ and
$z\in\overline{D(0,r_0)}$. It follows that
    \begin{equation*}
    \begin{split}
    \|f_{n_k}-f\|^q_{L^q(\mu)}
    &\le\e\mu\left(\overline{D(0,r_0)}\right)+\|f_{n_k}-f\|^q_{L^q(\mu_{r_0})}\\
    &\le\e\mu\left(\D\right)+K_{\mu_{r_0}}\|f_{n_k}-f\|^q_{A^p_\om}\\
    &\le\e(\mu\left(\D\right)+C),\quad n_k\ge n_0,
    \end{split}
    \end{equation*}
for some constant $C>0$, and so $I_d:\, A^p_\om\to L^q(\mu)$ is
compact.

Finally, we will prove \eqref{kmur}. By the assumption
\eqref{eq:s1compact}, for a given $\e>0$, there exists
$r_0\in(0,1)$ such that
    \begin{equation}\label{1.}
    \frac{\mu(S(I))}{\left(\om\left(S(I)\right)\right)^{\frac{q}{p}}}<\e
    \end{equation}
for all intervals $I\subset\T$ with $|I|\le 1-r_0$. Therefore, if
$r\ge r_0$, we have
    \begin{equation}\label{2.}
    \sup_{|I|\le1-r_0}\frac{\mu_r(S(I))}{\left(\om\left(S(I)\right)\right)^{\frac{q}{p}}}
    \le\sup_{|I|\le1-r_0}\frac{\mu(S(I))}{\left(\om\left(S(I)\right)\right)^{\frac{q}{p}}}\le\e.
    \end{equation}
If $|I|>1-r_0$, we choose $n=n(I)\in\N$ such that
$(n-1)(1-r_0)<|I|\le n(1-r_0)$, and let $I_k\subset\T$ be
intervals satisfying $|I_k|=1-r_0$ for all $k=1,\ldots,n$, and
$I\subset\cup_{k=1}^n I_k$. If now $r\ge r_0$, then \eqref{1.}
yields
    \begin{equation*}
    \begin{split}
    \mu_r(S(I))&\le\mu_{r_0}(S(I))\le\sum_{k=1}^n\mu_{r_0}(S(I_k))
    \le\e\sum_{k=1}^n\left(\om\left(S(I_k)\right)\right)^{\frac{q}{p}}\\
    &=\e n\left( (1-r_0)\int_{r_0}^1s\om(s)\,ds\right)^{\frac{q}{p}}
    \le\e\left(n
    (1-r_0)\int_{r_0}^1s\om(s)\,ds\right)^{\frac{q}{p}}\\
    &\le2^{\frac{q}{p}}\e\left(|I|\int_{r_0}^1s\om(s)\,ds\right)^{\frac{q}{p}}
    \le2^{\frac{q}{p}}\e\left(\om\left(S(I)\right)\right)^{\frac{q}{p}},\quad
    |I|>1-r_0.
    \end{split}
    \end{equation*}
This together with \eqref{2.} gives \eqref{kmur}. \hfill$\Box$

\subsection{Counterexample}\label{SubSection:counterexample}

Theorem~3.1 in \cite{OC}, Theorem~\ref{th:cm}(i) and
Lemma~\ref{le:RAp}(i) imply that condition \eqref{eq:s1} is
equivalent to
    \begin{equation}\label{eq:s1bolas}
    \sup_{a\in\D}\frac{\mu\left(D(a,\b(1-|a|)) \right)}{\left(\om\left(D(a,\b(1-|a|))
    \right)\right)^\frac{q}p}<\infty,\quad \b\in(0,1),
    \end{equation}
for all $q\ge p$ and $\om\in\R$. However, this is no longer true
if $\om\in\widetilde{\I}$. Indeed, one can show that
\eqref{eq:s1bolas} is a sufficient condition for $\mu$ to be a
$q$-Carleson measure for $A^p_{\om}$, but it turns out that it is
not necessary.

\begin{proposition}
Let $0<p\le q<\infty$ and $\om\in\widetilde{\I}$. Then there
exists a $q$-Carleson measure for $A^p_\om$ which does not
satisfy~\eqref{eq:s1bolas}.
\end{proposition}

\begin{proof}
Let us consider the measure
    $$
    d\mu(z)=(1-|z|)^{\frac{q}{p}-1}\left(\int_{|z|}^1\om(s)\,ds\right)^{\frac{q}{p}}\chi_{[0,1)}(z)d|z|,\quad
    \a\in(1,\infty),
    $$
which is supported on $[0,1)\subset\D$. It is clear that
    $$
    \mu(S(|a|))\le \left(\int_{|a|}^1\om(s)\,ds\right)^{\frac{q}{p}}\int_{|a|}^1(1-s)^{\frac{q}{p}-1}\,ds\asymp
    \left(\om(S(a))\right)^\frac{q}{p},\quad |a|\in[0,1).
    $$
Moreover, bearing in mind the proof of Lemma~\ref{le:condinte},
for each $\beta>0$, there exists $r_0=r_0(\b)$ such that the
function $h_{1/\beta}(r)=\frac{\int_{r}^1\om(s)\,ds}{(1-r)^\beta}$
is increasing on $[r_0,1)$. Therefore $h_{1/\beta}(r)$ is
essentially increasing on $[0,1)$, and hence
    $$
    \mu(S(|a|))\gtrsim\left(\frac{\int_{|a|}^1\om(s)\,ds}{(1-|a|)^\beta}\right)^{q/p}
    \int_{|a|}^1(1-s)^{\frac{q}{p}(1+\beta)-1}\,ds\asymp
    \left(\om(S(a))\right)^\frac{q}{p}.
    $$
Theorem~\ref{th:cm} now shows that $\mu$ is a $q$-Carleson measure
for $A^p_{\om}$.

It remains to show that \eqref{eq:s1bolas} fails. If $|a|$ is
sufficiently close to $1$, then, by
Lemma~\ref{le:cuadrado-tienda}, Lemma~\ref{le:sc1} and
\eqref{eq:r2} for $\om^\star$, we obtain
    \begin{equation*}
    \begin{split}
    \mu\left(D(|a|,\b(1-|a|))\right)
    &\asymp\int_{|a|-\b(1-|a|)}^{|a|+\b(1-|a|)}\frac{\left(\om^\star(s)\right)^{q/p}}{1-s}\,ds\\
    &\asymp\left(\om^\star(a)\right)^{q/p}\int_{|a|-\b(1-|a|)}^{|a|+\b(1-|a|)}\frac{ds}{1-s}
    \asymp\left(\om(S(a))\right)^\frac{q}{p}.
    \end{split}\end{equation*}
Moreover, since $\om\in\widetilde{\I}$ by the assumption,
\eqref{eq:r2} gives
    $$
    \om(D(|a|,\b(1-|a|)))\asymp(1-|a|)^2\om(a).
    $$
Therefore
    $$
    \lim_{|a|\to 1^{-}}\frac{\mu\left(D(|a|,\b(1-|a|))\right)}{ \om(D(|a|,\b(1-|a|)))^{q/p}}
    \asymp\lim_{|a|\to1^{-}}\left(\frac{\psi_\om(a)}{1-|a|}\right)^{q/p}=\infty,
    $$
because $\om\in\widetilde{\I}$, and consequently
\eqref{eq:s1bolas} does not hold.
\end{proof}

\chapter{Factorization and Zeros of Functions in
$A^p_\om$}\label{sec:factorizacion}\index{factorization}\index{${\mathcal
Inv}$}\index{invariant weight}

The main purpose of this chapter is to establish the following
factorization theorem for functions in the weighted Bergman space
$A^p_\om$, under the assumption that $\om$ is invariant and the
polynomials are dense in $A^p_\om$. Recall that a weight
$\om:\D\to(0,\infty)$ is invariant if $\om(z)\asymp\om(a)$ in each
pseudohyperbolic disc $\De(a,r)$ of a fixed radius $r\in(0,1)$.
The invariant weights and their relation to those $\om$ for which
the polynomials are dense in~$A^p_\om$ were discussed in
Section~\ref{Sec:DensityOfPolynomials}.

\begin{theorem}\label{Thm:FactorizationBergman}\index{factorization}
Let $0<p<\infty$ and $\omega\in{\mathcal Inv}$ such that the
polynomials are dense in $A^p_\om$. Let $f\in A^p_\omega$, and let
$0<p_1,p_2<\infty$ such that $p^{-1}=p_1^{-1}+p_2^{-1}$. Then
there exist $f_1\in A^{p_1}_\omega$ and $f_2\in A^{p_2}_\omega$
such that $f=f_1\cdot f_2$ and
    \begin{equation}\label{Eq:NormEstimateForFactorization}
   \|f_1\|_{A^{p_1}_\omega}^p\cdot\|f_2\|_{A^{p_2}_\omega}^p\le\frac{p}{p_1}\|f_1\|_{A^{p_1}_\omega}^{p_1}+\frac{p}{p_2}\|f_2\|_{A^{p_2}_\omega}^{p_2}\le C\|f\|_{A^p_\omega}^p
    \end{equation}
for some constant $C=C(p_1,p_2,\omega)>0$.\index{${\mathcal
Inv}$}\index{invariant weight}
\end{theorem}

If $\om$ is radial, then polynomials are dense in $A^p_\om$, and
hence in this case the assumption on $\om$ in
Theorem~\ref{Thm:FactorizationBergman} reduces to the requirement
\eqref{eq:r2}.

Theorem~\ref{Thm:FactorizationBergman} is closely related to the
zero distribution of functions in~$A^p_\om$. Therefore, and also
for completeness, we put some attention to this matter. Our
results follow the line of those due to
Horowitz~\cite{Horzeros,Horzeros1,Horzeros2}. Roughly speaking we
will study basic properties of unions, subsets and the dependence
on $p$ of the zero sets of functions in $A^p_\om$. In contrast to
the factorization result, here the density of polynomials need not
to be assumed. A good number of the results established are
obtained as by-products of some key lemmas used in the proof of
Theorem~\ref{Thm:FactorizationBergman} and sharp conditions on
some products defined on the moduli of the zeros. However, we do
not venture into generalizing the theory, developed among others
by Korenblum~\cite{Kor}, Hedenmalm~\cite{HedStPeter} and
Seip~\cite{S1,S2}, and based on the use of densities defined in
terms of partial Blaschke sums, Stolz star domains and
Beurling-Carleson characteristic of the corresponding boundary
set, which might be needed in order to obtain more complete
results. Since the angular distribution of zeros plays a role in a
description of the zero sets of functions in the classical
weighted Bergman space~$A^p_\alpha$, it is natural to expect that
the same happens also in $A^p_\om$, when $\om\in\I\cup\R$. In this
chapter we will also briefly discuss the zero distribution of
functions in the Bergman-Nevanlinna classes.

\section{Factorization of functions in $A^p_\om$}\label{SubSec:factorization}

Most of the results and proofs in this section are inspired by
those of Horowitz~\cite{HorFacto}, and Horowitz and
Schnaps~\cite{HorSchna}. It is worth noticing that
in~\cite{HorSchna} the factorization of functions in $A^p_\om$ is
considered when $\om$ belongs to a certain subclass of decreasing
radial weights. Consequently, it does not cover the class
${\mathcal Inv}$,\index{${\mathcal Inv}$}\index{invariant weight}
and in particular it does not contain the set
$\widetilde{\I}\cup\R$. One part of our contribution consists of
detailed analysis on the constant $C=C(p_1,p_2,\omega)>0$
appearing in \eqref{Eq:NormEstimateForFactorization}, see
Corollary~\ref{cor:FactorizationBergman} below. Later on, in
Chapter~\ref{SecVolterra}, this corollary shows its importance
when the bounded integral operators $T_g:A^p_\om\to A^q_\om$ are
characterized.

For $-1<\a<\infty$, the \emph{Bergman-Nevanlinna
class}\index{Bergman-Nevanlinna class} $\BN_\a$\index{$\BN_\a$}
consists of those $f\in\H(\D)$ such that
    $$
    \int_\D\log^+|f(z)|(1-|z|^2)^\alpha\,dA(z)<\infty.
    $$
If $\{z_k\}$ are the zeros of $f\in\BN_\a$, then
    \begin{equation}\label{Eq:Bergman-Nevanlinna-Zeros}\index{$N(r,f)$}\index{$N(r)$}
    \begin{split}
    &\sum_{k}(1-|z_k|)^{2+\a}<\infty,\\
    n(r,f)&=\op\left(\frac{1}{(1-r)^{2+\a}}\right),\quad
    r\to1^-,\\\index{$n(r,f)$}\index{$n(r)$}
    N(r,f)&=\op\left(\frac{1}{(1-r)^{1+\a}}\right),\quad r\to1^-,
    \end{split}
    \end{equation}
where $n(r,f)$ denotes the number of zeros of $f$ in $D(0,r)$,
counted according to multiplicity, and
    $$
    N(r,f)=\int_0^r\frac{n(s)-n(0,f)}{s}\,ds+n(0,f)\log r
    $$ is the
integrated counting function. If no confusion arises with regards
to $f$, we will simply write $n(r)$ and $N(r)$ instead of $n(r,f)$
and $N(r,f)$. The well-known facts
\eqref{Eq:Bergman-Nevanlinna-Zeros} are consequences of Jensen's
formula,\index{Jensen's formula} which states that
    \begin{equation}\label{Eq:Jensen-Formula}
    \log|f(0)|+\sum_{k=1}^n\log\frac{r}{|z_k|}=\frac{1}{2\pi}\int_0^{2\pi}\log|f(re^{i\t})|d\t,\quad0<r<1,
    \end{equation}
where $\{z_k\}$ are the zeros of $f$ in the disc $D(0,r)$,
repeated according to multiplicity and ordered by increasing
moduli. See Lemma~\ref{Lemma:NBzeros} for a more general result
from which \eqref{Eq:Bergman-Nevanlinna-Zeros} follows by choosing
$\om(r)=(1-r^2)^\a$.

We begin with showing that the class of invariant
weights\index{${\mathcal Inv}$}\index{invariant weight} is in a
sense a natural setting for the study of factorization of
functions in $A^p_\om$.

\begin{lemma}\label{InvcontenidoenB_0}\index{${\mathcal Inv}$}\index{invariant weight}
If $0<p<\infty$ and $\om\in{\mathcal Inv}$, then
$A^p_\omega\subset\BN_0$.
\end{lemma}

\begin{proof}
Let $f\in A^p_\om$, where $0<p<\infty$ and $\om\in{\mathcal Inv}$.
Let $u\in\D$ be fixed. Then Lemma~\ref{Lemma:InvariantWeights}
yields
    \begin{equation*}
    \begin{split}
    \int_\D\log^+|f(z)|\,dA(z)&\le\frac1p\int_\D\log^+(|f(z)|^p\omega(z))\,dA(z)+\frac1p\int_\D\log^+\frac1{\omega(z)}\,dA(z)\\
    &\le\frac1p\|f\|^p_{A^p_\omega}+\frac1p\int_\D\log^+C(\vp_u(z))\,dA(z)+\frac{1}{p}\log^+\frac1{\omega(u)}\\
    &\le\frac1p\|f\|^p_{A^p_\omega}+\frac1p\left(\frac{1+|u|}{1-|u|}\right)^2\int_\D\log^+C(w)\,dA(w)\\
    &\quad+\frac{1}{p}\log^+\frac1{\omega(u)}<\infty,
    \end{split}
    \end{equation*}
and the inclusion $A^p_\omega\subset\BN_0$
follows.\index{${\mathcal Inv}$}\index{invariant weight}
\end{proof}

The next result plays an important role in the proof of
Theorem~\ref{Thm:FactorizationBergman}.

\begin{lemma}\label{Lemma:factorization}\index{${\mathcal Inv}$}\index{invariant weight}
Let $0<p<q<\infty$ and $\omega\in{\mathcal Inv}$. Let $\{z_k\}$ be
the zero set of $f\in A^p_\omega$, and let
    $$
    g(z)=|f(z)|^p\prod_{k}\frac{1-\frac{p}{q}+\frac{p}{q}|\vp_{z_k}(z)|^q}{|\vp_{z_k}(z)|^p}.
    $$
Then there exists a constant $C=C(p,q,\omega)>0$ such that
\begin{equation}\label{constante}
\|g\|_{L^1_\omega}\le C\|f\|_{A^p_\omega}^p.
\end{equation}
Moreover, the constant $C$ has the following properties:
    \begin{enumerate}
    \item[\rm(i)] If $0<p<q\le 2$, then $C=C(\omega)$, that is, $C$ is independent of $p$ and $q$.
    \item[\rm(ii)] If $2<q<\infty$ and $\frac{q}{p}\ge1+\epsilon>1$, then $C=C_1qe^{C_1q}$, where $C_1=C_1(\epsilon,\omega)$.
    \end{enumerate}
\end{lemma}

\begin{proof}
Let us start with considering $g(0)$. To do this, assume
$f(0)\ne0$. Lemma~\ref{InvcontenidoenB_0} yields $f\in\BN_0$, and
hence $\sum_k(1-|z_k|)^2<\infty$. Moreover,
Bernouilli-l'H\^{o}pital theorem shows that
    \begin{equation}\label{Eq:Limit(1-r)^2}
    1-\frac{1-\frac1n+\frac1nr^{pn}}{r^p}\asymp(1-r)^2,\quad r\to1^-,
    \end{equation}
for all $n>1$, and hence the product
    $$
    \prod_{k}\frac{1-\frac{p}{q}+\frac{p}{q}|z_k|^q}{|z_k|^p}
    $$
converges. Now, an integration by parts gives
    \begin{equation}\label{45}
    \begin{split}\sum_{k}\log\left(\frac{1-\frac{p}{q}+\frac{p}{q}|z_k|^q}{|z_k|^p}\right)
    &=\int_0^1\log\left(\frac{1-\frac{p}{q}+\frac{p}{q}r^q}{r^p}\right)\,dn(r)\\
    &=\int_0^1\frac{(\frac{q}{p}-1)(1-r^q)}{\frac
    qp-1+r^q}\frac{pn(r)}{r}\,dr,
    \end{split}
    \end{equation}
where the last equality follows by \eqref{Eq:Limit(1-r)^2} and
\eqref{Eq:Bergman-Nevanlinna-Zeros} with $\alpha=0$. Another
integration by parts and Jensen's formula
\eqref{Eq:Jensen-Formula}\index{Jensen's formula} show that the
last integral in \eqref{45} equals to
    \begin{equation}\label{46}
    \begin{split}
    -\int_0^1pN(r)\,du(r)
    =\int_\D\log|f(z)|^p\,d\sigma(z)-\log|f(0)|^p,
    \end{split}
    \end{equation}
where
    $$
    d\sigma(z)=-u'(|z|)\frac{dA(z)}{2|z|},\quad u(r)=\frac{(\frac{q}{p}-1)(1-r^q)}{\frac{q}{p}-1+r^q},
    $$
is a positive measure of unit mass on $\D$. By combining
\eqref{45} and \eqref{46}, we deduce
    \begin{equation*}
    \begin{split}
    \log(g(0))&=\log\left(|f(0)|^p\prod_{k}\left(\frac{1-\frac{p}{q}+\frac{p}{q}|z_k|^q}{|z_k|^p}\right)\right)\\
    &=\log|f(0)|^p+\sum_{k}\log\left(\frac{1-\frac{p}{q}+\frac{p}{q}|z_k|^q}{|z_k|^p}\right)
    =\int_\D\log|f(z)|^p\,d\sigma(z).
    \end{split}
    \end{equation*}
Replacing now $f$ by $f\circ\vp_\zeta$, we obtain
    \begin{equation}\label{49}
    \log(g(\zeta))=\int_\D\log|f(\vp_\z(z))|^p\,d\sigma(z)
    \end{equation}
for $\zeta$ outside of zeros of $f$. Once this identity has been
established we will prove~(i).

Assume that $0<q\le2$. We claim that there exists a unique point
$x=x(p,q)\in(0,1)$ such that $-u'(x)=2x$, $2r\le-u'(r)$ on
$(0,x]$, and $2r\ge-u'(r)$ on $[x,1)$. We first observe that
    \begin{equation*}
    \begin{split}
    -u'(r)=\frac{q}{p}\left(\frac{q}{p}-1\right)\frac{qr^{q-1}}{\left(\frac{q}{p}-1+r^q\right)^2}&=2r\\
    \Leftrightarrow2r^{q+2}+4r^2\left(\frac{q}{p}-1\right)+2r^{2-q}\left(\frac{q}{p}-1\right)^2&=\frac{q^2}{p}\left(\frac{q}{p}-1\right).
    \end{split}
    \end{equation*}
Now the left hand side is increasing, it vanishes at the origin,
and at one it attains the value $2\left(\frac{q}{p}\right)^2$,
which is strictly larger than the constant on the right hand side
because $q\le2$. The existence of the point $x\in(0,1)$ with the
desired properties follows.

By dividing the radial integral in \eqref{49} into $(0,x)$ and
$(x,1)$, using the fact that
$\int_0^{2\pi}\log|f(\vp_\z(re^{i\t}))|^p\,d\t$ is increasing on
$(0,1)$, and bearing in mind that $-u'(r)dr$ and $2rdr$ are
probability measures on $(0,1)$, we deduce
    \begin{equation}\label{50}
    \int_\D\log|f(\vp_\z(z))|^p\,d\sigma(z)\le\int_\D\log|f(\vp_\z(z))|^p\,dA(z).
    \end{equation}
By combining \eqref{49} and \eqref{50}, we obtain
    \begin{equation}\label{47}
    \begin{split}
    \log(g(\z)\omega(\z))&\le\int_\D\log\left(|f(\vp_\z(z))|^p\omega(\z)\right)dA(z)\\
    &\le\int_\D\log\left(|f(\vp_\z(z))|^p\omega(\z)\frac{(1-|z|^2)}{C(z)}\right)dA(z)+C_1,
    \end{split}
    \end{equation}
where $C$ is the function from Lemma~\ref{Lemma:InvariantWeights},
and hence
    \begin{equation}\label{68}
    C_1=C_1(\omega)=\int_\D\log\left(\frac{C(z)}{1-|z|^2}\right)dA(z)<\infty
    \end{equation}
by \eqref{Eq:InvariantWeightsOldDefinitionIntegral}. Now Jensen's
inequality together with \eqref{Eq:InvariantWeightsOldDefinition}
yields
    \begin{equation}\label{48}
    \begin{split}
    g(\z)\omega(\z)&\le
    e^{C_1}\int_\D|f(\vp_\z(z))|^p\omega(\z)\frac{(1-|z|^2)}{C(z)}\,dA(z)\\
    &\le e^{C_1}\int_\D|f(\vp_\z(z))|^p\omega(\vp_\z(z))(1-|z|^2)\,dA(z)
    \end{split}
    \end{equation}
for $\zeta$ outside of zeros of $f$. By integrating this
inequality, changing a variable, and using Fubini's theorem and
\cite[Lemma~3.10]{Zhu} we deduce
    \begin{equation}\label{66}
    \begin{split}
    \|g\|_{L^1_\om}
    &\le e^{C_1}\int_\D|f(u)|^p\omega(u)\left(\int_\D|\vp_\z'(u)|^2(1-|\vp_\z(u)|^2)\,dA(\z)\right)dA(u)\\
    &\le C_2\|f\|_{A^p_\om}^p,
    \end{split}
    \end{equation}
where $C_2=C_2(\om)>0$ is a constant. This finishes the proof
of~(i).

Assume now that $q\ge 2$. In this case
    \begin{equation}\label{j1}
    -u'(r)=\frac{q}{p}\left(\frac{q}{p}-1\right)\frac{qr^{q-1}}{\left(\frac{q}{p}-1+r^q\right)^2}\le \frac{q^2}{q-p}r^{q-1}\le \frac{q^2}{q-p}r,
    \quad 0\le r\le1.
    \end{equation}
The identity \eqref{49} and \eqref{j1} give
    \begin{equation*}
    \begin{split}
    \log(g(\z)\omega(\z))&\le\int_\D\log\left(|f(\vp_\z(z))|^p\omega(\z)\frac{(1-|z|^2)}{C(z)}\right)d\sigma(z)\\
    &\quad+\int_\D\log\left(\frac{C(z)}{(1-|z|^2)}\right)d\sigma(z)\\
    &\le\int_\D\log\left(|f(\vp_\z(z))|^p\omega(\z)\frac{(1-|z|^2)}{C(z)}\right)d\sigma(z)+C_1(q,p,\om),
    \end{split}
    \end{equation*}
where $C$ is the function from Lemma~\ref{Lemma:InvariantWeights},
and hence
    $$
    C_1(q,p,\om)=\frac{q^2}{2(q-p)}\int_\D\log\left(\frac{C(z)}{1-|z|^2}\right)dA(z)<\infty
    $$
by \eqref{Eq:InvariantWeightsOldDefinitionIntegral}. By arguing
now similarly as in \eqref{48} and \eqref{66}, and using
\eqref{j1}, we obtain $\|g\|_{L^1_\om}\le C_2\|f\|_{A^p_\om}^p$,
where $C_2=C_2(p,q,\om)=\frac{q^2}{2(q-p)}e^{C_1(p,q,\om)}>0$ is a
constant. Consequently, joining this with (i), we deduce
\eqref{constante} for all $0<p<q<\infty$.

Finally we will prove (ii). Assume that $2<q<\infty$ and
$\frac{q}{p}\ge1+\epsilon>1$. A direct calculation shows that
    \begin{equation}\label{j3}
    -u'(r)=\frac{q}{p}\left(\frac{q}{p}-1\right)\frac{qr^{q-1}}{\left(\frac{q}{p}-1+r^q\right)^2}
    \le\frac{1+\epsilon}{\epsilon}qr,\quad
    0\le r\le1.
    \end{equation}
The identity \eqref{49} and \eqref{j3} give
    \begin{equation}\label{47b}
    \begin{split}
    \log(g(\z)\omega(\z))\le\int_\D\log\left(|f(\vp_\z(z))|^p\omega(\z)\frac{(1-|z|^2)}{C(z)}\right)d\sigma(z)+C_1,
    \end{split}
    \end{equation}
where
    $$
    C_1=C_1(p,q,\omega)=q\frac{1+\epsilon}{\epsilon}\int_\D\log\left(\frac{C(z)}{1-|z|^2}\right)dA(z).
    $$
Therefore $C_1=C_2C_3q$, where
$C_2=C_2(\epsilon)=\frac{1+\epsilon}{\epsilon}$ and
$C_3=C_3(\omega)<\infty$ by
\eqref{Eq:InvariantWeightsOldDefinitionIntegral}. Jensen's
inequality together with \eqref{47b} and
\eqref{Eq:InvariantWeightsOldDefinition} yield
    \begin{equation*}
    \begin{split}
    g(\z)\omega(\z)&\le
    e^{C_2C_3q}C_2q\int_\D|f(\vp_\z(z))|^p\omega(\z)\frac{(1-|z|^2)}{C(z)}\,dA(z)\\
    &\le e^{C_2C_3q}C_2q\int_\D|f(\vp_\z(z))|^p\omega(\vp_\z(z))(1-|z|^2)\,dA(z).
    \end{split}
    \end{equation*}
By arguing similarly as in \eqref{66} we finally obtain
$\|g\|_{L^1_\om}\le C_4\|f\|_{A^p_\om}^p$, where
$C_4=e^{C_2C_3q}C_2C_5q$ and $C_5>0$ is a constant. This finishes
the proof.
\end{proof}

Now we are in position to prove the announced factorization for
functions in~$A^p_\om$.

\subsection*{Proof of Theorem~\ref{Thm:FactorizationBergman}.}
Let $0<p<\infty$ and $\omega\in{\mathcal Inv}$ such that the
polynomials are dense in $A^p_\om$, and let $f\in A^p_\omega$.
Assume first that $f$ has finitely many zeros only. Such functions
are of the form $f=gB$, where $g\in A^p_\omega$ has no zeros and
$B$ is a finite Blaschke product. Let $z_1,\ldots,z_m$ be the
zeros of $f$ so that $B=\prod_{k=1}^mB_k$, where
$B_k=\frac{z_k}{|z_k|}\vp_{z_k}$. Write $B= B^{(1)}\cdot B^{(2)}$,
where the factors $B^{(1)}$ and $B^{(2)}$ are random subproducts
of $B_0,B_1,\ldots,B_m$, where $B_0\equiv1$. Setting
$f_j=\left(\frac{f}{B}\right)^\frac{p}{p_j}B^{(j)}$, we have
$f=f_1\cdot f_2$. We now choose $B^{(j)}$ probabilistically. For a
given $j\in\{1,2\}$, the factor $B^{(j)}$ will contain each $B_k$
with the probability $p/p_j$. The obtained $m$ random variables
are independent, so the expected value of $|f_j(z)|^{p_j}$ is
    \begin{equation}\begin{split}\label{eq:esp2}
    E(|f_j(z)|^{p_j})&=\left|\frac{f(z)}{B(z)}\right|^p\prod_{k=1}^m
    \left(1-\frac{p}{p_j}+\frac{p}{p_j}|\vp_{z_k}(z)|^{p_j}\right)\\
    &=\left|f(z)\right|^p\prod_{k=1}^m
    \frac{\left(1-\frac{p}{p_j}\right)+\frac{p}{p_j}|\vp_{z_k}(z)|^{p_j}}{|\vp_{z_k}(z)|^p}
    \end{split}\end{equation}
for all $z\in\D$ and $j\in\{1,2\}$. Now, bearing in mind
\eqref{eq:esp2} and Lemma~\ref{Lemma:factorization}, we find a
constant $C_1=C_1(p, p_1,\omega)>0$ such that
   \begin{equation*}\begin{split}
   \left\|E\left(f_1^{p_1}\right)\right\|_{L^{1}_\omega}
    & =
    \int_\D\left[
   \left|f(z)\right|^p\prod_{k=1}^m
    \frac{\left(1-\frac{p}{p_1}\right)+\frac{p}{p_1}|\vp_{z_k}(z)|^{p_1}}{|\vp_{z_k}(z)|^p}\right]\,\om(z)dA(z)
    \\ & =
    \int_\D\left[
   \left|f(z)\right|^p\prod_{k=1}^m
    \frac{\frac{p}{p_2}+\left(1-\frac{p}{p_2}\right)|\vp_{z_k}(z)|^{p_1}}{|\vp_{z_k}(z)|^p}\right]\,\om(z)dA(z)\le  C_{1}\|f\|_{A^p_\omega}^p.
   \end{split}\end{equation*}
Analogously, by \eqref{eq:esp2} and
Lemma~\ref{Lemma:factorization} there exists a constant
$C_2=C_{2}(p,p_2,\om)>0$ such that
    \begin{equation*}\begin{split}
   \left\|E\left(f_2^{p_2}\right)\right\|_{L^{1}_\omega}
    & =
    \int_\D\left[
   \left|f(z)\right|^p\prod_{k=1}^m
    \frac{\left(1-\frac{p}{p_2}\right)+\frac{p}{p_2}|\vp_{z_k}(z)|^{p_2}}{|\vp_{z_k}(z)|^p}\right]\,\om(z)dA(z)\le
    C_2\|f\|_{A^p_\omega}^p.
      \end{split}\end{equation*}
By combining the two previous inequalities, we obtain
 \begin{equation}\begin{split}\label{eq:lfn3}
    \left\|E\left(\frac{p}{p_1}f_1^{p_1}\right)\right\|_{L^{1}_\omega}
    +\left\|E\left(\frac{p}{p_2}f_2^{p_2}\right)\right\|_{L^{1}_\omega}
    \le\left(\frac{p}{p_1}C_{1}+\frac{p}{p_2}C_{2}\right)\|f\|_{A^p_\omega}^p.
      \end{split}\end{equation}
On the other hand,
      \begin{equation}\begin{split}\label{eq:lfn4}
   &  \left\|E\left(\frac{p}{p_1}f_1^{p_1}\right)\right\|_{L^{1}_\omega}
   +\left\|E\left(\frac{p}{p_2}f_2^{p_2}\right)\right\|_{L^{1}_\omega}
   \\ & =\frac{p}{p_1}\int_\D \left|\frac{f(z)}{B(z)}\right|^p
   \prod_{k=1}^m
    \left(\frac{p}{p_2}+\left(1-\frac{p}{p_2}\right)|\vp_{z_k}(z)|^{p_1}\right)\,\om(z)dA(z)
    \\ &   \quad+\frac{p}{p_2}\int_\D \left|\frac{f(z)}{B(z)}\right|^p
   \prod_{k=1}^m
    \left(\left(1-\frac{p}{p_2}\right)+\frac{p}{p_2}|\vp_{z_k}(z)|^{p_2}\right)
   \,\om(z)dA(z)\\
   &=\int_\D I(z)\omega(z)\,dA(z),
      \end{split}\end{equation}
where
    \begin{equation*}\begin{split}
    I(z)=\left|\frac{f(z)}{B(z)}\right|^p\Bigg[\frac{p}{p_1}\cdot&\prod_{k=1}^m
    \left(\frac{p}{p_2}+\left(1-\frac{p}{p_2}\right)|\vp_{z_k}(z)|^{p_1}\right)\\
    +\frac{p}{p_2}\cdot&\prod_{k=1}^m
    \left(\left(1-\frac{p}{p_2}\right)+\frac{p}{p_2}|\vp_{z_k}(z)|^{p_2}\right)\Bigg].
    \end{split}\end{equation*}
It is clear that the $m$ zeros of $f$ must be distributed to the
factors $f_1$ and $f_2$, so if $f_1$ has $n$ zeros, then $f_2$ has
the remaining $(m-n)$ zeros. Therefore
 \begin{equation}\label{idez}
    I(z)=\sum_{f_{l_1}\cdot f_{l_2}=f} \left(\left(1-\frac{p}{p_2}\right)^n\left(\frac{p}{p_2}\right)^{m-n}\left[\frac{p}{p_1}
|f_{l_1}(z)|^{p_1}+\frac{p}{p_2} |f_{l_2}(z)|^{p_2}
\right]\right).
 \end{equation}
This sum consists of $2^m$ addends, $f_{l_1}$ contains
$\left(\frac{f}{B}\right)^\frac{p}{p_1}$ and $n$ zeros of $f$, and
$f_{l_2}$ contains $\left(\frac{f}{B}\right)^\frac{p}{p_2}$ and
the remaining $(m-n)$ zeros of $f$, and thus $f=f_{l_1}\cdot
f_{l_2}$. Further, for a fixed $n=0,1,\ldots,m$, there are
$({m\atop n})$ ways to choose $f_{l_1}$ (once $f_{l_1}$ is chosen,
$f_{l_2}$ is determined). Consequently,
    \begin{equation}\begin{split}\label{eq:lfn5}
    \sum_{f_{l_1}\cdot f_{l_2}=f} \left(1-\frac{p}{p_2}\right)^n\left(\frac{p}{p_2}\right)^{m-n}
    =\sum_{n=0}^m\left({m\atop n}\right)\left(1-\frac{p}{p_2}\right)^n\left(\frac{p}{p_2}\right)^{m-n}=1.
    \end{split}\end{equation}
Now, by joining \eqref{eq:lfn3}, \eqref{eq:lfn4} and \eqref{idez},
we deduce
    \begin{equation*}\begin{split}
    &\sum_{f_{l_1}\cdot f_{l_2}=f}
    \left(1-\frac{p}{p_2}\right)^n\left(\frac{p}{p_2}\right)^{m-n}
    \left[\frac{p}{p_1}\|f_{l_1}\|_{A^{p_1}_\om}^{p_1}+\frac{p}{p_2}
    \|f_{l_2}\|_{A^{p_2}_\om}^{p_2} \right]\\
    &\quad\le\left(\frac{p}{p_1}C_{1}+\frac{p}{p_2}C_{2}\right)\|f\|_{A^p_\omega}^p.
    \end{split}\end{equation*}
This together with \eqref{eq:lfn5} shows that there must exist a
concrete factorization $f=f_1\cdot f_2$ such that
    \begin{equation}
    \begin{split}
    \label{eq:lfn6}
    \frac{p}{p_1}\|f_{1}\|_{A^{p_1}_\om}^{p_1}+\frac{p}{p_2}
    \|f_{2}\|_{A^{p_2}_\om}^{p_2} \le
    C(p_1,p_2,\om)\|f\|_{A^p_\omega}^p.
    \end{split}
    \end{equation}
By combining this with the inequality
     $$
     x^\alpha\cdot y^\beta\le \alpha x+\beta y,\quad x,y\ge 0,\quad \alpha+\beta=1,
     $$
we finally obtain \eqref{Eq:NormEstimateForFactorization} under
the hypotheses that $f$ has finitely many zeros only.

To deal with the general case, we first prove that every
norm-bounded family in $A^p_\omega$ is a normal family of analytic
functions. If $f\in A^p_\omega$, then
    \begin{equation}\label{Eq:RadialGrowthNEW}
    \begin{split}
    \|f\|_{A^p_\omega}^p&\ge\int_{D(0,\frac{1+\rho}{2})\setminus
    D(0,\rho)}|f(z)|^p\omega(z)\,dA(z)\\
    &\gtrsim
    M_p^p(\rho,f)\left(\min_{|z|\le\frac{1+\rho}{2}}\om(z)\right),\quad 0\le\rho<1,
    \end{split}
    \end{equation}
from which the well-known relation $M_\infty(r,f)\lesssim
M_p(\frac{1+r}{2},f)(1-r)^{-1/p}$ yields
    \begin{equation}\label{20NEW}
    M^p_\infty(r,f)
    \lesssim\frac{\|f\|_{A^p_\omega}^p}{(1-r)\left(\min_{|z|\le\frac{3+r}{4}}\om(z)\right)},\quad 0\le r<1.
    \end{equation}
Therefore every norm-bounded family in $A^p_\omega$ is a normal
family of analytic functions by Montel's theorem.

Finally, assume that $f\in A^p_\omega$ has infinitely many zeros.
Since polynomials are dense in $A^p_\omega$ by the assumption, we
can choose a sequence $f_l$ of functions with finitely many zeros
that converges to $f$ in norm, and then, by the previous argument,
we can factorize each $f_l=f_{l,1}\cdot f_{l,2}$ as earlier. Now,
since every norm-bounded family in~$A^p_\omega$ is a normal family
of analytic functions, by passing to subsequences of $\{f_{l,j}\}$
with respect to $l$ if necessary, we have $f_{l,j}\to f_j$, where
the functions $f_j$ form the desired bounded factorization $f=
f_1\cdot f_2$ satisfying \eqref{Eq:NormEstimateForFactorization}.
This finishes the proof of Theorem~\ref{Thm:FactorizationBergman}.
\hfill$\Box$

\medskip

It is worth noticing that the hypothesis on the density of
polynomials in Theorem~\ref{Thm:FactorizationBergman} can be
replaced by any assumption that ensures the existence of a dense
family of functions with finitely many zeros only.

At first glance the next result might seem a bit artificial.
However, it turns out to be the key ingredient in the proof of the
uniform boundedness of a certain family of linear operators, see
Proposition~\ref{PropSmallIndeces}. Usually this kind of
properties are established by using interpolation theorems similar
to~\cite[Theorem~2.34]{Zhucomplexball}
or~\cite[Theorem~4.29]{Zhu}. Indeed, it would be interesting to
know whether or not an extension of~\cite[Theorem~4.29]{Zhu}
remains valid for $A^p_\om$ if either $\om\in\I\cup\R$ or
$\om\in\widetilde{\I}\cup\R$.

\begin{corollary}\label{cor:FactorizationBergman}\index{factorization}
Let $0<p<2$ and $\omega\in{\mathcal Inv}$ such that the
polynomials are dense in $A^p_\om$.\index{${\mathcal
Inv}$}\index{invariant weight} Let $0<p_1\le2<p_2<\infty$ such
that $\frac1p=\frac1{p_1}+\frac1{p_2}$ and $p_2\ge2p$. If $f\in
A^p_\omega$, then there exist $f_1\in A^{p_1}_\omega$ and $f_2\in
A^{p_2}_\omega$ such that $f=f_1\cdot f_2$ and
    \begin{equation}\label{Eqco:NormEstimateForFactorization}
    \|f_1\|_{A^{p_1}_\omega}\cdot \|f_2\|_{A^{p_2}_\omega}\le C\|f\|_{A^p_\om}
    \end{equation}
for some constant $C=C(p_1,\om)>0$.
\end{corollary}

\begin{proof}
We will mimic the proof of Theorem~\ref{Thm:FactorizationBergman},
but we will have to pay special attention to the constants coming
from Lemma~\ref{Lemma:factorization}. Therefore we will keep the
notation used in the proof of
Theorem~\ref{Thm:FactorizationBergman}, and will only explain the
steps where the proof significantly differs from that of
Theorem~\ref{Thm:FactorizationBergman}.

It suffices to establish the desired factorization for those $f\in
A^p_\omega$ having finitely many zeros only. Following the proof
of Theorem~\ref{Thm:FactorizationBergman}, we write $f=gB$, where
$B= B^{(1)}\cdot B^{(2)}$, and
$f_j=\left(\frac{f}{B}\right)^\frac{p}{p_j}B^{(j)}$, so that
$f=f_1\cdot f_2$. Now, bearing in mind~\eqref{eq:esp2}, and the
cases (i) and (ii) of Lemma~\ref{Lemma:factorization}, we find
constants $C_1=C_1(\omega)>0$ and $C_2=C_{2}(\om)>0$ such that
    \begin{equation*}
    \begin{split}
    \left\|E\left(f_1^{p_1}\right)\right\|_{L^{1}_\omega}
    &=\int_\D\left(\left|f(z)\right|^p\prod_{k=1}^m
    \frac{\frac{p}{p_2}+\left(1-\frac{p}{p_2}\right)|\vp_{z_k}(z)|^{p_1}}{|\vp_{z_k}(z)|^p}\right)\,\om(z)dA(z)\le C_{1}\|f\|_{A^p_\omega}^p
    \end{split}
    \end{equation*}
and
    \begin{equation*}
    \begin{split}
    \left\|E\left(f_2^{p_2}\right)\right\|_{L^{1}_\omega}
    &=\int_\D\left(
    \left|f(z)\right|^p\prod_{k=1}^m
    \frac{\left(1-\frac{p}{p_2}\right)+\frac{p}{p_2}|\vp_{z_k}(z)|^{p_2}}{|\vp_{z_k}(z)|^p}\right)\,\om(z)dA(z)\\
    &\le C_2p_2e^{C_2p_2}\|f\|_{A^p_\omega}^p=M_2\|f\|_{A^p_\omega}^p,
    \end{split}
    \end{equation*}
where we may choose $C_2$ sufficiently large so that
$M_2=C_2p_2e^{C_2p_2}\ge 2$. These inequalities yield
    \begin{equation*}
    \begin{split}
    &\left\|E\left(\left(1-\frac{1}{M_2}\right)f_1^{p_1}\right)\right\|_{L^{1}_\omega}
    +\left\|E\left(\frac{1}{M_2}f_2^{p_2}\right)\right\|_{L^{1}_\omega}\le\left(C_{1}\left(1-\frac{1}{M_2}\right)+1\right)\|f\|_{A^p_\omega}^p,
    \end{split}
    \end{equation*}
where
   \begin{equation*}\begin{split}
   &\left\|E\left(\left(1-\frac{1}{M_2}\right)f_1^{p_1}\right)\right\|_{L^{1}_\omega}
   +\left\|E\left(\frac{1}{M_2}f_2^{p_2}\right)\right\|_{L^{1}_\omega}\\
   &=\left(1-\frac{1}{M_2}\right)\int_\D \left|\frac{f(z)}{B(z)}\right|^p
   \prod_{k=1}^m
   \left(\frac{p}{p_2}+\left(1-\frac{p}{p_2}\right)|\vp_{z_k}(z)|^{p_1}\right)\,\om(z)dA(z)\\
   &\quad+\frac{1}{M_2}\int_\D \left|\frac{f(z)}{B(z)}\right|^p
   \prod_{k=1}^m
   \left(\left(1-\frac{p}{p_2}\right)+\frac{p}{p_2}|\vp_{z_k}(z)|^{p_2}\right)\,\om(z)dA(z).
   \end{split}
   \end{equation*}
Therefore
    \begin{equation*}
    \begin{split}
    &\sum_{f_{l_1}\cdot f_{l_2}=f} \left(1-\frac{p}{p_2}\right)^n\left(\frac{p}{p_2}\right)^{m-n}
    \left[\left(1-\frac{1}{M_2}\right)
    \|f_{l_1}\|_{A^{p_1}_\om}^{p_1}+\frac{1}{M_2}\|f_{l_2}\|_{A^{p_2}_\om}^{p_2}\right]\\
    &\le\left(C_{1}\left(1-\frac{1}{M_2}\right)+1\right)\|f\|_{A^p_\omega}^p,
    \end{split}
    \end{equation*}
which together with~\eqref{eq:lfn5} implies that there exists a
concrete factorization $f=f_1\cdot f_2$ such that
    \begin{equation*}
    \begin{split}
    \left(1-\frac{1}{M_2}\right)\|f_{1}\|_{A^{p_1}_\om}^{p_1}+\frac{1}{M_2}\|f_{2}\|_{A^{p_2}_\om}^{p_2}
    \le\left(C_{1}\left(1-\frac{1}{M_2}\right)+1\right)\|f\|_{A^p_\omega}^p.
    \end{split}
    \end{equation*}
It follows that
      \begin{equation*}
      \begin{split}
      \|f_{1}\|_{A^{p_1}_\om}
      &\le\left(2\left(C_{1}\left(1-\frac{1}{M_2}\right)+1\right)\right)^\frac{1}{p_1}\|f\|_{A^p_\omega}^\frac{p}{p_1}
      \le K_1\|f\|_{A^p_\omega}^\frac{p}{p_1},
      \end{split}
      \end{equation*}
where $K_1=K_1(p_1,\omega)=\left(2(C_{1}+1)\right)^\frac{1}{p_1}$,
and
    \begin{equation*}
    \begin{split}\label{eq:lf7}
    \|f_{2}\|_{A^{p_2}_\om}
    \le\left(M_2+C_{1}\left(M_2-1\right)\right)^{\frac{1}{p_2}}\|f\|_{A^p_\omega}^{\frac{p}{p_2}}\le
    K_2\|f\|_{A^p_\omega}^{\frac{p}{p_2}}<\infty,
    \end{split}
    \end{equation*}
where
      $$
      K_2=K_2(\om)=\sup_{p_2>2}\left(M_2+ C_{1}\left(M_2-1\right)\right)^{\frac{1}{p_2}}<\infty.
      $$
By multiplying these inequalities we obtain
      $$
      \|f_{1}\|_{A^{p_1}_\om}\|f_{2}\|_{A^{p_2}_\om}
      \le K_1\cdot K_2\|f\|_{A^p_\omega}.
      $$
This establishes the assertion with
$C(p_1,\om)=K_1(p_1,\om)K_2(\om)$.
\end{proof}

\section{Zeros of functions in $A^p_\om$}\label{Sec:Zeros}

For a given space $X$ of analytic functions in $\D$, a sequence
$\{z_k\}$ is called an $X$-\emph{zero set}, if there exists a
function $f$ in $X$ such that $f$ vanishes precisely on the points
$z_k$ and nowhere else. As far as we know, it is still an open
problem to find a complete description of zero sets of functions
in the Bergman spaces $A^p=A^p_0$, but the gap between the known
necessary and sufficient conditions is very small. We refer
to~\cite[Chapter~4]{DurSchus} and \cite[Chapter~4]{HKZ} as well as
\cite{Kor,Lzeros96,S1,S2} for more information on this topic.

We begin with proving that a subset of an $A^p_\om$-zero set is
also an $A^p_\om$-zero set if~$\om$ is invariant.\index{${\mathcal
Inv}$}\index{invariant weight}

\begin{theorem}\label{Theorem:f/H}
Let $0<p<\infty$ and $\omega\in{\mathcal Inv}$.\index{${\mathcal
Inv}$}\index{invariant weight} Let $\{z_k\}$ be an arbitrary
subset of the zero set of $f\in A^p_\omega$, and let
    $$
    H(z)=\prod_{k}B_{k}(z)(2-B_{k}(z)),\quad
    B_k=\frac{z_k}{|z_k|}\vp_{z_k},
    $$
with the convention $z_k/|z_k|=1$ if $z_k=0$. Then there exists a
constant $C=C(\omega)>0$ such that $\|f/H\|_{A^p_\omega}^p\le
C\|f\|_{A^p_\omega}^p$. In particular, each subset of an
$A^p_\omega$-zero set is an $A^p_\omega$-zero set.
\end{theorem}

Since
    $$
    \frac{|f(z)|}{\prod_{k}|\vp_{z_k}(z)|(2-|\vp_{z_k}(z)|)}
    \ge\left|\frac{f(z)}{H(z)}\right|,
    $$
the assertion in Theorem~\ref{Theorem:f/H} follows by the
following lemma.

\begin{lemma}\label{fstar}
Let $0<p<\infty$ and $\omega\in{\mathcal Inv}$.\index{${\mathcal
Inv}$}\index{invariant weight} Let $\{z_k\}$ be an arbitrary
subset of the zero set of $f\in A^p_\omega$, and set
    $$
    \widehat{f}(z)=\frac{|f(z)|}{\prod_{k}|\vp_{z_k}(z)|(2-|\vp_{z_k}(z)|)}.
    $$
Then there exists a constant $C=C(\omega)>0$ such that
    $$
    \|\widehat{f}\|_{L^p_\omega}^p\le C\|f\|_{A^p_\omega}^p.
    $$
\end{lemma}

\begin{proof} Assume $f(0)\ne0$. By arguing as in~\eqref{45} and~\eqref{46}, and using Lemma~\ref{InvcontenidoenB_0} and \eqref{Eq:Bergman-Nevanlinna-Zeros}, we
obtain
    \begin{equation*}
    \begin{split}
    \sum_k\log\frac{1}{|z_k|(2-|z_k|)}&=\int_0^1\log\frac{1}{r(2-r)}\,dn(r)=\int_0^1\frac{2(1-r)}{2-r}\,\frac{n(r)}{r}\,dr\\
    &=\int_\D\log|f(z)|\frac{dA(z)}{(2-|z|)^2|z|}-\log|f(0)|\\
    &\le\int_\D\log|f(z)|\,dA(z)-\log|f(0)|,
    \end{split}
    \end{equation*}
where the last inequality can be established as \eqref{50}.
Replacing $f$ by $f\circ\vp_\zeta$, we obtain
    \begin{equation*}
    \log \widehat{f}(\z)=\log\frac{|f(\z)|}{\prod_{k}|\vp_{z_k}(\z)|(2-|\vp_{z_k}(\z)|)}
    \le\int_\D\log|f(\vp_\z(z))|\,dA(z)
    \end{equation*}
for all $\z$ outside of zeros of $f$. By proceeding as in
\eqref{47} and \eqref{48}, we deduce
    $$
    \|\widehat{f}\|_{L^p_\omega}^p\le
    e^{C_1}\int_\D|f(u)|^p\omega(u)\left(\int_\D|\vp_\z'(u)|^2(1-|\vp_\z(u)|^2)\,dA(\z)\right)dA(u),
    $$
where $C_1$ is given by \eqref{68} and is thus bounded by
Lemma~\ref{Lemma:InvariantWeights}. The assertion follows.
\end{proof}

Theorem~\ref{Theorem:f/H} gives a partial answer to a question
posed by Aleman and Sundberg in \cite[p.~10]{AleSun09}. They
asked:

\begin{enumerate}
\item[]\emph{If $\om$ is a radial continuous and decreasing weight
such that $\om(r)<1$ for all $0\le r<1$, does the condition
    \begin{equation}\label{eq:alesun}
    \limsup_{r\to 1^-}\frac{\log\om(r)}{\log(1-r)}<\infty
    \end{equation}
guarantee that any subset of an $A^p_\om$-zero set is necessarily
an $A^p_\om$-zero set?}
\end{enumerate}
The assumption \eqref{eq:alesun} on $\om$ is equivalent to saying
that $\om$ is continuous and decreasing, and there exists
$C=C(\om)>0$ such that $(1-r)^C\lesssim\om(r)<1$ for all $0\le
r<1$. By Theorem~\ref{Theorem:f/H} and the definition of invariant
weights the answer is affirmative if~$\om$ is a radial continuous
weight satisfying~\eqref{eq:r2}. In this case
$\psi_\om(r)\gtrsim(1-r)$ and $\om(r)\gtrsim(1-r)^\a$ for some
$\a>-1$, so \eqref{eq:alesun} is satisfied. However, there are
radial continuous weights $\om$ that satisfy \eqref{eq:alesun},
and for which~\eqref{eq:r2} fails. For example, consider $\om$
whose graph lies between those of $\om_1(r)=(1-r)/2$ and
$\om_2\equiv1/2$, and that is constructed in the following manner.
Let $r_k=1-e^{-e^k}$ and
$s_k=\frac{r_k+\frac{1}{2}}{1+\frac{r_k}{2}}$ for all $k\in\N$.
Then $0<r_k<s_k<r_{k+1}<1$ and
    \begin{equation}\label{rksk}
    \frac{1-r_k}{3}<1-s_k<1-r_k,\quad k\in\N.
    \end{equation}
Define the decreasing continuous function $\om$ by
    \begin{equation}\label{112}
    \om(r)=\left\{
        \begin{array}{cl}
        \frac{1-r}{2}, & \quad r\in[0,s_1],\\
        \frac{1-s_k}{2}, & \quad r\in[s_k,r_{k+1}],\quad k\in\N,\\
        \frac{1-s_k}{2}+\frac{s_k-s_{k+1}}{2(s_{k+1}-r_{k+1})}(r-r_{k+1}),&\quad
        r\in[r_{k+1}, s_{k+1}],\quad k\in\N.
        \end{array}\right.
    \end{equation}
Then \eqref{rksk} yields $1-r_k\asymp1-s_k$ and
    $$
    \frac{\om(r_k)}{\om(s_k)}=\frac{1-s_{k-1}}{1-s_k}
    \ge\frac{1-r_{k-1}}{3(1-r_k)}\to\infty,\quad k\to\infty,
    $$
and hence $\om$ does not satisfy \eqref{eq:r2}. Since clearly
$\om\not\in\I$, we also deduce $\om\not\in\I\cup\R\cup{\mathcal
Inv}$.

The next two results on zero sets of functions in~$A^p_\omega$ are
readily obtained by \cite[p.~208--209]{HorFacto} with
Lemma~\ref{fstar} in hand, and therefore the details are omitted.
These results say, in particular, that if $\{z_k\}$ is an
$A^p_\omega$-zero set and $\{w_k\}$ are sufficiently close to
$\{z_k\}$, then $\{w_k\}$ is also an $A^p_\omega$-zero set.

\begin{theorem}
Let $0<p<\infty$ and $\omega\in{\mathcal Inv}$.\index{${\mathcal
Inv}$}\index{invariant weight} Let $\{z_k\}$ be an arbitrary
subset of the zero set of $f\in A^p_\omega$. If $\{w_k\}$ is a
sequence such that
    \begin{equation}\label{eq:distzeros}
    \sum_k\left|\frac{w_k-z_k}{1-\overline{z}_kw_k}\right|<\infty,
    \end{equation}
then the function $g$, defined by
    $$
    g(z)=f(z)\prod_{k}\frac{B_{w_k}(z)}{B_{z_k}(z)},\quad z\in\D,
    $$
belongs to $A^p_\om$.
\end{theorem}

\begin{corollary}
Let $0<p<\infty$ and $\omega\in{\mathcal Inv}$.\index{${\mathcal
Inv}$}\index{invariant weight} Let $\{z_k\}$ be the zero set of
$f\in A^p_\omega$. If $\{w_k\}$ satisfies \eqref{eq:distzeros},
then $\{w_k\}$ is an $A^p_\om$-zero set.
\end{corollary}

The rest of the results in this section concern radial weights.
The first of them will be used to show that $A^p_\om$-zero
sets\index{$A^p_\om$-zero set} depend on $p$.

\begin{theorem}\label{Theorem:ZerosBergman1}
Let $0<p<\infty$ and let $\omega$ be a radial weight. Let $f\in
A^p_\omega$, $f(0)\ne0$, and let $\{z_k\}$ be its zero sequence
repeated according to multiplicity and ordered by increasing
moduli. Then
    \begin{equation}\label{j10}
    \prod_{k=1}^n\frac{1}{|z_k|}=\op\left(\left(\int_{1-\frac{1}n}^1\omega(r)\,dr\right)^{-\frac1p}\right),\quad
    n\to \infty.
    \end{equation}
\end{theorem}

\begin{proof}
Let $f\in A^p_\om$ and $f(0)\ne0$. By multiplying Jensen's formula
\eqref{Eq:Jensen-Formula}\index{Jensen's formula} by~$p$, and
applying the arithmetic-geometric mean inequality, we obtain
    \begin{equation}\label{44}
    |f(0)|^p\prod_{k=1}^n\frac{r^p}{|z_k|^p}\le M^p_p(r,f)
    \end{equation}
for all $0<r<1$ and $n\in\N$. Moreover,
    \begin{equation}\label{j20}
    \lim_{r\to 1^-}M^p_p(r,f)\int_r^1\om(s)\,ds\le \lim_{r\to 1^-}
    \int_r^1M^p_p(s,f)\om(s)\,ds=0,
    \end{equation}
so taking $r=1-\frac{1}{n}$ in~\eqref{44}, we deduce
    $$
    \prod_{k=1}^n\frac{1}{|z_k|}\lesssim M_p\left(1-\frac{1}{n},f\right)
    =\op\left(\left(\int_{1-\frac{1}n}^1\omega(r)\,dr\right)^{-\frac1p}\right),\quad
    n\to \infty,
    $$
as desired.
\end{proof}

The next result improves and generalizes
\cite[Theorem~4.6]{Horzeros}. Moreover, its proof shows that
condition~\eqref{j10} is a sharp necessary condition for $\{z_k\}$
to be an $A^p_\om$-zero set.

\begin{theorem}\label{th:zerosqp}
Let $0<q<\infty$ and $\om\in\I\cup\R$. Then there exists
$f\in\cap_{p<q}A^p_\om$ such that its zero sequence $\{z_k\}$,
repeated according to multiplicity and ordered by increasing
moduli, does not satisfy~\eqref{j10} with $p=q$. In particular,
there is a $\cap_{p<q} A^p_\om$-zero set which is not an
$A^q_\om$-zero set.
\end{theorem}

\begin{proof} The proof uses ideas
from~\cite[Theorem~3]{GNW}, see also \cite{Horzeros1,Horzeros2}.
Define
    \begin{equation}\label{zqp2}
    f(z)=\prod_{k=1}\sp\infty F_k(z),\quad z\in\D,
    \end{equation}
where
    $$
    F_k(z)=\frac{ 1+a_k z^{2^k}}{1+a^{-1}_k z^{2^k}},\quad z\in\D,\quad
    k\in\N,
    $$
and
    $$
    a_k=\left(\frac{\int_{1-2^{-k}}^1\om(s)\,ds}{\int_{1-2^{-(k+1)}}^1\om(s)\,ds}\right)^{1/q},\quad
    k\in\N.
    $$
By Lemma~\ref{le:condinte} there exists a constant
$C_1=C_1(q,\om)>0$ such that
    \begin{equation}\label{zqp1}
    1<a_k\le C_1<\infty,\quad k\in\N.
    \end{equation}
Therefore $\limsup_{k\to \infty}(a_k-a_k^{-1})^{2^{-k}}\le
\limsup_{k\to \infty}a_k^{2^{-k}}=1$, and hence the product
in~\eqref{zqp2} defines an analytic function in $\D$. The zero set
of $f$ is the union of the zero sets of the functions $F_k$, so
$f$ has exactly $2^k$ simple zeros on the circle
$\left\{z:|z|=a_k^{-2^{-k}}\right\}$ for each $k\in\N$. Let
$\{z_j\}_{j=1}^\infty$ be the sequence of zeros of $f$ ordered by
increasing moduli, and denote $N_n=2+2^2+\cdots+2^n$. Then $2^n\le
N_n\le 2^{n+1}$, and hence
    \begin{equation*}
    \prod_{k=1}^{N_n}\frac{1}{|z_k|}\ge\prod_{k=1}^n a_k
    =\left(\frac{\int_{\frac12}^1\om(s)\,ds}{\int_{1-2^{-(n+1)}}^1\om(s)\,ds}\right)^{1/q}
    \ge\left(\frac{\int_{\frac12}^1\om(s)\,ds}{\int_{1-\frac{1}{N_n}}^1\om(s)\,ds}\right)^{1/q}.
    \end{equation*}
It follows that $\{z_j\}_{j=1}^\infty$ does not
satisfy~\eqref{j10}, and thus $\{z_j\}_{j=1}^\infty$ is not an
$A^q_\om$-zero set by Theorem~\ref{Theorem:ZerosBergman1}.

We turn to prove that the function $f$ defined in \eqref{zqp2}
belongs to $A^p_\om$ for all $p\in(0,q)$. Set $r_n=e\sp{-2^{-n}}$
for $n\in\N$, and observe that
    \begin{equation}\label{zqp4}
    |f(z)|=\left| \prod_{k=1}^n a_k \frac{a_k^{-1}+z^{2^k}}{1+a^{-1}_k z^{2^k}}\right| \left| \prod_{j=1}^\infty \frac{ 1+a_{n+j} z^{2^{n+j}}}{1+a^{-1}_{n+j}
    z^{2^{n+j}}}\right|.
    \end{equation}
The function $h_1(x)=\frac{\a+x}{1+\a x}$ is increasing on $[0,1)$
for each $\a\in[0,1)$, and therefore
    \begin{equation}\label{67}
    \begin{split}
    \left|\frac{ 1+a_{n+j} z^{2^{n+j}}}{1+a^{-1}_{n+j} z^{2^{n+j}}}\right|&=a_{n+j}
    \left|\frac{ a^{-1}_{n+j}+ z^{2^{n+j}}}{1+a^{-1}_{n+j} z^{2^{n+j}}}\right|
    \le a_{n+j} \frac{ a^{-1}_{n+j}+ |z|^{2^{n+j}}}{1+a^{-1}_{n+j} |z|^{2^{n+j}}}\\
    &\le\frac{1+a_{n+j}\left(\frac1e\right)^{2^{j}}}{1+a^{-1}_{n+j}
    \left(\frac1e\right)^{2^{j}}},\quad |z|\le r_n,\quad
    j,\,n\in\N.
    \end{split}
    \end{equation}
Since $h_2(x)=\frac{1+x\alpha}{1+x^{-1}\alpha}$ is increasing on
$(0,\infty)$ for each $\alpha\in(0,\infty)$, \eqref{zqp1} and
\eqref{67} yield
    \begin{equation}
    \begin{split}\label{zqp5}
    \left|\prod_{j=1}^\infty\frac{1+a_{n+j} z^{2^{n+j}}}{1+a^{-1}_{n+j} z^{2^{n+j}}}\right|
    &\le\prod_{j=1}^\infty\frac{ 1+a_{n+j} \left(\frac1e\right)^{2^{j}}}{1+a^{-1}_{n+j}\left(\frac1e\right)^{2^{j}}}
    \le\prod_{j=1}^\infty\frac{1+C_1\left(\frac1e\right)^{2^{j}}}{1+C_1^{-1}\left(\frac1e\right)^{2^{j}}}
    =C_2<\infty,
    \end{split}
    \end{equation}
whenever $|z|\le r_n$ and $n\in\N$. So, by using \eqref{zqp4},
\eqref{zqp5}, Lemma~\ref{le:condinte} and the inequality
$e^{-x}\ge 1-x$, $x\ge0$, we obtain
    \begin{equation}
    \begin{split}\label{zqp6}
    |f(z)|&\le C_2\prod_{k=1}^n
    a_k\lesssim\left(\frac{1}{\int_{1-2^{-(n+1)}}^1\om(s)\,ds}\right)^{1/q}\lesssim\left(\frac{1}{\int_{1-2^{-n}}^1\om(s)\,ds}\right)^{1/q}\\
    &\le \left(\frac{1}{\int_{r_n}^1\om(s)\,ds}\right)^{1/q},\quad
    |z|\le r_n,\quad n\in\N.
    \end{split}
    \end{equation}
Let now $|z|\ge1/\sqrt{e}$ be given and fix $n\in\N$ such that
$r_n\le|z|<r_{n+1}$. Then \eqref{zqp6}, the inequality $1-x\le
e^{-x}\le 1-\frac{x}{2}$, $x\in[0,1]$, and Lemma~\ref{le:condinte}
give
    \begin{equation*}
    \begin{split}
    |f(z)|&\le
    M_\infty(r_{n+1},f)\lesssim\left(\frac{1}{\int_{r_{n+1}}^1\om(s)\,ds}\right)^{1/q}\\
    &\le \left(\frac{1}{\int_{1-2^{-(n+2)}}^1\om(s)\,ds}\right)^{1/q}
    \lesssim\left(\frac{1}{\int_{1-2^{-n}}^1\om(s)\,ds}\right)^{1/q}\\
    &\le\left(\frac{1}{\int_{r_n}^1\om(s)\,ds}\right)^{1/q}\le \left(\frac{1}{\int_{|z|}^1\om(s)\,ds}\right)^{1/q},
    \end{split}
    \end{equation*}
and hence
    \begin{equation*}
    M_\infty(r,f)\lesssim\left(\frac{1}{\int_{r}^1\om(s)\,ds}\right)^{1/q},\quad 0<r<1.
    \end{equation*}
This and the identity
$\psi_{\widetilde{\om}}(r)=\frac{1}{1-\a}\psi_\om(r)$\index{$\widetilde{\om}(r)$}
of Lemma~\ref{le:RAp}(iii), with $\a=p/q<1$ and $r=0$, yield
    $$
    \|f\|^p_{A^p_\om}\lesssim\int_0^1\frac{\om(r)\,dr}{\left(\int_{r}^1\om(s)\,ds\right)^{p/q}}
    =\int_0^1\widetilde{\om}(r)\,dr=\frac{q}{q-p}\left(\int_{0}^1\om(s)\,ds\right)^{\frac{q-p}{q}}<\infty.
    $$
This finishes the proof.
\end{proof}

The proof of Theorem~\ref{th:zerosqp} remains valid for any radial
weight satisfying~\eqref{eq:Integral2} for some $\beta>0$. By
observing the proof of Lemma~\ref{le:condinte}, this is the case,
in particular, if $\psi_\om(r)\gtrsim(1-r)$ for all $r$ close to
one.

The proof of Theorem~\ref{th:zerosqp} implies that the union of
two $A^p_\om$-zero sets is not necessarily an $A^p_\om$-zero set
if $\om\in\I\cup\R$. More precisely, by arguing similarly as in
the proof of \cite[Theorem~5.1]{Horzeros} we obtain the following
result.

\begin{corollary}\label{co:unioneszeros}
Let $0<p<\infty$ and $\om\in\I\cup\R$. Then the union of two
$A^p_\omega$-zero sets is an $A^{p/2}_\om$-zero set. However,
there are two $\cap_{p<q} A^p_\om$-zero sets such that their union
is not an $A^{q/2}_\om$-zero set.
\end{corollary}

\begin{proof}
The first assertion is an immediate consequence of the
Cauchy-Schwarz inequality. To see the second one, use the proof of
Theorem~\ref{th:zerosqp} to find $f\in\cap_{p<q} A^p_\om$,
$f(0)\ne0$, such that its zero sequence $\{z_k\}$ satisfies
    \begin{equation}\label{73}
    \prod_{k=1}^n\frac{1}{|z_k|}\gtrsim\left(\int_{1-\frac1n}^1\om(r)\,dr\right)^\frac1q.
    \end{equation}
Choose $\t$ such that $\{e^{i\t}z_k\}\cap\{z_k\}=\emptyset$, and
let $\{w_k\}$ be the union $\{e^{i\t}z_k\}\cup\{z_k\}$ organized
by increasing moduli. Then \eqref{73} and Lemma~\ref{le:condinte}
yield
    $$
    \prod_{k=1}^{2n}\frac{1}{|w_k|}=\left(\prod_{k=1}^n\frac{1}{|z_k|}\right)^2
    \gtrsim\left(\int_{1-\frac1n}^1\om(r)\,dr\right)^\frac2q
    \gtrsim\left(\int_{1-\frac1{2n}}^1\om(r)\,dr\right)^\frac2q.
    $$
Therefore $\{w_k\}$ is not an $A^{q/2}_\om$-zero set by
Theorem~\ref{Theorem:ZerosBergman1}.
\end{proof}

The next result can be obtained by bearing in mind
Lemma~\ref{le:condinte} and using arguments similar to those in
\cite[p.~101--103]{HKZ}. One shows that if $f\in A^p_\om$, then
\eqref{71} is satisfied, and this together with
Lemma~\ref{le:condinte} implies \eqref{72}.

\begin{proposition}\label{pr:numberzeros}
Let $0<p<\infty$, $\om\in\I\cup\R$ and $f\in A^p_\om$. Then
    \begin{equation}\label{71}
    N(r)\lesssim\left(\log\frac{1}{\int_r^1\om(s)\,ds}\right),\quad r\to
    1^-,
    \end{equation}
and
    \begin{equation}\label{72}
    n(r)\lesssim\left(\frac{1}{1-r}\log\frac{1}{\int_r^1\om(s)\,ds}\right),\quad r\to 1^-.
    \end{equation}
\end{proposition}

It is well known that if $\om$ is any radial weight, then there is
an $A^p_\om$-zero set which does not satisfy the Blaschke
condition~\cite[p.~94]{DurSchus}. However, if $\{z_k\}$ is an
$A^p_\a$-zero set\index{$A^p_\a$-zero set}, then
    \begin{equation}\label{Eq:ZerosInGrowthSpaces}
    \sum_{k}(1-|z_k|)\left(\log\frac{1}{1-|z_k|}\right)^{-1-\e}<\infty
    \end{equation}
for all $\e>0$ by~\cite[p.~95]{DurSchus}. This together with the
observations (ii) and (iii) to Lemma~\ref{le:condinte} show that the same is
true for $A^p_\om$-zero sets if $\om\in\I\cup\R$. We will improve
this last statement. To do so we let $\om$ be a radial weight such
that $\int_0^1 \om(r)\,dr<1$, and consider the increasing
differentiable function
    \begin{equation}\label{hderre}
    h(r)=\log\frac{1}{\int_{r}^1\om(s)\,ds},\quad
    0\le r<1,
    \end{equation}
that satisfies $h'(r)=1/\psi_\om(r)$ and $h(r)>0$ for all $0\le
r<1$.

\begin{lemma}\label{Lem:tau}
Let $f\in\H(\D)$ such that its ordered sequence of zeros $\{z_k\}$
satisfies $N(r)\asymp h(r)$, as $r\to1^-$. Further, let
$\om\in\I\cup\R$ with $\int_0^1 \om(r)\,dr<1$, and let
$\tau:[0,\infty)\to[0,\infty)$ be an increasing function such that
$x\lesssim\tau(x)$ and $\tau'(x)\lesssim \tau(x)+1$. Then
    \begin{equation}\label{sblascke1}
    \sum_{k}\frac{1-|z_k|}{\tau\left(\log\frac{1}{\int_{|z_k|}^1\om(s)\,ds}\right)}<\infty
    \end{equation}
if and only if
    \begin{equation}\label{69}
    \int_{R_0}^\infty\frac{\tau'(x)}{\tau^2(x)}x\,dx<\infty
    \end{equation}
for some $R_0\in(0,\infty)$.
\end{lemma}

\begin{proof}
Let $f$, $\{z_k\}$, $\om$ and $\tau$ be as in the statement, and
let $h$ be the function given in \eqref{hderre}. Without loss of
generality we may assume that $f(0)\ne0$. Since $N(r)\asymp h(r)$,
we have $n(r)\lesssim h(r)/(1-r)$ by the paragraph just before
Proposition~\ref{pr:numberzeros}. Therefore \eqref{eq:r3} for
$\om\in\R$ and \eqref{eq:I} for $\om\in\I$, and two integrations
by parts yield
    \begin{equation*}
    \begin{split}
    \sum_{|z_k|\ge\frac{1}{2}}\frac{1-|z_k|}{\tau\left(\log\frac{1}{\int_{|z_k|}^1\om(s)\,ds}\right)}
    &\le\int_{\frac{1}{2}}^1\frac{(1-r)}{\tau(h(r))}\,dn(r)\\
    &\lesssim1+\int_{\frac{1}{2}}^1\frac{1}{\tau(h(r))}
    \left(1+\frac{\tau'(h(r))h'(r)(1-r)}{\tau(h(r))}\right)\,n(r)\,dr\\
    &\lesssim1+\int_{\frac{1}{2}}^1\frac{1}{\tau(h(r))}\,\frac{n(r)}{r}\,dr\\
    &\lesssim1+\int_{\frac{1}{2}}^1\frac{\tau'(h(r))h'(r)}{\tau^2(h(r))}N(r)\,dr
    \lesssim1+\int_{h(\frac12)}^\infty\frac{\tau'(x)}{\tau^2(x)}x\,dx,
    \end{split}
    \end{equation*}
which shows that \eqref{69} implies \eqref{sblascke1}.

A reasoning similar to that above gives
    \begin{equation*}
    \begin{split}\label{j24}
    \sum_{k}\frac{1-|z_k|}{\tau\left(\log\frac{1}{\int_{|z_k|}^1\om(s)\,ds}\right)}
    &\ge\int_{r_0}^1\frac{(1-r)}{\tau(h(r))}\,dn(r)\\
    &\gtrsim-1+\int_{r_0}^1\frac{1}{\tau(h(r))}
    \left(1+\frac{\tau'(h(r))h'(r)(1-r)}{\tau(h(r))}\right)\,n(r)\,dr\\
    &\ge-1+r_0\int_{r_0}^1\frac{1}{\tau(h(r))}\,\frac{n(r)}{r}\,dr\\
    &\gtrsim-1+\int_{r_0}^1\frac{\tau'(h(r))h'(r)}{\tau^2(h(r))}N(r)\,dr
    \gtrsim-1+\int_{h(r_0)}^\infty\frac{\tau'(x)}{\tau^2(x)}x\,dx
    \end{split}
    \end{equation*}
for all sufficiently large $r_0$, and thus \eqref{sblascke1}
implies \eqref{69}.
\end{proof}

It is worth noticing that \eqref{69} is equivalent to
    \begin{equation}\label{85}
    \int_{R_0}^\infty\frac{dx}{\tau(x)}<\infty.
    \end{equation}
This can be seen by integrating the identity
    $$
    \frac{\tau'(x)}{\tau^2(x)}x+\frac{d}{dx}\left(\frac{x}{\tau(x)}\right)=\frac1{\tau(x)}
    $$
from $R_0$ to $R$, and then letting $R\to\infty$.

\begin{theorem}\label{th:subsblaschkecond}
Let $0<p<\infty$ and $\om\in\I\cup\R$ such that $\int_0^1
\om(r)\,dr<1$. Let $\tau:[0,\infty)\to[0,\infty)$ be an increasing
function such that $x\lesssim\tau(x)$  and $\tau'(x)\lesssim
\tau(x)+1$.
\begin{itemize}
\item[\rm(i)] If $f\in A^p_\om$, then its ordered sequence of
zeros $\{z_k\}$ satisfies \eqref{sblascke1} for each~$\tau$ that
obeys \eqref{85}.

\item[\rm(ii)] There exists $f\in A^p_\om$ such that its ordered
sequence of zeros $\{z_k\}$ satisfies
    \begin{equation}\label{sblascke1s}
    \sum_{k}\frac{1-|z_k|}{\tau\left(\log\frac{1}{\int_{|z_k|}^1\om(s)\,ds}\right)}=\infty
    \end{equation}
for each $\tau$ for which \eqref{85} fails for some
$R_0\in(0,\infty)$.
\end{itemize}
\end{theorem}

\begin{proof}
The assertion (i) follows by the first part of the proof of
Proposition~\ref{pr:numberzeros}, Lemma~\ref{Lem:tau} and the
discussion related to \eqref{85}. To prove~(ii) it suffices to
find $f\in A^p_\om$ such that $N(r)\gtrsim h(r)$, where $h$ is the
function defined in~\eqref{hderre}. To do so, we will use
arguments similar to those in the proof of \cite[Theorem~6]{GNW}.
Take $\a>p$, and let $\{r_n\}$ be the increasing sequence defined
by
    \begin{equation}\label{rn}
    \int_{r_n}^1\om(r)\,dr=\frac{1}{2^{n\a}}.
    \end{equation}
Denote $M_n=E\left(\frac{1}{1-r_n}\right)$, where $E(x)$ is the
integer such that $E(x)\le x< E(x)+1$. By the proof of
Lemma~\ref{le:condinte} there exists $\b=\beta(\om)>0$ such that
\eqref{eq:Integral2} with $C=1$ holds for all $r$ sufficiently
close to $1$. Hence there exist $\lambda>1$ and $n_0\in\N$ such
that
    \begin{equation}
    \begin{split}\label{j51}
    \frac{M_{n+1}}{M_n}&\ge\frac{1-r_n}{1-r_{n+1}}-(1-r_n)
    \ge\left(\frac{\int_{r_n}^1\om(r)\,dr}{\int_{r_{n+1}}^1\om(r)\,dr}\right)^{\frac{1}{\beta}}-(1-r_n)\\
    &\ge2^{\frac{1}{\a\beta}}-(1-r_n)
    \ge\lambda,\quad n\ge n_0.
  \end{split}\end{equation}
Therefore the analytic function
    $$
    g(z)=\sum_{n=n_0}^\infty 2^n z^{M_n},\quad z\in\D,
    $$
is a lacunary series, and hence $M_p(r,g)\asymp M_2(r,g)$ for all
$0<p<\infty$~\cite{Zygmund59}. \index{lacunary series}

Next, we claim that there exists $r_0\in(0,1)$ such that
\begin{equation}\label{j50}
M_2(r,g)\asymp\left( \int_{r}^1
    \om(r)\,dr\right)^{-\frac{1}{\a}}, \quad r\ge r_0,
\end{equation}
which together Lemma~\ref{le:RAp}(iii) for $\frac{p}{\a}<1$
implies $g\in A^p_\om$.\index{$\widetilde{\om}(r)$}

Since  $g$ is a lacunary series, arguing as in the proof of
\cite[Theorem~6]{GNW}, we deduce
    \begin{equation}\label{Eq:obsT(r,f)}
    \frac{1}{2\pi}\int_{0}^{2\pi}\log^{+}|g(re^{i\theta})|\,d\theta\gtrsim h(r),\quad r\in
    [r_1,1),
    \end{equation}
for some $r_1\in (0,1)$. This together with \cite[Theorem on
p.~276]{Nevanlinnabook} implies that there exists
$a\in\mathbb{C}$ such that
    \begin{equation}\label{j23}
    N(r,g-a)\gtrsim h(r),\quad r_2\le r<1,
    \end{equation}
for some $r_2\in[r_1,1)$. Therefore $f=g-a\in A^p_\om$ has the
desired properties. Finally, we will prove \eqref{j50}. We begin
with proving \eqref{j50} for $r=r_N$, where $N\ge n_0$. To do
this, note first that
    \begin{equation}\label{serie2}
    \begin{split}
    \sum_{n=n_0}^N2^{2n}r_N^{2E\left(\frac{1}{1-r_n}\right)}
    \le\sum_{n=0}^N2^{2n}
    \le\frac{4}{3}2^{2N}
    =\frac{4}{3}\left( \int_{r_N}^1
    \om(r)\,dr\right)^{-\frac{2}{\a}}.
    \end{split}
    \end{equation}
To deal with the remainder of the sum, we observe that \eqref{j51}
implies
    $$
    \frac{1-r_n}{1-r_{n+j}}\ge \left(\frac{\int_{r_n}^1\om(r)\,dr}{\int_{r_{n+j}}^1\om(r)\,dr}\right)^{1/\b}
    = 2^{\frac{j\a}{\b}},\quad n\ge n_0.
    $$
This and the inequality $1-r\le\log\frac1r$ give
    \begin{equation*}
    \begin{split}
    \sum_{n=N+1}^\infty 2^{2n}r_N^{E\left(\frac{1}{1-r_n}\right)}
    &\le 2^{2N}\sum_{j=1}^\infty 2^{2j}e^{-C\frac{1-r_N}{1-r_{N+j}}}
    \le 2^{2N}\sum_{j=1}^\infty 2^{2j}e^{-C2^{\frac{j\a}{\b}}}\\
    &=C(\b,\a,\om)\left( \int_{r_N}^1
    \om(s)\,ds\right)^{-\frac{2}{\a}}.
    \end{split}
    \end{equation*}
Since $\b=\b(\om)$, this together with \eqref{serie2} yields
    \begin{equation}\label{1111}
    M_2^2(r_N,g)\lesssim\left(\int_{r_N}^1
    \om(r)\,dr\right)^{-\frac{2}{\a}}, \quad N\ge n_0.
    \end{equation}

Let now $r\in[r_{n_0},1)$ be given and fix $N\ge n_0$ such that
$r_N\le r< r_{N+1}$. Then, by \eqref{1111}, there exists
$C=C(\a,\om)$ such that
    \begin{equation*}
  M^2_2(r,g)  \le M^2_2(r_{N+1},g)\le C2^{2}2^{2N}\le C\left( \int_{r}^1  \om(s)\,ds\right)^{-\frac{2}{\a}}.
    \end{equation*}
 Further,
    \begin{equation*}
    \begin{split}
    M_2^2(r_N,g) &\ge
    \sum_{n=n_0}^N2^{2n}r_N^{2E\left(\frac{1}{1-r_n}\right)}
    \ge C 2^{2N}
    \ge C
    \left( \int_{r_N}^1
    \om(r)\,dr\right)^{-\frac{2}{\a}}
    \end{split}
    \end{equation*}
for all $N\ge n_0$. Now for a given $r\in[r_{n_0},1)$, choose
$N\in\N$ such that $r_N\le r<r_{N+1}$. It follows that
    \begin{equation*}
    \begin{split}
    M_2^2(r,g)& \ge  M_2^2(r_N,g)
    \ge C 2^{2N+2}
   \ge C
    \left( \int_{r_{N+1}}^1
    \om(r)\,dr\right)^{-\frac{2}{\a}}\ge  \left( \int_{r}^1
    \om(r)\,dr\right)^{-\frac{2}{\a}},
    \end{split}
    \end{equation*}
    which gives \eqref{j50}.
\end{proof}

The proof of Theorem~\ref{th:subsblaschkecond}(ii) has a
consequence which is of independent interest. To state it, we
recall two things. First, the \emph{proximity function} of an
analytic (or meromorphic) function $g$ in $\D$ is defined as
    $$\index{proximity function}\index{$m(r,f)$}
    m(r,g)=\frac{1}{2\pi}\int_{0}^{2\pi}\log^{+}|g(re^{i\theta})|\,d\theta,\quad
    0<r<1.
    $$
This is the function appearing in \eqref{Eq:obsT(r,f)}. Second,
each convex function $\Phi$ can be written as
$\Phi(r)=\int_0^r\phi(t)\,dt$, where $\phi$ is continuous
increasing unbounded function, see, for example,
\cite[p.~24]{Zygmund59}.

\begin{theorem}\label{Thm:T(r,f)}
Let $\Phi(r)=\int_0^r\phi(t)\,dt$ be an increasing unbounded
convex function on $(0,1)$ such that either $\phi(r)(1-r)\asymp1$
or $\lim_{r\to1^-}\phi(r)(1-r)=0$. Then there exists $g\in\H(\D)$
such that
    $$
    \log M_\infty(r,g)\asymp m(r,g)\asymp\Phi(r),\quad r\to1^-.
    $$
\end{theorem}

\begin{proof}
For a given $\om\in\I\cup\R$, the function $g\in\H(\D)$
constructed in the proof of Theorem~\ref{th:subsblaschkecond}
satisfies
    $$
    m(r,g)\asymp\log\frac{1}{\int_r^1\om(s)\,ds},\quad r\to1^-.
    $$
Namely, by \eqref{Eq:obsT(r,f)} the growth of $m(r,g)$ has the
correct lower bound, and the same upper bound follows by Jensen's
inequality and \eqref{j50}. Moreover, \eqref{j50} and a minor
modification in its proof show that
    $$
    M_\infty(r,g)\asymp\frac{1}{\left(\int_r^1\om(s)\,ds\right)^\frac1\a},\quad
    r\to1^-.
    $$
If we now take
    $$
    \om(r)=\phi(r)\exp\left(-\int_0^r\phi(t)\,dt\right),
    $$
then $\psi_\om(r)=(\phi(r))^{-1}$ and
    $$
    \log\frac{1}{\int_r^1\om(s)\,ds}=\int_0^r\phi(t)\,dt=\Phi(r).
    $$
Therefore $\om\in\I\cup\R$ by the assumptions, and $g$ has the
desired properties.
\end{proof}

Theorem~\ref{Thm:T(r,f)} shows that we can find $g\in\H(\D)$ such
that $m(r,g)$ grows very slowly. For example, by choosing
$\om\in\I$ appropriately, there exists $g\in\H(\D)$ such that
    $$
    m(r,g)\asymp \log\log\log\frac{1}{1-r},\quad r\to1^-.
    $$

If the function $\Phi$ in Theorem~\ref{Thm:T(r,f)} exceeds
$\log\frac{1}{1-r}$ in growth, then the existence of $g$ such that
$m(r,g)\asymp\Phi(r)$ was proved by Shea~\cite[Theorem~1]{Shea},
see also related results by Clunie~\cite{Clunie} and
Linden~\cite{Linden}.

Typical examples of functions $\tau$ satisfying \eqref{85} are
    $$
    \tau_{N,\e}(x)=x\prod_{n=1}^{N-1}\log_n(\exp_n1+x)(\log_N(\exp_N1+x))^{1+\e},\quad
    N\in\N,\quad \e>0.
    $$
To see a concrete example, consider $\om=v_\a$\index{$v_\a(r)$}
and $\tau=\tau_{1,\e}$, where $1<\a<\infty$ and $\e>0$. If $f\in
A^p_{v_\alpha}$ and $\{z_k\}$ is its ordered sequence of zeros,
then
    $$
    \sum_{k}\frac{1-|z_k|}{\log\log\frac{\exp_21}{1-|z_k|}\left(\log\log\log\frac{\exp_31}{1-|z_k|}\right)^{1+\e}}<\infty
    $$
for all $\e>0$. The functions $\tau_{N,0}$ do not satisfy
\eqref{85}, and therefore there exists $f\in
A^p_{v_\alpha}$\index{$v_\a(r)$} such that its ordered sequence of
zeros $\{z_k\}$ satisfies
    $$
    \sum_{k}\frac{1-|z_k|}{\log\log\frac{\exp_21}{1-|z_k|}\cdot\log\log\log\frac{\exp_31}{1-|z_k|}}=\infty.
    $$
Whenever $\om\in\I\cup\R$, Theorem~\ref{th:subsblaschkecond} is
the best we can say about the zero distribution of functions in
$A^p_\om$ in terms of conditions depending on their moduli only.

\section{Zeros of functions in the Bergman-Nevanlinna class $\BN_\om$}\label{Sec:ZerosBergman-Nevanlinna}

We now turn back to consider the condition
\eqref{Eq:ZerosInGrowthSpaces} for further reference. It is not
hard to find a space of analytic functions on $\D$ for which the
zero sets are characterized by this neat condition. To give the
precise statement, we say that $f\in\H(\D)$ belongs to the
\emph{Bergman-Nevanlinna class}
$\BN_\om$,\index{$\BN_\om$}\index{Bergman-Nevanlinna class} if
    $$
    \|f\|_{\BN_\om}=\int_\D\log^+|f(z)|\om(z)\,dA(z)<\infty.\index{$\Vert \cdot\Vert_{\BN_\om}$}
    $$

\begin{proposition}\label{Prop:bergman-nevanlinna}\index{$\BN_\om$}
Let $\om\in\I\cup\R$. Then $\{z_k\}$ is a $\BN_\om$-zero set if
and only if
    \begin{equation}\label{70}
    \sum_k\om^\star(z_k)<\infty.
    \end{equation}
\end{proposition}

\begin{proof}
If $f\in\BN_\om$ and $\{z_k\}$ is its zero set, then by
integrating Jensen's formula we obtain
$2\sum_k\om^\star(z_k)\le\|f\|_{\BN_\om}$. Conversely, let
$\{z_k\}$ be a sequence such that~\eqref{70} is satisfied. By
following the construction in \cite[p.~131--132]{HKZ} and using
Lemma~\ref{Lemma:Zhu-type}, we find $f\in\BN_\om$ such that it
vanishes at the points $z_k$ and nowhere else. We omit the
details.
\end{proof}

Lemma~\ref{Lemma:NBzeros} shows that the condition \eqref{70} is
equivalent to the convergence of certain integrals involving the
counting functions. We will use this result in
Section~\ref{Sec:SolutionsBergmanNevanlinna} when studying linear
differential equations with solutions in Bergman-Nevanlinna
classes.

\begin{lemma}\label{Lemma:NBzeros}\index{$\widehat{\om}(r)$}
Let $\om$ be a radial weight and denote
$\widehat{\om}(r)=\int_r^1\om(s)\,ds$. Let $f\in\H(\D)$ and let
$\{z_k\}$ be its zero sequence. Then the following conditions are
equivalent:
\begin{itemize}
\item[\rm(1)]
$\displaystyle\sum_k\om^\star(z_k)<\infty$;\item[\rm(2)]
$\displaystyle\int_0^1N(r,f)\,\om(r)\,dr<\infty$; \item[\rm(3)]
$\displaystyle\int_0^1n(r,f)\,\widehat{\om}(r)\,dr<\infty$.
\end{itemize}
\end{lemma}

\begin{proof}
Let $\rho\in(\rho_0,1)$, where $\rho_0\in(0,1)$ is fixed. An
integration by parts gives
    \begin{equation}\label{81}
    \begin{split}
    \int_{\rho_0}^\rho
    N(r,f)\om(r)\,dr&=-N(\rho,f)\widehat{\om}(\rho)+N(\rho_0,f)\widehat\om(\rho_0)\\
    &\quad+\int_{\rho_0}^\rho\frac{n(r,f)}{r}\widehat\om(r)\,dr.
    \end{split}
    \end{equation}
If (2) is satisfied, then $N(\rho,f)\widehat{\om}(\rho)\to0$, as
$\rho\to1^-$. Therefore \eqref{81} shows that (2) and (3) are
equivalent. Moreover,
    \begin{equation}\label{82}
    \begin{split}
    \int_{\rho_0}^\rho
    n(r,f)\widehat\om(r)\,dr&=-n(\rho,f)\int_{\rho}^1\widehat{\om}(s)\,ds+n(\rho_0,f)\int_{\rho_0}^1\widehat\om(s)\,ds\\
    &\quad+\sum_{\rho_0<|z_k|<\rho}\int_{|z_k|}^1\widehat\om(s)\,ds,
    \end{split}
    \end{equation}
where
    $$
    \int_{|z_k|}^1\widehat\om(s)\,ds=\int_{|z_k|}^1\om(t)(t-|z_k|)\,dt\asymp\om^\star(z_k),\quad
    |z_k|\to1^-.
    $$
We deduce from \eqref{82} that also (1) and (3) are equivalent.
\end{proof}

By Proposition~\ref{Prop:bergman-nevanlinna} the condition
\eqref{Eq:ZerosInGrowthSpaces} characterizes zero sets in
$\BN_{v_{2+\e}}$.
Proposition~\ref{Prop:bergman-nevanlinna}\index{$\BN_\om$} further
implies that for each increasing differentiable unbounded function
$\lambda:[0,1)\to(0,\infty)$ such that
$\om_\lambda=\lambda'/\lambda^2\in\I\cup\R$, the condition
    $$
    \sum_k\frac{1-|z_k|}{\lambda(z_k)}<\infty
    $$
characterizes the zero sets in $\BN_{\om_\lambda}$. This because
$\om^\star_\lambda(r)\asymp\frac{1-r}{\lambda(r)}$ by
Lemma~\ref{le:cuadrado-tienda}.

\chapter{Integral Operators and Equivalent
Norms}\label{SecVolterra}

One of our main objectives is to characterize those symbols
$g\in\H(\D)$ such that the integral operator
    \begin{displaymath}
    T_g(f)(z)=\int_{0}^{z}f(\zeta)\,g'(\zeta)\,d\zeta,\quad
    z\in\D,
    \end{displaymath}
is bounded or compact from $A^p_\om$ to $A^q_\om$, when $\om$ is a
rapidly increasing weight. Pommerenke was probably one of the
first authors to consider the operator $T_g$. He used it
in~\cite{Pom} to study the space $\BMOA$, \index{$\BMOA$} which
consists of those functions in the Hardy space $H^1$ that have
\emph{bounded mean oscillation} \index{bounded mean oscillation}
on the boundary $\T$~\cite{Ba86(2),GiBMO}. The space
$\BMOA$\index{$\BMOA$}\index{bounded mean oscillation} can be
equipped with several different equivalent norms~\cite{GiBMO}. We
will use the one given by\index{$\Vert\cdot\Vert_{\BMOA}$}
    $$
    \|g\|^2_{\BMOA}=\sup_{a\in\D}\frac{\int_{S(a)}|g'(z)|^2(1-|z|^2)\,dA(z)}{1-|a|}+|g(0)|^2.\index{$\|\cdot\|_{\BMOA}$}
    $$
The operator $T_g$ has received different names in the literature:
the Pommerenke operator, a Volterra type operator (the choice
$g(z)=z$ gives the usual Volterra operator\index{Volterra
operator}), the generalized Ces\`{a}ro operator\index{Ces\`{a}ro
operator} (the usual Ces\`{a}ro operator is obtained when
$g(z)=-\log(1-z)$), a Riemann-Stieltjes type operator, or simply
an integral operator. The operator $T_g$ began to be extensively
studied after the appearance of the works by Aleman and
Siskakis~\cite{AS0,AS}, see also \cite{AHpreview, Sisreview} for
two surveys on the topic. If $\om$ is either regular or rapidly
decreasing,\index{rapidly decreasing weight} then a description of
those $g\in\H(\D)$ for which $T_g:\,A^p_\om\to A^q_\om$ is bounded
or compact can be found in~\cite{AlCo,AS,PP,PPVal}. In particular,
for these weights the boundedness of $T_g:A^p_\om\to A^p_\om$ is
equivalent to a growth condition on $M_\infty(r,g')$. It follows
from Lemma~\ref{le:RAp}(i) and \cite[Theorem~4.1]{AlCo} that if
$\om$ is regular, then $T_g:A^p_\om\to A^p_\om$ is bounded if and
only if $g$ is a Bloch function. Recall that the \emph{Bloch
space}\index{Bloch space} $\B$~\cite{ACP}\index{$\B$} consists of
those $f\in\H(\D)$ such that\index{$\Vert\cdot\Vert_{\B}$}
    $$
    \|f\|_{\B}=\sup_{z\in\D}|f'(z)|(1-|z|^2)+|f(0)|<\infty.\index{$\|\cdot\|_{\B}$}
    $$
We will see soon that if $\om$ is rapidly increasing, then the
boundedness of $T_g:A^p_\om\to A^p_\om$ can not be characterized
by a simple condition depending only on $M_\infty(r,g')$. The
above-mentioned results on the boundedness and compactness of
$T_g:A^p_\om\to A^q_\om$ are usually discovered by the aid of a
Littlewood-Paley type formula\index{Littlewood-Paley} that allows
one to get rid of the integral appearing from the definition of
$T_g$. However, this kind of approach does not work for $A^p_\om$
when $\om\in\I$, because such a Littlewood-Paley type formula for
$A^p_\om$, with $\om\in\I$ and $p\ne2$, does not exist by
Proposition~\ref{pr:NOL-P}. Therefore we are forced to search for
alternative norms in $A^p_\om$ in terms of the first derivative.
This leads us to employ ideas that are closely related to the
theory of Hardy spaces~\cite{AC,AHpreview,AS0}. It is also worth
noticing that if $0<p\le1$ and $\om\in\I$, then the Riesz
projection is not necessarily bounded on $L^p_\om$.

We say that $g\in\H(\D)$ belongs to $\CC^{q,\,p}(\om^\star)$,
$0<p,q<\infty$, if the measure $|g'(z)|^2\om^\star(z)\,dA(z)$ is a
$q$-Carleson measure for $A^p_\om$. Moreover,
$g\in\CC^{q,\,p}_0(\om^\star)$ if the identity operator
$I_d:A^p_\omega\to L^q(|g'|^2\omega^\star dA)$ is compact. If
$q\ge p$ and $\om\in\I\cup\R$, then Theorem~\ref{th:cm} shows that
these spaces only depend on the quotient $\frac{q}{p}$.
Consequently, for $q\ge p$ and $\om\in\I\cup\R$, we simply write
$\CC^{q/p}(\om^\star)$ instead of $\CC^{q,\,p}(\om^\star)$. Thus,
if $\alpha\ge 1$ and $\om\in\I\cup\R$, then
$\CC^{\alpha}(\om^\star)$ consists of those $g\in\H(\D)$ such
that\index{$\Vert\cdot\Vert_{\CC^{\alpha}(\om^\star)}$}
    \begin{equation}\label{calpha}
    \|g\|^2_{\CC^{\alpha}(\om^\star)}=|g(0)|^2+\sup_{I\subset\T}\frac{\int_{S(I)}|g'(z)|^2\om^\star(z)\,dA(z)}
    {\left(\om\left(S(I)\right)\right)^{\alpha}}<\infty.\index{$\CC^\a(\om^\star)$}
    \end{equation}
An analogue of this identity is valid for the little space
$\CC^{\alpha}_0(\om^\star)$.  The above characterization of the
Banach space
$\left(\CC^{\alpha}(\om^\star),\|\cdot\|_{\CC^{\alpha}(\om^\star)}
\right)$ has all the flavor of known characterizations of
$\BMOA$~\cite{GiBMO},\index{$\BMOA$} the Lipschitz class
$\Lambda_\alpha$~\cite{AC} or the Bloch space
$\B$~\cite{X2}\index{$\B$} when $\a$ is chosen appropriately. In
fact, Proposition~\ref{pr:blochcpp} shows that
$\CC^1(\om^\star)=\B$\index{$\CC^1(\om^\star)$} for all
$\om\in\R$, and
    \begin{equation}\label{42}
    \BMOA\subset\CC^1(\om^\star)\subset\B,\quad \om\in\I,\index{$\CC^1(\om^\star)$}
    \end{equation}
where both embeddings can be strict at the same time.
Unlike~$\B$,\index{$\B$} the space
$\CC^1(\om^\star)$\index{$\CC^1(\om^\star)$} can not be described
by a simple growth condition on the maximum modulus of $g'$ if
$\om\in\I$. This follows by Proposition~\ref{pr:blochcpp} and the
fact that $\log(1-z)\in A^p_\om$ for all $\om\in\I$. However, if
$\alpha>1$ and $\om\in\I$, then $g\in\CC^{\alpha}(\omega^\star)$
if and only if
    $$
    M_\infty(r,g')\lesssim\frac{(\omega^\star(r))^{\frac{\alpha-1}{2}}}{1-r},\quad
    0<r<1,
    $$
by Proposition~\ref{PropRadialq>p}.

The spaces $\BMOA$\index{$\BMOA$} and $\B$\index{$\B$} are
conformally invariant. This property has been used, among other
things, in describing those symbols $g\in\H(\D)$ for which $T_g$
is bounded on $H^p$ or $A^p_\alpha$. However, the space
$\CC^1(\omega^\star)$\index{$\CC^1(\om^\star)$} is not necessarily
conformally invariant by
Proposition~\ref{PropConformallyInvariant}, and therefore
different techniques must be employed in the case of $A^p_\om$
with $\om\in\I$.

The next result gives a complete characterization of when
$T_g:A^p_\om\to A^q_\omega$ is bounded, provided
$\om\in\widetilde{\I}\cup\R$.

\begin{theorem}\label{Thm-integration-operator-1}
Let $0<p,q<\infty$, $\om\in\I\cup\R$ and $g\in\H(\D)$.
    \begin{itemize}
    \item[\rm(i)] The following conditions are equivalent:
    \begin{enumerate}
    \item[\rm(ai)]\, $T_g:A^p_\om\to A^p_\om$ is bounded;
    \item[\rm(bi)]\,
    $g\in\CC^1(\om^\star)$.\index{$\CC^1(\om^\star)$}
    \end{enumerate}
    \item[\rm(ii)] If $0<p<q$ and $\frac1p-\frac1q<1$, then the
following conditions are equivalent:
    \begin{enumerate}
    \item[\rm(aii)]\,$T_g:A^p_\om\to A^q_\om$ is bounded;
    \item[\rm(bii)] $\displaystyle M_\infty(r,g')\lesssim\frac{(\omega^\star(r))^{\frac1p-\frac1q}}{1-r},\quad
    r\to1^-;
    $
    \item[\rm(cii)]
    $g\in\CC^{2\left(\frac{1}{p}-\frac1q\right)+1}(\omega^\star)$.
    \end{enumerate}
    \item[\rm(iii)] If $\frac1p-\frac1q\ge 1$, then $T_g:A^p_\om\to A^q_\om$ is bounded if and only if $g$ is constant.
    \item[\rm(iv)] If $0<q<p<\infty$ and $\om\in\widetilde{\I}\cup\R$,
    then the following conditions are equivalent:
    \begin{enumerate}
    \item[\rm(aiv)]\,$T_g:A^p_\om\to A^q_\om$ is bounded;
    \item[\rm(biv)] $g\in A^s_\omega$, where $\frac{1}{s}=\frac1q-\frac1p$.
    \end{enumerate}
    \end{itemize}
\end{theorem}

It is worth noticing that the regularity assumption \eqref{eq:r2},
that all the weights in $\widetilde{\I}\cup\R$ satisfy, is only
used in the proof of Theorem~\ref{Thm-integration-operator-1} when
showing that (aiv)$\Rightarrow$(biv), the rest of the proof is
valid under the hypothesis $\om\in\I\cup\R$.

A description of the symbols $g\in\H(\D)$ such that
$T_g:A^p_\om\to A^q_\om$ is compact is given in
Theorem~\ref{Thm-integration-operator-2}.

In Section~\ref{Hardy} we will show how the techniques developed
on the way to the proof of
Theorem~\ref{Thm-integration-operator-1} can be applied to the
case when $T_g$ acts on the Hardy spaces. In particular, we will
give a new proof for the fact that $T_g:H^p\to H^p$ is bounded if
and only if $g\in\BMOA$.

\section{Equivalent norms on $A^p_\om$}\label{ssLP}

For a large class of radial weights, which includes any
differentiable decreasing weight and all the standard ones, the
most appropriate way to obtain a useful norm involving the first
derivative is to establish a kind of
Littlewood-Paley\index{Littlewood-Paley} type formula~\cite{PavP}.
However, if $\om\in\I$ and $p\neq 2$, this is not possible in
general by Proposition~\ref{pr:NOL-P} below. Consequently, we will
equip the space $A^p_\om$ with other norms that are inherited from
different equivalent $H^p$ norms. The one that we will obtain by
using square functions via the classical Fefferman-Stein
estimate~\cite{FC} appears to be the most useful for our purposes.
The square function that arises naturally is obtained by
integrating over the lens type region $\Gamma(u)$, see
\eqref{eq:gammadeu} for the definition, and also
Picture~\ref{fig:2}.

\begin{theorem}\label{ThmLittlewood-Paley}\index{$\omega^\star(z)$}
Let $0<p<\infty$, $n\in\N$ and $f\in\H(\D)$, and let $\omega$ be a
radial weight. Then
    \begin{equation}\label{HSB}
    \|f\|_{A^p_\omega}^p=p^2\int_{\D}|f(z)|^{p-2}|f'(z)|^2\omega^\star(z)\,dA(z)+\omega(\D)|f(0)|^p,
    \end{equation}
    and
\begin{equation}\label{normacono}
    \begin{split}
    \|f\|_{A^p_\omega}^p&\asymp\int_\D\,\left(\int_{\Gamma(u)}|f^{(n)}(z)|^2
    \left(1-\left|\frac{z}{u}\right|\right)^{2n-2}\,dA(z)\right)^{\frac{p}2}\omega(u)\,dA(u)\\
    &\quad+\sum_{j=0}^{n-1}|f^{(j)}(0)|^p,
    \end{split}
    \end{equation}
where the constants of comparison depend only on $p$, $n$ and $\om$. In particular,
    \begin{equation}\label{eq:LP2}
    \|f\|_{A^2_\omega}^2=4\|f'\|_{A^2_{\omega^\star}}^2+\omega(\D)|f(0)|^2.
    \end{equation}
\end{theorem}

\begin{proof}
Hardy-Stein-Spencer
identity~\cite{Garnett1981}\index{Hardy-Stein-Spencer identity}
    $$
    \|f\|_{H^p}^p=\frac{p^2}{2}\int_\D|f(z)|^{p-2}|f'(z)|^2\log\frac{1}{|z|}\,dA(z)+|f(0)|^p,
    $$
applied to $f_r(z)=f(rz)$, and Fubini's theorem yield
    \begin{equation*}
    \begin{split}
    \|f\|_{A^p_\omega}^p&=\int_\D|f(z)|^p\omega(z)\,dA(z)=2\int_0^1\|f_r\|_{H^p}^p\omega(r)r\,dr\\
    &=p^2\int_0^1\left(\int_\D|f(rz)|^{p-2}|f'(rz)|^2r^2\log\frac{1}{|z|}\,dA(z)\right)\omega(r)r\,dr\\
    &\quad+\omega(\D)|f(0)|^p\\
    &=p^2\int_0^1\left(\int_0^r\int_0^{2\pi}|f(se^{i\t})|^{p-2}|f'(se^{i\t})|^2d\theta\log\frac{r}{s}\,s\,ds\right)\omega(r)r\,dr\\
    &\quad+\omega(\D)|f(0)|^p\\
    &=p^2\int_0^1\int_0^{2\pi}|f(se^{i\t})|^{p-2}|f'(se^{i\t})|^2d\theta\left(\int_s^1\log\frac{r}{s}\,\omega(r)r\,dr\right)s\,ds\\
    &\quad+\omega(\D)|f(0)|^p\\
    &=p^2\int_\D|f(z)|^{p-2}|f'(z)|^2\omega^\star(z)\,dA(z)+\omega(\D)|f(0)|^p,
    \end{split}
    \end{equation*}
which proves~\eqref{HSB}. In particular, the
identity~\eqref{eq:LP2} is valid.

Analogously, the extension of the
Littlewood-Paley\index{Littlewood-Paley} identity to any $H^p$,
proved by Ahern and Bruna~\cite{AhernBruna1988}, states that
    \begin{equation}\label{eq:FC}
    \begin{split}
    \|f\|_{H^p}&\asymp\int_\T\,\left(\int_{\Gamma(e^{i\theta})}|f^{(n)}(z)|^2(1-|z|)^{2n-2}\,dA(z)\right)^{p/2}d\theta
    +\sum_{j=0}^{n-1}|f^{(j)}(0)|,
    \end{split}
   \end{equation}
see also the earlier results in \cite{FC} and \cite[Vol~II.
Chapter~14]{Zygmund59}. This and the change of variable $rz=\xi$
give
    \begin{equation*}
    \begin{split}
    \|f\|_{A^p_\omega}^p&=2\int_0^1\|f_r\|_{H^p}^p\omega(r)r\,dr\\
    &\asymp\int_0^1\left(\int_\T\,\left(\int_{\Gamma(\z)}|f^{(n)}(rz)|^2r^{2n}(1-|z|)^{2n-2}\,dA(z)\right)^{p/2}|d\z|\right)
    \omega(r)r\,dr\\
    &\quad+\sum_{j=0}^{n-1}|f^{(j)}(0)|^p\\
    &=\int_0^1\int_\T\,\left(\int_{\Gamma(r\z)}|f^{(n)}(\xi)|^2r^{2n-2}
    \left(1-\left|\frac{\xi}{r}\right|\right)^{2n-2}\,dA(\xi)\right)^{p/2}|d\z|\,\omega(r)r\,dr\\
    &\quad+\sum_{j=0}^{n-1}|f^{(j)}(0)|^p\\
    &\asymp
    \int_\D\,\left(\int_{\Gamma(u)}|f^{(n)}(\xi)|^2\left(1-\left|\frac{\xi}{u}\right|\right)^{2n-2}\,dA(\xi)\right)^{p/2}\omega(u)\,dA(u)\\
    &\quad+\sum_{j=0}^{n-1}|f^{(j)}(0)|^p,
    \end{split}
    \end{equation*}
and thus also \eqref{normacono} is valid.
\end{proof}

Theorem~\ref{ThmLittlewood-Paley} is of very general nature
because $\om$ is only assumed to be radial. In particular, the
Littlewood-Paley formula \eqref{eq:LP2}\index{Littlewood-Paley} is
valid for all radial weights~$\om$. The next result shows that no
asymptotic Littlewood-Paley formula can be found for $A^p_\om$
when $\om\in\I$ and $p\ne2$.

\begin{proposition}\label{pr:NOL-P}
Let $p\ne2$. Then there exists $\om\in\I$ such that, for any
function $\vp:[0,1)\to(0,\infty)$, the relation
\begin{equation}\label{eq:NOL-P}
    \|f\|^p_{A^p_{\om}}\asymp
    \int_\D|f'(z)|^p\varphi(|z|)^p\om(z)\,dA(z)+|f(0)|^p
    \end{equation}
can not be valid for all $f\in\H(\D)$.
\end{proposition}

\begin{proof} Let first $p>2$ and consider the weight $v_\a$,\index{$v_\a(r)$} where $\alpha$ is fixed
such that $2<2(\alpha-1)\le p$. Assume on the contrary to the
assertion that \eqref{eq:NOL-P} is satisfied for all $f\in\H(\D)$.
Applying this relation to the function $h_n(z)=z^n$, we obtain
    \begin{equation}\label{eq:NOL-P1}
    \int_0^1r^{np}v_\a(r)\,dr\asymp
    n^p\int_0^1r^{(n-1)p}\varphi(r)^pv_\a(r)\,dr,\quad
    n\in\N.
    \end{equation}
Consider now the lacunary series $h(z)=\sum_{k=0}^\infty z^{2^k}$.
It is easy to see that
    \begin{equation}\label{1}
    M_p(r,h)\asymp\left(\log\frac{1}{1-r}\right)^{1/2},\quad
    M_p(r,h')\asymp\frac{1}{1-r},\quad 0\le r<1.
    \end{equation}
By combining the relations \eqref{eq:NOL-P1}, \eqref{1} and
    $$
    \left(\frac{1}{1-r^p}\right)^p\asymp\sum_{n=1}^\infty
    n^{p-1}r^{(n-1)p},\quad\log\frac{1}{1-r^p}\asymp\sum_{n=1}^\infty
    n^{-1}r^{np},\quad0\le r<1,
    $$
we obtain
    \begin{equation*}\begin{split}
    \int_\D|h'(z)|^p\varphi(z)^pv_\a(z)\,dA(z) &\asymp
    \int_0^1\left(\frac{1}{1-r^p}\right)^p\varphi(r)^pv_\a(r)\,dr \\ &
    \asymp \int_0^1\left(\sum_{n=1}^\infty
    n^{p-1}r^{(n-1)p}\right)\varphi(r)^pv_\a(r)\,dr\\
    &\asymp \sum_{n=1}^\infty n^{p-1}\int_0^1
    r^{(n-1)p}\varphi(r)^pv_\a(r)\,dr\\
    &\asymp \sum_{n=1}^\infty n^{-1}\int_0^1
    r^{np}v_\a(r)\,dr\\
    &\asymp \int_0^1\left(\sum_{n=1}^\infty n^{-1}r^{np}\right)v_\a(r)\,dr\\
    &\asymp \int_0^1\log\frac{1}{1-r^p}\,v_\a(r)\,dr,
    \end{split}\end{equation*}
where the last integral is convergent because $\a>2$. However,
    $$
    \|h\|^p_{A^p_{v_\a}}\asymp\int_0^1\left(\log\frac{1}{1-r}\right)^{p/2}v_\a(r)\,dr=\infty,
    $$
since $p\ge2(\alpha-1)$, and therefore \eqref{eq:NOL-P} fails for
$h\in\H(\D)$. This is the desired contradiction.

If $0<p<2$, we again consider $v_\a$,\index{$v_\a(r)$} where $\a$
is chosen such that $p<2(\alpha-1)\le2$, and use an analogous
reasoning to that above to prove the assertion. Details are
omitted.
\end{proof}

Let $f\in\H(\D)$, and define the \emph{non-tangential maximal
function} of $f$ in the (punctured) unit disc by
    $$\index{$N(f)$}\index{non-tangential maximal function}
    N(f)(u)=\sup_{z\in\Gamma(u)}|f(z)|,\quad
    u\in\D\setminus\{0\}.
    $$
The following equivalent norm will be used in the proof of
Theorem~\ref{Thm-integration-operator-1}.

\begin{lemma}\label{le:funcionmaximalangular}
Let $0<p<\infty$ and let $\om$ be a radial weight. Then there
exists a constant $C>0$ such that
    $$
    \|f\|^p_{A^p_\om}\le\|N(f)\|^p_{L^p_\om}\le C\|f\|^p_{A^p_\om}
    $$
for all $f\in\H(\D)$.\index{$N(f)$}\index{non-tangential maximal
function}
\end{lemma}

\begin{proof}
It follows from \cite[Theorem~3.1 on p.~57]{Garnett1981} that
there exists a constant $C>0$ such that the \emph{classical
non-tangential maximal function}\index{classical non-tangential
maximal function}
    $$
    f^\star(\z)=\sup_{z\in\Gamma(\z)}|f(z)|\index{$f^\star$},\quad
    \z\in\T,
    $$
satisfies
    \begin{equation}\label{Eq:Maximal}\index{$f^\star$}
    \|f^\star\|^p_{L^p(\T)}\le C\|f\|^p_{H^p}
    \end{equation}
for all $0<p<\infty$ and $f\in\H(\D)$. Therefore
    \begin{equation*}\index{$N(f)$}\index{non-tangential maximal function}
    \begin{split}
    \|f\|^p_{A^p_\om}&\le\|N(f)\|^p_{L^p_\om}=\int_\D(N(f)(u))^p\om(u)\,dA(u)\\
    &=\int_0^1\om(r)r\int_{\T}((f_r)^\star(\z))^p\,|d\z|\,dr\\
    &\le C\int_0^1\om(r)r\int_{\T}f(r\z)^p\,|d\z|\,dr
    =C\|f\|^p_{A^p_\om},
    \end{split}
    \end{equation*}
and the assertion is proved.
\end{proof}

With these preparations we are now in position to prove our main
results on the boundedness and compactness of the integral
operator $T_g:A^p_\om\to A^q_\om$ with $\om\in\I\cup\R$.

\section{Integral operator $T_g$ on the weighted Bergman space $A^p_\om$}\label{sec:Volterra}

We begin with the following lemma which says, in particular, that
the symbol $g\in\H(\D)$ must belong to the Bloch space
$\B$\index{$\B$} whenever the operator $T_g$ is bounded on
$A^p_\om$ and $\om\in\I\cup\R$.

\begin{lemma}\label{le:tgminfty}
Let $0<p,q<\infty$ and $\omega\in\I\cup\R$.
\begin{itemize}
\item[\rm(i)] If $T_g:A^p_\omega\to A^q_\omega$ is bounded, then
\begin{equation}\label{Eq:Radialqp}
    M_\infty(r,g')\lesssim\frac{(\omega^\star(r))^{\frac{1}{p}-\frac{1}{q}}}{1-r},\quad
    0<r<1.
    \end{equation}
\item[\rm(ii)] If $T_g:A^p_\omega\to A^q_\omega$ is compact, then
    \begin{equation}\label{Eq:Radialqp2}
    M_\infty(r,g')=\op\left(\frac{(\omega^\star(r))^{\frac{1}{p}-\frac{1}{q}}}{1-r}\right),\quad
    r\to 1^{-}.
    \end{equation}
\end{itemize}
\end{lemma}

\begin{proof}
(i) Let $0<p,q<\infty$ and $\omega\in\I\cup\R$, and assume that
$T_g:A^p_\omega\to A^q_\omega$ is bounded. Consider the functions
    $$
    f_{a,p}(z)=\frac{(1-|a|)^{\frac{\gamma+1}{p}}}{(1-\overline{a}z)^{\frac{\gamma+1}{p}}\om\left(S(a)\right)^{\frac1p}},\quad
    a\in\D,\index{$f_{a,p}$}
    $$
defined in~\eqref{testfunctions}. By choosing $\gamma>0$ large
enough and arguing as in the proof of Theorem~\ref{th:cm}(ii), we
deduce
$\sup_{a\in\D}\|f_{a,p}\|_{A^p_\om}<\infty$\index{$f_{a,p}$} and
also $f_{a,p}\to0$ uniformly on compact subsets of $\D$ as $|a|\to
1^-$. Since
    \begin{equation*}
    \begin{split}
    \|h\|_{A^q_\omega}^q&\ge\int_{\D\setminus
    D(0,r)}|h(z)|^q\omega(z)\,dA(z)
    \gtrsim M_q^q(r,h)\int_r^1\omega(s)\,ds,\quad r\ge\frac12,
    \end{split}
    \end{equation*}
for all $h\in A^q_\omega$, we obtain
    \begin{equation*}\index{$f_{a,p}$}
    \begin{split}
    M^q_q(r,T_g(f_{a,p}))&\lesssim\frac{\|T_g(f_{a,p})\|_{A^q_\omega}^q}{\int_r^1\omega(s)\,ds}
    \le\frac{\|T_g\|^q_{(A^p_\om,A^q_\om)} \cdot\left(\sup_{a\in\D}\|f_{a,p}\|^q_{A^p_\om}\right)}{\int_r^1\omega(s)\,ds}\\
    &\lesssim\frac{1}{\int_r^1\omega(s)\,ds},\quad r\ge\frac12,
    \end{split}
    \end{equation*}
for all $a\in\D$. This together with the well-known relations
$M_\infty(r,f)\lesssim M_q(\rho,f)(1-r)^{-1/q}$ and
$M_q(r,f')\lesssim M_q(\rho,f)/(1-r)$, $\rho=(1+r)/2$,
Lemma~\ref{le:condinte} and Lemma~\ref{le:cuadrado-tienda} yield
    \begin{equation*}
    \begin{split}\index{$f_{a,p}$}
    |g'(a)|&\asymp(\omega^\star(a))^\frac1p|T_g(f_{a,p})'(a)|
    \lesssim(\omega^\star(a))^\frac1p\frac{M_q\left(\frac{1+|a|}{2},\left(T_g(f_{a,p})\right)'\right)}{\left(1-|a|\right)^{\frac1q}}\\
    &\lesssim (\omega^\star(a))^\frac1p\frac{M_q((3+|a|)/4,T_g(f_{a,p}))}{\left(1-|a|\right)^{1+\frac1q}}
    \asymp\frac{(\omega^\star(a))^{\frac1p-\frac1q}}{1-|a|},\quad |a|\ge\frac12.
    \end{split}
    \end{equation*}
The assertion follows from this inequality.

(ii) We will need the following standard lemma whose proof will be
omitted.

\begin{lemma}\label{le:compacidadtg}
Let $0<p,q<\infty$, and let $\omega$ be a radial weight. Then
$T_g:\, A^p_\om\to A^q_\om$ is compact if and only if
$\lim_{n\to\infty}\|T_g(f_n)\|_{A^q_\om}=0$ for all sequences
$\{f_n\}\subset A^p_\om$ such that
$\sup_n\|f_n\|_{A^p_\om}<\infty$ and $\lim_{n\to\infty} f_n(z)=0$
uniformly on compact subsets of $\D$.
\end{lemma}

Assume now that $T_g:\, A^p_\om\to A^q_\om$ is compact, and
consider again the functions~$f_{a,p}$.\index{$f_{a,p}$} Recall
that $\sup_{a\in\D}\|f_{a,p}\|_{A^p_\om}<\infty$ and also
$f_{a,p}\to0$ uniformly on compact subsets of $\D$ as $|a|\to
1^-$. Therefore $\lim_{|a|\to 1^-}\|T_g(f_{a,p})\|_{A^q_\om}=0$ by
Lemma~\ref{le:compacidadtg}. A reasoning similar to that in the
proof of~(i) yields the assertion.
\end{proof}

\subsection*{Proof of Theorem~\ref{Thm-integration-operator-1}} The
proof is split into several parts. We will first prove (iii) and
(ii) because they are less involved than (i). Part (iv) will be
proved last.

\smallskip

(iii) If $0<p<q$ and $\frac1p-\frac1q\ge 1$, then
\eqref{Eq:Radialqp} implies $g'\equiv0$, and so $g$ is a constant.
The converse is trivial.

\smallskip

(ii) If $0<p<q$ and $\frac1p-\frac1q<1$, then
Lemma~\ref{le:tgminfty} gives the implication
(aii)$\Rightarrow$(bii). The equivalence
(bii)$\Leftrightarrow$(cii) follows from the next more general
result.

\begin{proposition}\label{PropRadialq>p}
Let $0<\alpha<\infty$, $\omega\in\I\cup\R$ and $g\in\H(\D)$. Then
the following assertions hold:
\begin{itemize}
\item[\rm(i)] $g\in\CC^{2\alpha+1}(\omega^\star)$ if and only if
    \begin{equation}\label{Eq:Radialq>p}
    M_\infty(r,g')\lesssim\frac{(\omega^\star(r))^{\alpha}}{1-r},\quad
    0<r<1;
    \end{equation}
\item[\rm(ii)] $g\in\CC^{2\alpha+1}_0(\omega^\star)$ if and only
if
    $$
    M_\infty(r,g')=\op\left(\frac{(\omega^\star(r))^{\alpha}}{1-r}\right),\quad
    r\to1^-.
    $$
\end{itemize}
\end{proposition}

\begin{proof}
(i) Let $0<\alpha<\infty$ and $\omega\in\I\cup\R$, and assume
first that $g\in\CC^{2\alpha+1}(\omega^\star)$. We observe that
$\omega^\star\in\R$ by Lemma~\ref{le:sc1}. Let $z\in\D$, and
assume without loss of generality that $|z|>1/2$. Set
$z^\star=\frac{3|z|-1}{2}e^{i\arg z}$ so that
$D(z,\frac12(1-|z|))\subset S(z^\star)$. This inclusion together
with the subharmonicity property of $|g'|^2$ and the fact
$\omega^\star\in\R$ gives
    \begin{equation*}
    \begin{split}
    |g'(z)|^2 &\lesssim\frac{1}{(1-|z|)^2\omega^\star(z)}\int_{D(z,\frac12(1-|z|))}|g'(\z)|^2\omega^\star(\z)\,dA(\z)\\
    &\le\frac{1}{(1-|z|)^2\omega^\star(z)}\int_{S(z^\star)}|g'(\z)|^2\omega^\star(\z)\,dA(\z).
    \end{split}
    \end{equation*}
A standard reasoning involving \eqref{calpha},
Lemma~\ref{le:condinte} and Lemma~\ref{le:cuadrado-tienda} now
show that~\eqref{Eq:Radialq>p} is satisfied.

Conversely, \eqref{Eq:Radialq>p} together with
Lemma~\ref{le:cuadrado-tienda} yield
    \begin{equation*}
    \begin{split}
    \int_{S(a)}|g'(z)|^2\omega^\star(z)\,dA(z)
    &\lesssim
    (1-|a|)\int_{|a|}^1\frac{(\omega^\star(r))^{2\alpha+1}}{(1-r)^2}\,dr\\
    &\asymp(1-|a|)\int_{|a|}^1\frac{\left((1-r)\int_{r}^1\om(s)\,ds\right)^{2\alpha+1}}{(1-r)^{2}}\,dr\\
    &\lesssim \left(\om\left(S(a)\right)\right)^{2\alpha+1}\asymp \left(\omega^\star(a)\right)^{2\alpha+1},
    \quad |a|\ge\frac12,
    \end{split}
    \end{equation*}
and it follows that $g\in\CC^{2\alpha+1}(\omega^\star)$.

The second assertion (ii) can be proved by appropriately modifying
the reasoning above.
\end{proof}

We now return to the proof of the case (ii) of
Theorem~\ref{Thm-integration-operator-1}. It remains to show that
(aii) is satisfied if either of the equivalent conditions (bii)
and (cii) is valid. We first observe that, if $q=2$, then
(aii)$\Leftrightarrow$(cii) by Theorem~\ref{ThmLittlewood-Paley},
the definition of~$\CC^{2/p}(\om^\star)$ and \eqref{calpha}. The
cases $q>2$ and $q<2$ are now treated separately.

Let $q>2$ and assume that
$g\in\CC^{2\left(\frac{1}{p}-\frac1q\right)+1}(\omega^\star)$. In
this case $L^{q/2}_\omega$ can be identified with the dual of
$L^{\frac{q}{q-2}}_\omega$, that is,
$L^{q/2}_\omega=\left(L^{\frac{q}{q-2}}_\omega\right)^\star$.
Therefore formula \eqref{normacono} in
Theorem~\ref{ThmLittlewood-Paley} shows that $T_g:A^p_\omega\to
A^q_\omega$ is bounded if and only if
    \begin{equation*}
    \left|
    \int_\D\,h(u)\left(\int_{\Gamma(u)}|f(z)|^2|g'(z)|^2\,dA(z)\right)\omega(u)\,dA(u)
    \right|\lesssim\|h\|_{L^{\frac{q}{q-2}}_\omega}\|f\|^2_{A^p_\om}
    \end{equation*}
for all $h\in L^{\frac{q}{q-2}}_\omega$. To see this, we use first
Fubini's theorem, Lemma~\ref{le:cuadrado-tienda} and the
definition of the maximal function $M_\om(|h|)$\index{weighted
maximal function}\index{$M_\om(\vp)$} to obtain
    \begin{equation*}
    \begin{split}
    &\left|
    \int_\D\,h(u)\left(\int_{\Gamma(u)}|f(z)|^2|g'(z)|^2\,dA(z)\right)\omega(u)\,dA(u)\right|\\
    &\le\int_\D|f(z)|^2|g'(z)|^2\left(\int_{T(z)}|h(u)|\omega(u)\,dA(u)\right)\,dA(z)\\
    &\asymp\int_\D|f(z)|^2|g'(z)|^2\omega\left(T(z)\right)\quad\cdot\left(\frac{1}{\omega\left(S(z)\right)}\int_{T(z)}|h(u)|\omega(u)\,dA(u)\right)\,dA(z)\\
    &\lesssim\int_\D|f(z)|^2M_\om(|h|)(z)|g'(z)|^2\omega^\star(z)\,dA(z).
    \end{split}
    \end{equation*}
But $g\in\CC^{2\left(\frac{1}{p}-\frac1q\right)+1}(\omega^\star)$
by the assumption, and
$2(\frac1p-\frac1q)+1=(2+p-\frac{2p}{q})/p$, so we may estimate
the last integral upwards by H\"{o}lder's inequality and
Corollary~\ref{co:maxbou} to
    \begin{equation*}\index{weighted maximal
function}\index{$M_\om(\vp)$}
    \begin{split}
    &\left(\int_\D|f(z)|^{2+p-\frac{2p}{q}}|g'(z)|^2\omega^\star(z)\,dA(z)\right)^{\frac{2q}{(2+p)q-2p}}\\
    &\quad\cdot\left(\int_\D\left(M_\om(|h|)(z)\right)^{1+\frac{2q}{p(q-2)}}|g'(z)|^2\omega^\star(z)\,dA(z)\right)^{\frac{1}{1+\frac{2q}{p(q-2)}}}
    \lesssim\|f\|^2_{A^p_\om}\|h\|_{L^{\frac{q}{q-2}}_\omega}.
    \end{split}
    \end{equation*}
These estimates give the desired inequality for all $h\in
L^{\frac{q}{q-2}}_\omega$, and thus $T_g:A^p_\omega\to A^q_\omega$
is bounded.

Let now $0<q<2$, and assume again that
$g\in\CC^{2\left(\frac{1}{p}-\frac1q\right)+1}(\omega^\star)$. We
will use ideas from \cite{Cohn}. Then \eqref{normacono},
H\"older's inequality, Lemma~\ref{le:funcionmaximalangular},
Fubini's theorem, and Lemma~\ref{le:cuadrado-tienda} give
    \begin{equation*}\index{$N(f)$}\index{non-tangential maximal function}
    \begin{split}
    \|T_g(f)\|_{A^q_\omega}^q
    &\asymp \int_\D\left(\int_{\Gamma(u)}|f(z)|^2|g'(z)|^2\,dA(z)\right)^\frac{q}2\omega(u)\,dA(u)\\
    &\le\int_\D N(f)(u)^{\frac{p(2-q)}{2}}\cdot\left(\int_{\Gamma(u)}|f(z)|^{2-\frac{2p}{q}+p}|g'(z)|^2\,dA(z)\right)^\frac{q}2\omega(u)\,dA(u)\\
    &\le\left(\int_\D
    N(f)(u)^p\omega(u)\,dA(u)\right)^\frac{2-q}{2}\\
    &\quad\cdot\left(\int_\D\int_{\Gamma(u)}|f(z)|^{2-\frac{2p}{q}+p}|g'(z)|^2\,dA(z)\omega(u)\,dA(u)\right)^\frac{q}2\\
    &\lesssim\|f\|_{A^p_\omega}^{\frac{p(2-q)}{2}}\left(\int_\D|f(z)|^{2-\frac{2p}{q}+p}|g'(z)|^2\omega(T(z))\,dA(z)\right)^\frac{q}2\\
    &\asymp\|f\|_{A^p_\omega}^{\frac{p(2-q)}{2}}\left(\int_{\D}|f(z)|^{2-\frac{2p}{q}+p}|g'(z)|^2\omega^\star(z)\,dA(z)\right)^\frac{q}2
    \lesssim\|f\|_{A^p_\omega}^q,
    \end{split}
    \end{equation*}
and thus $T_g:A^p_\omega\to A^q_\omega$ is bounded. The proof of
the case (ii) is now complete.

\smallskip

(i) This case requires more work than the previous ones mainly
because $\CC^1(\om^\star)$\index{$\CC^1(\om^\star)$} can not be
characterized by a simple radial condition like
$\CC^\alpha(\om^\star)$ for each $\alpha>1$. We first observe
that, if $p=2$, then (ai)$\Leftrightarrow$(bi) by
Theorem~\ref{ThmLittlewood-Paley}, the definition of
$\CC^1(\om^\star)$\index{$\CC^1(\om^\star)$} and \eqref{calpha}.
Moreover, by following the two different cases of the proof of
(cii)$\Rightarrow$(aii) we obtain (bi)$\Rightarrow$(ai) for
$p\ne2$. It remains to show that (ai)$\Rightarrow$(bi) for
$p\ne2$. This will be done in two pieces.

Let first $p>2$, and assume that $T_g:A^p_\omega\to A^p_\omega$ is
bounded. By \eqref{normacono} in Theorem~\ref{ThmLittlewood-Paley}
this is equivalent to
    \begin{equation*}
    \|T_g(f)\|_{A^p_\omega}^p
    \asymp\int_\D\left(\int_{\Gamma(u)}|f(z)|^2|g'(z)|^2\,dA(z)\right)^\frac{p}2\omega(u)\,dA(u)
    \lesssim\|f\|_{A^p_\omega}^p
    \end{equation*}
for all $f\in A^p_\om$. By using this together with
Lemma~\ref{le:cuadrado-tienda}, Fubini's theorem, H\"older's
inequality and Lemma~\ref{le:funcionmaximalangular}, we obtain
    \begin{eqnarray*}\index{$N(f)$}\index{non-tangential maximal function}
    &&\int_{\D}|f(z)|^p|g'(z)|^2\om^\star(z)\,dA(z)
    \asymp\int_{\D}|f(z)|^p|g'(z)|^2\om(T(z))\,dA(z)\\
    &&=\int_\D\int_{\Gamma(u)}|f(z)|^p|g'(z)|^2dA(z)\om(u)\,dA(u)\\
    &&\le\int_\D N(f)(u)^{p-2}\int_{\Gamma(u)}|f(z)|^2|g'(z)|^2\,dA(z)\om(u)\,dA(u)\\
    &&\le\left(\int_\D
    N(f)(u)^{p}\om(u)\,dA(u)\right)^\frac{p-2}{p}\\
    &&\quad\cdot\left(\int_\D\left(\int_{\Gamma(u)}|f(z)|^2|g'(z)|^2dA(z)\right)^\frac{p}{2}\om(u)\,dA(u)\right)^\frac{2}{p}
    \lesssim\|f\|_{A^p_\om}^p
    \end{eqnarray*}
for all $f\in A^p_\om$. Therefore $|g'(z)|^2\omega^\star(z)dA(z)$
is a $p$-Carleson measure for $A^p_\omega$, and thus
$g\in\CC^1(\om^\star)$\index{$\CC^1(\om^\star)$} by the
definition. This implication can also be proved by using
Theorem~\ref{th:cm}, Theorem~\ref{ThmLittlewood-Paley} and the
functions $F_{a,p}$ of
Lemma~\ref{testfunctions1}.\index{$F_{a,p}$}

\smallskip

Let now $0<p<2$, and assume that $T_g:A^p_\omega\to A^p_\omega$ is
bounded. Then Lemma~\ref{le:tgminfty} and its proof imply
$g\in\mathcal{B}$ and
    \begin{equation}\label{eq:tgbloch}
    \|g\|_{\mathcal{B}}\lesssim\|T_g\|.
    \end{equation}
Choose $\gamma>0$ large enough, and consider the functions
$F_{a,p}=\left(\frac{1-|a|^2}{1-\overline{a}z}\right)^{\frac{\gamma+1}{p}}$
of Lemma~\ref{testfunctions1}.\index{$F_{a,p}$} Let
$1<\a,\b<\infty$ such that $\b/\a=p/2<1$, and let $\a'$ and $\b'$
be the conjugate indexes of $\a$ and $\b$. Then
Lemma~\ref{le:cuadrado-tienda}, Fubini's theorem, H\"older's
inequality, \eqref{eq:tf1} and \eqref{normacono} yield
    \begin{equation}\index{$F_{a,p}$}
    \begin{split}\label{eq:tgb1}
    &\int_{S(a)}|g'(z)|^2\omega^\star(z)\,dA(z)\\
    &\asymp\int_\D\left(\int_{S(a)\cap\Gamma(u)}|g'(z)|^2|F_{a,p}(z)|^2\,dA(z)\right)^{\frac1\a+\frac1{\a'}}\omega(u)\,dA(u)\\
    &\le\left(\int_\D\left(\int_{\Gamma(u)}|g'(z)|^2|F_{a,p}(z)|^2\,dA(z)\right)^\frac{\b}{\a}\omega(u)\,dA(u)\right)^\frac1\b\\
    &\quad\cdot\left(\int_\D\left(\int_{\Gamma(u)\cap
    S(a)}|g'(z)|^2\,dA(z)\right)^\frac{\b'}{\a'}\omega(u)\,dA(u)\right)^\frac1{\b'}\\
    &\asymp\|T_g(F_{a,p})\|_{A^p_\omega}^\frac{p}{\b}\|S_g(\chi_{S(a)})\|_{L_\omega^\frac{\b'}{\a'}}^\frac{1}{\a'},\quad
    a\in\D,
    \end{split}
    \end{equation}
where
    \begin{equation*}
    S_g(\varphi)(u)=\int_{\Gamma(u)}|\varphi(z)|^2|g'(z)|^2\,dA(z),\quad
    u\in\D\setminus\{0\},
    \end{equation*}
for any bounded function $\varphi$ on $\D$. Since $\b/\a=p/2<1$,
we have $\frac{\b'}{\a'}>1$ with the conjugate exponent
$\left(\frac{\b'}{\a'}\right)'=\frac{\b(\a-1)}{\a-\b}>1$.
Therefore
      \begin{equation}
      \begin{split}\label{eq:tgb2}
      \|S_g(\chi_{S(a)})\|_{L_\omega^\frac{\b'}{\a'}}
      =\sup_{\|h\|_{L_\omega^{\frac{\b(\a-1)}{\a-\b}}}\le1}
      \left|\int_\D h(u)S_g(\chi_{S(a)})(u)\omega(u)\,dA(u)\right|.
      \end{split}
      \end{equation}
By using Fubini's theorem, Lemma~\ref{le:cuadrado-tienda},
H\"older's inequality and Corollary~\ref{co:maxbou}, we deduce
     \begin{equation}\index{weighted maximal
function}\index{$M_\om(\vp)$}
     \begin{split}\label{eq:tgb3}
     &\left|\int_\D h(u)S_g(\chi_{S(a)})(u)\omega(u)\,dA(u)\right|\\
     &\le\int_\D|h(u)|\int_{\Gamma(u)\cap
     S(a)}|g'(z)|^2\,dA(z)\,\omega(u)\,dA(u)\\
     &\lesssim\int_{S(a)}M_\omega(|h|)(z)|g'(z)|^2\omega^\star(z)\,dA(z)\\
     &\le\left(\int_{S(a)}|g'(z)|^2\omega^\star(z)\,dA(z)\right)^\frac{\a'}{\b'}\\
     &\quad\cdot\left(\int_\D
     M_\omega(|h|)^{\left(\frac{\b'}{\a'}\right)'}|g'(z)|^2\omega^\star(z)\,dA(z)\right)^{1-\frac{\a'}{\b'}}\\
     &\lesssim\left(\int_{S(a)}|g'(z)|^2\omega^\star(z)\,dA(z)\right)^\frac{\a'}{\b'}\\
     &\quad\cdot\left(\sup_{a\in\D}\frac{\int_{S(a)}|g'(z)|^2\omega^\star(z)\,dA(z)}{\omega(S(a))}\right)^{1-\frac{\a'}{\b'}}
     \|h\|_{L_\omega^{\left(\frac{\b'}{\a'}\right)'}}.
     \end{split}
     \end{equation}
By replacing $g(z)$ by $g_r(z)=g(rz)$, $0<r<1$, and combining
\eqref{eq:tgb1}--\eqref{eq:tgb3}, we obtain
    \begin{equation*}\index{$F_{a,p}$}
    \begin{split}
    \int_{S(a)}|g_r'(z)|^2\omega^\star(z)\,dA(z)&\lesssim\|T_{g_r}(F_{a,p})\|_{A^p_\omega}^\frac{p}{\b}
    \left(\int_{S(a)}|g_r'(z)|^2\omega^\star(z)\,dA(z)\right)^\frac{1}{\b'}\\
    &\quad\cdot\left(\sup_{a\in\D}\frac{\int_{S(a)}|g_r'(z)|^2\omega^\star(z)\,dA(z)}{\omega(S(a))}\right)^{\frac{1}{\a'}\left(1-\frac{\a'}{\b'}\right)}.
    \end{split}
    \end{equation*}
We now claim that there exists a constant $C=C(\om)>0$ such that
    \begin{equation}\label{eq:dilatadas}\index{$F_{a,p}$}
    \sup_{0<r<1}\|T_{g_r}(F_{a,p})\|_{A^p_\omega}^p\le
    C\|T_{g}\|_{A^p_\omega}^p\omega(S(a)),\quad a\in\D.
    \end{equation}
Taking this for granted for a moment, we deduce
    \begin{equation*}
    \left(\frac{\int_{S(a)}|g_r'(z)|^2\omega^\star(z)\,dA(z)}{\omega(S(a))}\right)^\frac{1}{\b}
    \lesssim\|T_{g}\|_{A^p_\omega}^\frac{p}{\b}
    \left(\sup_{a\in\D}\frac{\int_{S(a)}|g_r'(z)|^2\omega^\star(z)\,dA(z)}{\omega(S(a))}\right)^{\frac{1}{\a'}\left(1-\frac{\a'}{\b'}\right)}
    \end{equation*}
for all $0<r<1$ and $a\in\D$. This yields
    $$
    \frac{\int_{S(a)}|g_r'(z)|^2\omega^\star(z)\,dA(z)}{\omega(S(a))}\lesssim\|T_{g}\|^2,\quad
    a\in\D,
    $$
and so
    \begin{equation*}
    \sup_{a\in\D}\frac{\int_{S(a)}|g(z)|^2\omega^\star(z)\,dA(z)}{\omega(S(a))}\le\sup_{a\in\D}\liminf_{r\to
1^-}\left(\frac{\int_{S(a)}|g_r'(z)|^2\omega^\star(z)\,dA(z)}{\omega(S(a))}\right)
    \lesssim  \|T_{g}\|^2
    \end{equation*}
by Fatou's lemma. Therefore
$g\in\CC^1(\om^\star)$\index{$\CC^1(\om^\star)$} by
Theorem~\ref{th:cm}.

It remains to prove~\eqref{eq:dilatadas}. To do this, let
$a\in\D$. If $|a|\le r_0$, where $r_0\in(0,1)$ is fixed, then the
inequality in~\eqref{eq:dilatadas} follows by Theorem~\ref{ThmLittlewood-Paley}, the change of
variable $rz=\z$, the fact
    \begin{equation}\label{conos}
    \Gamma(ru)\subset \Gamma(u),\quad 0<r<1,
    \end{equation}
and the assumption that $T_g:A^p_\om\to A^p_\om$ is bounded. If
$a\in\D$ is close to the boundary, we consider two separate cases.

Let first $\frac12<|a|\le\frac1{2-r}$. Then
    $$
    |1-\overline{a}z|\le\left|1-\overline{a}\frac{z}{r}\right|+\frac{1-r}{2-r}\le
    2\left|1-\overline{a}\frac{z}{r}\right|,\quad |z|\le r.
    $$
Therefore Theorem~\ref{ThmLittlewood-Paley}, \eqref{conos} and
\eqref{eq:tf2} yield
    \begin{equation}\label{78}\index{$F_{a,p}$}
    \begin{split}
    \|T_{g_{r}}(F_{a,p})\|_{A^p_\omega}^p
    &\asymp \int_\D \left(\int_{\Gamma(u)}r^2|g'(rz)|^2
    \left|F_{a,p}(z)\right|^2\,dA(z)\right)^{\frac{p}{2}}\om(u)\,du\\
    &=\int_\D \left(\int_{\Gamma(ru)}|g'(z)|^2
    \left|F_{a,p}\left(\frac{z}{r}\right)\right|^2
    \,dA(z)\right)^{\frac{p}{2}}\om(u)\,du\\
    &\le2^{\gamma+1}\int_\D \left(\int_{\Gamma(ru)}|g'(z)|^2
    \left|F_{a,p}(z)\right|^2
    \,dA(z)\right)^{\frac{p}{2}}\om(u)\,du\\
    &\le2^{\gamma+1}\int_\D \left(\int_{\Gamma(u)}|g'(z)|^2
    \left|F_{a,p}(z)\right|^2
    \,dA(z)\right)^{\frac{p}{2}}\om(u)\,du\\
    &\asymp\|T_{g}(F_{a,p})\|_{A^p_\omega}^p\lesssim\|T_{g}\|_{A^p_\omega}^p\omega(S(a)),
    \end{split}
    \end{equation}
and hence
    \begin{equation}\index{$F_{a,p}$}
    \begin{split}\label{eq:dil1}
    \|T_{g_{r}}(F_{a,p})\|_{A^p_\omega}^p\lesssim \|T_{g}\|_{A^p_\omega}^p\omega(S(a))
    ,\quad\frac12<|a|\le\frac1{2-r}.
    \end{split}
    \end{equation}

Let now $|a|>\max\{\frac1{2-r},\frac12\}$. Then, by
Theorem~\ref{ThmLittlewood-Paley}, \eqref{eq:tgbloch} and
Lemma~\ref{Lemma:Zhu-type}, with $\gamma+2-p$ in place of
$\gamma$, we deduce
    \begin{equation}\index{$F_{a,p}$}
    \begin{split}\label{eq:dil2}
    \|T_{g_{r}}(F_{a,p})\|_{A^p_\omega}^p
    &\asymp\int_\D\left(\int_{\Gamma(u)}r^2|g'(rz)|^2
    \left|F_{a,p}(z)\right|^2\,dA(z)\right)^{\frac{p}{2}}\om(u)\,du\\
    &\le M_\infty(r,g')^p \int_\D \left(\int_{\Gamma(u)}\left|F_{a,p}(z)\right|^2\,dA(z)\right)^{\frac{p}{2}}\om(u)\,du
\\ &\lesssim
 M_\infty\left( 2-\frac{1}{|a|},g'\right)^p(1-|a|)^p\left\|\left(\frac{1-|a|^2}{1-\overline{a}z}\right)^{\frac{\gamma+1}{p}-1}\right\|_{A^p_\omega}^p
\\ & \lesssim \|g\|^p_{\mathcal{B}}\,\omega(S(a))\lesssim\|T_{g}\|_{A^p_\omega}^p\,\omega(S(a))
    \end{split}
    \end{equation}
for $\gamma>0$ large enough. This together with \eqref{eq:dil1}
gives \eqref{eq:dilatadas}. The proof of (i) is now complete.

(iv) Let first $g\in A^s_{\omega}$, where $s=\frac{pq}{p-q}$. Then
Theorem~\ref{ThmLittlewood-Paley}, H\"older's inequality and
Lemma~\ref{le:funcionmaximalangular} yield
    \begin{equation}\label{58}\index{$N(f)$}\index{non-tangential maximal function}
    \begin{split}
    \|T_g(f)\|_{A^q_\omega}^q&\asymp\int_\D\left(\int_{\Gamma(u)}|f(z)|^2|g'(z)|^2\,dA(z)\right)^\frac{q}{2}\omega(u)\,dA(u)\\
    &\le\int_\D(N(f)(u))^q\left(\int_{\Gamma(u)}|g'(z)|^2\,dA(z)\right)^\frac{q}{2}\omega(u)\,dA(u)\\
    &\le\left(\int_\D(N(f)(u))^p\omega(u)\,dA(u)\right)^\frac{q}{p}\\
    &\quad\cdot\left(\int_\D\left(\int_{\Gamma(u)}|g'(z)|^2\,dA(z)\right)^\frac{pq}{2(p-q)}\omega(u)\,dA(u)\right)^\frac{p-q}{p}\\
    &\le C_1^{q/p}C_2(p,q,\om)\|f\|_{A^p_\omega}^q\|g\|_{A^s_\omega}^q.
    \end{split}
    \end{equation}
Thus $T_g:A^p_\omega\to A^q_\omega$ is bounded.

To prove the converse implication, we will use ideas
from~\cite[p.~170--171]{AC}, where~$T_g$ acting on Hardy spaces is
studied. We begin with the following result whose proof relies on
Corollary~\ref{cor:FactorizationBergman}.

\begin{proposition}\label{PropSmallIndeces}\index{factorization}
Let $0<q<p<\infty$ and $\omega\in\widetilde{\I}\cup\R$, and let
$T_g:A^p_\omega\to A^q_\omega$ be bounded. Then $T_g: A^{
\hat{p}}_\omega\to A^{ \hat{q}}_\omega$ is bounded for any $
\hat{p}<p$ and $ \hat{q}<q$ with $\frac{1}{ \hat{q}}-\frac{1}{
\hat{p}}=\frac{1}{q}-\frac{1}{p}$. Further, if $0<p\le 2$, then
there exists $C=C(p,q,\omega)>0$ such that
    \begin{equation}\label{eq:tqq<p1}
    \limsup_{ \hat{p}\to p^-}\|T_g\|_{\left(A^{ \hat{p}}_\om,A^{ \hat{q}}_\om\right)}\le C\|T_g\|_{\left(A^{p}_\om,A^{q}_\om\right)}.
    \end{equation}
\end{proposition}

\begin{proof} Theorem~\ref{Thm:FactorizationBergman} shows that for any $f\in
A^{\hat{p}}_\om$, there exist $f_1\in A^p_\om$ and $f_2\in
A^{\frac{ \hat{p}p}{p- \hat{p}}}_\om$ such that
    \begin{equation}\label{eq:fact}
    f=f_1f_2\quad\text{and}\quad\|f_1\|_{A^p_\om}\cdot\|f_2\|_{A^{\frac{ \hat{p}p}{p- \hat{p}}}_\om}\le C_3\|f\|_{A^{ \hat{p}}_\om}
    \end{equation}
for some constant $C_3=C_3(p, \hat{p},\om)>0$. We observe that
$T_g(f)=T_F(f_2)$, where $F=T_g(f_1)$. Since $T_g:A^p_\omega\to
A^q_\omega$ is bounded,
    \begin{equation}\label{eq:fact2}
    \|F\|_{ A^q_\omega}=\|T_g(f_1)\|_{
    A^q_\omega}\le\|T_g\|_{\left(A^{p}_\om,A^{q}_\om\right)}\|f_1\|_{A^p_\om}<\infty,
    \end{equation}
and hence $F\in A^q_\omega$. Then \eqref{58} and the identity
$\frac{1}{q}=\frac{1}{ \hat{q}}-\frac{1}{\frac{ \hat{p}p}{p-
\hat{p}}}$ yield
    \begin{equation*}
    \|T_g(f)\|_{ A^{ \hat{q}}_\omega}=\|T_F(f_2)\|_{ A^{ \hat{q}}_\omega}\le C_1^{\frac{1}{ \hat{p}}-\frac{1}{p}}C_2
    \|f_2\|_{A^{\frac{ \hat{p}p}{p- \hat{p}}}_\om}\|F\|_{ A^q_\omega},
    \end{equation*}
where $C_2=C_2(q,\om)>0$. This together with \eqref{eq:fact} and
\eqref{eq:fact2} gives
    \begin{equation}\label{j11}
    \begin{split}
    \|T_g(f)\|_{A^{ \hat{q}}_\omega}&\le
    C_1^{\frac{1}{ \hat{p}}-\frac{1}{p}}C_2\|T_g\|_{\left(A^{p}_\om,A^{q}_\om\right)}\|f_1\|_{A^p_\om}
    \cdot\|f_2\|_{A^{\frac{ \hat{p}p}{p- \hat{p}}}_\om}\\
    &\le C_1^{\frac{1}{ \hat{p}}-\frac{1}{p}}C_2\,C_3\,\|T_g\|_{\left(A^{p}_\om,A^{q}_\om\right)}\|f\|_{A^{ \hat{p}}_\om}.
    \end{split}
    \end{equation}
Therefore $T_g: A^{ \hat{p}}_\omega\to A^{ \hat{q}}_\omega$ is
bounded.

To prove~\eqref{eq:tqq<p1}, let $0<p\le 2$ and let $0< \hat{p}<2$
be close enough to $p$ such that
    $$
    \min\left\{\frac{p}{p- \hat{p}}, \frac{ \hat{p}p}{p- \hat{p}}\right\}>2.
    $$
If $f\in A^{ \hat{p}}_\om$, then
Corollary~\ref{cor:FactorizationBergman} shows that
\eqref{eq:fact} holds with $C_3=C_3(p,\omega)$. Therefore the
reasoning in the previous paragraph  and \eqref{j11} give
\eqref{eq:tqq<p1}.
\end{proof}

With this result in hand, we are ready to prove
(aiv)$\Rightarrow$(biv). Let $0<q<p<\infty$ and
$\omega\in\widetilde{\I}\cup\R$, and let $T_g: A^{p}_\omega\to
A^{q}_\omega$ be bounded. Denote
$\frac{1}{s}=\frac{1}{q}-\frac{1}{p}$. By the first part of
Proposition~\ref{PropSmallIndeces}, we may assume that $p\le2$. We
may also assume, without loss of generality, that $g(0)=0$. Define
$t^*=\sup\{t:g\in A^t_\om\}$. Since the constant function $1$
belongs to $A^p_\om$, we have $g=T_g(1)\in A^q_\om$, and hence
$t^*\ge q>0$. Fix a positive integer $m$ such that
$\frac{t^*}{m}<p$. For each $t<t^*$, set $ \hat{p}=
\hat{p}(t)=\frac{t}{m}$, and define $ \hat{q}= \hat{q}(t)$ by the
equation $\frac{1}{s}=\frac{1}{ \hat{q}}-\frac{1}{ \hat{p}}$. Then
$ \hat{p}<p$,  $ \hat{q}<q$ and $T_g: A^{ \hat{p}}_\omega\to A^{
\hat{q}}_\omega$ is bounded by Proposition~\ref{PropSmallIndeces}.
Since $g^m=g^{\frac{t}{ \hat{p}}}\in A^{ \hat{p}}_\omega$, we have
$g^{m+1}=(m+1)T_g(g^m)\in A^{ \hat{q}}_\om$ and
    $$
    \|g^{m+1}\|_{ A^{ \hat{q}}_\omega}\le
    (m+1)\|T_g\|_{\left(A^{ \hat{p}}_\om,A^{ \hat{q}}_\om\right)}\|g^m\|_{A^{ \hat{p}}_\om},
    $$
that is,
    \begin{equation}\label{eq:tgq<p2}
    \|g\|^{m+1}_{ A^{(m+1) \hat{q}}_\omega}\le (m+1)\|T_g\|_{\left(A^{ \hat{p}}_\om,A^{ \hat{q}}_\om\right)}\|g\|^m_{A^{t}_\om}.
    \end{equation}

Suppose first that for some $t<t^*$, we have
    $$
    t\ge (m+1) \hat{q}=\left(\frac{t}{ \hat{p}}+1 \right) \hat{q}= \hat{q}+t\left(1-\frac{ \hat{q}}{s}
    \right).
    $$
Then $s\le t<t^*$, and the result follows from the definition of
$t^*$. It remains to consider the case in which $t<(m+1) \hat{q}$
for all $t<t^*$. By H\"older's inequality, $\|g\|^m_{A^{t}_\om}\le
C_1(m,\om)\|g\|^m_{ A^{(m+1) \hat{q}}_\omega}$. This and
\eqref{eq:tgq<p2} yield
    \begin{equation}\label{j12}
    \|g\|_{ A^{(m+1) \hat{q}}_\omega}\le
    C_2(m,\om)\|T_g\|_{\left(A^{ \hat{p}}_\om,A^{ \hat{q}}_\om\right)},
    \end{equation}
where $C_2(m,\om)=C_1(m,\om)(m+1)$. Now, as $t$ increases to
$t^*$, $ \hat{p}$ increases to $\frac{t^*}{m}$ and $ \hat{q}$
increases to $\frac{t^*s}{t^*+ms}$, so by \eqref{j12} and
\eqref{eq:tqq<p1} we deduce
    \begin{equation*}
    \begin{split}
    \|g\|_{ A^{\frac{(m+1)t^*s}{t^*+ms}}_\omega}
    &\le\limsup_{t\to t^*}\|g\|_{ A^{(m+1) \hat{q}}_\omega}
    \le C_2(m,\om)\limsup_{ \hat{p}\to
    p^-}\|T_g\|_{\left(A^{ \hat{p}}_\om,A^{ \hat{q}}_\om\right)}\\
    &\le C(p,q,m,\om)\|T_g\|_{\left(A^{p}_\om,A^{q}_\om\right)}<\infty.
    \end{split}
    \end{equation*}
The definition of $t^*$ implies $\frac{(m+1)t^*s}{t^*+ms}\le t^*$,
and so $t^*\ge s$. This finishes the proof. \hfill$\Box$

\medskip

We next characterize the symbols $g\in\H(\D)$ such that
$T_g:A^p_\om\to A^q_\om$ is compact.

\begin{theorem}\label{Thm-integration-operator-2}
Let $0<p,q<\infty$, $\om\in\I\cup\R$ and $g\in\H(\D)$.
\begin{itemize}
\item[\rm(i)] The following conditions are equivalent:
\begin{enumerate}
 \item[\rm(ai)]\,$T_g:A^p_\om\to A^p_\om$ is compact;
 \item[\rm(bi)]\,$g\in\CC^1_0(\om^\star)$.\index{$\CC^1_0(\om^\star)$}
\end{enumerate}
\item[\rm(ii)] If $0<p<q$ and $\frac1p-\frac1q<1$, then the
following conditions are equivalent:
\begin{enumerate}
\item[\rm(aii)]\,$T_g:A^p_\om\to A^q_\om$ is compact;
\item[\rm(bii)]\, $\displaystyle
    M_\infty(r,g')=\op\left(\frac{(\omega^\star(r))^{\frac1p-\frac1q}}{1-r}\right),\quad
    r\to1^-;
    $
\item[\rm(cii)]
$g\in\CC_0^{2\left(\frac{1}{p}-\frac1q\right)+1}(\omega^\star).$
\end{enumerate}
\item[\rm(iii)] If $0<q<p<\infty$ and
$\om\in\widetilde{\I}\cup\R$, then the following conditions are
equivalent:
\begin{enumerate}
\item[\rm(aiii)]\,$T_g:A^p_\om\to A^q_\om$ is compact;
\item[\rm(biii)]\,$T_g:A^p_\om\to A^q_\om$ is bounded;
\item[\rm(ciii)] $g\in A^s_\omega$, where
$\frac{1}{s}=\frac1q-\frac1p$.
\end{enumerate}
\end{itemize}
\end{theorem}

\begin{proof}
(ii) Appropriate modifications in the proofs of
Theorem~\ref{th:cm}(ii) and
Theorem~\ref{Thm-integration-operator-1}(ii) together with
Lemma~\ref{le:tgminfty}(ii), Lemma~\ref{le:compacidadtg} and
Proposition~\ref{PropRadialq>p}(ii) give the assertion.

(i) By using Lemma~\ref{le:compacidadtg}, the proof of
Theorem~\ref{th:cm}(ii) and arguing as in the proof of
Theorem~\ref{Thm-integration-operator-1}(i), with appropriate
modifications, we obtain (bi)$\Rightarrow$(ai) for all
$0<p<\infty$, and the converse implication (ai)$\Rightarrow$(bi)
for $p\ge2$. To prove the remaining case, let $0<p<2$ and assume
that $T_g:A^p_\omega\to A^p_\omega$ is compact. Recall that the
functions $f_{a,p}$\index{$f_{a,p}$} defined
in~\eqref{testfunctions} satisfy $\|f_{a,p}\|_{A^p_\omega}\asymp1$
and $f_{a,p}\to 0$, as $|a|\to1^-$, uniformly in compact subsets
of $\D$. Therefore $\|T_g(f_{a,p})\|_{A^p_\omega}\to 0$, as
$|a|\to 1^-$, by Lemma~\ref{le:compacidadtg}. Now, let
$1<\a,\b<\infty$ such that $\b/\a=p/2<1$. Arguing as in
\eqref{eq:tgb1} we deduce
    \begin{equation*}\index{$f_{a,p}$}
    \frac{1}{(\omega(S(a)))^\frac2p}\int_{S(a)}|g'(z)|^2\omega^\star(z)\,dA(z)
    \lesssim\|T_g(f_{a,p})\|_{A^p_\omega}^\frac{p}{\b}
    \|S_g(\chi_{S(a)}f_{a,p})\|_{L_\omega^{\frac{\b'}{\a'}}}^\frac{1}{\a'}
    \end{equation*}
for all $a\in\D$. Following the reasoning in the proof of
Theorem~\ref{Thm-integration-operator-1}(i) further, we obtain
    \begin{equation*}\index{$f_{a,p}$}
    \frac{\int_{S(a)}|g'(z)|^2\omega^\star(z)\,dA(z)}{(\omega(S(a)))^\frac{2}{p}}
    \lesssim\|T_g(f_{a,p})\|_{A^p_\omega}^\frac{p}{\b}
    \frac{\left(\int_{S(a)}|g'(z)|^2\omega^\star(z)\,dA(z)\right)^{\frac{\a'}{\b'}\cdot\frac{1}{\a'}}}
    {(\omega(S(a)))^{\frac{2}{p}\cdot\frac{1}{\a'}}},
    \end{equation*}
which is equivalent to
    $$
    \frac{\int_{S(a)}|g'(z)|^2\omega^\star(z)\,dA(z)}{\omega(S(a))}
    \lesssim\|T_g(f_{a,p})\|_{A^p_\omega}^p.
    $$\index{$f_{a,p}$}
Thus Theorem~\ref{th:cm} implies
$g\in\CC^1_0(\omega^\star)$.\index{$\CC^1_0(\om^\star)$}

(iii) The equivalence (biii)$\Leftrightarrow$(ciii) is
Theorem~\ref{Thm-integration-operator-1}(iv), and
 (aiii)$\Rightarrow$(biii) is obvious. To prove the remaining implication
 (ciii)$\Rightarrow$(aiii), let $g\in A^s_{\omega}$, where $s=\frac{pq}{p-q}$, and let  $\{f_n\}\subset A^p_\om$ such that
$\sup_n\|f_n\|_{A^p_\om}<\infty$ and $\lim_{n\to\infty} f_n(z)=0$
uniformly on compact subsets of $\D$. Let $\e>0$. By using
\eqref{normacono} in Theorem~\ref{ThmLittlewood-Paley}, we can
find $r_0\in(0,1)$ such that
    $$
    \int_{r_0\le |z|<1}\left(\int_{\Gamma(u)}|g'(z)|^2\,dA(z)\right)^\frac{pq}{2(p-q)}\omega(u)\,dA(u)<\e^{\frac{p}{p-q}}.
    $$
Now, take $n_0\in\N$ such that $\sup_{n\ge n_0}|f_n(z)|<\e^{1/q}$
for all $z\in D(0,r_0)$. Then Theorem~\ref{ThmLittlewood-Paley},
H\"older's inequality and Lemma~\ref{le:funcionmaximalangular}
yield
    \begin{equation*}
    \begin{split}
    \|T_g(f_n)\|_{A^q_\omega}^q&\asymp\int_\D\left(\int_{\Gamma(u)}|f_n(z)|^2|g'(z)|^2\,dA(z)\right)^\frac{q}{2}\omega(u)\,dA(u)\\
    &\le\e\int_{|z|<r_0}\left(\int_{\Gamma(u)}|g'(z)|^2\,dA(z)\right)^\frac{q}{2}\omega(u)\,dA(u)\\
    &\quad+\int_{r_0\le |z|<1}\left(\int_{\Gamma(u)}|f_n(z)|^2|g'(z)|^2\,dA(z)\right)^\frac{q}{2}\omega(u)\,dA(u)\\
    &\lesssim\e\|g\|^q_{A^q_\om}
    +\int_{r_0\le |z|<1}(N(f_n)(u))^q\left(\int_{\Gamma(u)}|g'(z)|^2\,dA(z)\right)^\frac{q}{2}\omega(u)\,dA(u)\\
    &\le\e\|g\|^q_{A^q_\om}+
    \e\left(\int_{r_0\le |z|<1}(N(f_n)(u))^p\omega(u)\,dA(u)\right)^\frac{q}{p}\\
    &\lesssim\e\left(\|g\|^q_{A^q_\om}+\left(\sup_n\|f_n\|_{A^p_\om}\right)^q\right)\lesssim\e
    \end{split}
    \end{equation*}
for all $n\ge n_0$. In the last step we used the fact $g\in
A^q_\om$, which follows by
Theorem~\ref{Thm-integration-operator-1}(iv). Therefore,
$\lim_{n\to\infty}\|T_g(f_n)\|_{A^q_\omega}=0$, and so
$T_g:A^p_\om\to A^q_\om$ is compact by
Lemma~\ref{le:compacidadtg}. This finishes the proof of
Theorem~\ref{Thm-integration-operator-2}.
\end{proof}

It is worth noticing that
Theorem~\ref{Thm-integration-operator-1}(iii) is a consequence of
the fact that condition \eqref{Eq:Radialq>p} implies $g'\equiv0$
for all $\alpha\in [1,\infty)$. If $\om$ is regular, then this
implication remains valid also for some $\alpha<1$ by~\eqref{64}.
However, no such conclusion can be made if $\om$ is rapidly
increasing by the observation (ii) to Lemma~\ref{le:condinte}.

If $\om$ is a regular weight, a description of those $g\in\H(\D)$
such that $T_g:A^p_\om\to A^q_\om$ is bounded follows by
\cite[Theorem~4.1]{AlCo} and Lemma~\ref{le:RAp}(i), see also
\cite{AS}. The reasoning in the proof of
Theorem~\ref{Thm-integration-operator-1} gives a different way to
establish these results, in particular, when $q\le p$. In the
proofs found in the existing literature, the correct necessary
conditions for $\mu$ to be a $q$-Carleson measure for $A^p_\om$,
with $q<p$, are achieved by using Luecking's approach based on
Kinchine's inequality~\cite{Lu93}. In contrast to this, the
corresponding part of the proof of
Theorem~\ref{Thm-integration-operator-1}(iv) relies on an argument
inherited from \cite{AC} and an appropriate estimate for the
constant in the reverse type H\"older's inequality, which is
obtained through factorization of functions in $A^p_\om$ in
Corollary~\ref{cor:FactorizationBergman}. If $q=p$ and $\om$ is
rapidly increasing, then the proof of
Theorem~\ref{Thm-integration-operator-1} is much more involved
than in the case when $\om$ is regular. This is due to the fact
that $\CC^1(\om^\star)$\index{$\CC^1_0(\om^\star)$} is a proper
subspace of the Bloch space~$\B$\index{$\B$} by
Proposition~\ref{pr:blochcpp}, it is not necessarily conformally
invariant and it can not be characterized by a simple growth
condition on $M_\infty(r,g')$.

\section{Integral operator $T_g$ on the Hardy space $H^p$}\label{Hardy}

The question of when $T_g:\,H^p\to H^q$ is bounded was completely
solved by Aleman, Cima and Siskakis~\cite{AC,AS0}, see also
\cite{AHpreview,Cohn2,Pom,Sisreview}. We next quote their result
for further reference.

\begin{lettertheorem}\label{Thm-integration-operator-3}
Let $0<p,q<\infty$ and $g\in\H(\D)$.
\begin{itemize}
\item[\rm(i)] The following conditions are equivalent:
    \begin{enumerate}
    \item[\rm(ai)] $T_g:H^p\to H^p$ is bounded;
    \item[\rm(bi)] $g\in\BMOA$.\index{$\BMOA$}
    \end{enumerate}
\item[\rm(ii)] If $0<p<q$ and $\frac1p-\frac1q\le 1$, then the
following conditions are equivalent:
     \begin{enumerate}
    \item[\rm(aii)] $T_g:H^p\to H^q$ is bounded;
    \item[\rm(bii)] $\displaystyle M_\infty(r,g')\lesssim\left(\frac{1}{1-r}\right)^{1-\left(\frac1p-\frac1q\right)},\quad
    r\to1^-$;
     \item[\rm(cii)] $g\in\Lambda(\frac1p-\frac1q)$;
      \item[\rm(dii)] The measure $d\mu_g(z)=|g'(z)|^2(1-|z|^2)\,dA(z)$ satisfies
        $$
        \sup_{I\subset \T}\frac{\mu_g\left(S(I)\right)}{|I|^{2\left(\frac1p-\frac1q\right)+1}}<\infty.
        $$
      \end{enumerate}
     \item[\rm(iii)] If $\frac1p-\frac1q>1$, then $T_g:H^p\to H^q$ is bounded if and only if $g$ is constant.
\item[\rm(iv)] If $0<q<p<\infty$, then the following conditions
are equivalent:
    \begin{enumerate}
    \item[\rm(aiv)] $T_g:H^p\to H^q$ is bounded;
    \item[\rm(biv)] $g\in H^s$, where $\frac{1}{s}=\frac1q-\frac1p$.
    \end{enumerate}

\end{itemize}
\end{lettertheorem}

For $0<\alpha\le 1$, the \emph{Lipschitz space}\index{Lipschitz
space} $\Lambda(\alpha)$\index{$\Lambda(\alpha)$} consists of
those $g\in\H(\D)$, having a non-tangential limit $g(e^{i\theta})$
almost everywhere, such that
    $$
    \sup_{\t\in[0,2\pi],\,0<t<1}\frac{|g(e^{i(\theta+t)})-g(e^{i\theta})|}{t^\alpha}<\infty.
    $$
The proof of Theorem~\ref{Thm-integration-operator-3}
in~\cite{AC,AS0} uses several well-known properties of
$\BMOA$\index{$\BMOA$} and $H^p$ such as the conformal invariance
of $\BMOA$, Fefferman's duality identity
$(H^1)^\star=\BMOA$~\cite{Ba86(2),GiBMO}, a Riesz-Thorin type
interpolation theorem for $H^p$~\cite{Zygmund59,Pabook}, and the
inner-outer factorization of $H^p$-functions. However, the proof
of Theorem~\ref{Thm-integration-operator-1} does not rely on such
properties for $\CC^1(\om^\star)$\index{$\CC^1(\om^\star)$} and
$A^p_\om$. In fact, we will see in Chapter~\ref{Sec:SpaceCCpp}
that the space $\CC^1(\om^\star)$\index{$\CC^1(\om^\star)$} is not
necessarily conformally invariant if $\om$ is rapidly increasing.
The objective of this section is to offer an alternative proof for
some cases of Theorem~\ref{Thm-integration-operator-3} by using
the techniques developed on the way to the proof of
Theorem~\ref{Thm-integration-operator-1}. We will omit analogous
steps and we will deepen only in the cases when the proof
significantly differs from the original one.

We begin with recalling some definitions and known results needed.
For $\beta>0$, a positive Borel measure on $\D$ is a
\emph{$\beta$-classical Carleson measure}\index{$\beta$-classical
Carleson measure} if
    \begin{equation*}
    \sup_{I\subset\T}\frac{\mu\left(S(I)\right)}{|I|^\beta}<\infty.
    \end{equation*}
If $\beta\ge 1$ and $0<p<\infty$, then $H^p\subset
L^{p\beta}(\mu)$ if and only if $\mu$ is a $\beta$-classical
Carleson measure~\cite[Section~9.5]{Duren1970}. For $0<\a<\infty$
and a $2\pi$-periodic function $\vp(e^{i\t})\in L^1(\T)$, the
\emph{Hardy-Littlewood maximal function} is defined
by\index{Hardy-Littlewood maximal function}
    $$
    M(\vp)(z)=\sup_{I:\,z\in S(I)}\frac{1}{|I|}\int_{I}|\vp(\z)|\,\frac{|d\z|}{2\pi},\quad
    z\in\D.
    $$

The following result can be obtained by carefully observing either
\cite[Section~9.5]{Duren1970} or the proof of Theorem~\ref{th:cm}.

\begin{corollary}\label{co:maxbouhp}
Let $0<p\le q<\infty$ and $0<\alpha<\infty$ such that $p\alpha>1$.
Let $\mu$ be a positive Borel measure on $\D$. Then
$[M((\cdot)^{\frac{1}{\alpha}})]^{\alpha}:L^p(\T)\to L^q(\mu)$ is
bounded if and only if $\mu$ is a $\frac{q}{p}$-classical Carleson
measure. Moreover,
    $$
    \|[M((\cdot)^{\frac{1}{\alpha}})]^{\alpha}\|^q\asymp\sup_{I\subset\T}\frac{\mu\left(S(I) \right)}{|I|^\frac{q}p}.
    $$
\end{corollary}

\subsection*{Proof of Theorem~\ref{Thm-integration-operator-3}} We
will prove in detail only the implication (ai)$\Rightarrow$(bi)
and discuss (aiv)$\Rightarrow$(biv) because their proofs are
significantly different from the original ones, and because some
adjustments to the arguments in the proof of
Theorem~\ref{Thm-integration-operator-1} should be made. After
these proofs, we will indicate how the reader can modify the proof
of Theorem~\ref{Thm-integration-operator-1} to obtain the
remaining implications.

(ai)$\Rightarrow$(bi). If $p=2$, then the equivalence
(ai)$\Leftrightarrow$(bi) follows by using the fact that
$g\in\BMOA$\index{$\BMOA$} if and only if $|g'(z)|^2(1-|z|^2)$ is
a $1$-classical Carleson measure~\cite{Garnett1981,GiBMO}.

Let $p>2$, and assume that $T_g: H^p\to H^p$ is bounded. Then
\eqref{eq:FC}, with $n=1$, yields
    \begin{equation*}
    \|T_g(f)\|_{H^p}^p
    \asymp\int_\T\left(\int_{\Gamma(\z)}|f(z)|^2|g'(z)|^2\,dA(z)\right)^\frac{p}2\,|d\z|
    \lesssim\|f\|_{H^p}^p
    \end{equation*}
for all $f\in {H^p}$. This together with Fubini's theorem,
H\"older's inequality and \eqref{Eq:Maximal} give
    \begin{eqnarray*}\index{$f^\star$}
    &&\int_{\D}|f(z)|^p|g'(z)|^2(1-|z|^2)\,dA(z)\asymp\int_{\D}|f(z)|^p|g'(z)|^2||I_z|\,dA(z)\\
    &&\asymp\int_\T\int_{\Gamma(\z)}|f(z)|^p|g'(z)|^2dA(z)\,|d\z|\\
    &&\le\int_\T \left(f^\star(\z)\right)^{p-2}\left(\int_{\Gamma(\z)}|f(z)|^2|g'(z)|^2\,dA(z)\right)\,|d\z|\\
    &&\le\left(\int_\T
    \left(f^\star(\z)\right)^{p}\,|d\z|\right)^\frac{p-2}{p}\left(\int_\D\left(\int_{\Gamma(\z)}|f(z)|^2|g'(z)|^2dA(z)\right)^\frac{p}{2}\,|d\z|\right)^\frac{2}{p}
    \lesssim\|f\|_{H^p}^p,
    \end{eqnarray*}
and thus $|g'(z)|^2(1-|z|^2)dA(z)$ is a $p$-Carleson measure for
$H^p$. Therefore $g\in\BMOA$\index{$\BMOA$} by
\cite[Theorem~9.4]{Duren1970}.

Let now $0<p<2$, and assume that $T_g:H^p\to H^p$ is bounded.
First, we will show that $g$ is a Bloch function and
    \begin{equation}\label{eq:tgblochhp}
    \|g\|_{\mathcal{B}}\lesssim\|T_g\|_{(H^p,H^p)}.
    \end{equation}
For $a\in\D$, let $D_a=\{z:~|z-a|<\frac{1-|a|}{2}\}$ and consider
the functions
$F_{a,p}(z)=\left(\frac{1-|a|^2}{1-\overline{a}z}\right)^{\frac{\gamma+1}{p}}$,\index{$F_{a,p}$}
where $\gamma>0$ is fixed. Clearly,
$\|F_{a,p}\|_{H^p}^p\asymp(1-|a|)$ and $|F_{a,p}(z)|\asymp1$ for
all $z\in D_a$. Moreover, there exists $r_0\in(0,1)$ such that
    $$
    \left|\{\z\in \T:~ D_a\subset\Gamma(\z)\}\right|\asymp(1-|a|),\quad |a|\ge
    r_0.
    $$\index{$F_{a,p}$}
Therefore
    \begin{equation*}\index{$F_{a,p}$}
    \begin{split}
    (1-|a|)\left(\int_{D_a}\left|g'(z)\right|^2dA(z)\right)^{p/2}
    &\lesssim\int_\T\left(\int_{\Gamma(\zeta)}|F_{a,p}(z)|^2|g'(z)|^2\,dA(z)\right)^{p/2}|d\z|\\
    &\lesssim\|T_g\|^p_{(H^p,H^p)} \|F_{a,p}\|_{H^p}^p\\
    &\asymp\|T_g\|^p_{(H^p,H^p)}(1-|a|),\quad |a|\ge r_0,
    \end{split}
    \end{equation*}
from which the subharmonicity of $|g'|^2$ yields
    $$
    |g'(a)|^2(1-|a|^2)^2\lesssim\int_{D_a}\left|g'(z)\right|^2\,dA(z)\lesssim\|T_g\|^2_{(H^p,H^p)}<\infty.
    $$
This implies $g\in\B$\index{$\B$} and \eqref{eq:tgblochhp}.

Let now $1<\a,\b<\infty$ such that $\b/\a=p/2<1$, and let $\a'$
and $\b'$ be the conjugate indexes of $\a$ and $\b$. Then Fubini's
theorem, \eqref{eq:tf1} and H\"older's inequality yield
    \begin{equation}\index{$F_{a,p}$}
    \begin{split}\label{eq:tgb1hp}
    &\int_{S(a)}|g'(z)|^2(1-|z|^2)\,dA(z)\\
    &\asymp\int_\T\left(\int_{S(a)\cap\Gamma(\z)}|g'(z)|^2|F_{a,p}(z)|^2\,dA(z)\right)^{\frac1\a+\frac1{\a'}}\,|d\z|\\
    &\le\left(\int_\T\left(\int_{\Gamma(\z)}|g'(z)|^2|F_{a,p}(z)|^2\,dA(z)\right)^\frac{\b}{\a}\,|d\z|\right)^\frac1\b\\
    &\quad\cdot\left(\int_\T\left(\int_{\Gamma(\z)\cap
    S(a)}|g'(z)|^2\,dA(z)\right)^\frac{\b'}{\a'}\,|d\z|\right)^\frac1{\b'}\\
    &=\|T_g(F_{a,p})\|_{H^p}^\frac{p}{\b}\|S_g(\chi_{S(a)})\|_{L^\frac{\b'}{\a'}(\T)}^\frac{1}{\a'},\quad a\in\D,
    \end{split}
    \end{equation}
where
    $$
    S_g(\varphi)(\zeta)=\int_{\Gamma(\zeta)}|\varphi(z)|^2|g'(z)|^2\,dA(z),\quad
    \zeta\in\T,
    $$
for any bounded function $\varphi$ on $\D$. Now
$\left(\frac{\b'}{\a'}\right)'=\frac{\b(\a-1)}{\a-\b}>1$, and
hence
      \begin{equation}\begin{split}\label{eq:tgb2hp}
    \|S_g(\chi_{S(a)})\|_{L^\frac{\b'}{\a'}(\T)}
    =\sup_{\|h\|_{L^{\frac{\b(\a-1)}{\a-\b}}(\T)}\le1}
    \left|\int_\T h(\z)S_g(\chi_{S(a)})(\z)\,|d\z|\right|.
    \end{split}\end{equation}
Next, using Fubini's theorem, H\"older's inequality and
Corollary~\ref{co:maxbouhp}, we deduce
     \begin{equation}
     \begin{split}\label{eq:tgb3hp}
    &\left|\int_\T h(\z)S_g(\chi_{S(a)})(\z)\,|d\z|\right|\\
    &\le\int_\T|h(\z)|\int_{\Gamma(\z)\cap
    S(a)}|g'(z)|^2\,dA(z)\,|d\z|\\
    &\lesssim\int_{S(a)}|g'(z)|^2(1-|z|^2)M(|h|)(z)\,dA(z)\\
    &\le\left(\int_{S(a)}|g'(z)|^2(1-|z|^2)\,dA(z)\right)^\frac{\a'}{\b'}\\
    &\quad\cdot\left(\int_\D
    M(|h|)^{\left(\frac{\b'}{\a'}\right)'}|g'(z)|^2(1-|z|^2)\,dA(z)\right)^{1-\frac{\a'}{\b'}}\\
    &\lesssim\left(\int_{S(a)}|g'(z)|^2(1-|z|^2)\,dA(z)\right)^\frac{\a'}{\b'}\\
    &\quad\cdot\left(\sup_{a\in\D}\frac{\int_{S(a)}|g'(z)|^2(1-|z|^2)\,dA(z)}{1-|a|}\right)^{1-\frac{\a'}{\b'}}
    \|h\|_{L^{\left(\frac{\b'}{\a'}\right)'}(\T)}
    \end{split}
    \end{equation}
By replacing $g$ by $g_r$ in \eqref{eq:tgb1hp}--\eqref{eq:tgb3hp},
we obtain
    \begin{equation}\label{59}\index{$F_{a,p}$}
    \begin{split}
    &\int_{S(a)}|g_r'(z)|^2(1-|z|^2)\,dA(z)\\
    &\lesssim\|T_{g_r}(F_{a,p})\|_{H^p}^\frac{p}{\b}
    \left(\int_{S(a)}|g_r'(z)|^2(1-|z|^2)\,dA(z)\right)^\frac{1}{\b'}\\
    &\quad\cdot\left(\sup_{a\in\D}\frac{\int_{S(a)}|g_r'(z)|^2(1-|z|^2)\,dA(z)}{1-|a|}\right)^{\frac{1}{\a'}\left(1-\frac{\a'}{\b'}\right)}.
    \end{split}
    \end{equation}
By arguing as in the proof of \eqref{eq:dilatadas} we find a
constant $C>0$ such that
    $$
    \sup_{0<r<1}\|T_{g_r}(F_{a,p})\|_{H^p}^p\le
    C\|T_{g}\|_{(H^p,H^p)}^p(1-|a|),\quad a\in\D.
    $$\index{$F_{a,p}$}
This combined with \eqref{59} and Fatou's lemma yield
    $$
    \sup_{a\in\D}\frac{\int_{S(a)}|g'(z)|^2(1-|z|^2)\,dA(z)}{1-|a|}\lesssim\|T_{g}\|_{(H^p,H^p)}^2,
    $$
and so $g\in\BMOA$.\index{$\BMOA$} Now the proof of
(ai)$\Rightarrow$(bi) is complete.

(aiv)$\Rightarrow$(biv). One of the key ingredients in the
original proof of this implication is \cite[Proposition
p.~170]{AC} that is stated in weaker form as
Proposition~\ref{PropSmallIndeceshp} below. Our contribution
consists of indicating how this weaker result can be established
by using the standard factorization of $H^p$-functions instead of
appealing to the interpolation theory. The disadvantage of this
method is that we will not obtain the sharp constant $C=1$ that
follows by interpolating.

\begin{letterproposition}\label{PropSmallIndeceshp}
Let $0<q<p<\infty$ and let $T_g:H^p\to H^q$ be bounded. Then $T_g:
H^{ \hat{p}}\to H^{ \hat{q}}$ is bounded for any $ \hat{p}<p$ and
$\hat{q}<q$ with $\frac{1}{ \hat{q}}-\frac{1}{
\hat{p}}=\frac{1}{q}-\frac{1}{p}$. Further, if $0<p<\infty$, then
there exists $C=C(q)>0$ such that
    \begin{equation}\label{eq:tqq<p1hp}
    \limsup_{ \hat{p}\to p^-}\|T_g\|_{\left(H^{ \hat{p}},H^{ \hat{q}}\right)}\le C\|T_g\|_{\left(H^{p},H^{q}\right)}.
    \end{equation}
\end{letterproposition}

Indeed, if $f\in H^{ \hat{p}}$, then $f=Bg$, where $B$ is the
Blaschke product whose zeros are those of $f$, and $g$ is a
non-vanishing analytic function with $\|g\|_{H^{
\hat{p}}}=\|f\|_{H^{ \hat{p}}}$. So, if  we take $f_1=Bg^{\frac{
\hat{p}}{p}}\in H^{p}$ and $f_2=g^{\frac{p- \hat{p}}{p}}\in
H^{\frac{ \hat{p}p}{p- \hat{p}}}$, then
    \begin{equation*}
    f=f_1f_2\quad\text{and}\quad\|f_1\|_{H^p}\cdot\|f_2\|_{H^{\frac{ \hat{p}p}{p- \hat{p}}}}=\|f\|_{H^{ \hat{p}}}.
    \end{equation*}
Consequently, \eqref{eq:tqq<p1hp} can be proved in the same way as
\eqref{eq:tqq<p1}.

The remaining parts of Theorem~\ref{Thm-integration-operator-3}
can be proved by appropriately modifying the proof of
Theorem~\ref{Thm-integration-operator-1}. Basically one has to use
the norm \eqref{eq:FC} in $H^p$ when the norm \eqref{normacono} in
$A^p_\om$ is used, and use the maximal function
$f^\star$\index{$f^\star$} whenever $N(f)$
appears.\index{$N(f)$}\index{non-tangential maximal function}
Further, the maximal operator $M$ should be used instead of
$M_\om$\index{weighted maximal function}\index{$M_\om(\vp)$} and
the weight $1-|z|$ should appear instead of $\om^\star(z)$.
Furthermore, one will need to use \cite[Theorem~9.3]{Duren1970},
\cite[Theorem~9.4]{Duren1970}, \cite[Theorem~5.1]{Duren1970}, and
Corollary~\ref{co:maxbouhp}. With this guidance on the remaining
implications we consider Theorem~\ref{Thm-integration-operator-3}
proved. \hfill$\Box$

\medskip

We finish the section by two observations on the proof of
Theorem~\ref{Thm-integration-operator-3}. First, the techniques
can be also used to establish \cite[Corollary~1]{AC}. Second, in
the proof we used repeatedly the fact that for $h\in L^p(\T)$,
$1<p\le q$, and a $q/p$-classical Carleson measure $\mu$, we have
    \begin{equation*}
    \begin{split}
    \int_{\D}\left(\frac{1}{1-|z|}\int_{I_z}h(\z)\,dm(\z)\right)\,d\mu(z)
    &\lesssim\int_{\D}M(h)(z)\,d\mu(z)\\
    &\lesssim\|M(h)\|_{L^{q}(\mu)}\lesssim\|h\|_{L^p(\T)}<\infty.
    \end{split}
    \end{equation*}
It is also worth noticing that here the maximal function $M(h)$
can be replaced by the Poisson integral $P(h)$ of~$h$.

\chapter{Non-conformally Invariant Space Induced by $T_g$
on~$A^p_\om$}\label{Sec:SpaceCCpp}\index{conformally invariant}

The boundedness of the integral operator $T_g$ on the Hardy space
$H^p$ and the classical weighted Bergman space $A^p_\a$ are
characterized by the conditions $g\in\BMOA$\index{$\BMOA$} and
$g\in\B$,\index{$\B$} respectively. Moreover, we saw in
Chapter~\ref{SecVolterra} that if $\om$ is rapidly increasing,
then $T_g$ is bounded on $A^p_\om$ if and only if
$g\in\CC^1(\om^\star)$.\index{$\CC^1(\om^\star)$} Since
$H^p\subset A^p_\om\subset A^p_\a$ for each $\om\in\I$ and
$\a>-1$, it is natural to expect that
$\CC^1(\om^\star)$\index{$\CC^1(\om^\star)$} lies somewhere
between $\BMOA$\index{$\BMOA$} and $\B$.\index{$\B$} In this
chapter we will give some insight to the structural properties of
the spaces $\CC^1(\omega^\star)$\index{$\CC^1(\om^\star)$} and
$\CC^1_0(\omega^\star)$.\index{$\CC^1_0(\om^\star)$} We will also
study their relations to several classical spaces of analytic
functions on $\D$. In particular, we will confirm the expected
inclusions
$\BMOA\subset\CC^1(\om^\star)\subset\B$.\index{$\B$}\index{$\CC^1(\om^\star)$}
Moreover, we will show that whenever a rapidly increasing weight
$\om$ admits certain regularity, then
$\CC^1(\om^\star)$\index{$\CC^1(\om^\star)$} is not conformally
invariant, but contains non-trivial inner functions.

\section{Inclusion relations}

The first result shows the basic relations between
$\BMOA$,\index{$\BMOA$}
$\CC^1(\om^\star)$,\index{$\CC^1(\om^\star)$} $\B$\index{$\B$} and
$A^p_\om$, when $\om\in\I\cup\R$. In some parts we are forced to
require additional regularity on $\om\in\I$ due to technical
reasons induced by the fact that rapidly increasing weights may
admit a strong oscillatory behavior, as was seen in
Chapter~\ref{S2}. Recall that $h:\,[0,1)\to (0,\infty)$ is
essentially increasing on $[0,1)$ if there exists a constant $C>0$
such that $h(r)\le C h(t)$ for all $0\le r\le t<1$.

\begin{proposition}\label{pr:blochcpp}
\begin{itemize}
\item[\rm(A)] If $\om\in\I\cup\R$, then
$\CC^1(\omega^\star)\subset\cap_{0<p<\infty}A^p_\omega$.\index{$\CC^1(\om^\star)$}

\item[\rm(B)] If $\om\in\I\cup\R$, then $\BMOA\subset
\CC^1(\omega^\star)\subset\B$.

\item[\rm(C)] If $\om\in\R$, then $\CC^1(\omega^\star)=\B$.

\item[\rm(D)] If $\om\in\I$, then
$\CC^1(\omega^\star)\subsetneq\B$.

\item[\rm(E)] If $\om\in\I$ and both $\om(r)$ and
$\frac{\psi_\om(r)}{1-r}$ are essentially increasing on $[0,1)$,
then
$\BMOA\subsetneq\CC^1(\omega^\star)$.\index{$\BMOA$}\index{$\B$}
\end{itemize}
\end{proposition}

\begin{proof} (A). Let $g\in\CC^1(\omega^\star)$.\index{$\CC^1(\om^\star)$} Theorem~\ref{th:cm} shows that
$|g'(z)|^2\omega^{\star}(z)\,dA(z)$ is a $p$-Carleson measure for
$A^p_\omega$ for all $0<p<\infty$. In particular,
$|g'(z)|^2\omega^{\star}(z)\,dA(z)$ is a finite measure and hence
$g\in A^2_\omega$ by \eqref{eq:LP2}. Therefore \eqref{HSB} yields
    \begin{equation*}
    \|g\|_{A^4_\omega}^4=4^2\int_{\D}|g(z)|^{2}|g'(z)|^2\omega^\star(z)\,dA(z)+|g(0)|^4\lesssim\|g\|_{A^2_\omega}^2+|g(0)|^4,
    \end{equation*}
and thus $g\in A^4_\omega$. Continuing in this fashion, we deduce
$g\in A^{2n}_\omega$ for all $n\in\N$, and the assertion follows.

(B). If $g\in\BMOA$,\index{$\BMOA$} then
$|g'(z)|^2\log\frac{1}{|z|}\,dA(z)$ is a classical Carleson
measure~\cite{Garnett1981} (or \cite[Section~8]{GiBMO}), that is,
    $$
    \sup_{I\subset\T}\frac{\int_{S(I)}|g'(z)|^2\log\frac{1}{|z|}\,dA(z)}{|I|}<\infty.
    $$
Therefore
    \begin{equation*}
    \begin{split}
    \int_{S(I)}|g'(z)|^2\om^\star(z)\,dA(z)
    &\le\int_{S(I)}|g'(z)|^2\log\frac{1}{|z|}\left(\int_{|z|}^1\om(s)s\,ds\right)\,dA(z)\\
    &\le\left(\int_{1-|I|}^1\om(s)s\,ds\right)\int_{S(I)}|g'(z)|^2\log\frac{1}{|z|}\,dA(z)\\
    &\lesssim\left(\int_{1-|I|}^1\om(s)s\,ds\right)|I|\asymp\om\left(S(I)\right),
    \end{split}
    \end{equation*}
which together with Theorem~\ref{th:cm} gives
$g\in\CC^1(\omega^\star)$ for all
$\om\in\I\cup\R$.\index{$\CC^1(\om^\star)$}

Let now $g\in\CC^1(\omega^\star)$ with $\om\in\I\cup\R$. It is
well known that $g\in\H(\D)$ is a Bloch function if and only if
    $$
    \int_{S(I)}|g'(z)|^2(1-|z|^2)^\gamma\,dA(z)\lesssim|I|^\gamma,\quad
    I\subset\T,
    $$
for some (equivalently for all) $\gamma>1$, see~\cite{X2}. Fix
$\b=\b(\omega)>0$ and $C=C(\b,\om)>0$ as in
Lemma~\ref{le:condinte}. Then Lemma~\ref{le:cuadrado-tienda} and
Lemma~\ref{le:condinte} yield
    \begin{equation*}
    \begin{split}
    \int_{S(I)}|g'(z)|^2(1-|z|)^{\b+1}\,dA(z)
    &=\int_{S(I)}|g'(z)|^2\om^\star(z)\frac{(1-|z|)^{\b+1}}{\om^\star(z)}\,dA(z)\\
    &\asymp\int_{S(I)}|g'(z)|^2\om^\star(z)\frac{(1-|z|)^{\b}}{\int_{|z|}^1\om(s)s\,ds}\,dA(z)\\
    &\le\frac{C|I|^\b}{\int_{1-|I|}^1\om(s)s\,ds}\int_{S(I)}|g'(z)|^2\om^\star(z)\,dA(z)\\
    &\lesssim |I|^{\b+1},\quad |I|\le\frac12,
    \end{split}\index{$\CC^1(\om^\star)$}
    \end{equation*}
and so $g\in\B$.\index{$\B$}

(C). By Part (B) it suffices to show that
$\B\subset\CC^1(\omega^\star)$\index{$\B$} for $\om\in\R$. To see
this, let $g\in\B$\index{$\B$} and $\om\in\R$. By \eqref{eq:r1}
there exists $\a=\alpha(\omega)>0$ such that
$h(r)=\frac{\int_r^1\om(s)\,ds}{(1-r)^\a}$ is decreasing on
$[0,1)$. This together with Lemma~\ref{le:cuadrado-tienda} gives
    \begin{equation*}\index{$\CC^1(\om^\star)$}
    \begin{split}
    \int_{S(I)}|g'(z)|^2\om^\star(z)\,dA(z)
    &=\int_{S(I)}|g'(z)|^2\frac{\om^\star(z)}{(1-|z|)^{\alpha+1}}(1-|z|)^{\alpha+1}\,dA(z)\\
    &\asymp\int_{S(I)}|g'(z)|^2\frac{\int_{|z|}^1\om(s)s\,ds}{(1-|z|)^{\alpha}}(1-|z|)^{\alpha+1}\,dA(z)\\
    &\le\frac{\int_{1-|I|}^1\om(s)\,ds}{|I|^\alpha}\int_{S(I)}|g'(z)|^2(1-|z|)^{\alpha+1}\,dA(z)\\
    &\lesssim\om\left(S(I)\right),\quad |I|\le\frac12,
    \end{split}
    \end{equation*}
and therefore $g\in\CC^1(\omega^\star)$.

(D). Let $\om\in\I$, and assume on the contrary to the assertion
that $\B\subset\CC^1(\omega^\star)$.\index{$\B$} Ramey and
Ullrich~\cite[Proposition~5.4]{RU} constructed
$g_1,g_2\in\B$\index{$\B$} such that
$|g'_1(z)|+|g'_2(z)|\ge(1-|z|)^{-1}$ for all $z\in\D$. Since
$g_1,g_2\in\CC^1(\omega^\star)$ by the antithesis,
Lemma~\ref{le:cuadrado-tienda} yields
    \begin{equation}\label{56}\index{$\CC^1(\om^\star)$}
    \begin{split}
    \|f\|_{A^2_\om}^2&\gtrsim\int_{\D}|f(z)|^2\left(|g'_1(z)|^2+|g'_2(z)|^2\right)\om^\star(z)\,dA(z)\\
    &\ge\frac{1}{2}\int_{\D}|f(z)|^2\left(\vert g'_1(z)\vert +\vert
    g'_2(z)\vert\right)^2\om^\star(z)\,dA(z)\\
    &\ge\frac{1}{2}\int_{\D}|f(z)|^2\frac{\om^\star(z)}{(1-|z|)^2}\,dA(z)
    \asymp\int_{\D}|f(z)|^2\frac{\int_{|z|}^1\om(s)\,ds}{(1-|z|)}\,dA(z)\\
    &=\int_{\D}|f(z)|^2\frac{\psi_\om(|z|)}{1-|z|}\,\om(z)\,dA(z)
    \end{split}
    \end{equation}
for all $f\in\H(\D)$. If
$\int_{\D}\frac{\psi_\om(|z|)}{1-|z|}\,\om(z)\,dA(z)=\infty$, we
choose $f\equiv1$ to obtain a contradiction. Assume now that
    $
    \int_{\D}\frac{\psi_\om(|z|)}{1-|z|}\,\om(z)\,dA(z)<\infty,
    $
and replace $f$ in \eqref{56} by the test function $F_{a,2}$ from
Lemma~\ref{testfunctions1}. Then \eqref{eq:tf2} and
Lemma~\ref{le:cuadrado-tienda} yield
    \begin{equation*}
    \om^\star(a)\gtrsim\int_0^1\frac{(1-|a|)^{\gamma+1}}{(1-|a|r)^\gamma}\frac{\psi_\om(r)}{1-r}\,\om(r)\,dr
    \gtrsim(1-|a|)\int_{|a|}^1\frac{\psi_\om(r)}{1-r}\,\om(r)\,dr,
    \end{equation*}
and hence
    $$
    \int_{|a|}^1\frac{\psi_\om(r)}{1-r}\,\om(r)\,dr\lesssim\int_{|a|}^1\om(r)\,dr,\quad
    a\in\D.
    $$
By letting $|a|\to1^-$, Bernouilli-l'H\^{o}pital theorem and the
assumption $\om\in\I$ yield a contradiction. For completeness with
regards to \eqref{56} we next construct $f\in A^2_\om$ such that
    \begin{equation}\label{eq:ccppb}
    \int_{\D}|f(z)|^2\frac{\psi_\om(|z|)}{1-|z|}\,\om(z)\,dA(z)=\infty,
    \end{equation}
provided $\om\in\I$ such that $\frac{\psi_\om(r)}{1-r}$ is
essentially increasing on $[0,1)$. To do this, denote
    $$
    \omega_k=\int_0^1s^{2k}\omega(s)s\,ds,\quad
    M_k=\int_0^1s^{2k}\frac{\psi_\om(s)}{1-s}\,\om(s)s\,ds,\quad k\in\N.
    $$
Since $\om\in\I$, there exists an increasing sequence
$\{n_k\}\subset\N$ such that
    \begin{equation}\label{eq:nk}
    n_k\psi_\om(1-n_k^{-1})\ge k^2,\quad k\in\N.
    \end{equation}
We claim that the analytic function
    $$
    f(z)=\sum_{k=1}^\infty \frac{z^{n_k}}{M_{n_k}^{1/2}}
    $$
has the desired properties. Parseval's identity shows that $f$
satisfies \eqref{eq:ccppb}. To prove $f\in A^2_\om$, note first
that by the proof of Lemma~\ref{le:condinte}, there exists
$r_0\in(0,1)$ such that $h(r)=\frac{\int_r^1\om(s)\,ds}{1-r}$ is
increasing on $[r_0,1)$, and hence $h$ is essentially increasing
on $[0,1)$. This together with the assumption on
$\frac{\psi_\om(r)}{1-r}$ and Lemmas~\ref{le:condinte}
and~\ref{le:momentos} yields
    \begin{equation}\label{eq:wk}
    \om_k\asymp \int_{1-\frac{1}{k}}^1\om(s)s\,ds,\quad M_k\asymp
    \int_{1-\frac{1}{k}}^1\frac{\psi_\om(s)}{1-s}\,\om(s)s\,ds,\quad
    k\in\N.
    \end{equation}
Parseval's identity, \eqref{eq:wk}, the assumption on
$\frac{\psi_\om(r)}{1-r}$ and \eqref{eq:nk} finally give
    \begin{equation*}
    \begin{split}
    \|f\|_{A^2_\om}^2&=\sum_k\frac{\om_k}{M_{n_k}}\asymp\sum_{k=1}^\infty\frac{\int_{1-\frac{1}{n_k}}^1\om(s)s\,ds}
    {\int_{1-\frac{1}{n_k}}^1\frac{\psi_\om(s)}{1-s}s\om(s)\,ds}\\
    &\lesssim\sum_{k=1}^\infty\frac{1}{ n_k\psi_\om(1-n_k^{-1})}
    \le\sum_{k=1}^\infty\frac{1}{k^{2}}<\infty.
    \end{split}
    \end{equation*}
(E) Recall that
$\BMOA\subset\CC^1(\om^\star)$\index{$\BMOA$}\index{$\CC^1(\om^\star)$}
by Part~(B). In Proposition~\ref{pr:blochcppcero}(E) below we will
prove that
$\CC^1_0(\om^\star)\not\subset\BMOA$,\index{$\BMOA$}\index{$\CC^1_0(\om^\star)$}
which yields the desired strict inclusion
$\BMOA\subsetneq\CC^1(\om^\star)$.\index{$\BMOA$}\index{$\CC^1(\om^\star)$}
\end{proof}

Results analogous to those given in Proposition~\ref{pr:blochcpp}
are valid for $\CC_0^1(\omega^\star)$. Recall that $\VMOA$
consists of those functions in the Hardy space $H^1$ that have
\emph{vanishing mean oscillation} on the
boundary~$\T$\index{vanishing mean oscillation}\index{$\VMOA$},
and $f\in\H(\D)$ belongs to the \emph{little Bloch
space}~$\B_0$\index{little Bloch space}\index{$\B_0$} if
$f'(z)(1-|z|^2)\to0$, as $|z|\to1^-$. It is well known that
$\VMOA\subsetneq\B_0$.

\begin{proposition}\label{pr:blochcppcero}
\begin{itemize}
\item[\rm(A)] $\CC_0^1(\omega^\star)$ is a closed subspace of
$\left(\CC^1(\omega^\star),\|\cdot\|_{\CC^1(\omega^\star)}\right)$.
\item[\rm(B)] If $\om\in\I\cup\R$, then $\VMOA\subset
\CC_0^1(\omega^\star)\subset\B_0$.\index{$\B_0$}\index{$\CC^1(\om^\star)$}\index{$\CC^1_0(\om^\star)$}

\item[\rm(C)] If $\om\in\R$, then
$\CC^1_0(\omega^\star)=\B_0$.\index{$\B_0$}

\item[\rm(D)] If $\om\in\I$, then
$\CC^1_0(\omega^\star)\subsetneq\B_0$.\index{$\B_0$}

\item[\rm(E)] If $\om\in\I$ and both $\om(r)$ and
$\frac{\psi_\om(r)}{1-r}$ are essentially increasing, then
$\VMOA\subsetneq\CC^1_0(\omega^\star)$. Moreover, there exists a
function $g\in \CC^1_0(\omega^\star)$ such that
$g\notin\BMOA$.\index{$\BMOA$}
\end{itemize}
\end{proposition}\index{$\VMOA$}

\begin{proof}
Parts~(A), (B) and (C) follow readily from the proof of
Proposition~\ref{pr:blochcpp}. Part~(D) can be proved by following
the argument in Proposition~\ref{pr:blochcpp}(D), with
$g_1,g_2\in\B$\index{$\B$} being replaced by
$(g_1)_r,(g_2)_r\in\B_0$,\index{$\B$} and using Fatou's lemma at
the end of the proof. It remains to prove the second assertion
in~(E), which combined with~(B) implies
$\VMOA\subsetneq\CC^1_0(\omega^\star)$. To do this, we assume
without loss of generality that $\int_r^1 \om(s)\,ds<1$. Consider
the lacunary series
    \begin{equation}\label{eq:nordn}
    g(z)=\sum_{k=0}^\infty\frac{z^{2^k}}{
    2^{k/2}\left(\psi_\om(1-2^{-k})\log\left(\frac{e}{\int_{1-2^{-k}}^1 \om(s)\,ds}\right)\right)^{1/2}}.
    \end{equation}
The assumption on $\frac{\psi_\om(r)}{1-r}$ implies
    \begin{equation*}\begin{split}
    \infty&=\lim_{\rho\to 1^-}\log\log\left(\frac{e}{\int_\rho^1\om(s)\,ds}\right)\asymp\int_{0}^1\frac{dr}{\psi_\om(r)\log\left(\frac{e}{\int_r^1
    \om(s)\,ds}\right)}\\
    &=\sum_{k=0}^\infty\int_{1-2^{-k}}^{1-2^{-k-1}}\frac{dr}{\psi_\om(r)\log\left(\frac{e}{\int_r^1 \om(s)\,ds}\right)}\\
    &\lesssim\sum_{k=0}^\infty \frac{1}{2^{k}\psi_\om(1-2^{-k})\log\left(\frac{e}{\int_{1-2^{-k}}^1 \om(s)\,ds}\right)}
    =\|g\|_{H^2}^2,
    \end{split}\end{equation*}
and hence $g\notin\BMOA$.\index{$\BMOA$}

It remains to show that $g\in\CC_0^1(\omega^\star)$. Consider
$z=re^{i\theta}\in\D$ with $r\ge\frac12$, and take $N=N(r)\in\N$
such that $1-\frac{1}{2^N}\le r<1-\frac{1}{2^{N+1}}$. Since
$\om(r)$ is essentially increasing on $[0,1)$ and
$h(x)=x\log\frac{e}{x}$ is increasing on $(0,1]$, we deduce
    \begin{equation}
    \begin{split}\label{eq:cppbmo1n}
    &\sum_{k=0}^N \frac{2^{k/2}r^{2^k}}{\left(\psi_\om(1-2^{-k})\log\left(\frac{e}{\int_{1-2^{-k}}^1
    \om(s)\,ds}\right)\right)^{1/2}}\\
    &\lesssim\frac{1}{\left(\psi_\om(1-2^{-N})\log\left(\frac{e}{\int_{1-2^{-N}}^1 \om(s)\,ds}\right)\right)^{1/2}}\sum_{k=0}^N 2^{k/2}\\
    &\lesssim\frac{2^{N/2}}{\left(\psi_\om(r)\log\left(\frac{e}{\int_{r}^1 \om(s)\,ds}\right)\right)^{1/2}}\le\frac{1}{\left((1-r)\psi_\om(r)\log\left(\frac{e}{\int_{r}^1 \om(s)\,ds}\right)\right)^{1/2}}.
    \end{split}
    \end{equation}
Moreover, the assumptions on $\om(r)$ and
$\frac{\psi_\om(r)}{1-r}$ and Lemma~\ref{le:condinte}(ii) yield
    \begin{equation}
    \begin{split}\label{eq:cppbmo2n}
    &\sum_{k=N+1}^\infty \frac{2^{k/2}r^{2^k}}{\left(\psi_\om(1-2^{-k})
    \log\left(\frac{e}{\int_{1-2^{-k}}^1 \om(s)\,ds}\right)\right)^{1/2}}\\
    &\lesssim\frac{1}{2^{N/2}\left(\psi_\om(1-2^{-N})\log\left(\frac{e}{\int_{1-2^{-N}}^1 \om(s)\,ds}\right)\right)^{1/2}}\sum_{k=N+1}^\infty2^{k}r^{2^k}\\
    &\lesssim\frac{2^{N/2}}{\left(\psi_\om(1-2^{-N})\log\left(\frac{e}{\int_{1-2^{-N}}^1 \om(s)\,ds}\right)\right)^{1/2}}\sum_{j=1}^\infty2^{j}e^{-2^j}\\
    &\lesssim\frac{1}{\left((1-r)\psi_\om(r)\log\left(\frac{e}{\int_{r}^1 \om(s)\,ds}\right)\right)^{1/2}}.
    \end{split}
    \end{equation}
Consequently, joining \eqref{eq:cppbmo1n} and \eqref{eq:cppbmo2n},
we deduce
    $$
    M_\infty(r,g')\lesssim
    \frac{1}{\left((1-r)\psi_\om(r)\log\left(\frac{e}{\int_{r}^1 \om(s)\,ds}\right)\right)^{1/2}},\quad r\ge\frac12,
    $$
in particular,
    $$
    M^2_\infty(r,g')=\op\left(\frac{1}{\psi_\om(r)(1-r)}\right)
    \asymp\left(\frac{\om(r)}{\om^\star(r)}\right),\quad r\to 1^-,
    $$
by Lemma~\ref{le:cuadrado-tienda}. Hence, for each $\e>0$, there
exists $r_0\in(0,1)$ such that
$M^2_\infty(r,g')\le\e\frac{\om(r)}{\om^\star(r)}$ for all
$r\in[r_0,1)$. It follows that
    \begin{equation*}
    \begin{split}
    \int_{S(I)}|g'(z)|^2\om^\star(z)\,dA(z)&\le\int_{S(I)}M^2_\infty(|z|,g')\om^\star(z)\,dA(z)\\
    &\le\e\int_{S(I)}\om(z)\,dA(z)=\e\om(S(I)),\quad |I|\le1-r_0,
    \end{split}
    \end{equation*}
and thus $g\in\CC_0^1(\omega^\star)$ by Theorem~\ref{th:cm}(ii).
\end{proof}

The proof of Part~(E) shows that for any $\om\in\I$, with both
$\om(r)$ and $\frac{\psi_\om(r)}{1-r}$ essentially increasing,
there are functions in $\CC^1_0(\omega^\star)$ which have finite
radial limits almost nowhere on $\T$.\index{$\CC^1_0(\om^\star)$}

\section{Structural properties of $\CC^1(\om^\star)$}\index{$\CC^1(\om^\star)$}

We begin with the following auxiliary result which, in view of
Theorem~\ref{th:cm}, yields a global integral characterization of
$q$-Carleson measures for $A^p_\omega$ when $0<p\le q<\infty$ and
$\omega\in\I\cup\R$.

\begin{lemma}\label{LemmaGlobalCarleson}
Let $0<s<\infty$ and $\omega\in\I\cup\R$, and let $\mu$ be a
positive Borel measure on $\D$. Then there exists
$\eta=\eta(\omega)>1$ such that
    $$
    I_1(\mu)=\sup_{a\in\D}\frac{\mu(S(a))}{(\omega(S(a)))^s}
    \asymp\sup_{a\in\D}\int_\D\left(\frac{1}{\omega(S(a))}\left(\frac{1-|a|}{|1-\overline{a}z|}\right)^\eta\right)^s\,d\mu(z)
    =I_2(\mu)
    $$
and, for $s\ge1$,
    $$
    \lim_{|a|\to1^-}\frac{\mu(S(a))}{(\omega(S(a)))^s}=0
    \Leftrightarrow\lim_{|a|\to1^-}\int_\D\left(\frac{1}{\omega(S(a))}\left(\frac{1-|a|}{|1-\overline{a}z|}\right)^\eta\right)^s\,d\mu(z)=0.
    $$
In particular, if $\omega\in\I$, then the above assertions are
valid for all $1<\eta<\infty$.
\end{lemma}

\begin{proof}
Clearly,
    \begin{equation}\label{Eq:I1<I2}
    \begin{split}
    \int_\D\left(\frac{1}{\omega(S(a))}\left(\frac{1-|a|}{|1-\overline{a}z|}\right)^\eta\right)^s\,d\mu(z)
    &\ge\int_{S(a)}\left(\frac{1}{\omega(S(a))}\left(\frac{1-|a|}{|1-\overline{a}z|}\right)^\eta\right)^s\,d\mu(z)\\
    &\asymp\frac{\mu(S(a))}{(\omega(S(a)))^s},
    \end{split}
    \end{equation}
and consequently $I_1(\mu)\lesssim I_2(\mu)$.

For $a\in\D\setminus\{0\}$ and $k\in\N\cup\{0\}$, denote
    $$
    S_k(a)=\{z\in\D:|z-a/|a||<2^k(1-|a|)\}.
    $$
Then $|1-\overline{a}z|\asymp1-|a|$, if $z\in S_0(a)$, and
$|1-\overline{a}z|\asymp2^k(1-|a|)$, if $z\in S_k(a)\setminus
S_{k-1}(a)$ and $k\ge1$. Let $\eta>\b+1$, where $\b=\b(\omega)>0$
is from Lemma~\ref{le:condinte}. Then
    \begin{equation}\label{12}
    \begin{split}
    &\int_\D\left(\frac{1}{\omega(S(a))}\left(\frac{1-|a|}{|1-\overline{a}z|}\right)^\eta\right)^s\,d\mu(z)\\
    &=\frac{1}{(\omega(S(a)))^s}\sum_k\int_{S_k(a)\setminus S_{k-1}(a)}\left(\frac{1-|a|}{|1-\overline{a}z|}\right)^{\eta
    s}\,d\mu(z)\\
    &\lesssim\frac{1}{(\omega(S(a)))^s}\sum_k\frac{1}{2^{ks\eta}}\int_{S_k(a)}\,d\mu(z)\\
    &\lesssim\frac{I_1(\mu)}{(\omega(S(a)))^s}\sum_k\frac{1}{2^{ks\eta}}(2^k(1-|a|))^s\left(\int_{1-2^k(1-|a|)}^1\omega(s)s\,ds\right)^s\\
    &=\frac{I_1(\mu)}{(\omega(S(a)))^s}\sum_k\frac{1}{2^{ks(\eta-1)}}(1-|a|)^s\left(\int_{1-2^k(1-|a|)}^1\omega(s)s\,ds\right)^s\\
    &\lesssim\frac{I_1(\mu)}{(\omega(S(a)))^s}\sum_k\frac{1}{2^{ks(\eta-\b-1)}}(1-|a|)^s\left(\int_{|a|}^1\omega(s)s\,ds\right)^s\\
    &\le
    I_1(\mu)\sum_{k=0}^\infty\frac{1}{2^{ks(\eta-\b-1)}}\lesssim
    I_1(\mu),
    \end{split}
\end{equation}
because $\eta>\b+1$. It follows that $I_2(\mu)\lesssim I_1(\mu)$,
and thus the first assertion is proved.

Assume now that
    \begin{equation}\label{13}
    \lim_{|a|\to1^-}\frac{\mu(S(a))}{(\omega(S(a)))^s}=0,
    \end{equation}
and hence $I_1(\mu)<\infty$. Then \eqref{12} implies
    \begin{eqnarray*}
    &&\int_\D\left(\frac{1}{\omega(S(a))}\left(\frac{1-|a|}{|1-\overline{a}z|}\right)^\eta\right)^s\,d\mu(z)\\
    &&\lesssim\frac{1}{(\omega(S(a)))^s}\frac{(1-|a|)^{\eta s}}{(1-r)^{\eta s}}\mu(\D)
    +\sup_{a\in\D}\frac{\mu(S(a)\setminus D(0,r))}{(\omega(S(a)))^s}\\
    &&\lesssim\left(\frac{(1-|a|)^{\eta-1}}{\int_{|a|}^1\omega(s)s\,ds}\right)^s\frac{\mu(\D)}{(1-r)^{\eta s}}
    +\sup_{|a|\ge r}\frac{\mu(S(a))}{(\omega(S(a)))^s},
    \end{eqnarray*}
where the obtained upper bound can be made arbitrarily small by
fixing sufficiently large $r$ by \eqref{13} first and then
choosing $|a|$ to be close enough to 1, see the proof of
Lemma~\ref{le:condinte} for an argument similar to the last step.
The opposite implication follows by~\eqref{Eq:I1<I2}.
\end{proof}

Lemma~\ref{LemmaGlobalCarleson} yields an alternative description
of $\CC^{\a}(\om^\star)$ when $\a\ge 1$. Indeed, for each
$\om\in\I\cup\R$ and $\eta=\eta(\om)>1$ large
enough,\index{$\Vert\cdot\vert_{\CC^{\a}(\om^\star),\eta}$} the
quantity
    $$
    \|g\|^2_{\CC^{\a}(\om^\star),\eta}
    =|g(0)|^2+ \sup_{a\in\D}\int_\D\left(\frac{1}{\omega(S(a))}\left(\frac{1-|a|}{|1-\overline{a}z|}\right)^\eta\right)^\a\,
    |g'(z)|^2\om^\star(z)\,dA(z)
    $$
is equivalent to $\|g\|^2_{\CC^{\a}(\om^\star)}$. We also deduce
that $g\in \CC^{\a}_0(\om^\star)$ if and only if
    $$
    \lim_{|a|\to 1^-}\int_\D\left(\frac{1}{\omega(S(a))}\left(\frac{1-|a|}{|1-\overline{a}z|}\right)^\eta\right)^\a\,
    |g'(z)|^2\om^\star(z)\,dA(z)=0.
    $$

Recall that an inner function is called
\emph{trivial}\index{trivial inner function} if it is a finite
Blaschke product. It is known that only inner functions in $\VMOA$
are the trivial ones~\cite{GiBMO}, but this is not true in the
case of $\B_0$~\cite{Bishop1990}.\index{$\B_0$} Therefore it is
natural to ask whether or not
$\CC^1_0(\om^\star)$\index{$\CC^1_0(\om^\star)$} contains
non-trivial inner functions when $\om\in\I$? The next result gives
an affirmative answer to this question and offers an alternative
proof for the first assertion in
Proposition~\ref{pr:blochcppcero}(E) under the additional
assumption on the monotonicity of
$\psi_\om(r)/(1-r)$.\index{$\VMOA$}

\begin{proposition}\label{pr:inner}
If $\om\in\I$ such that $\om(r)$ is essentially increasing and
$\psi_\om(r)/(1-r)$ is increasing, then there are non-trivial
inner functions in $\CC^1_0(\om^\star)$.
\end{proposition}\index{$\CC^1_0(\om^\star)$}

\begin{proof}
We may assume without loss of generality that $\int_0^1
\om(r)\,dr<1$. We will show that there exists a non-trivial inner
function $I$ such that
    \begin{equation}\label{j22}
    M_\infty(r,I')\lesssim
    \frac{1}{\left((1-r)\psi_\om(r)\log\left(\frac{e}{\int_{r}^1 \om(s)\,ds}\right)\right)^{1/2}},\quad r\in (0,1),
    \end{equation}
and therefore $I\in\CC^1_0(\om^\star)$ by
Theorem~\ref{th:cm}.\index{$\CC^1_0(\om^\star)$}

Consider the auxiliary function $\Phi$ defined by
    $$
    \Phi^2(r)=\frac{r}{\psi_\om(1-r)\log\left(\frac{e}{\int_{1-r}^1 \om(s)\,ds}\right)},\quad r\in
    (0,1],\quad\Phi(0)=0.
    $$
It is clear that $\Phi$ is continuous and increasing since
$\psi_\om(r)/(1-r)$ is increasing on $(0,1)$. Further,
    \begin{equation*}
    \int_0^1\frac{\Phi^2(t)}{t}\,dt=\int_0^1 \frac{1}{\psi_\om(r)\log\left(\frac{e}{\int_{r}^1
    \om(s)\,ds}\right)}\,dr=\infty.
    \end{equation*}
By using the assumption on $\om$ and the monotonicity
of~$h(r)=r\log\frac{e}{r}$, we obtain
    \begin{equation*}
    \begin{split}
    &t\int_t^1\frac{\Phi(s)}{s^2}\,ds+t\Phi(1)\le t\int_t^1\frac{\Phi(s)}{s^2}\,ds+\Phi(t)\\
    &\lesssim t\int_{0}^{1-t}\frac{dr}
    {\left((1-r)^3\psi_\om(r)\log\left(\frac{e}{\int_{r}^1 \om(s)\,ds}\right)\right)^{1/2}}+ \Phi(t)\\
    &\lesssim t\om^{1/2}(1-t)\int_{0}^{1-t}\frac{dr}
    {\left((1-r)^3\left(\int_{r}^1 \om(s)\,ds\right)\log\left(\frac{e}{\int_{r}^1 \om(s)\,ds}\right)\right)^{1/2}}
    + \Phi(t)\\
    &\le\frac{t}{\left(\psi_\om(1-t)\log\left(\frac{e}{\int_{1-t}^1 \om(s)\,ds}\right)\right)^{1/2}}\int_{0}^{1-t}(1-r)^{-3/2}\,dr
    + \Phi(t)\lesssim\Phi(t).
    \end{split}
    \end{equation*}
Consequently, \cite[Theorem~6(b)]{AlAnNi} implies that there
exists a non-trivial inner function $I$ such that
    $$
    (1-|z|^2)\frac{|I'(z)|}{1-|I(z)|^2}\lesssim\Phi(1-|z|),\quad
    z\in\D.
    $$
This implies \eqref{j22} and finishes the proof.
\end{proof}

It is natural to expect that polynomials are not dense in
$\CC^1(\omega^\star)$.\index{$\CC^1(\om^\star)$} Indeed we will
prove that the closure of polynomials in
$\left(\CC^1(\omega^\star),\|\cdot\|_{\CC^1(\omega^\star)}\right)$
is nothing else
but~$\CC_0^1(\omega^\star)$.\index{$\CC^1(\om^\star)$}\index{$\CC^1_0(\om^\star)$}

\begin{proposition}\label{density}
Let $\om\in\I\cup\R$ and $g\in\CC^1(\omega^\star)$. Then the
following assertions are equivalent:
\begin{itemize}
\item[\rm(i)] $g\in \CC^1_0(\omega^\star)$;

\item[\rm(ii)] $\lim_{r\to 1^-}
\|g-g_r\|_{\CC^1(\omega^\star)}=0$;

\item[\rm(iii)] There is a sequence of polynomials $\{P_n\}$ such
that
$\displaystyle\lim_{n\to\infty}\|g-P_n\|_{\CC^1(\omega^\star)}=0$.
\end{itemize}\index{$\CC^1_0(\om^\star)$}
\end{proposition}

\begin{proof}
Since $\CC_0^1(\om^\star)$ is a closed subspace of
$\CC^1(\om^\star)$ by Proposition~\ref{pr:blochcppcero}(A),
and~$H^\infty$ is continuously embedded in $\CC^1(\om^\star)$,
standard arguments give the implications
(ii)$\Rightarrow$(iii)$\Rightarrow$(i). It remains to prove
(i)$\Rightarrow$(ii). To do this, take $\gamma=\gamma(\om)>0$
large enough and consider the test functions
    $$
    f_{a,2}(z)=\left(\frac{1}{\omega(S(a))}\left(\frac{1-|a|}{|1-\overline{a}z|}\right)^{\gamma+1}\right)^{1/2},
    $$
defined in \eqref{testfunctions}. Since
$g\in\CC^1_0(\omega^\star)$ by the assumption,
Theorem~\ref{th:cm}, Lemma~\ref{LemmaGlobalCarleson} and
Theorem~\ref{ThmLittlewood-Paley} yield
    \begin{equation}\label{j16}\index{$\CC^1_0(\om^\star)$}
    \lim_{|a|\to 1^-}\|T_g(f_{a,2})\|_{A^2_\om}^2\asymp\lim_{|a|\to 1^-}\int_\D|f_{a,2}(z)|^2
    |g'(z)|^2\,\om^\star(z)dA(z)=0.
    \end{equation}
Since
    $$
    \lim_{r\to 1^-}\|g-g_r\|_{\CC^1(\omega^\star)}\asymp\lim_{r\to 1^-}
    \sup_{a\in\D}\|T_{g-g_r}(f_{a,2})\|_{A^2_\om}^2,
    $$
it is enough to prove
    \begin{equation}\label{j17}\index{$\CC^1_0(\om^\star)$}
    \lim_{r\to 1^-}
    \sup_{a\in\D}\|T_{g-g_r}(f_{a,2})\|_{A^2_\om}^2=0.
    \end{equation}
Let $\e>0$ be given. By \eqref{j16} and
Proposition~\ref{pr:blochcppcero}(B), there exists $r_0\in(1/2,1)$
such that
    \begin{equation}\label{j18}
    \|T_g(f_{a,2})\|_{A^2_\om}^2<\e,\quad |a|\ge r_0,
    \end{equation}
and
    \begin{equation}\label{j19}
    |g'(a)|^2(1-|a|)^2<\e,\quad|a|\ge r_0.
    \end{equation}
If $r_0\le|a|\le1/(2-r)$, then \eqref{78} and \eqref{j18} yield
    \begin{equation*}
    \begin{split}
    \|T_{g-g_r}(f_{a,2})\|_{A^2_\om}^2\lesssim\|T_{g}(f_{a,2})\|_{A^2_\om}^2+\|T_{g_r}(f_{a,2})\|_{A^2_\om}^2
    \lesssim\|T_{g}(f_{a,2})\|_{A^2_\om}^2<\e.
    \end{split}
    \end{equation*}
If $|a|\ge\max\{r_0,1/(2-r)\}$, then \eqref{j19} and
Lemma~\ref{Lemma:Zhu-type}(i), applied to $\om^\star\in\R$, give
    \begin{equation*}
    \begin{split}
    \|T_{g-g_r}(f_{a,2})\|_{A^2_\om}^2&\lesssim\|T_{g}(f_{a,2})\|_{A^2_\om}^2+\|T_{g_r}(f_{a,2})\|_{A^2_\om}^2\\
    &\lesssim\e+M_\infty^2(r,g')\int_\D|f_{a,2}(z)|^2\om^\star(z)\,dA(z)\\
    &\lesssim\e+M_\infty^2\left(2-\frac{1}{|a|},g'\right)(1-|a|)^2\lesssim\e.
    \end{split}
    \end{equation*}
Consequently,
    \begin{equation}
    \begin{split}\label{j21}
    \sup_{r\in(0,1),|a|\ge
    r_0}\|T_{g-g_r}(f_{a,2})\|_{A^2_\om}^2\lesssim\e,
    \end{split}
    \end{equation}
and since clearly,
    \begin{equation*}
    \lim_{r\to 1^-}\sup_{|a|<r_0}\|T_{g-g_r}(f_{a,2})\|_{A^2_\om}^2
    \lesssim\frac{1}{\omega(S(r_0))}\lim_{r\to 1^-}\|g-g_r\|_{A^2_\om}^2=0,
    \end{equation*}
we obtain \eqref{j17}.
\end{proof}

A space $X\subset\H(\D)$ equipped with a seminorm $\rho$ is called
\emph{conformally invariant}\index{conformally invariant} or
\emph{M\"{o}bius invariant}\index{M\"{o}bius invariant} if there
exists a constant $C>0$ such that
    $$
    \sup_{\vp}\rho(g\circ\varphi)\le C\rho(g),\quad
    g\in X,
    $$
where the supremum is taken on all M\"{o}bius transformations
$\vp$ of $\D$ onto itself.

$\BMOA$\index{$\BMOA$} and $\B$\index{$\B$} are conformally
invariant spaces. This is not necessarily true for
$\CC^1(\omega^\star)$ if $\omega\in\I$ as the following result
shows. Typical examples satisfying the hypothesis of
Proposition~\ref{PropConformallyInvariant} are the weights
\eqref{Eq:PesosEnV-alpha} and $v_\a$,\index{$v_\a(r)$}
$1<\a<\infty$.\index{$\CC^1_0(\om^\star)$}

\begin{proposition}\label{PropConformallyInvariant}
Let $\om\in\I$ such that both $\om(r)$ and
$\frac{\psi_\om(r)}{1-r}$ are essentially increasing on $[0,1)$,
and
    \begin{equation}\label{41}
    \int_r^1\omega(s)s\,ds\lesssim\int_{\frac{2r}{1+r^2}}^1\omega(s)s\,ds,\quad0\le
    r<1.
    \end{equation}
Then $\CC^1(\omega^\star)$ is not conformally
invariant.\index{$\CC^1(\om^\star)$}
\end{proposition}

\begin{proof}
Let $\omega\in\I$ be as in the assumptions. Recall first that
$g\in\CC^1(\omega^\star)$ if and only if
    $$
    \sup_{b\in\D}\frac{(1-|b|)^2}{\omega(S(b))}\int_\D\frac{|g'(z)|^2}{|1-\overline{b}z|^2}\omega(S(z))\,dA(z)<\infty
    $$
by Lemma~\ref{le:cuadrado-tienda} and
Lemma~\ref{LemmaGlobalCarleson} with $s=1$ and $\eta=2$. Let
$g\in\CC^1(\omega^\star)\setminus H^2$ be the function constructed
in the proof of Proposition~\ref{pr:blochcpp}(E). Then
    \begin{equation}\label{43}\index{$\CC^1(\om^\star)$}
    \begin{split}
    &\sup_{b\in\D}\frac{(1-|b|)^2}{\omega(S(b))}\int_\D\frac{|(g\circ\vp_a)'(z)|^2}{|1-\overline{b}z|^2}\omega(S(z))\,dA(z)\\
    &\ge\frac{(1-|a|)^2}{\omega(S(a))}\int_\D\frac{|g'(\zeta)|^2}{|1-\overline{a}\vp_a(\zeta)|^2}
    \omega(S(\vp_a(\zeta)))\,dA(\z)\\
    &\ge\int_{D(0,|a|)}|g'(\z)|^2(1-|\z|)\left(\frac{\omega(S(\vp_a(\zeta)))}{\omega(S(a))}\frac{|1-\overline{a}\z|^2}{1-|\z|}\right)dA(\z),
    \end{split}
    \end{equation}
where
    \begin{equation*}
    \begin{split}
    \frac{\omega(S(\vp_a(\zeta)))}{\omega(S(a))}\frac{|1-\overline{a}\z|^2}{1-|\z|}
    &=\frac{(1-|\vp_a(\zeta)|)\int_{|\vp_a(\zeta)|}^1\omega(s)s\,ds}{(1-|a|)\int_{|a|}^1\omega(s)s\,ds}
    \frac{|1-\overline{a}\z|^2}{1-|\z|}\\
    &\gtrsim\frac{\int_{\frac{2|a|}{1+|a|^2}}^1\omega(s)s\,ds}{\int_{|a|}^1\omega(s)s\,ds}\gtrsim1,\quad
    |\z|\le|a|,
    \end{split}
    \end{equation*}
by Lemma~\ref{le:cuadrado-tienda} and \eqref{41}. Since $g\not\in
H^2$, the assertion follows by letting $|a|\to1^-$ in \eqref{43}.
\end{proof}

\chapter{Schatten Classes of the Integral Operator $T_g$ on $A^2_\om$}\label{Sec:Schatten}\index{Schatten class}

Let $H$ be a separable Hilbert space. For any non-negative integer
$n$, the $n$:th \emph{singular value} of a \index{singular value}
bounded operator $T:H\to H$ is defined by
    $$
    \lambda_n(T)=\inf\left\{\|T-R\|:\,\text{rank}(R)\le
    n\right\},\index{$\lambda_n(T)$}
    $$
where $\|\cdot\|$ denotes the operator norm. It is clear that
    $$
    \|T\|=\lambda_0(T)\ge\lambda_1(T)\ge\lambda_2(T)\ge\dots\ge 0.
    $$
For $0<p<\infty$, the \emph{Schatten $p$-class} $\mathcal{S}_
p(H)$ \index{$\mathcal{S}_ p(H)$} consists of those compact
operators $T:H\to H$ whose sequence of singular values
$\{\lambda_n\}_{n=0}^\infty$ belongs to the space~$\ell^p$ of
$p$-summable sequences. For $1\le p<\infty$, the Schatten
$p$-class $\mathcal{S}_p(H)$ is a Banach space with respect to the
norm
$|T|_p=\|\{\lambda_n\}_{n=0}^\infty\|_{\ell^p}$.\index{$\vert\cdot\vert_p$}
Therefore all finite rank operators belong to every
$\mathcal{S}_p(H)$, and the membership of an operator in
$\mathcal{S}_p(H)$ measures in some sense the size of the
operator. We refer to~\cite{DS} and \cite[Chapter~1]{Zhu} for more
information about $\mathcal{S}_p(H)$.

The membership of the integral operator $T_g$ in the Schatten
$p$-class $\SSS_p(H)$ has been characterized when $H$ is either
the Hardy space $H^2$~\cite{AS0}, the classical weighted Bergman
space $A^2_\alpha$~\cite{AS}, or the weighted Bergman space
$A^2_\om$, where $\omega$ is a rapidly decreasing
weight~\cite{PP,PPVal}\index{rapidly decreasing weight} or a
Bekoll\'e-Bonami weight~\cite{OC} (the case
$p\ge2$).\index{Bekoll\'e-Bonami weight} We note that if $\om$ is
regular and $p\ge2$, then \cite[Theorem~5.1]{OC} yields a
characterization of those $g\in\H(\D)$ for which
$T_g\in\SSS_p(A^2_\om)$. This because by Lemma~\ref{le:RAp}(i),
for each $p_0=p_0(\omega)>1$ there exists
$\eta=\eta(p_0,\omega)>-1$ such that $\frac{\om(z)}{(1-|z|)^\eta}$
belongs to $B_{p_0}(\eta)$.\index{Bekoll\'e-Bonami
weight}\index{$B_{p_0}(\eta)$}

The main purpose of this chapter is to provide a complete
description of those symbols $g\in\H(\D)$ for which the integral
operator $T_g$ belongs to the Schatten $p$-class
$\SSS_p(A^2_\om)$, where $\omega\in\I\cup\R$. For this aim, we
recall that, for $1<p<\infty$, the \emph{Besov space} $B_p$
consists of those $g\in\H(\D)$\index{Besov
space}\index{$\Vert\cdot\Vert_{B_p}$} such that
    $$
    \|g\|_{B_p}^p=\int_{\D}|g'(z)|^p(1-|z|^2)^{p-2}\,dA(z)+|g(0)|^p<\infty.
    $$

\begin{theorem}\label{th:shI}
Let $\om\in\I\cup\R$ and $g\in\H(\D)$. If $p>1$, then $T_g\in
\SSS_p(A^2_\om)$ if and only if $g\in B_p$. If  $0<p\le 1$, then
$T_g\in \SSS_p(A^2_\om)$ if and only if $g$ is constant.
\end{theorem}

The assertion in Theorem~\ref{th:shI} is by no means a surprise.
This because of the inclusions $H^2\subset A^2_\om\subset A^2_\a$,
where $\a=\a(\om)\in(-1,\infty)$, and the fact that the membership
of $T_g$ in both $\SSS_p(H^2)$ and $\SSS_p(A^2_\a)$ is
characterized by the condition $g\in B_p$.

This chapter is mainly devoted to proving Theorem~\ref{th:shI}. It
appears that the hardest task here is to make the proof work for
the rapidly increasing weights, the regular weights could be
treated in an alternative way. This again illustrates the
phenomenon that in the natural boundary between $H^2$ and each
$A^2_\a$, given by the spaces $A^2_\om$ with $\om\in\I$, things
work in a different way than in the classical weighted Bergman
space $A^2_\a$. It appears that on the way to the proof of
Theorem~\ref{th:shI} for $\om\in\I$, we are forced to study the
Toeplitz operator~$T_\mu$, induced by a complex Borel measure
$\mu$ and a reproducing kernel, in certain Dirichlet type spaces
that are defined by using the associated weight $\om^\star$. Our
principal findings on $T_\mu$ are gathered in
Theorem~\ref{th:sufschpmayor2}, whose proof occupies an important
part of the chapter.

\section{Preliminary results}

We will need several definitions and auxiliary lemmas. We first
observe that by Lemma~\ref{le:suf1} each point evaluation $L_
a(f)=f(a)$\index{point evaluation}\index{$L_a(f)$} is a bounded
linear functional on $A^p_\om$ for all $0<p<\infty$ and
$\om\in\I\cup\R$. Therefore there exist reproducing
kernels\index{reproducing kernel} $B^\omega_a\in
A^2_\om$\index{$B^\omega_a$} with $\|L_ a\|=\|B^\omega_
a\|_{A^2_\om}$ such that
    \begin{equation}\label{rk}
    f(a)=\langle f, B^\omega_ a\rangle_{A^2_\om} =\int_{\D}
    f(z)\,\overline{B^\omega_a(z)}\,\om(z)\,dA(z),\quad f\in A^2_\om.
    \end{equation}
We will write $b^\omega_a=\frac{B^\omega_ a}{\|B^\omega_
a\|_{A^2_\om}}$ for the normalized reproducing
kernels.\index{$b^\omega_a$} In fact, Lemma~\ref{le:suf1} shows
that the norm convergence implies the uniform convergence on
compact subsets of $\D$. It follows that the space $A^p_\om$ is a
Banach space for all $1\le p<\infty$ and $\om\in\I\cup\R$.

\begin{lemma}\label{kernels}
If $\om\in\I\cup\R$, then the reproducing kernel $B^\omega_a\in
A^2_\om$ satisfies
    \begin{equation}\label{eq:kernels}
    \|B^\omega_a\|^2_{A^2_\om}\asymp \frac{1}{\om\left(S(a)\right)}
    \end{equation}
for all $a\in\D$.
\end{lemma}

\begin{proof}
Consider the functions $F_{a,2}$ of Lemma~\ref{testfunctions1}.
The relations \eqref{eq:tf1} and \eqref{eq:tf2} yield
    \begin{equation}\label{8}
    \|B^\omega_ a\|_{A^2_\om}=\|L_
    a\|\ge\frac{|F_{a,2}(a)|}{\|F_{a,2}\|_{A^2_\om}}\asymp\frac{1}{\om\left(S(a)\right)^{1/2}}
    \end{equation}
for all $a\in\D$. Moreover, Lemma~\ref{le:suf1} gives
    $$
    |f(a)|^2\lesssim M_\om(|f|^2)(a)\le \left(\sup_{I:a\in
    S(I)}\frac{1}{\om\left(S(I)\right)}\right)\|f\|^2_{A^2_\om}=\frac{\|f\|^2_{A^2_\om}}{\om\left(S(a)\right)},
    $$\index{weighted maximal
function}\index{$M_\om(\vp)$} and so
    $$
    \|B^\omega_ a\|_{A^2_\om}=\|L_ a\|\lesssim\frac{1}{\om\left(S(a)\right)^{1/2}}.
    $$
This together with \eqref{8} yields \eqref{eq:kernels}.
\end{proof}

It is known that if $\{e_n\}$ is an orthonormal basis of a Hilbert
space $H\subset\H(\D)$ with reproducing kernel $K_z$,
then\index{reproducing kernel}
    \begin{equation}\label{RKformula}
    K_z(\zeta)=\sum_n e_ n(\zeta)\,\overline{e_ n(z)}
    \end{equation}
for all $z,\z\in\D$, see e.g.~\cite[Theorem~4.19]{Zhu}. It follows
that
    \begin{equation}\label{eqRK1}
    \sum_n|e_n(z)|^2\le\|K_ z\|_{H}^2
    \end{equation}
for any orthonormal set $\{e_n\}$ of $H$, and equality in
\eqref{eqRK1} holds if $\{e_n\}$ is a basis of~$H$.

For $n\in\N\cup\{0\}$ and a radial weight $\omega$, set
    $$
    \om_n=\int_0^1r^{2n+1}\om(r)\,dr.\index{$\om_n$}
    $$

\begin{lemma}\label{le:sc2} If $0<\alpha<\infty$, $n\in\N$ and
$\om\in\I\cup\R$, then
    \begin{equation}\label{4}
    \int_0^1r^n\om_{\alpha-2}^\star(r)\,dr\asymp
    \frac{\om^\star\left(1-\frac{1}{n+1}\right)}{(n+1)^{\alpha-1}},
    \end{equation}
and
    \begin{equation}\label{5}
    (n+1)^{2-\alpha}\om^\star_n\asymp\int_0^1r^{2n+1}\om_{\alpha-2}^\star(r)\,dr.
    \end{equation}
\end{lemma}

\begin{proof}
Now $\om_{\alpha-2}^\star\in\R$ by Lemma~\ref{le:sc1}, and so
\eqref{eq:r1} yields
    \begin{equation}\begin{split}\label{eq:sc21}
    \int_{1-\frac{1}{n+1}}^1r^n(1-r)^{\alpha-2}\om^\star(r)\,dr\lesssim
    \frac{\om^\star\left(1-\frac{1}{n+1}\right)}{(n+1)^{\alpha-1}}.
    \end{split}\end{equation}
Moreover, by Lemma~\ref{le:cuadrado-tienda} and the proof of
Lemma~\ref{le:condinte}, there exists $\b=\b(\omega)>1$ such that
$(1-r)^{-\beta}\om^\star(r)$ is essentially increasing on
$[1/2,1)$. It follows that
    \begin{equation}\begin{split}\label{eq:sc22}
    &\int_0^{1-\frac{1}{n+1}}r^n(1-r)^{\alpha-2}\om^\star(r)\,dr\\
    &\quad\lesssim
    (n+1)^\beta\om^\star\left(1-\frac{1}{n+1}\right)\int_0^{1-\frac{1}{n+1}}r^n(1-r)^{\beta+\alpha-2}\,dr\\
    &\quad\lesssim
    (n+1)^\beta\om^\star\left(1-\frac{1}{n+1}\right)\int_0^{1}r^n(1-r)^{\beta+\alpha-2}\,dr\asymp
    \frac{\om^\star\left(1-\frac{1}{n+1}\right)}{(n+1)^{\alpha-1}}
    \end{split}\end{equation}
and
    \begin{equation}\begin{split}\label{eq:sc23}
    &\int_0^1r^n(1-r)^{\alpha-2}\om^\star(r)\,dr\ge\int_{1-\frac{1}{n+1}}^1r^n(1-r)^{\alpha-2}\om^\star(r)\,dr\\
    &\quad\gtrsim
    (n+1)^\beta\om^\star\left(1-\frac{1}{n+1}\right)\int_{1-\frac{1}{n+1}}^1(1-r)^{\beta+\alpha-2}\,dr\asymp
    \frac{\om^\star\left(1-\frac{1}{n+1}\right)}{(n+1)^{\alpha-1}}.
    \end{split}\end{equation}
Combining \eqref{eq:sc21}--\eqref{eq:sc23} we obtain \eqref{4}.

The second assertion is a consequence of \eqref{4}. Namely,
\eqref{4} with $\alpha=2$ gives
    $$
    \omega^\star\left(1-\frac1{n+1}\right)\asymp(n+1)\int_0^1
    r^n\omega^\star(r)\,dr,
    $$
which combined with \eqref{4} yields
    $$
    \int_0^1r^n\omega_{\a-2}^\star(r)\,dr\asymp(n+1)^{2-\alpha}{\int_0^1r^n\omega^\star(r)\,dr}
    $$
for any $n\in\N$. By replacing $n$ by $2n+1$ we obtain~\eqref{5}.
\end{proof}

\begin{lemma}\label{le:Phi}
For $a\in\D$, $-\infty<\alpha<1$ and
$\omega\in\I\cup\R$,\index{$\Phi^{\om}_a$} define
    \begin{equation}\label{eq:Phi}
    \Phi^{\om_{-\a}}_a(z)=\overline{a}\int_0^z
    B^{\omega^\star_{-\a}}_{a}(\xi)\,d\xi,\quad
    \Phi^{\om_{-\a}}_a(0)=0,\quad z\in\D.
    \end{equation}
Then there exists $\delta=\delta(\alpha,\om)>0$ such that
    $$
    |\Phi^{\om_{-\a}}_a(z)|\gtrsim\frac{1}{(1-|a|)\om_{-\a}^\star(a)},\quad z\in
    D(a,\delta(1-|a|)),
    $$
for all $|a|\ge\frac{1}{2}$.
\end{lemma}

\begin{proof}
Taking the orthonormal basis
    $
    \left\{\left(2\,(\omega_{-\a}^\star)_n\right)^{-1/2}z^n\right\}_{n=0}^\infty
    $
in $A^2_{\omega^\star_{-\a}}$ and using \eqref{RKformula} we
deduce
    $$
    \Phi^\om_a(z)=\sum_{n=0}^\infty\frac{(\overline{a}z)^{n+1}}{2(n+1)(\omega_{-\a}^\star)_n},\quad z\in\D.
    $$
Let $|a|\ge\frac{1}{2}$ be given and fix the integer $N\ge2$ such
that $1-\frac{1}{N}<|a|\le1-\frac{1}{N+1}$. Now
$\omega_{-\a}^\star\in\R$ by Lemma~\ref{le:sc1}, and hence
Lemma~\ref{le:sc2} and \eqref{eq:r2} yield
$(n+1)\,(\omega_{-\a}^\star)_n\asymp
(\omega_{-\a-1}^\star)_n\asymp\omega^\star_{-\a}(1-\frac1{n+1})$.
Since $\om^\star$ is decreasing, we deduce
    \begin{equation*}
    \begin{split}
    \Phi_a^{\om_{-\a}}(a) & \gtrsim\sum_{n=N}^\infty\frac{|a|^{2n+2}}{(\omega_{-\a-1}^\star)_n}
    \ge \frac{1}{(\omega_{-\a-1}^\star)_N}\sum_{n=N}^\infty|a|^{2n+2}
 \\ &    \gtrsim\frac{1}{\om^\star_{-\a}\left(1-\frac{1}{N+1}\right)}\sum_{n=N}^\infty|a|^{2n+2}\asymp\frac{1}{(1-|a|)\om_{-\a}^\star(a)}
    \end{split}
    \end{equation*}
for all $|a|\ge\frac12$.

Next, bearing in mind \eqref{RKformula}, Lemma~\ref{kernels},
\eqref{22} and Lemma~\ref{le:sc1}, we deduce
    \begin{equation}\label{86}
    B_\zeta^{\omega_{-\a}^\star}(\zeta)
    =\sum_{n=0}^\infty\frac{|\zeta|^{2n}}{2(\omega_{-\a}^\star)_n}
    =\|B_\zeta^{\omega_{-\a}^\star}\|_{A^2_{\om^\star_{-\a}}}^2
    \asymp \frac{1}{(1-|\zeta|)^2\omega_{-\a}^\star(\zeta)}
    \end{equation}
for each $\zeta\in\D$. Let now $z\in D(a,\delta(1-|a|))$, where
$|a|\ge\frac12$ and $0<\delta<1$. Then the Cauchy-Schwarz
inequality, \eqref{86}, the relation $1-|a|\asymp 1-|\zeta|$ for
all $\zeta\in[a,z]$, and the observation~(i) to
Lemma~\ref{le:condinte} yield
    \begin{equation*}
    \begin{split}
    |B_a^{\omega_{-\a}^\star}(\zeta)|
    &=\left|\sum_{n=0}^\infty\frac{\left(\overline{a}\zeta\right)^{n}}{2(\omega_{-\a}^\star)_n}\right|
    \le\left(\sum_{n=0}^\infty\frac{|a|^{2n}}{2(\omega_{-\a}^\star)_n}\right)^{\frac{1}{2}}
    \left(\sum_{n=0}^\infty\frac{|\zeta|^{2n}}{2(\omega_{-\a}^\star)_n}\right)^{\frac{1}{2}}\\
    &\asymp|B_a^{\omega_{-\a}^\star}(a)|^{\frac{1}{2}}\left(\frac{1}{(1-|\zeta|)^2\omega_{-\a}^\star(\zeta)}\right)^{\frac{1}{2}}
    \asymp |B_a^{\omega_{-\a}^\star}(a)|.
\end{split}
\end{equation*}
Consequently,
\begin{equation*}
    \begin{split}
    |\Phi_a^{\om_{-\a}}(z)-\Phi_a^{\om_{-\a}}(a)|&\le\max_{\zeta\in[a,z]}|(\Phi_a^{\om_{-\a}})'(\zeta)||z-a|
    \le\max_{\zeta\in[a,z]}|B_a^{\omega_{-\a}^\star}(\zeta)|\delta(1-|a|)\\
    &\lesssim|B_a^{\omega_{-\a}^\star}(a)|\delta(1-|a|)\asymp\frac{\delta}{(1-|a|)\omega_{-\a}^\star(a)}.
    \end{split}
    \end{equation*}
By choosing $\delta>0$ sufficiently small, we deduce
    $$
    |\Phi_a^{\om_{-\a}}(z)|\ge\Phi_a^{\om_{-\a}}(a)-|\Phi_a^{\om_{-\a}}(z)-\Phi_a^{\om_{-\a}}(a)|\gtrsim\frac{1}{(1-|a|)\om_{-\a}^\star(a)}
    $$
for all $z\in D(a,\delta(1-|a|))$.
\end{proof}

A sequence $\{a_k\}_{k=0}^\infty$ of points in $\D$ is called
\emph{uniformly discrete}\index{uniformly discrete sequence} if it
is separated in the pseudohyperbolic metric, that is, if there
exists a constant $\gamma>0$ such that
$\varrho(a_j,a_k)=\left|\frac{a_j-a_k}{1-\overline{a}_ja_k}\right|\ge\gamma$
for all $k$. For $0<\e<1$, a sequence $\{a_k\}_{k=0}^\infty$ is
called an \emph{$\e$-net}\index{$\e$-net} if $\D =
\bigcup_{k=0}^\infty \Delta(a_k,\e)$. A sequence
$\{a_k\}_{k=0}^\infty\subset\D$ is a
\emph{$\delta$-lattice}\index{$\delta$-lattice} if it is uniformly
discrete with constant $\gamma=\delta/5$ and if it is a
$5\delta$-net. With these preparations we are ready for the next
lemma.

\begin{lemma}\label{le:wo}
For $a\in\D$ and $\om\in\I\cup\R$,\index{$\phi^\om_{a}$} define
    $$
    \phi^\om_{a}(z)=\frac{\Phi^\om_{a}(z)}{\|B^{\om^\star}_a\|_{A^2_{\om^\star}}},\quad
    z\in\D.
    $$
Let $\{a_j\}$ be a uniformly discrete sequence and $\{e_ j\}$ be
an orthonormal set in~$A^2_\om$. Let $E=[\{e_ j\}]$ denote the
subspace generated by $\{e_ j\}$ and equipped with the norm of
$A^2_\omega$, and consider the linear operator
$\widetilde{J}:\,E\to A^2_\omega$, defined by
$\widetilde{J}(e_j)=\phi^\om_{a_j}$. If~$P$ is the orthogonal
projection from $A^2_\om$ to $E$, then $J=\widetilde{J}\circ P$ is
bounded on $A^2_\omega$.
\end{lemma}

\begin{proof}
Let $\om\in\I\cup\R$. Since $P:A^2_\omega\to E$ is bounded, it
suffices to show that
    \begin{equation}\label{9}
    \left \| \widetilde{J} \left(\sum_{j} c_ j e_ j \right) \right \|_{A^2_\om}\lesssim \left
    (\sum_ j |c_ j|^2\right )^{1/2}
    \end{equation}
for all sequences $\{c_j\}$ in $\ell^2$. To prove \eqref{9}, note
first that $\om^\star\in\R$ by Lemma~\ref{le:sc1}, and hence
\eqref{rk} can be applied for $f'\in A^2_{\omega^\star}$. Using
this, the definitions of $\widetilde{J}$ and~$\phi^\om_{a}$, the
polarization of the identity \eqref{eq:LP2} and the Cauchy-Schwarz
inequality, we obtain
    \begin{equation}\label{6}
    \begin{split}
    \left|\left \langle \widetilde{J} \left(\sum_{j} c_ j e_ j \right), f\right
    \rangle_{A^2_\om} \right|
    &= \left| \left \langle \sum_{j} c_ j \phi^\om_{a_ j},f
    \right \rangle_{A^2_\om} \right |=4\left |\sum_{j} \overline{a_j}c_ j \big \langle
b^{\omega^\star}_{a_ j},f' \big \rangle_{A^2_{\om^\star}}\right |
    \\ &\le4\sum_{j} |c_ j|\,\frac{|f'(a_
    j)|}{\|B^{\om^\star}_{a_ j}\|_{A^2_{\om^\star}}}\\
    &\le4\left (\sum_{j}|c_j|^2\right)^{1/2}\left(\sum_{j}|f'(a_
    j)|^2 \,\|B^{\om^\star}_{a_ j}\|_{A^2_{\om^\star}}^{-2}\right)^{1/2}
    \end{split}
    \end{equation}
for all $f\in A^2_\om$. Applying now Lemma~\ref{kernels} and
Lemma~\ref{le:cuadrado-tienda} to $\om^\star\in\R$, and
Lemma~\ref{le:sc1}, with $\a=2$, we deduce
    \begin{equation}\label{7}
    \begin{split}
    \|B^{\om^\star}_{a_
    j}\|_{A^2_{\om^\star}}^{-2}&\asymp\omega^\star(S(a_j))\asymp\omega^{\star\star}(a_j)\asymp(1-|a_j|)^2\om^\star(a_j),\quad |a_j|\ge\frac12.
    \end{split}
    \end{equation}
Let $\gamma>0$ be the pseudohyperbolic separation constant of
$\{a_j\}$. Then \eqref{7} together with the subharmonicity of
$|f'|^2$, \eqref{eq:r2} for $\om^\star$ and
Theorem~\ref{ThmLittlewood-Paley} yield
    \begin{equation}\label{10}
    \begin{split}
    \sum_{|a_j|\ge\frac12} |f'(a_j)|^2 \,\|B^{\om^\star}_{a_ j}\|_{A^2_{\om^\star}}^{-2}
    &\asymp\sum_{|a_j|\ge\frac12}|f'(a_ j)|^2(1-|a_j|)^2\om^\star(a_j)\\
    &\lesssim \sum_{|a_j|\ge\frac12}\om^\star(a_j)\int_{\Delta\left(a_ j,\frac{\gamma}{2}\right)}
    |f'(z)|^2\,dA(z)\\
    &\lesssim \sum_{|a_j|\ge\frac12}\int_{\Delta\left(a_ j,\frac{\gamma}{2}\right)}
    |f'(z)|^2\om^\star(z)\,dA(z)\\
    &\lesssim\|f'\|^2_{A^2_{\om^\star}}\lesssim\|f\|^2_{A^2_{\om}}.
    \end{split}
    \end{equation}
By combining \eqref{6} and \eqref{10}, we finally obtain
    $$
    \left|\left \langle \widetilde{J} \left(\sum_{j} c_ j e_ j \right), f\right
    \rangle_{A^2_\om}\right|
    \lesssim\left
    (\sum_ j |c_ j|^2\right )^{1/2}\|f\|_{A^2_\omega},\quad f\in
    A^2_\omega,
    $$
which in turn implies \eqref{9}.
\end{proof}

\section{Proofs of the main results}

With the lemmas of the previous section in hand we are ready to
start proving Theorem~\ref{th:shI}. We first deal with the case
$0<p\le1$, and show that $g\in B_p$ is a necessary condition for
$T_g$ to belong to $\SSS_p(A^2_\om)$ for all $1<p<\infty$.

\begin{proposition}\label{pr:necshI}
Let $\om\in\I\cup\R$ and $g\in\H(\D)$. If $p>1$ and $T_g\in
\SSS_p(A^2_\om)$, then $g\in B_p$. If $0<p\le1$, then $T_g\in
\SSS_p(A^2_\om)$ if and only if $g$ is constant.
\end{proposition}

\begin{proof}
Let first $p\ge1$ and $T_g\in \SSS_p\left(A^2_\om\right)$. Let
$\{e_ j\}$ be an orthonormal set in~$A^2_\om$. Let $\{a_ j\}$ be a
uniformly discrete sequence and consider the linear operator
$J=\widetilde{J}\circ P$, where $P$ is the orthogonal projection
from $A^2_\omega$ to $[\{e_j\}]$ and the linear operator
$\widetilde{J}:[\{e_j\}]\to A^2_\omega$ is defined by
$\widetilde{J}(e_j)=\phi^\om_{a_j}$. The operator $J$ is bounded
on~$A^2_\om$ by Lemma~\ref{le:wo}. Since $\mathcal{S}_p(A^2_\om)$
is a two-sided ideal in the space of bounded linear operators on
$A^2_\om$, we have $J^\star T_ g J\in\mathcal{S}_p(A^2_\om)$ by
\cite[p.~27]{Zhu}. Hence \cite[Theorem~1.27]{Zhu} yields
    \begin{displaymath}
    \sum_ j \left|\langle T_ g(\phi^\om_{a_ j}),\phi^\om_{a_ j})\rangle_{A^2_\omega}\right|^p
    =\sum_ j \left|\langle (J^\star T_gJ)(e_j),e_j\rangle_{A^2_\omega}\right|^p<\infty.
    \end{displaymath}
Then, by the polarization of the identity \eqref{eq:LP2},
\eqref{rk} for $\omega^\star$, Lemma~\ref{le:Phi} and \eqref{7},
we obtain
    \begin{equation}
    \begin{split}
    \infty&>\sum_{j}\left|\langle T_ g(\phi^\om_{a_ j}),\phi^\om_{a_ j})\rangle_{A^2_\omega}\right|^p
    =4\sum_{j}|a_j|^{p}\left|\langle g'\phi^\om_{a_j},b^{\omega^\star}_{a_j}\rangle_{A^2_{\om^\star}}\right|^p\\
    &\gtrsim\sum_{|a_j|\ge\frac12}
    \left|\frac{g'(a_j)\Phi^\omega_{a_j}(a_j)}{\|B^{\omega^\star}_{a_j}\|^2_{A^2_{\om^\star}}}\right|^p
    \gtrsim\sum_{|a_j|\ge\frac12}\left|g'(a_j)\right|^p(1-|a_j|)^p.
    \end{split}
    \end{equation}
Therefore for any uniformly discrete sequence $\{a_j\}$, and hence
in particular for any $\delta$-lattice, we have
    $$
    \sum_{|a_j|\ge\frac12}\left|g'(a_j)\right|^p(1-|a_j|)^p<\infty.
    $$
Arguing as in~\cite[p.~917]{RoInd}, this in turn implies
    \begin{equation}\label{11111}
    \int_{\D}|g'(z)|^p(1-|z|^2)^{p-2}\,dA(z)<\infty,
    \end{equation}
which is the assertion for $p>1$.

If $0<p\le1$ and $T_g\in \SSS_p\left(A^2_\om\right)$, then $T_g\in
\SSS_1\left(A^2_\om\right)$, and hence the first part of the proof
gives \eqref{11111} with $p=1$, which implies that $g$ is
constant. Since $T_g=0$ if $g$ is constant, we deduce the
assertion also for $0<p\le1$.
\end{proof}

We note that if $\om$ is regular and $1<p<2$, then the assertion
in Proposition~\ref{pr:necshI} can be proved in an alternative way
by following the argument in \cite[Proposition~7.15]{Zhu} and
using \cite[Theorem~1.26]{Zhu}, Lemmas~\ref{le:cuadrado-tienda}
and~\ref{kernels}, \eqref{22} and
Theorem~\ref{ThmLittlewood-Paley}.

The proof of the fact that $g\in B_p$ is a sufficient condition
for $T_g$ to belong to $\SSS_p(A^2_\om)$ is more involved, in
particular when $p>2$. We begin with the case $1<p<2$ which will
be proved by using ideas from~\cite{PP}. We need an auxiliary
result.

\begin{lemma}\label{le:sp1}
If $\om\in\I\cup\R$, then
    $$
    \left\|\frac{\partial}{\partial\overline{z}}B^\omega_z\right\|_{A^2_\om}
    \lesssim\frac{\,\|B^\omega_ z\|_{A^2_\omega}}{1-|z|},\quad
    z\in\D.
    $$
\end{lemma}

\begin{proof}
Let $\{e_n\}_{n=0}^\infty$ be the orthonormal basis of
$A^2_\omega$ given by $e_n(z)=(2\,\om_n)^{-1/2}z^n$, where
$\om_n=\int_0^1 r^{2n+1}\om(r)\,dr$. If $f(z)=\sum_{n=0}^\infty
e_n(z)$, then Parseval's identity yields
    \begin{equation*}\begin{split}
    M^2_2(r,f)=\frac{1}{2\pi}\int_{0}^{2\pi}\left\vert f(re\sp{i\theta})\right\vert^2d\theta
    =\sum_{n=0}^\infty
    \frac{r^{2n}}{2\,\om_n}
    =\|B^\omega_z\|^2_{A^2_\om},\quad |z|=r.
    \end{split}\end{equation*}
But now
    \begin{displaymath}
    \begin{split}
    \left\|\frac{\partial}{\partial\overline{z}}B^\omega_z\right\|^2_{A^2_\om}
    =\sum_{n=1}^\infty|e'_n(z)|^2
    =\sum_{n=1}^\infty \frac{n^2r^{2n-2}}{2\,\om_n}
    =M^2_2(r,f'),\quad |z|=r,
    \end{split}
    \end{displaymath}
and hence an application of the Cauchy integral formula, Lemma
\ref{kernels} and Lemma~\ref{le:condinte}, with $t=\frac{1+r}{2}$,
give
    \begin{displaymath}
    \left\|\frac{\partial}{\partial\overline{z}}B^\omega_z\right\|_{A^2_\omega}=M_
    2(r,f')\lesssim\frac{M_
    2\left(\frac{1+r}{2},f\right)}{(1-r)}\asymp
    \frac{\,\|B^\omega_z\|_{A^2_\omega}}{1-|z|},\quad |z|=r,
    \end{displaymath}
which is the desired estimate.
\end{proof}

\begin{proposition}\label{pr:shR1}
Let $1<p<2$ and $\om\in\I\cup\R$. If $g\in B_p$, then $T_g\in
\SSS_p(A^2_\om)$.
\end{proposition}

\begin{proof}
If $1<p<\infty$, then $T_g\in\mathcal{S}_p(A^2_\om)$ if and only
if
    \begin{displaymath}
    \sum_n\left|\langle T_g(e_n),e_n\rangle_{A^2_\om}\right|^p<\infty
    \end{displaymath}
for any orthonormal set $\{e_n\}$, see~\cite[Theorem~1.27]{Zhu}.
Let $1<p<2$ and let~$\{e_n\}$ be an orthonormal set in $A^2_\om$.
Then the polarization of the identity \eqref{eq:LP2}, two
applications of H\"older's inequality, Lemma~\ref{le:sp1},
Lemma~\ref{kernels} and Lemma~\ref{le:cuadrado-tienda} yield
    \begin{equation}
    \begin{split}\label{eq:1}
    &\sum_n\left|\langle T_g(e_n), e_n\rangle_{A^2_\om}\right|^p \\ &
    \lesssim
    \sum_n \left(\int_\D|g'(z)||e_n(z)||e'_n(z)|\om^\star(z)\,dA(z)
    \right)^p\\
    &\le\sum_n\left(\int_\D|g'(z)|^p|e_n(z)|^p|e'_n(z)|^{2-p}\om^\star(z)\,dA(z)\right)\\
    &\quad\cdot\left(\int_\D|e'_n(z)|^2\om^\star(z)\,dA(z)\right)^{p-1}\\
    &\lesssim\int_\D|g'(z)|^p\left(\sum_n|e_n(z)|^p|e'_n(z)|^{2-p}\right)\om^\star(z)\,dA(z)\\
    &\le\int_\D|g'(z)|^p\left(\sum_n|e_n(z)|^2\right)^{\frac{p}{2}}\left(\sum_n|e'_n(z)|^{2}\right)^{1-\frac{p}{2}}
    \om^\star(z)\,dA(z)
    \\ & \lesssim\int_\D|g'(z)|^p\frac{\|B^\omega_z\|^2_{A^2_\omega}}{(1-|z|)^{2-p}}\,\om^\star(z)\,dA(z)
    \asymp\|g\|_{B_p}^p.
    \end{split}\end{equation}
Thus, $T_g\in \SSS_p(A^2_\om)$ if $g\in B_p$.
\end{proof}

We have now proved Theorem~\ref{th:shI} when $\om\in\I\cup\R$ and
$0<p<2$. As mentioned earlier, the case $\om\in\R$ and $2\le
p<\infty$ follows by \cite[Theorem~5.1]{OC} and
Lemma~\ref{le:RAp}(i). We now proceed to the remaining case
$\om\in\I$ and $2\le p<\infty$. The proof we are going to present
works also for $\om\in\R$.

\subsection{Dirichlet type spaces induced by $\om^\star$}\index{Dirichlet type space}

For $\alpha\in\mathbb{R}$ and $\om\in\I\cup\R$, let us consider
the Hilbert space
$H_\alpha(\om^\star)$,\index{$H_\alpha(\om^\star)$}\index{$\omega_\alpha(z)$}
which consists of those $f\in\H(\D)$ whose Maclaurin series
$\sum_{n=0}^\infty a_n z^n$ satisfies
    $$
    \sum_{n=0}^\infty
    (n+1)^{\alpha+2}\om^\star_n|a_{n+1}|^2<\infty.
    $$

Bearing in mind Theorem~\ref{ThmLittlewood-Paley}, we deduce the
identity $H_0(\om^\star)=A^2_\omega$, which is a special case of
the following result. Recall that
$\omega_\alpha(z)=(1-|z|)^\alpha\omega(z)$ for all
$\alpha\in\mathbb{R}$.

\begin{lemma}\label{le:sc3}
If $-\infty<\alpha<2$ and $\om\in\I\cup\R$, then
    \begin{equation}\label{eq:sc31}\index{$H_\alpha(\om^\star)$}
    H_\alpha(\omega^\star)=\left\{f\in\H(\D):\,\int_\D |f'(z)|^2\,\om^\star_{-\alpha}(z)\,dA(z)<\infty\right\}.
    \end{equation}
In particular, if $\alpha<0$, then
$H_\alpha(\omega^\star)=A^2_{\omega_{-\alpha-2}^\star}$.
\end{lemma}\index{$H_\alpha(\om^\star)$}

\begin{proof}
If $f(z)=\sum_{n=0}^\infty a_nz^n\in\H(\D)$, then
Lemma~\ref{le:sc2} and Parseval's identity give
    \begin{equation*}\begin{split}
    \sum_{n=0}^\infty(n+1)^{\alpha+2}\omega^\star_n|a_{n+1}|^2
    &=\sum_{n=0}^\infty(n+1)^2(n+1)^{2-(-\alpha+2)}\omega^\star_n|a_{n+1}|^2\\
    &\asymp
    2\sum_{n=0}^\infty(n+1)^2|a_{n+1}|^2\int_0^1r^{2n+1}\om_{-\alpha}^\star(r)\,dr\\
    &=\int_\D |f'(z)|^2\,\om^\star_{-\alpha}(z)\,dA(z),
    \end{split}\end{equation*}
which proves \eqref{eq:sc31}.

The identity
$H_\alpha(\omega^\star)=A^2_{\omega_{-\alpha-2}^\star}$,
$-\infty<\alpha<0$, is an immediate consequence of
\eqref{eq:sc31}, Theorem~\ref{ThmLittlewood-Paley} and
Lemma~\ref{le:sc1}.
\end{proof}\index{$H_\alpha(\om^\star)$}

Lemma~\ref{le:sc3} allows us to define an inner product on
$H_\alpha(\om^\star)$, $-\infty<\alpha<2$, by
    \begin{equation}
    \begin{split}\label{Eq:InnerProduct}
    \langle f,g\rangle_{H_\alpha(\om^\star)}
    &=f(0)\overline{g(0)}\om^\star_{-\alpha}(\D)\\
    &\quad+2\sum_{n=0}^\infty(n+1)^2a_{n+1}\overline{b_{n+1}}\int_0^1r^{2n+1}\om_{-\alpha}^\star(r)\,dr\\
    &=f(0)\overline{g(0)}\om^\star_{-\alpha}(\D)+\int_\D
    f'(z)\overline{g'(z)}\,\om^\star_{-\alpha}(z)\,dA(z),
    \end{split}
    \end{equation}
where $\sum_{n=0}^\infty a_nz^n$ and $\sum_{n=0}^\infty b_nz^n$
are the Maclaurin series of $f$ and $g$ in $\D$, respectively. It
also follows from Lemma~\ref{le:sc3} that each point evaluation
$L_a(f)=f(a)$ is a bounded linear functional on
$H_\alpha(\om^\star)$ for all $-\infty<\alpha<2$.
\index{$H_\alpha(\om^\star)$} Therefore there exist reproducing
kernels\index{reproducing kernel} $K^\alpha_a\in
H_\alpha(\om^\star)$\index{$K^\alpha_a$} with $\|L_
a\|=\|K_a^\a\|_{H_\alpha(\om^\star)}$ such that $L_a(f)=\langle f,
K^\alpha_ a\rangle_{H_\alpha(\om^\star)}$, and thus
    \begin{equation}\label{eq:repro}
    f(a)=f(0)\overline{K^\alpha_a(0)}\om^\star_{-\alpha}(\D)
    +\int_\D f'(z)\overline{\frac{\partial K^\alpha_a(z)}{\partial z}}\,\om^\star_{-\alpha}(z)\,dA(z)
    \end{equation}
for all $f\in H_\alpha(\om^\star)$.

\begin{lemma}\label{le:rker}
If $-\infty<\alpha<1$ and $\om\in\I\cup\R$, then
    \begin{equation}\label{Eq:le:rker}
    \|K^\alpha_a\|_{H_\alpha(\om^\star)}^2\asymp\frac{1}{\om_{-\alpha}^\star(a)},\quad |a|\ge\frac12.
    \end{equation}
\end{lemma}

\begin{proof}
If $-\infty<\alpha<0$, then Lemma~\ref{le:sc3},
Lemma~\ref{kernels}, Lemma~\ref{le:cuadrado-tienda}, and
Lemma~\ref{le:sc1} yield
    \begin{equation*}\index{$H_\alpha(\om^\star)$}
    \begin{split}
    \|K^\alpha_a\|^2_{H_\alpha(\om^\star)}
    &\asymp\|B^{\omega^\star_{-\alpha-2}}_a\|^2_{A^2_{\omega^\star_{-\alpha-2}}}
    \asymp\frac{1}{\omega^\star_{-\alpha-2}(S(a))}\\
    &\asymp\frac{1}{(\omega^\star_{-\alpha-2})^\star(a)}
    \asymp\frac{1}{\omega^\star_{-\alpha}(a)},\quad |a|\ge\frac12.
    \end{split}
    \end{equation*}
A similar reasoning involving Theorem~\ref{ThmLittlewood-Paley}
and Lemma~\ref{le:cuadrado-tienda} gives the assertion for
$\alpha=0$.

To prove the case $0<\a<1$, consider the orthonormal basis
$\{e_n\}_{n=0}^\infty$ of $H_\alpha(\om^\star)$ given by
    $$
    e_0(z)=\left(\frac{1}{\omega_{-\alpha}^\star(\D)}\right)^{1/2},\quad
    e_n(z)=\frac{z^n}{n\left(2\int_0^1r^{2n-1}\om_{-\alpha}^\star(r)\,dr\right)^{1/2}},\quad
    n\in\N.
    $$
Lemma~\ref{le:sc2} yields
    \begin{equation}\begin{split}\label{eq:ker1}
    \|K^\alpha_a\|^2_{H_\alpha(\om^\star)}&=K^\alpha_a(a)=
    \frac{1}{\omega_{-\alpha}^\star(\D)}+\sum_{n=0}^\infty\frac{|a|^{2n+2}}{2(n+1)^2\int_0^1r^{2n+1}\om_{-\alpha}^\star(r)\,dr}\\
    &\asymp1+\sum_{n=0}^\infty\frac{|a|^{2n+2}}{(n+1)^{2+\alpha}\om^\star_n},\quad
    a\in\D.
    \end{split}
    \end{equation}
Take $N\in\N$ such that $1-\frac{1}{N}\le |a|<1-\frac{1}{N+1}$.
Using \eqref{eq:LP2}, with $f(z)=z^{n+1}$, we deduce
    $$
    4(n+1)^2\om^\star_n=\int_0^1r^{2n+3}\om(r)\,dr,\quad n\in\N\cup\{0\},
    $$
and hence $\{\om^\star_n(n+1)^2\}_{n=0}^\infty$ is a decreasing
sequence. This together with the fact $\om^\star\in\R$, which
follows by Lemma~\ref{le:sc1}, and Lemma~\ref{le:sc2} gives
    \begin{equation}
    \begin{split}\label{eq:ker2}
    \sum_{n=0}^N\frac{|a|^{2n+2}}{(n+1)^{2+\alpha}\om^\star_n}
    &\le\frac{1}{(N+1)^{2}\om^\star_N}\sum_{n=0}^N\frac{|a|^{2n+2}}{(n+1)^{\alpha}}\\
    &\lesssim\frac{1}{(N+1)\om^\star\left(1-\frac{1}{N+1}\right)}\sum_{n=0}^\infty\frac{|a|^{2n+2}}{(n+1)^{\alpha}}\\
    &\lesssim\frac{1-|a|}{\om^\star(a)}(1-|a|)^{\alpha-1}=\frac{1}{\om_{-\alpha}^\star(a)},\quad
    a\in\D,
    \end{split}
    \end{equation}
and
    \begin{equation}
    \begin{split}\label{eq:ker2999}
    \sum_{n=N}^\infty\frac{|a|^{2n+2}}{(n+1)^{2+\alpha}\om^\star_n}
    &\ge\frac{1}{(N+1)^{2}\om^\star_N}\sum_{n=N}^\infty\frac{|a|^{2n+2}}{(n+1)^{\alpha}}\\
    &\asymp\frac{1}{(N+1)\om^\star\left(1-\frac{1}{N+1}\right)}\sum_{n=N}^\infty\frac{|a|^{2n+2}}{(n+1)^{\alpha}}\\
    &\gtrsim\frac{1-|a|}{\om^\star(a)}(1-|a|)^{\alpha-1}=\frac{1}{\om_{-\alpha}^\star(a)},\quad
    |a|\ge\frac12.
    \end{split}
    \end{equation}
Moreover, \eqref{Eq:le:rker} with $\alpha=0$ yields
    \begin{equation}\label{11}
    \begin{split}
    \sum_{n=N+1}^\infty\frac{|a|^{2n+2}}{(n+1)^{2+\alpha}\om^\star_n}
    &\lesssim \frac{1}{(N+1)^{\alpha}} \sum_{n=0}^\infty\frac{|a|^{2n+2}}{(n+1)^{2}\om^\star_n}\\
    &\lesssim (1-|a|)^\alpha\|K^0_a\|^2_{H_0(\om^\star)}\asymp\frac{1}{\om_{-\alpha}^\star(a)}.
    \end{split}
    \end{equation}
By combining \eqref{eq:ker1}--\eqref{11} we obtain the assertion
for $0<\alpha<1$.
\end{proof}

\subsection{Toeplitz operator $T_\mu$ and complex interpolation
technique}\index{Toeplitz operator}

We next consider a decomposition of $\D$ into disjoint sets of
roughly equal size in the hyperbolic sense. Let $\Upsilon$ denote
the family of all dyadic arcs of $\T$. Every dyadic arc
$I\subset\T$ is of the form
    $$
    I_{n,k}=\left\{e^{i\theta}:\,\frac{2\pi k}{2^n}\le
    \theta<\frac{2\pi(k+1)}{2^n}\right\},
    $$
where $k=0,1,2,\dots,2^n-1$ and $n\in\N\cup\{0\}$. For each
$I\subset\T$, set
    $$
    R(I)=\left\{z\in\D:\,\frac{z}{|z|}\in I,\,\,1-\frac{|I|}{2\pi}\le r<1-\frac{|I|}{4\pi}\right\}.
    $$
Then the family $\left\{R(I):\,\,I\in\Upsilon\right\}$ consists of
pairwise disjoint sets whose union covers~$\D$. For
$I_j\in\Upsilon\setminus\{I_{0,0}\}$, we will write $z_j$ for the
unique point in $\D$ such that $z_j=(1-|I_j|/2\pi)a_j$, where
$a_j\in\T$ is the midpoint of $I_j$. For convenience, we associate
the arc $I_{0,0}$ with the point $1/2$.

Now, if $\mu$ is a complex Borel measure on $\D$, let us consider
the\index{$H_\alpha(\om^\star)$} operator
    $$
    T_\mu(f)(w)=\int_\D f(z)K^\alpha(w,z)\,d\mu(z),\quad f\in
    H_\alpha(\om^\star),\index{$T_\mu(f)$}
    $$
where $K^\alpha(w,z)=K_z^\alpha(w)=\overline{K_w^\alpha(z)}$. This
operator has been studied, for instance, in~\cite{Lu87}.

\begin{theorem}\label{th:sufschpmayor2}
Let $1\le p<\infty$ and $-\infty<\alpha<1$ such that $p\alpha<1$.
Let $\omega\in\I\cup\R$, and let $\mu$ be a complex Borel measure
on $\D$. If
    \begin{equation}\label{eq:sl}\index{$H_\alpha(\om^\star)$}
    \sum_{R_j\in\Upsilon}
    \left(\frac{|\mu|(R_j)}{\omega^\star_{-\alpha}(z_j)}\right)^p<\infty,
    \end{equation}
then $T_\mu\in\SSS_p(H_\alpha(\om^\star))$, and there exists a
constant $C>0$ such that
\begin{equation}\label{eq:sl2}
|T_\mu|_p^p\le C \sum_{R_j\in\Upsilon}
    \left(\frac{|\mu|(R_j)}{\omega^\star_{-\alpha}(z_j)}\right)^p.
    \end{equation}
Conversely, if $\mu$ is a positive Borel measure on $\D$ and
$T_\mu\in\SSS_p(H_\alpha(\om^\star))$, then \eqref{eq:sl} is
satisfied.
\end{theorem}\index{$H_\alpha(\om^\star)$}

We consider this result of its own interest because it allows us
to complete the proof of Theorem~\ref{th:shI} for rapidly
increasing weights.

\begin{corollary}\label{co:shIpmayor2}\index{Besov space}\index{$\Vert\cdot\Vert_{B_p}$}
Let $2\le p<\infty$, $\om\in\I\cup\R$ and $g\in\H(\D)$. Then,
$T_g\in\SSS_p(A^2_\om)$ if and only if $g\in B_p$.
\end{corollary}

\begin{proof}
By the proof of \cite[Theorem~2]{AS0}, $g\in B_p$ if and only if
    \begin{equation}\label{eq:as}
    \sum_{R_j\in\Upsilon}
    \left(\int_{R_j}|g'(z)|^2\,dA(z)\right)^{p/2}<\infty,
    \end{equation}
which is equivalent to
    $$
    \sum_{R_j\in\Upsilon}
    \left(\frac{\int_{R_j}|g'(z)|^2\om^\star(z)\,dA(z)}{\omega^\star(z_j)}\right)^{p/2}<\infty
    $$
since $\om^\star\in\R$ by Lemma~\ref{le:sc1}. Therefore we may
apply Theorem~\ref{th:sufschpmayor2} with $\alpha=0$ to the
measure $d\mu_g(z)=|g'(z)|^2\om^\star(z)\,dA(z)$ to deduce
$T_{\mu_g}\in\SSS_{p/2}(H_0(\omega^\star))$ if and only if $g\in
B_p$. But $H_0(\omega^\star)=A^2_\omega$ by
Theorem~\ref{ThmLittlewood-Paley}, and hence
$T_{\mu_g}\in\SSS_{p/2}(A^2_\omega)$. On the other hand,
\eqref{eq:repro} gives
    \begin{equation*}
    \langle (T_g^\star T_g)(f),h\rangle_{H_0(\om^\star)}
    =\langle  T_g(f),T_g(h)\rangle_{H_0(\om^\star)}
    =\int_{\D}f(z)\overline{h(z)}|g'(z)|^2\om^\star(z)\,dA(z),
    \end{equation*}
so, by taking $h(z)=K^0_w(z)$, we deduce
    $$
    (T_g^\star
    T_g)(f)(w)=\int_{\D}f(z)K^0(w,z)|g'(z)|^2\om^\star(z)\,dA(z)=T_{\mu_g}(f)(w).
    $$
Therefore $g\in B_p\Leftrightarrow T_g^\star T_g\in
\SSS_{p/2}(A^2_\om)\Leftrightarrow T_g\in\SSS_p(A^2_\om)$, where
in the last equivalence we use \cite[Theorem~1.26]{Zhu}.
\end{proof}

Propositions~\ref{pr:necshI} and~\ref{pr:shR1}, and
Corollary~\ref{co:shIpmayor2} yield Theorem~\ref{th:shI}. Finally,
we prove Theorem~\ref{th:sufschpmayor2}.

\subsection*{Proof of Theorem~\ref{th:sufschpmayor2}.} We borrow the
argument from the proof of the main theorem in
\cite[p.~352--355]{Lu87}. First, we will show that \eqref{eq:sl}
implies $T_\mu\in\SSS_p(H_\alpha(\om^\star))$ and \eqref{eq:sl2}.
This part of the proof will be divided into two
steps.\index{$H_\alpha(\om^\star)$}

\smallskip
\noindent\textbf{First Step.} We begin with showing by a limiting
argument that it is enough to prove the assertion for measures
with compact support. For simplicity we only consider positive
Borel measures, a similar reasoning works also for complex Borel
measures.

Assume that there exists a constant $C>0$ such that \eqref{eq:sl2}
is satisfied for all positive compactly supported Borel measures
on $\D$ that fulfill \eqref{eq:sl}. Let $\mu$ be a positive Borel
measure on $\D$ such that
    $$
    M_p=\sum_{R_j\in\Upsilon}\left(\frac{\mu(R_j)}{\omega^\star_{-\alpha}(z_j)}\right)^p<\infty.
    $$
For each $k\in\N$, let $\mu_k$ denote the measure defined by
    $$
    \mu_k(E)=\mu\left(E\cap
    D\left(0,1-\frac{1}{k}\right)\right)
    $$
for any Borel set $E\subset\D$, and set
    $$
    M^k_p=\sum_{R_j\in\Upsilon}\left(\frac{\mu_k(R_j)}{\omega^\star_{-\alpha}(z_j)}\right)^p.
    $$
Since $\mu_k$ has compact support and $M_p^k\le M_p<\infty$ for
all $k\in\N$, the assumption yields
$T_{\mu_k}\in\SSS_p(H_\alpha(\om^\star))$ for all $k\in\N$, and
    \begin{equation}\label{eq:st11}\index{$H_\alpha(\om^\star)$}
    |T_{\mu_k}|_p^p\le C M^k_p\le C M_p<\infty.
    \end{equation}
In particular, $T_{\mu_k}$ is a compact operator on
$H_\alpha(\om^\star)$ with $\|T_{\mu_k}\|^p\le CM_p$ for all
$k\in\N$.\index{$H_\alpha(\om^\star)$}

Consider the identity operator $I_d:H_\alpha(\om^\star)\to
L^2(\mu)$.\index{$H_\alpha(\om^\star)$} The
definition~\eqref{Eq:InnerProduct}, Fubini's theorem and
\eqref{eq:repro} yield
    \begin{equation}\label{eq:st12}\index{$H_\alpha(\om^\star)$}
    \begin{split}
    \langle T_\mu(g),f\rangle_{H_\alpha(\om^\star)}
    &=T_\mu(g)(0)\overline{f(0)}\,\om^\star_{-\alpha}(\D)\\
    &\quad+\int_\D
    \left(\int_\D g(z)\frac{\partial K^\alpha(\zeta,z)}{\partial \zeta}\,d\mu(z)\right)
    \overline{f'(\zeta)}\,\om_{-\alpha}^\star(\zeta)\,dA(\zeta)\\
    &=\int_\D g(z)\left(\overline{f(0)}K^\alpha(0,z)\om_{-\alpha}^\star(\D)\right)\,d\mu(z)\\
    &\quad+\int_\D g(z)\left(\int_\D\frac{\partial
    K^\alpha(\zeta,z)}{\partial \zeta}\overline{f'(\zeta)}\om^\star_{-\alpha}(\zeta)\,dA(\zeta)\right)d\mu(z)\\
    &=\int_\D g(z)\overline{f(z)}\,d\mu(z)=\langle I_d(g),
    I_d(f)\rangle_{L^2(\mu)}\\
    &=\langle (I_d^\star I_d)(g),f\rangle_{H_\alpha(\om^\star)}
    \end{split}
    \end{equation}
for all $g$ and $f$ in $H_\alpha(\om^\star)$, and thus
$T_\mu=I_d^\star I_d$. Therefore $T_\mu$ is bounded (resp.
compact) on $H_\alpha(\om^\star)$ if and only if
$I_d:\,H_\alpha(\om^\star)\to L^2(\mu)$ is bounded (resp.
compact). Since the same is true if $\mu$ is replaced by $\mu_k$,
we deduce that $I_d:\,H_\alpha(\om^\star)\to L^2(\mu_k)$ is
compact with $\|I_d\|^2\asymp\|T_{\mu_k}\|\le CM_p$ for all
$k\in\N$ by the previous paragraph. Now, by using the monotone
convergence theorem, we obtain
    \begin{equation*}\index{$H_\alpha(\om^\star)$}
    \begin{split}
    \int_\D|f(z)|^2\,d\mu(z)&=\int_\D\left(\lim_{k\to\infty}|f(z)|^2\chi_{D(0,1-\frac{1}{k})}(z)\right)\,d\mu(z)=\lim_{k\to\infty}\int_\D|f(z)|^2\,d\mu_k(z)\\
    &\lesssim\left(\lim_{k\to\infty}\|T_{\mu_k}\|\right)\|f\|^2_{H_\alpha(\om^\star)}
    \le C M_p\|f\|^2_{H_\alpha(\om^\star)}
    \end{split}
    \end{equation*}
for all $f\in H_\alpha(\om^\star)$. Therefore
$I_d:H_\alpha(\om^\star)\to L^2(\mu)$ is bounded and so is $T_\mu$
on $H_\alpha(\om^\star)$. Furthermore, by using standard tools and
\eqref{eq:st12}, we deduce
    \begin{equation}\label{eq:st13}\index{$H_\alpha(\om^\star)$}
    \begin{split}
    &\|T_{\mu}-R\|\\
    &=\sup_{\|f\|_{H_\alpha(\om^\star)}\le1,\,\|g\|_{H_\alpha(\om^\star)}\le1}
    \left|\langle
    (T_{\mu}-R)(f),g\rangle_{H_\alpha(\om^\star)}\right|\\
    &=\sup_{\|f\|_{H_\alpha(\om^\star)}\le1,\,\|g\|_{H_\alpha(\om^\star)}\le1}
    \left|\int_\D f(z)\overline{g(z)}\,d\mu(z)-\langle R(f),g
    \rangle_{H_\alpha(\om^\star)}\right|\\
    &=\sup_{\|f\|_{H_\alpha(\om^\star)}\le1,\,\|g\|_{H_\alpha(\om^\star)}\le1}\lim_{k\to\infty}
    \left|\int_\D f(z)\overline{g(z)}\,d\mu_k(z)-\langle R(f),g
    \rangle_{H_\alpha(\om^\star)}\right|\\
    &=\lim_{k\to\infty}\sup_{\|f\|_{H_\alpha(\om^\star)}\le1,\,\|g\|_{H_\alpha(\om^\star)}\le1}
    \left|\int_\D f(z)\overline{g(z)}\,d\mu_k(z)-\langle R(f), g \rangle_{H_\alpha(\om^\star)}\right|\\
    &=\lim_{k\to\infty}\|T_{\mu_k}-R\|
    \end{split}
    \end{equation}
for any bounded operator $R$ on $H_\alpha(\om^\star)$. Therefore
$\lim_{k\to\infty}\|T_{\mu_k}-T_{\mu}\|=0$, which implies that
$T_{\mu}$ is compact. We also  deduce from \eqref{eq:st13} that
$\lambda_n(T_\mu)=\lim_{k\to\infty}\lambda_n(T_{\mu_k}),$\, which
together with \eqref{eq:st11} and Fatou's lemma gives
    \begin{equation*}
    \begin{split}
    |T_\mu|^p_p&=\sum_{n=0}^\infty \lambda_n(T_\mu)^p
    =\sum_{n=0}^\infty \left(\lim_{k\to\infty}\lambda_n(T_{\mu_k})^p\right)\\
    &\le\liminf_{k\to\infty}\left(\sum_{n=0}^\infty \lambda_n(T_{\mu_k})^p \right)
    \le\limsup_{k\to\infty}|T_{\mu_k}|^p_p\le CM_p.
    \end{split}
    \end{equation*}
This completes the proof of the first step.

\smallskip
\noindent\textbf{Second Step.} Assume that $\mu$ is a compactly
supported complex Borel measure on $\D$. The proof is by complex
interpolation. First, we will obtain the assertion for $p=1$.
Since $T_\mu$ is compact on $H_\alpha(\om^\star)$, there are
orthonormal sets $\{e_n\}$ and $\{f_n\}$ on $H_\alpha(\om^\star)$
such that
    $$\index{$H_\alpha(\om^\star)$}
    T_\mu(f)=\sum_n\lambda_n\langle f,e_n\rangle_{H_\alpha(\om^\star)}f_n
    $$
for all $f\in H_\alpha(\om^\star)$.\index{$H_\alpha(\om^\star)$}
Therefore the formula \eqref{Eq:InnerProduct}, Fubini's theorem,
the Cauchy-Schwarz inequality, \eqref{eqRK1} and
Lemma~\ref{le:rker} yield
    \begin{equation}\label{eq:sufp1}
    \begin{split}
    |T_\mu|_{\SSS_1(H_\alpha(\om^\star))}
    &=\sum_n\left|\langle
    T_\mu(e_n),f_n\rangle_{H_\alpha(\om^\star)}\right|
    =\sum_n \bigg|T_\mu(e_n)(0)\overline{f_n(0)}\,\om^\star_{-\alpha}(\D)\\
    &\quad+\int_\D\left(\int_\D e_n(z)\frac{\partial K^\alpha(w,z)}{\partial w}\,d\mu(z)\right)
    \overline{f'_n(w)}\om_{-\alpha}^\star(w)\,dA(w)\bigg|\\
    &=\sum_n\bigg|\int_\D e_n(z)\bigg(\overline{f_n(0)}K^\alpha(0,z)\om_{-\alpha}^\star(\D)\\
    &\quad+\int_\D\frac{\partial K^\alpha(w,z)}{\partial w}\overline{f'_n(w)}\,\om^\star_{-\alpha}(w)\,dA(w)\bigg)\,d\mu(z)\bigg|\\
    &=\sum_n\left|\int_\D e_n(z)\overline{f_n(z)}\,d\mu(z)\right|\\
    &\le\int_\D
    \left(\sum_n|e_n(z)|^2\right)^{1/2}\left(\sum_n|f_n(z)|^2\right)^{1/2}\,d|\mu|(z)\\
    &\le\int_\D\|K^\alpha_z\|^2_{H_\alpha(\om^\star)}\,d|\mu|(z)
    \asymp\int_\D\frac{d|\mu|(z)}{\om_{-\alpha}^\star(z)}
    \asymp\sum_{R_j\in\Upsilon}\frac{|\mu|(R_j)}{\om_{-\alpha}^\star(z_j)},
    \end{split}\index{$H_\alpha(\om^\star)$}
    \end{equation}
because $\omega_{-\alpha}^\star\in\R$ by Lemma~\ref{le:sc1}. Thus
the assertion is proved for $p=1$ and $-\infty<\alpha<1$.

Let now $1<p<\infty$ and $-\infty<\alpha<1$ with $p\alpha<1$. Take
$\varepsilon>0$ such that
$\alpha<2\varepsilon<\frac{1-\alpha}{p-1}$. For $\zeta$ in the
strip $\Lambda=\{w:0\le \Re w\le 1\}$, define the differentiation
operator $Q_\zeta$ on $\H(\D)$ by
    $$
    Q_\zeta\left(\sum_{n=0}^\infty a_nz^n\right)=\sum_{n=0}^\infty
    (n+1)^{\varepsilon(1-p\zeta)}a_nz^n,
    $$
and let $\gamma=\alpha-2\varepsilon(1-p\Re\zeta)$. It is easy to
see that $Q_\zeta$ is a bounded invertible operator from
${H_\alpha(\om^\star)}$\index{$H_\alpha(\om^\star)$} to
$H_\gamma(\om^\star)$ and $\|Q_\zeta\|=\|Q_{\Re\zeta}\|$. If
$\Re\zeta=0$, then
$H_\gamma(\om^\star)=H_{\alpha-2\varepsilon}(\om^\star)=A^2_{\om^\star_{-(\alpha-2\varepsilon)-2}}$
by Lemma~\ref{le:sc3}, because $\alpha-2\varepsilon<0$. If
$\Re\zeta=1$, then $\gamma<1$ because
$2\varepsilon<\frac{1-\alpha}{p-1}$. For each $\zeta\in\Lambda$,
define the measure-valued function $\mu_\zeta$ by
    $$
    \mu_\zeta=\sum_{R_j\in\Upsilon}\left(\frac{|\mu|(R_j)}{\om_{-\alpha}^\star(z_j)}
    \right)^{p\zeta-1}\chi_{R_j}\mu,
    $$
where the coefficient of $\chi_{R_j}$ is taken to be zero if
$\mu(R_j)=0$. Further, define the analytic operator-valued
function\index{$H_\alpha(\om^\star)$} $S_\zeta$ on
${H_\alpha(\om^\star)}$ by
    $$
    S_\zeta(f)(w)=\int_\D
    Q_\zeta(f)(z)\,\overline{Q_{\overline{\zeta}}(K_w^\alpha)(z)}
    (1-|z|)^{2\varepsilon(1-p\zeta)}\,d\mu_{\zeta}(z).
    $$
Since $\mu$ has compact support, $\|S_\zeta\|$ is uniformly
bounded on $\Lambda$. Moreover, $S_{\frac1p}=T_\mu$. By
\cite[Theorem~13.1]{Gohberg}, it suffices to find constants
$M_0>0$ and $M_1>0$ such that $\|S_\zeta\|\le M_0$ when
$\Re\zeta=0$, and $|S_\zeta|_1\le M_1$ when $\Re\zeta=1$, because
then it follows that $T_\mu\in \SSS_p(H_\alpha(\om^\star))$ and
$|T_\mu|_p\le M_0^{1-1/p}M_1^{1/p}$.\index{$H_\alpha(\om^\star)$}

Let us find $M_0$ first. By using Fubini's theorem,
\eqref{eq:repro}, the identity
$Q_\zeta(K_z^\alpha)(w)=\overline{Q_{\overline{\zeta}}(K_w^\alpha)(z)}$
and arguing as in \eqref{eq:st12} we obtain
    \begin{equation}\label{28}
    \langle S_\zeta(f),g\rangle_{H_\a(\om^\star)}
    =\int_\D
    Q_\zeta(f)(z)\overline{Q_{\overline{\zeta}}(g)(z)}(1-|z|)^{2\varepsilon(1-p\zeta)}\,d\mu_\zeta(z)
    \end{equation}
for all $f$ and $g$ in
$H_\alpha(\om^\star)$,\index{$H_\alpha(\om^\star)$} and
consequently,
    \begin{equation}\label{25}
    \begin{split}
    \left|\langle S_\zeta(f),g\rangle_{H_\a(\om^\star)}\right|
    &\le\left(\int_\D|Q_\zeta(f)(z)|^2(1-|z|)^{2\varepsilon}\,d|\mu_\zeta|(z)\right)^{1/2}\\
    &\quad\cdot\left(\int_\D|Q_{\overline{\zeta}}(g)(z)|^2(1-|z|)^{2\varepsilon}\,d|\mu_\zeta|(z)\right)^{1/2}
    \end{split}
    \end{equation}
if $\Re\zeta=0$. Now $\alpha-2\varepsilon<0$ by our choice, so
Lemma~\ref{le:sc3} yields
$H_{\alpha-2\varepsilon}(\om^\star)=A^2_{\om^\star_{-(\alpha-2\varepsilon)-2}}$,
where $\om^\star_{-(\alpha-2\varepsilon)-2}\in\R$ by
Lemma~\ref{le:sc1}. Moreover, both $Q_\zeta(f)$ and
$Q_{\overline{\zeta}}(g)$ belong to
$H_{\alpha-2\varepsilon}(\om^\star)=A^2_{\om^\star_{-(\alpha-2\varepsilon)-2}}$
because $\gamma=\alpha-2\varepsilon$ if $\Re\zeta=0$. Since
$|Q_\zeta(f)|^2$ is subharmonic and
$\om^\star_{-(\alpha-2\varepsilon)-2}\in\R$, we obtain
    \begin{equation}\label{33}
    |Q_\zeta(f)(z)|^2
    \lesssim\frac{\int_{\Delta(z,r)}|Q_\zeta(f)(\xi)|^2\omega^\star_{-(\alpha-2\varepsilon)-2}(\xi)\,dA(\xi)}
    {\omega^\star_{-(\alpha-2\varepsilon)}(z)}
    \end{equation}
for any fixed $r\in(0,1)$. This together with Fubini's theorem
yields
    \begin{equation}\label{36}
    \begin{split}
    &\int_\D|Q_\zeta(f)(z)|^2(1-|z|)^{2\varepsilon}\,d|\mu_\zeta|(z)\\
    &\lesssim\int_\D|Q_\zeta(f)(\xi)|^2\omega^\star_{-(\a-2\e)-2}(\xi)\left(\int_{\Delta(\xi,r)}
    \frac{d|\mu_\zeta|(z)}{\omega^\star_{-\a}(z)}\right)dA(\xi)\\
    &\le\sup_{\xi\in\D}\left(\int_{\Delta(\xi,r)}
    \frac{d|\mu_\zeta|(z)}{\omega^\star_{-\a}(z)}\right)
    \int_\D|Q_\zeta(f)(\xi)|^2\omega^\star_{-(\a-2\e)-2}(\xi)dA(\xi)\\
    &\asymp\sup_{R_j\in\Upsilon}\frac{|\mu_\zeta|(R_j)}{\omega^\star_{-\alpha}(z_j)}\,
    \|Q_\zeta(f)\|^2_{A^2_{\omega^\star_{-(\alpha-2\e)-2}}},
    \end{split}
    \end{equation}
and the same is true for $Q_{\overline{\zeta}}(g)$ in place of
$Q_{{\zeta}}(f)$. Therefore \eqref{25} gives
    \begin{equation}\label{26}
    \begin{split}
    \left|\langle S_\zeta(f),g\rangle_{H_\a(\om^\star)}\right|
    &\lesssim C^2(\mu_\zeta)\|Q_\zeta(f)\|_{A^2_{\om^\star_{-(\alpha-2\varepsilon)-2}}}
    \|Q_{\overline{\zeta}}(g)\|_{A^2_{\om^\star_{-(\alpha-2\varepsilon)-2}}}\\
    &\lesssim C^2(\mu_\zeta)\|Q_0\|^2\|f\|_{H_{\alpha}(\om^\star)}\|g\|_{H_{\alpha}(\om^\star)},
    \end{split}
    \end{equation}
where
    \begin{equation*}
    C(\mu_\zeta)=\sup_{R_j\in\Upsilon}\frac{|\mu_\zeta|(R_j)}{\omega^\star_{-\alpha}(z_j)}.
    \end{equation*}
Since
    \begin{equation}\label{27}
    \frac{|\mu_\zeta|(R_j)}{\om^\star_{-\alpha}(z_j)}
    =\left(\frac{|\mu|(R_j)}{\om^\star_{-\alpha}(z_j)}\right)^{p\Re\zeta-1}
    \frac{|\mu|(R_j)}{\om^\star_{-\alpha}(z_j)}
    =\left(\frac{|\mu|(R_j)}{\om^\star_{-\alpha}(z_j)}\right)^{p\Re\zeta},
    \end{equation}
we obtain $|\mu_\zeta|(R_j)|\le\om^\star_{-\alpha}(z_j)$ if
$\Re\zeta=0$, and hence $C(\mu_\zeta)\lesssim1$. Therefore the
existence of $M_0>0$ with the desired properties follows by
\eqref{26}.

Finally, we will show that $S_\zeta\in\SSS_1$ if $\Re\zeta=1$.  In
this case the fact $\om_{-\alpha}^\star\in\R$ and \eqref{27} yield
    \begin{equation}\label{eq:b1}
    \int_\D\frac{d|\mu_\zeta|(z)}{\om_{-\alpha}^\star(z)}
    \asymp\sum_{R_j\in\Upsilon}\frac{|\mu_\zeta|(R_j)}{\om_{-\alpha}^\star(z_j)}
    =\sum_{R_j\in\Upsilon}\left(\frac{|\mu|(R_j)}{\om_{-\alpha}^\star(z_j)}\right)^p.
    \end{equation}
Let $\zeta\in\Lambda$ with $\Re\zeta=1$ be fixed, and choose
bounded invertible operators $A$ and $B$ on
$H_\alpha(\om^\star)$\index{$H_\alpha(\om^\star)$} so that
$V=Q_\zeta A$ and $W=Q_{\overline{\zeta}}B$ are unitary operators
from $H_{\alpha}(\om^\star)$ to $H_{\gamma}(\om^\star)$, where
$\gamma=\alpha+2\varepsilon(p-1)$. Let $T=B^\star S_\zeta A$. By
applying \eqref{28}, with $f$ and $g$ being replaced by $A(f)$ and
$B(g)$, respectively, we obtain
    \begin{equation}\label{29}\index{$H_\alpha(\om^\star)$}
    \begin{split}
    \langle T(f),g\rangle_{H_\a(\om^\star)}
    &=\langle (S_\zeta A)(f),B(g)\rangle_{H_\a(\om^\star)}\\
    &=\int_\D V(f)(z)\overline{W(g)(z)}(1-|z|)^{2\varepsilon(1-p\zeta)}\,d\mu_\zeta(z).
    \end{split}
    \end{equation}
Let now $\{f_n\}$ and $\{g_n\}$ be orthonormal sets on
$H_\alpha(\om^\star)$. Since $V$ and $W$ are unitary operators,
the sets $\{e_n=V(f_n)\}$ and $\{h_n=W(g_n)\}$ are orthonormal on
$H_{\gamma}(\om^\star)$. Since $-\infty<\gamma<1$, the inequality
\eqref{eqRK1} and Lemma~\ref{le:rker} yield
    $$
    \sum_{n}|e_n(z)|^2\le\|K^\gamma_z\|^2_{H_{\gamma}(\om^\star)}
    \asymp\frac{1}{\om^\star_{-\gamma}(z)},
    $$
and similarly for $\{h_n\}$. These inequalities together with
\eqref{29}, the Cauchy-Schwarz inequality and \eqref{eq:b1} give
    \begin{equation*}
    \begin{split}
   & \sum_n\left|\langle T(f_n), g_n \rangle_{H_\a(\om^\star)}\right|
   \\ &\le\sum_n\int_\D|e_n(z)||h_n(z)|\,(1-|z|)^{2\varepsilon(1-p)}\,d|\mu_\zeta|(z)\\
    &\le\int_\D\left(\sum_n|e_n(z)|^2\right)^{1/2}\left(\sum_n|h_n(z)|^2\right)^{1/2}
    (1-|z|)^{2\varepsilon(1-p)}\,d|\mu_\zeta|(z)\\
    &\lesssim\int_\D\frac{(1-|z|)^{2\varepsilon(1-p)}}{\om^\star_{-\gamma}(z)}\,d|\mu_\zeta|(z)=\int_\D \frac{d|\mu_\zeta|(z)}{\om^\star_{-\alpha}(z)}
    \asymp\sum_{R_j\in\Upsilon}\left(\frac{|\mu|(R_j)}{\om_{-\alpha}^\star(z_j)}\right)^p.
    \end{split}
    \end{equation*}
By \cite[Theorem~1.27]{Zhu} this implies
$T\in\SSS_1(H_\alpha(\om^\star))$ with
    $$\index{$H_\alpha(\om^\star)$}
    |T|_1\lesssim\sum_{R_j\in\Upsilon}
    \left(\frac{|\mu|(R_j)}{\om_{-\alpha}^\star(z_j)}\right)^p,
    $$
from which the existence of $M_1>0$ with the desired properties
follows by the inequality
    $$
    |S_\zeta|_1\le\|(B^\star)^{-1}\||T|_1\|A^{-1}\|,
    $$
see \cite[p.~27]{Zhu}. This finishes the proof of the first
assertion in Theorem~\ref{th:sufschpmayor2}.

\medskip
It remains to show that if $\mu$ is a positive Borel measure such
that
$T_\mu\in\SSS_p(H_\alpha(\om^\star))$,\index{$H_\alpha(\om^\star)$}
then \eqref{eq:sl} is satisfied. We only give an outline of the
proof. A similar reasoning as in the beginning of
Proposition~\ref{pr:necshI} together with \eqref{eq:st12} shows
that the assumption $T_\mu\in\SSS_p(H_\alpha(\om^\star))$ yields
    \begin{equation}\label{34}\index{$H_\alpha(\om^\star)$}
    \sum_n\left(\int_\D|v_n(z)|^2\,d\mu(z)\right)^p<\infty,
    \end{equation}
where $\{v_n\}$ are the images of orthonormal vectors $\{e_n\}$ in
$H_{\a}(\omega^\star)$ under any bounded linear operator
$\widetilde{J}:[\{e_n\}]\to H_\a(\omega^\star)$. By choosing
    \begin{equation}\label{35}\index{$H_\alpha(\om^\star)$}
    v_n(z)=\frac{\Phi_{a_n}^{\omega_{-\a}}(z)}{\|B_{a_n}^{\omega_{-\a}^\star}\|_{A^2_{\omega_{-\a}^\star}}},
    \end{equation}
where $\{a_n\}$ is uniformly discrete and $a_n\ne0$ for all $n$, a
similar reasoning as in the proof of Lemma~\ref{le:wo}, with
$\omega^\star_{-\a}$ in place of $\om^\star$, shows that
$\widetilde{J}:[\{e_n\}]\to H_\a(\omega^\star)$ is
bounded.\index{$H_\alpha(\om^\star)$}

Now, take $\delta=\delta(\alpha,\om)>0$ obtained in
Lemma~\ref{le:Phi} and consider the covering
$\Omega=\left\{D(a_n,\delta(1-|a_n|)\right\}_{n\ge 0}$ of $\D$,
induced by a uniformly discrete sequence $\{a_n\}$ with $a_n\ne0$,
such that
    \begin{enumerate}
    \item[\rm(i)] Each $R_j\in\Upsilon$ intersects at most
    $N=N(\delta)$ discs of
    $\Omega$.
    \item[(ii)] Each disc of $\Omega$
    intersects at most $M=M(\delta)$ squares of $\Upsilon$.
\end{enumerate}
Then, by combining \eqref{34} and \eqref{35}, and using
Lemmas~\ref{le:Phi} and \ref{kernels} as well as \eqref{22} and
\eqref{eq:r2} for $\omega_{-\a}^\star\in\R$ we obtain
    $$
    \sum_{R_j\in\Upsilon}\left(\frac{|\mu|(R_j)}{\om_{-\alpha}^\star(z_j)}\right)^p
    \lesssim C(N,M)\sum_n\left(\int_\D
    \frac{\left|\Phi_{a_n}^{\omega_{-\a}}(z)\right|^2}{\|B_{a_n}^{\omega_{-\a}^\star}\|^2_{A^2_{\omega_{-\a}^\star}}}\,d\mu(z)\right)^p<\infty,
    $$
which finishes the proof. \hfill$\Box$

\medskip

We finish the chapter by the following result concerning the case
$p=\infty$ in Theorem~\ref{th:sufschpmayor2}. The proof of this
proposition is straightforward and its principal ingredients can
be found in the proof of Theorem~\ref{th:sufschpmayor2}.

\begin{proposition}
Let $-\infty<\alpha<0$ and $\omega\in\I\cup\R$, and let $\mu$ be a
complex Borel measure on $\D$. If
    \begin{equation}\label{Eq:p=infinity}
    \sup_{R_j\in\Upsilon}\frac{|\mu|(R_j)}{\om_{-\alpha}^\star(z_j)}<\infty,
    \end{equation}
then $T_\mu$ is bounded on $H_\alpha(\om^\star)$, and there exists
a constant $C>0$ such that
    \begin{equation}\label{Eq:p=infinity2}\index{$H_\alpha(\om^\star)$}
    \|T_\mu\|\le C\sup_{R_j\in\Upsilon}\frac{|\mu|(R_j)}{\om_{-\alpha}^\star(z_j)}.
    \end{equation}
Conversely, if $\mu$ is a positive Borel measure on $\D$ and
$T_\mu$ is bounded on $H_\alpha(\om^\star)$, then
\eqref{Eq:p=infinity} is satisfied.
\end{proposition}\index{$H_\alpha(\om^\star)$}

\chapter{Applications to Differential
Equations}\label{difequ}\index{linear differential
equation}\index{$\BN_\om$}

In this chapter we study linear differential equations with
solutions in either the weighted Bergman space $A^p_\om$ or the
Bergman-Nevalinna class $\BN_\omega$. Our primary interest is to
relate the growth of coefficients to the growth and the zero
distribution of solutions. Apart from tools commonly used in the
theory of complex differential equations in the unit disc, this
chapter relies also strongly on results and techniques from
Chapters~\ref{S2}--\ref{Sec:SpaceCCpp}.

\section{Solutions in the weighted Bergman space $A^p_\om$}\label{Sec:SolutionsBergman}

The aim of this section is to show how the results and techniques
developed in the preceding chapters can be used to find a set of
sufficient conditions for the analytic coefficients $B_j$ of the
linear differential equation
    \begin{equation}\label{lde}
    f^{(k)}+B_{k-1}f^{(k-1)}+\cdots+B_1f'+B_0f=B_k
    \end{equation}
forcing all solutions to be in the weighted Bergman
space~$A^p_\omega$. In fact, the linearity of the equation is not
in essence, yet it immediately guarantees the analyticity of all
solutions if the coefficients are analytic. At the end of the
section we will shortly discuss certain non-linear equations, and
obtain results on their analytic solutions.

The approach taken here originates from Pommerenke's work~\cite{P}
in which he uses Parseval's identity and Carleson's theorem to
study the question of when all solutions of the equation
$f''+Bf=0$ belong to the Hardy space $H^2$. For the case of
\eqref{lde} with solutions in $H^p$, we refer to \cite{R2}. For
the theory of complex differential equations the reader is invited
to see~\cite{Laine,Laine2}.

The main result of this section is Theorem~\ref{Thm:DE-1}. Its
proof is based on Theorem~\ref{th:cm}, the equivalent norms given
in Theorem~\ref{ThmLittlewood-Paley}, the proof of
Theorem~\ref{Thm-integration-operator-1}, and the estimate on the
zero distribution of functions in $A^p_\om$ given in
Theorem~\ref{th:subsblaschkecond}.

\begin{theorem}\label{Thm:DE-1}
Let $0<p<\infty$ and  $\om\in\I\cup\R$. Then there exists a
constant $\a=\a(p,k,\om)>0$ such that if the coefficients
$B_j\in\H(\D)$ of \eqref{lde} satisfy
    \begin{equation}\label{CondForB0}
    \sup_{I\subset\T}\frac{\int_{S(I)}|B_0(z)|^2(1-|z|)^{2k-2}\omega^\star(z)\,dA(z)}{\omega(S(I))}\le\a,\quad
    \end{equation}
and
    \begin{equation}\label{CondForOthers}\index{linear differential equation}
    \sup_{z\in\D}|B_j(z)|(1-|z|^2)^{k-j}\le\a,\quad j=1,\ldots,k-1,
    \end{equation}
and a $k$:th primitive $B_k^{(-k)}$ of $B_k$ belongs to
$A^p_\omega$, then all solutions $f$ of \eqref{lde} belong to
$A^p_\omega$. Moreover, the ordered zero sequence $\{z_k\}$ of
each solution satisfies~\eqref{sblascke1}.
\end{theorem}

\begin{proof}
Let $f$ be a solution of \eqref{lde}, and let first $0<p<2$. We
may assume that $f$ is continuous up to the boundary; if this is
not the case we replace $f(z)$ by $f(rz)$ and let $r\to1^-$ at the
end of the proof. The equivalent norm \eqref{normacono}
in~$A^p_\omega$, the equation~\eqref{lde} and the assumption
\eqref{CondForOthers} yield
    \begin{equation}\label{30}\index{linear differential equation}
    \begin{split}
    \|f\|_{A^p_\omega}^p&\asymp\int_\D\,\left(\int_{\Gamma(u)}|f^{(k)}(z)|^2
    \left(1-\left|\frac{z}{u}\right|\right)^{2k-2}\,dA(z)\right)^{\frac{p}2}\omega(u)\,dA(u)\\
    &\quad+\sum_{j=0}^{k-1}|f^{(j)}(0)|^p\\
    &\lesssim\sum_{j=0}^{k-1}\int_\D\,\left(\int_{\Gamma(u)}|f^{(j)}(z)|^2|B_j(z)|^2
    \left(1-\left|\frac{z}{u}\right|\right)^{2k-2}\,dA(z)\right)^{\frac{p}2}\omega(u)\,dA(u)\\
    &\quad+\|B_k^{(-k)}\|_{A^p_\omega}^p+\sum_{j=0}^{k-1}|f^{(j)}(0)|^p\\
    &\lesssim\int_\D\,\left(\int_{\Gamma(u)}|f(z)|^2|B_0(z)|^2
    \left(1-\left|\frac{z}{u}\right|\right)^{2k-2}\,dA(z)\right)^{\frac{p}2}\omega(u)\,dA(u)\\
    &\quad+\a^p\|f\|_{A^p_\om}^p+\|B_k^{(-k)}\|_{A^p_\omega}^p+\sum_{j=0}^{k-1}|f^{(j)}(0)|^p.
    \end{split}
    \end{equation}
Arguing as in the last part of the proof of
Theorem~\ref{Thm-integration-operator-1}(ii), with $q=p<2$, we
obtain
    \begin{eqnarray*}
    &&\int_\D\,\left(\int_{\Gamma(u)}|f(z)|^2|B_0(z)|^2
    \left(1-\left|\frac{z}{u}\right|\right)^{2k-2}\,dA(z)\right)^{\frac{p}2}\omega(u)\,dA(u)\\
    &&\lesssim\|f\|_{A^p_\om}^\frac{p(2-p)}{2}\left(\int_\D|f(z)|^p|B_0(z)|^2\int_{T(z)}
    \omega(u)\left(1-\left|\frac{z}{u}\right|\right)^{2k-2}\,dA(u)\,dA(z)\right)^{\frac{p}2}\\
    &&=\|f\|_{A^p_\om}^\frac{p(2-p)}{2}\left(\int_\D|f(z)|^p|B_0(z)|^2\left(\int_{|z|}^1
    \omega(s)\left(1-\frac{|z|}{s}\right)^{2k-1}s\,ds\right)\,dA(z)\right)^{\frac{p}2}\\
    &&\lesssim\|f\|_{A^p_\om}^\frac{p(2-p)}{2}\left(\int_\D|f(z)|^p|B_0(z)|^2
   (1-|z|)^{2k-2} \omega^\star(z)\,dA(z)\right)^{\frac{p}2},
    \end{eqnarray*}
where the last step follows by Lemma~\ref{le:cuadrado-tienda}.
Putting everything together, and applying Theorem~\ref{th:cm} and
the assumption \eqref{CondForB0}, we get
    \begin{equation}\label{31}\index{linear differential equation}
    \|f\|_{A^p_\omega}^p\le C\left(\a^p\|f\|_{A^p_\om}^p+\|B_k^{(-k)}\|_{A^p_\omega}^p+\sum_{j=0}^{k-1}|f^{(j)}(0)|^p\right)
    \end{equation}
for some constant $C=C(p,k,\omega)>0$. By choosing $\a>0$
sufficiently small, and rearranging terms, we deduce $f\in
A^p_\omega$.

A proof of the case $p=2$ can be easily constructed by imitating
the reasoning above and taking into account Theorem~\ref{th:cm}
and the Littlewood-Paley formula \eqref{eq:LP2}.

Let now $2<p<\infty$. Since the reasoning in \eqref{30} remains
valid, we only need to deal with the last integral in there.
However, we have already considered a similar situation in the
proof of Theorem~\ref{Thm-integration-operator-1}(ii). A careful
inspection of that proof, with $q=p>2$ and
$|B_0(z)|^2\left(1-\left|z\right|\right)^{2k-2}$ in place of
$|g'(z)|^2$, shows that \eqref{31} remains valid for some
$C=C(p,k,\omega)>0$, provided $2<p<\infty$ and
$\omega\in\I\cup\R$. It follows that $f\in A^p_\omega$.

Finally, we complete the proof by observing that the assertion on
the zeros of solutions is an immediate consequence of
Theorem~\ref{th:subsblaschkecond}.
\end{proof}

Theorem~\ref{Thm:DE-1} is of very general nature because it is
valid for all weights~$\omega\in\I\cup\R$. If $\omega$ is one of
the exponential type radial weights $w_{\gamma,\alpha}(z)$,
defined in \eqref{Eq:ExponentialWeights}, or tends to zero even
faster as $|z|$ approaches $1$, then the maximum modulus of a
solution~$f$ might be of exponential growth, and hence a natural
way to measure the growth of $f$ is by means of the terminology of
Nevalinna theory. This situation will be discussed in
Section~\ref{Sec:SolutionsBergmanNevanlinna}.

If $\om$ is rapidly increasing, then \eqref{CondForB0} is more
restrictive than the other assumptions \eqref{CondForOthers} in
the sense that the supremum in \eqref{CondForOthers} with $j=0$ is
finite if \eqref{CondForB0} is satisfied, but the converse is not
true in general. Moreover, if $\om$ is regular, then
\eqref{CondForB0} reduces to a radial condition as the following
consequence of Theorem~\ref{Thm:DE-1} shows. These facts can be
deduced by carefully observing the proofs of Parts~(B)-(D) in
Proposition~\ref{pr:blochcpp}.

\begin{corollary}\label{Cor:DE}\index{linear differential equation}
Let $0<p<\infty$ and $\om\in\R$. Then there exists a constant
$\a=\a(p,k,\omega)>0$ such that if the coefficients $B_j\in\H(\D)$
of \eqref{lde} satisfy
    \begin{equation}\label{CondForOthers++}
    \sup_{z\in\D}|B_j(z)|(1-|z|^2)^{k-j}\le\a,\quad j=0,\ldots,k-1,
    \end{equation}
and a $k$:th primitive $B_k^{(-k)}$ of $B_k$ belongs to
$A^p_\omega$, then all solutions of \eqref{lde} belong to
$A^p_\omega$.
\end{corollary}

It is clear that the assertions in Theorem~\ref{Thm:DE-1}  and
Corollary~\ref{Cor:DE} remain valid if the assumptions are
satisfied only near the boundary. By this we mean that in
\eqref{CondForB0} it is enough to take the supremum over the
intervals whose length does not exceed a pregiven
$\delta\in(0,1)$, and that in \eqref{CondForOthers} we may replace
$\sup_{z\in\D}$ by $\sup_{|z|\ge1-\delta}$. In the proofs one just
needs to split the integral over $\D$ into two pieces,
$D(0,1-\delta)$ and the remainder, and proceed in a natural
manner. This observation yields the following result.

\begin{corollary}\label{Cor:DE-2}\index{linear differential equation}
Let $\om\in\I\cup\R$. If the coefficients $B_j\in\H(\D)$ of
\eqref{lde} satisfy
    \begin{equation}\label{CondForB0+n}
    \lim_{|I|\to0}\frac{\int_{S(I)}|B_0(z)|^2(1-|z|)^{2k-2}\omega^\star(z)\,dA(z)}{\omega(S(I))}=0,
    \end{equation}
and
    \begin{equation}\label{CondForOthers+n}
    \lim_{|z|\to1^-}|B_j(z)|(1-|z|^2)^{k-j}=0,\quad j=1,\ldots,k-1,
    \end{equation}
and a $k$:th primitive $B_k^{(-k)}$ of $B_k$ belongs to
$\cap_{0<p<\infty}A^p_\omega$, then all solutions of \eqref{lde}
belong to $\cap_{0<p<\infty}A^p_\omega$.
\end{corollary}

For simplicity we now settle to analyze the assertion in
Theorem~\ref{Thm:DE-1} in the case of the most simple, yet
generally explicitly unsolvable, homogeneous second order linear
differential equation
    \begin{equation}\label{Eq:Second-Order}
    f''+Bf=0,\quad B\in\H(\D).
    \end{equation}
In this case the condition \eqref{CondForB0} reduces to
    \begin{equation}\label{32}\index{linear differential equation}
    \sup_{I\subset\T}\frac{\int_{S(I)}|B(z)|^2(1-|z|)^{2}\omega^\star(z)\,dA(z)}{\omega(S(I))}\le\a.
    \end{equation}
We concentrate on the most interesting case $\omega\in\I$ so that
the assumption on the only coefficient $B$ does not reduce to a
radial condition as in Corollary~\ref{Cor:DE}. Since the weight
$v_\alpha$\index{$v_\a(r)$} are one of the typical members in
$\I$, we choose it as an example. To this end, let $0<p<\infty$
and $1<\b\le2$. Then the function
$f(z)=(1-z)^{-\frac1p}(\log\frac{e}{1-z})^{\frac{\b-2}{p}}$ is a
solution of \eqref{Eq:Second-Order}, where
    $$
    B(z)=-\frac{\frac1p(\frac1p+1)}{(1-z)^2}-\frac{(\b-2)(2+p)}{p^2(1-z)^2\log\frac{e}{1-z}}
    -\frac{(\b-2)(\b-2+p)}{p^2(1-z)^2\left(\log\frac{e}{1-z}\right)^2}
    $$
is analytic in $\D$. Now $f\not\in A^p_{v_\b}$, and therefore the
constant $\alpha=\alpha(p,\om)$ in \eqref{32} must satisfy
$\a(p,\om)\lesssim\frac1p(\frac1p+1)$ if $\omega=v_\b$ and
$1<\b\le2$. In particular, $\alpha(p,\om)\lesssim\frac1p$, as
$p\to\infty$.

We finish the section by shortly discussing certain non-linear
equations for which the used methods work. For simplicity, let us
consider the equation
    \begin{equation}\label{Eq:non-linear1}\index{linear differential equation}
    (f'')^2+Bf=0,\quad
    B\in\H(\D),
    \end{equation}
and the question of when its analytic solutions belong to the
Hilbert space~$A^2_\omega$. If~$f$ is such an analytic solution,
then the Littlewood-Paley formula \eqref{eq:LP2},
Lemma~\ref{le:sc1} and H\"older's inequality yield
    \begin{equation*}
    \begin{split}
    \|f\|_{A^2_\omega}^2&\asymp\int_\D|f''(z)|^2(1-|z|)^2\omega^\star(z)\,dA(z)+|f(0)|^2+|f'(0)|^2\\
    &=\int_\D|f(z)||B(z)|(1-|z|)^2\omega^\star(z)\,dA(z)+|f(0)|^2+|f'(0)|^2\\
    &\le\|f\|_{A^2_\omega}\left(\int_\D\frac{|B(z)|^2(1-|z|)^4(\omega^\star(z))^2}{\omega(z)}\,dA(z)\right)^\frac12
    +|f(0)|^2+|f'(0)|^2,
    \end{split}
    \end{equation*}
and we immediately deduce $f\in A^2_\omega$ if the last integral
is sufficiently small. This method works also for more general
equations. Namely, it yields a set of sufficient condition for the
coefficients of
    $$
    (f^{(k)})^{n_k}+B_{k-1}(f^{(k-1)})^{n_{k-1}}+\cdots+B_1(f')^{n_1}+B_0f=B_k,
    $$
where $n_j\ge1$ for all $j=1,\ldots,k$, forcing all analytic
solutions to be in $A^p_\omega$. Details are left to the reader.

\section{Solutions in the Bergman-Nevanlinna class $\BN_\om$}\label{Sec:SolutionsBergmanNevanlinna}\index{$\BN_\om$}

In this section we study the equation
    \begin{equation}\label{ldeHOMO}\index{linear differential equation}
    f^{(k)}+B_{k-1}f^{(k-1)}+\cdots+B_1f'+B_0f=0
    \end{equation}
with solutions in the Bergman-Nevalinna class $\BN_\omega$ induced
by a rapidly increasing or regular weight $\om$. We will see that
in this case it is natural to measure the growth of the
coefficients $B_j$ by the containment in the weighted Bergman
spaces depending on the index $j$ and $\om$. The Bergman-Nevalinna
class $\BN_\omega$ is much larger than the weighted Bergman space
$A^p_\om$, and therefore we will use methods that are different
from those applied in Section~\ref{Sec:SolutionsBergman}. The
tools employed here involve general growth estimates for
solutions~\cite{HKR:AASF}, a representation of coefficients in
terms of ratios of linearly independent solutions~\cite{Kim},
integrated generalized logarithmic derivative
estimates~\cite{CHR2009}, the order reduction
procedure~\cite{Gary2} as well as lemmas on weights proved in
Section~\ref{Sec:LemmasOnWeights}, and results on zero
distribution of functions in $\BN_\om$ from
Section~\ref{Sec:Zeros}. Also results from the Nevanlinna value
distribution theory are explicitly or implicitly present in many
instances.\index{$\BN_\om$}

The order reduction produces linear equations with meromorphic
coefficients and solutions. Therefore we say that a meromorphic
function $f$ in $\D$ belongs to the \emph{Bergman-Nevanlinna
class} $\BN_\omega$ if
    $$\index{$\BN_\om$}\index{Bergman-Nevanlinna class}
    \int_0^1T(r,f)\,\om(r)\,dr<\infty,
    $$
where
    $$
    T(r,f)=m(r,f)+N(r,1/f)=\frac1{2\pi}\int_0^{2\pi}\log^+|f(re^{i\t})|\,d\t+N(r,1/f)
    $$
is the \emph{Nevanlinna characteristic}\footnote{The integrated
counting function of poles is usually denoted by $N(r,f)$. Since
we have already conserved this symbol for zeros, we will use the
unorthodox notation $N(r,1/f)$ for poles.}.\index{Nevalinna
characteristic}\index{$T(r,f)$} If $f\in\H(\D)$, then the
integrated counting function
    $$
    N(r,1/f)=\int_0^r\frac{n(s,1/f)-n(0,1/f)}{s}\,ds+n(0,1/f)\log
    r
    $$
of poles vanishes, and thus this definition is in accordance with
that given in Section~\ref{SubSec:factorization} for analytic
functions. For the Nevanlinna value distribution theory the reader
is invited to see~\cite{Hayman,Laine}.

For each radial weight $\omega$, we write
$\overline{\om}(r)=\om(r)(1-r)$ for short, and recall that
    $\index{$\widehat{\om}(r)$}\index{$\overline{\om}(r)$}
    \widehat{\omega}(r)=\int_r^1\om(s)\,ds.
    $
If $\om$ is regular, then
$\widehat{\om}(r)\asymp\overline{\om}(r)$, and
$\widehat{\om}(r)/\overline{\om}(r)\to\infty$, as $r\to1^-$, for
each rapidly increasing weight $\om$.

We begin with applying the growth estimates for solutions to
obtain sufficient conditions for coefficients forcing all
solutions to belong to $\BN_\om$.

\begin{proposition}\label{Prop:DE-implication}\index{linear differential equation}
Let $\om$ be a radial weight. If $B_j\in
A^{1/(k-j)}_{\widehat{\om}}$ for all $j=0,\ldots,k-1$, then all
solutions of \eqref{ldeHOMO} belong to $\BN_\om$, and the zero
sequence $\{z_k\}$ of each non-trivial solution satisfies
    \begin{equation}\label{Eq:ZerosOfSolutionsInBN}
    \sum_k\om^\star(z_k)<\infty.
    \end{equation}
\end{proposition}

\begin{proof}
Let $f$ be a non-trivial solution of \eqref{ldeHOMO}. By
\cite[Corollary~5.3]{HKR:AASF}, there exists a constant $C>0$,
depending only on $k$ and initial values of $f$, such that
    \begin{equation*}
    m(r,f)\le C\left(\sum_{j=0}^{k-1}\int_0^r\int_0^{2\pi}
    |B_j(se^{i\t})|^\frac{1}{k-j}\,d\t ds+1\right).
    \end{equation*}
Therefore Fubini's theorem yields
    \begin{equation*}
    \|f\|_{\BN_\om}\lesssim\sum_{j=0}^{k-1}\|B_j\|_{A^{1/(k-j)}_{\widehat{\om}}}^{1/(k-j)}+\int_0^1\om(s)\,ds,
    \end{equation*}
and thus $f\in\BN_\om$. By
Proposition~\ref{Prop:bergman-nevanlinna} this in turn implies
that the zero sequence $\{z_k\}$ of $f$ satisfies
\eqref{Eq:ZerosOfSolutionsInBN}.
\end{proof}

We next study the growth of coefficients when all solutions belong
to $\BN_\om$ and will find out a clear difference, that is not
explicitly present in Proposition~\ref{Prop:DE-implication},
between regular and rapidly increasing weights. This difference
will be underscored in detail at the end of the section by means
of examples. At this point it is also worth mentioning that
Theorem~\ref{Thm:De-converse-growth} will be needed when the
growth of coefficients is studied under the hypothesis
$\sum_k\om^\star(z_k)<\infty$ on all zero sequences $\{z_k\}$ of
solutions.

\begin{theorem}\label{Thm:De-converse-growth}\index{linear differential equation}
Let $f_{1},\ldots,f_{k}$ be linearly independent meromorphic
solutions of \eqref{ldeHOMO}, where $B_{0},\ldots,B_{k-1}$ are
meromorphic in $\D$, such that $f_j\in\BN_\om$ for all
$j=1,\ldots,k-1$.
\begin{itemize}
\item[\rm(i)] If $\om\in\R$, then
$\|B_j\|_{L^\frac1{k-j}_{\widehat{\om}}}<\infty$ for all
$j=0,\ldots,k-1$. \item[\rm(ii)] If $\om\in\widetilde{\I}$ such
that
    $$
    \int_0^1\log\frac{1}{1-r}\,\om(r)\,dr<\infty,
    $$
then $\|B_j\|_{L^\frac1{k-j}_{\overline{\om}}}<\infty$ for all
$j=0,\ldots,k-1$.
\end{itemize}
\end{theorem}

By combining Proposition~\ref{Prop:DE-implication} and
Theorem~\ref{Thm:De-converse-growth} we obtain the following
immediate consequence.

\begin{corollary}\index{linear differential equation}
Let $\om\in\R$ and $B_j\in\H(\D)$ for all $j=0,\ldots,k-1$. Then
all solutions of \eqref{ldeHOMO} belong to $\BN_\om$ if and only
if $B_j\in A^{1/(k-j)}_{\widehat{\om}}$ for all $j=0,\ldots,k-1$.
\end{corollary}

Theorem~\ref{Thm:De-converse-growth} is proved by applying the
standard order reduction procedure. To do so, we will need the
following three lemmas of which the first one is contained in
\cite[Lemma~6.4]{G-S-W}.\index{order reduction procedure}

\begin{letterlemma}\label{orderred-lemma}\index{linear differential equation}
Let $f_{0,1},\ldots,f_{0,m}$ be $m\ge2$ linearly independent
meromorphic solutions of
\begin{equation}\label{94}
f^{(k)}+B_{0,k-1}f^{(k-1)}+\cdots +B_{0,1}f'+B_{0,0}f=0,\quad
k\geq m,
\end{equation}
where $B_{0,0},\ldots,B_{0,k-1}$ are meromorphic in $\D$. For
$1\le q\leq m-1$, set
    \begin{equation}\label{ratkaisut}
    f_{q,j}=\left(\frac{f_{q-1,j+1}}{f_{q-1,1}}\right)',\quad
    j=1,\ldots,m-q.
    \end{equation}
Then $f_{q,1},\ldots,f_{q,m-q}$ are linearly independent
meromorphic solutions of
    \begin{equation}\label{kertalukuapudotettu}
    f^{(k-q)}+B_{q,k-q-1}f^{(k-q-1)}+\cdots
    +B_{q,1}f'+B_{q,0}f=0,
    \end{equation}
where
    \begin{equation}\label{pudotetutkertoimet}
    B_{q,j}=\sum_{l=j+1}^{k-q+1}\binom{l}{j+1}B_{q-1,l}
    \frac{f_{q-1,1}^{(l-j-1)}}{f_{q-1,1}}
    \end{equation}
for $j=0,\ldots,k-1-q$. Here $B_{j,k-j}\equiv 1$ for all
$j=0,\ldots,q$.
\end{letterlemma}

The coefficients $B_{q,j}$ of \eqref{kertalukuapudotettu} contain
terms of the type $g^{(j)}/g$, where $g$ is a meromorphic function
given by \eqref{ratkaisut}. These quotients are known as
\emph{generalized logarithmic derivatives}\index{generalized
logarithmic derivative} and they will be estimated by using the
following result.\index{logarithmic derivative estimates}

\begin{lemma}\label{Lem:log-deriv-estimate}\index{linear differential equation}
Let $g$ be a meromorphic function in $\BN_\om$, and let $k>j\ge0$
be integers.
\begin{itemize}
\item[\rm(i)] If $\om\in\R$, then
    $$
    \left\|\frac{g^{(k)}}{g^{(j)}}\right\|_{L^{\frac{1}{k-j}}_{\widehat{\om}}}^{\frac{1}{k-j}}
    =\int_\D\left|\frac{g^{(k)}(z)}{g^{(j)}(z)}\right|^{\frac{1}{k-j}}\widehat{\om}(z)\,dA(z)<\infty.
    $$
\item[\rm(ii)] If $\om\in\widetilde{\I}$ such that
    $$
    \int_0^1\log\frac1{1-r}\,\om(r)\,dr<\infty,
    $$
then
    $$
    \left\|\frac{g^{(k)}}{g^{(j)}}\right\|_{L^{\frac{1}{k-j}}_{\overline{\om}}}^{\frac{1}{k-j}}
    =\int_\D\left|\frac{g^{(k)}(z)}{g^{(j)}(z)}\right|^{\frac{1}{k-j}}\overline{\om}(z)\,dA(z)<\infty.
    $$
\end{itemize}
\end{lemma}

\begin{proof}
(i) Let $r_n=1-2^{-n}$ for $n\in\N\cup\{0\}$, and denote
$A(r,\rho)=\{z:r<|z|<\rho\}$. Under the assumptions there exists
$d\in(0,1)$ such that
    \begin{equation}\label{Eq:logartihmicDerivativeEstimate}
    \begin{split}
    \int_{A(r_n,r_{n+1})}\left|\frac{g^{(k)}(z)}{g^{(j)}(z)}\right|^{\frac{1}{k-j}}\,dA(z)
    &\lesssim
    T\left(s(r_{n+3}),g\right)+\log\frac{e}{1-r_{n+3}},
    \end{split}
    \end{equation}
where $s(r)=1-d(1-r)$. See \cite[Theorem~5]{CHR2009} and
\cite[p.~7--8]{HR2011} for a proof. Therefore
    \begin{equation}
    \begin{split}
    \left\|\frac{g^{(k)}}{g^{(j)}}\right\|_{L^{\frac{1}{k-j}}_{\widehat{\om}}}^{\frac{1}{k-j}}
    &\le\sum_{n=0}^\infty\widehat{\om}(r_n)\int_{A(r_n,r_{n+1})}\left|\frac{g^{(k)}(z)}{g^{(j)}(z)}\right|^{\frac{1}{k-j}}\,dA(z)\\
    &\lesssim\sum_{n=0}^\infty\widehat{\om}(r_n)\left(T\left(s(r_{n+3}),g\right)+\log\frac{e}{1-r_{n+3}}\right)\\
    &\lesssim\sum_{n=0}^\infty\widehat{\om}(r_n)\left(T\left(s(r_{n}),g\right)+\log\frac{e}{1-r_{n}}\right).
    \end{split}
    \end{equation}
Now $\om\in\R$ by the assumption, and hence the observation (i) to
Lemma~\ref{le:condinte}(i) gives
    $$
    \widehat{\om}(s(r_n))\asymp\om(s(r_n))(1-s(r_n))=2\om(s(r_n))\int_{s(r_n)}^{s(r_{n+1})}dr
    \asymp\int_{s(r_n)}^{s(r_{n+1})}\om(r)dr
    $$
for all $n\in\N\cup\{0\}$. Moreover, $T(r,g)$ is an increasing
function by \cite[p.~8]{Hayman},\index{linear differential
equation} and hence
    \begin{equation}
    \begin{split}
    \left\|\frac{g^{(k)}}{g^{(j)}}\right\|_{L^{\frac{1}{k-j}}_{\widehat{\om}}}^{\frac{1}{k-j}}
    \lesssim\int_0^1T(r,g)\,\om(r)\,dr+\int_0^1\log\frac{e}{1-r}\om(r)\,dr.
    \end{split}
    \end{equation}
The first integral in the right hand side is finite by the
assumption $g\in\BN_\om$ and the second one converges by the
observation (ii) to Lemma~\ref{le:condinte}.

(ii) The assertion follows by minor modifications in the proof of
(i).
\end{proof}

We will also need the following standard result.

\begin{lemma}\label{Lem:differentialfield}\index{linear differential equation}
Let $\om\in\I\cup\R$ such that
    \begin{equation}\label{Extra}
    \int_0^1\log\frac{1}{1-r}\,\om(r)\,dr<\infty.
    \end{equation}
If $f$ and $g$ are meromorphic functions in $\BN_\om$, then $fg$,
$f/g$ and $f'$ belong to~$\BN_\om$.
\end{lemma}

\begin{proof}
Let $f,g\in\BN_\om$, where $\om\in\I\cup\R$ satisfies
\eqref{Extra}. Since $\log^+(xy)\le\log^+ x+\log^+y+\log2$ for all
$x,y>0$, we obtain
    \begin{equation}\label{log+}
    T(r,fg)\le T(r,f)+T(r,g)+\log2,
    \end{equation}
and so $fg\in\BN_\om$. The inequality \eqref{log+} and the first
fundamental theorem of Nevanlinna~\cite[Theorem~2.1.10]{Laine}
give
    $$
    T(r,f/g)\lesssim T(r,f)+T(r,g)+1,
    $$
and hence $f/g\in\BN_\om$. Moreover, a standard lemma of the
logarithmic derivative~\cite[Theorem~2.3.1]{Laine} together
with~\eqref{log+} yields
    $$
    T(r,f')\lesssim
    T(r,f)+\log^+T(\rho,f)+\log\frac{1}{\rho-r}+1,\quad
    \rho=\frac{1+r}{2}.
    $$
Since $\om\in\I\cup\R$, we obtain $\log^+
T(r,f)\lesssim\log\frac1{1-r}$, and hence
    $$
    T(r,f')\lesssim T(r,f)+\log\frac{1}{1-r}+1.
    $$
By using the assumption \eqref{Extra} we deduce $f'\in\BN_\om$.
\end{proof}

By Lemma~\ref{orderred-lemma} and
Lemma~\ref{Lem:differentialfield} we deduce that if
$f_{0,1},\ldots,f_{0,m}$, $m\ge2$, belong to $\BN_\om$ for all
$j=1,\ldots,m$, then $f_{q,j}\in\BN_\om$ for all $j=1,\ldots,m-q$
and $q=1,\ldots,m-1$.

\subsection*{Proof of Theorem~\ref{Thm:De-converse-growth}} (i) Assume
that at least one of the coefficients does not satisfy the
assertion and denote
$q=\max\{j:\|B_j\|_{L^\frac1{k-j}_{\widehat{\om}}}=\infty\}$.

If $q=0$, then \eqref{ldeHOMO}, standard estimates and H\"older's
inequality yield
    $$\index{linear differential equation}
    \|B_{0}\|_{L^\frac1{k}_{\widehat{\om}}}^\frac1{k}
    \lesssim\left\|\frac{f^{(k)}}{f}\right\|_{L^\frac1{k}_{\widehat{\om}}}^\frac1{k}
    +\sum_{j=1}^{k-1}\|B_{j}\|_{L^\frac1{k-j}_{\widehat{\om}}}^\frac{1}{k}
    \left\|\frac{f^{(j)}}{f}\right\|_{L^\frac1{j}_{\widehat{\om}}}^\frac
    1{k}.
    $$
Since $q=0$, we have
$\|B_{j}\|_{L^\frac1{k-j}_{\widehat{\om}}}<\infty$ for all
$j=1,\ldots,k-1$. Moreover, the norms of the generalized
logarithmic derivatives are finite by
Lemma~\ref{Lem:log-deriv-estimate}(i). Therefore the right hand
side is finite and we have a contradiction.

Assume now that $q=1$. Denote $B_j=B_{0,j}$ for $j=0,\ldots,k-1$
so that the notation for coefficients in \eqref{94} coincide with
those in \eqref{ldeHOMO}. The order of \eqref{94} is now reduced
down once to equation \eqref{kertalukuapudotettu}, where the
coefficients $B_{1,0},\ldots,B_{1,k-2}$ are given by
\eqref{pudotetutkertoimet}. H\"older's inequality and
\eqref{pudotetutkertoimet} now yield
    $$
    \|B_{0,1}\|_{L^\frac1{k-1}_{\widehat{\om}}}^\frac1{k-1}
    \lesssim\|B_{1,0}\|_{L^\frac1{k-1}_{\widehat{\om}}}^\frac1{k-1}
    +\sum_{j=2}^{k-1}\|B_{0,j}\|_{L^\frac1{k-j}_{\widehat{\om}}}^\frac{1}{k-1}
    \left\|\frac{f_1^{(j-1)}}{f_1}\right\|_{L^\frac1{j-1}_{\widehat{\om}}}^\frac
    1{k-1}
    +\left\|\frac{f_1^{(k-1)}}{f_1}\right\|_{L^\frac1{k-1}_{\widehat{\om}}}^\frac1{k-1}.
    $$
All the norms of the generalized logarithmic derivatives are
finite by Lemma~\ref{Lem:log-deriv-estimate}(i), and
$\|B_{0,j}\|_{L^\frac1{k-j}_{\widehat{\om}}}<\infty$ for all
$j=2,\ldots,k-1$ since $q=1$. Moreover,
\eqref{kertalukuapudotettu} gives
    $$
    \|B_{1,0}\|_{L^\frac1{k-1}_{\widehat{\om}}}^\frac1{k-1}
    \lesssim\left\|\frac{f_{1,1}^{(k-1)}}{f_{1,1}}\right\|_{L^\frac1{k-1}_{\widehat{\om}}}^\frac1{k-1}
    +\sum_{j=1}^{k-2}\|B_{1,j}\|_{L^\frac1{k-1-j}_{\widehat{\om}}}^\frac{1}{k-1}
    \left\|\frac{f_{1,1}^{(j)}}{f_{1,1}}\right\|_{L^\frac1{j}_{\widehat{\om}}}^\frac
    1{k-1},
    $$
where again all the norms of the generalized logarithmic
derivatives are finite by Lemma~\ref{Lem:log-deriv-estimate}(i)
and Lemma~\ref{Lem:differentialfield}. Applying once again
\eqref{pudotetutkertoimet}, we obtain
    \begin{equation}\index{linear differential equation}
    \begin{split}
    \|B_{1,j}\|_{L^\frac1{k-1-j}_{\widehat{\om}}}^\frac{1}{k-1-1}
    &\lesssim\|B_{0,j+1}\|_{L^\frac1{k-1-j}_{\widehat{\om}}}^\frac1{k-1-j}
    +\sum_{l=j+2}^{k-1}\|B_{0,l}\|_{L^\frac1{k-l}_{\widehat{\om}}}^\frac{1}{k-1-j}
    \left\|\frac{f_{1}^{(l-j-1)}}{f_{1}}\right\|_{L^\frac1{l-j-1}_{\widehat{\om}}}^\frac
    1{k-1-j}\\
    &\quad+\left\|\frac{f_{1}^{(k-j-1)}}{f_{1}}\right\|_{L^\frac1{k-j-1}_{\widehat{\om}}}^\frac1{k-1-j}
    \end{split}
    \end{equation}
for all $j=1,\ldots,k-2$. As earlier, we see that this upper bound
is finite. Putting everything together we deduce
$\|B_{0,1}\|_{L^\frac1{k-1}_{\widehat{\om}}}<\infty$ which is in
contradiction with the hypothesis $q=1$.

If $q>1$ we must do $q$ order reduction steps and proceed as in
the case $q=1$ until the desired contradiction is reached. We omit
the details.

(ii) Proof is essentially identical to that of (i) and is
therefore omitted. \hfill$\Box$

\medskip

To study the converse implication with regards to oscillation of
solutions in Proposition~\ref{Prop:DE-implication}, we first
observe that a standard substitution shows that the coefficient
$B_{k-1}$ can be ignored when studying the zero distribution of
solutions of \eqref{ldeHOMO}. Namely, if we replace $f$ by $u$ in
\eqref{ldeHOMO} and let $\Phi$ be a primitive function
of~$B_{k-1}$, then the substitution $f=ue^{-\frac1k \Phi}$, which
does not change zeros, shows that $f$ satisfies
    \begin{equation}\label{lde2}\index{linear differential equation}
    f^{(k)}+B^\star_{k-2}f^{(k-2)}+\cdots +B^\star_1f'+B^\star_0f=0,\quad k\geq 2,
    \end{equation}
with coefficients $B^\star_0,\ldots,B^\star_{k-2}\in\H(\D)$.
Therefore, from now on we concentrate on \eqref{lde2}.

The principal findings of this section are gathered in the
following result.

\begin{theorem}\label{Thm:DE-Main}
Let $\om\in\I\cup\R$ and
$B^\star_0,\ldots,B^\star_{k-2}\in\H(\D)$, and consider the
assertions:
    \begin{enumerate}
    \item[(1)] $B_j^\star\in
A^{1/(k-j)}_{\widehat{\om}}$ for all $j=0,\ldots,k-2$; \item[(2)]
All solutions of \eqref{lde2} belong to $\BN_\om$; \item[(3)] The
zero sequence $\{z_k\}$ of each non-trivial solution of
\eqref{lde2} satisfies \eqref{Eq:ZerosOfSolutionsInBN}; \item[(4)]
$B_j^\star\in A^{1/(k-j)}_{\overline{\om}}$ for all
$j=0,\ldots,k-2$.
    \end{enumerate}
With this notation the following assertions hold:
    \begin{itemize}
    \item[\rm(i)]
    If $\om\in\R$, then
    {\rm(1)}$\Leftrightarrow${\rm(2)}$\Leftrightarrow${\rm(3)}$\Leftrightarrow${\rm(4)}.
    \item[\rm(ii)] If $\om\in\I$, then {\rm(1)}$\Rightarrow${\rm(2)}$\Rightarrow${\rm(3)}.
    Moreover, if $\om\in\widetilde\I$ satisfies
        $$
        \int_0^1\log\frac1{1-r}\,\om(r)<\infty,
        $$
then {\rm(3)}$\Rightarrow${\rm(4)}.
    \end{itemize}
\end{theorem}

\begin{proof}
(i) By Proposition~\ref{Prop:DE-implication},
Proposition~\ref{Prop:bergman-nevanlinna} and the relation
$\widehat{\om}(r)\asymp\overline{\om}(r)$ for $\om\in\R$, it
suffices to show that (3) implies (1). Let $f_1,\ldots,f_k$ be
linearly independent solutions of \eqref{lde2}, where
$B^\star_0,\ldots,B^\star_{k-2}\in\H(\D)$, and denote
    \begin{equation}\label{ratios-plane}
    y_j=\frac{f_j}{f_k},\quad j=1,\ldots, k-1.
    \end{equation}
The second main theorem of Nevanlinna~\cite[p.~84]{Nevanlinna}
yields
    $$
    T(r,y_j)\le N(r,y_j)+N(r,1/y_j)+N(r,y_j-1)+S(r,y_j),
    $$
where
    $$
    S(r,y_j)\lesssim\log^+T(r,y_j)+\log\frac{1}{1-r},
    $$
outside of a possible exceptional set $F\subset[0,1)$ satisfying
$\int_F\frac{dr}{1-r}<\infty$. By \cite[Lemma~5]{Gary2}, there
exists a constant $d\in(0,1)$ such that
    \begin{equation}\label{Eq:NevanlinnaII}\index{linear differential equation}
    \begin{split}
    T(r,y_j)&\lesssim
    N(s(r),y_j)+N(s(r),1/y_j)+N(s(r),y_j-1)\\
    &\quad+\log^+T(s(r),y_j))+\log\frac{1}{1-r}+\log\frac1d,\quad
    r\in[0,1),
    \end{split}
    \end{equation}
where $s(r)=1-d(1-r)$. Now $\om^\star\in\R$ by Lemma~\ref{le:sc1},
and hence the observation~(ii) to Lemma~\ref{le:condinte} shows
that there exists $\a>1$ such that $(1-r)^\a\lesssim\om^\star(r)$.
Therefore the assumption (3) implies that the zero sequence
$\{z_k\}$ of each solution $f$ of \eqref{lde2} satisfies
    $$
    \sum_k(1-|z_k|)^\a\lesssim\sum_k\om^\star(z_k)<\infty.
    $$
Therefore each solution $f$ satisfies $T(r,f)\lesssim
(1-r)^{-\a-\e}$ for all $\e>0$ by \cite[Corollary~1.4]{HR2011},
and thus $\log^+T(s(r),y_j))\lesssim\log\frac1{1-r}-\log d$.
Therefore the term $\log^+T(s(r),y_j))$ can be erased from
\eqref{Eq:NevanlinnaII}.

The zeros and poles of $y_j=\frac{f_j}{f_k}$ are the zeros of
$f_j$ and $f_k$, respectively, and the 1-points of $y_j$ are
precisely the zeros of $f_j-f_k$, which is also a solution of
\eqref{lde2}. Therefore the assumption~(3),
\eqref{Eq:NevanlinnaII} and Lemma~\ref{Lemma:NBzeros} yield
    \begin{equation}\label{92}\index{linear differential equation}
    y_j\in\BN_\om,\quad j=1,\ldots,k-1.
    \end{equation}

By \cite[Theorem~2.1]{Kim} each coefficient of \eqref{lde2} can be
represented in the form
    \begin{equation}\label{Eq:RepresentationOfCoefficients}
    B^\star_j=\sum_{i=0}^{k-j} (-1)^{2k-i}\delta_{ki}\left(\begin{array}{c} k-i\\ k-i-j\end{array}\right)
    \frac{W_{k-i}}{W_k} \frac{\left(\sqrt[k]{W_k}\right)^{(k-i-j)}}{\sqrt[k]{W_k}},
    \end{equation}
where $\delta_{kk}=0$ and $\delta_{ki}=1$ otherwise, and
    \begin{equation}\label{determinant-plane}
    W_j=\left|
    \begin{array}{cccc}
    y_1'\ & y_2'\ & \cdots \ & y_{k-1}'\\
    \vdots &&&\\
    y_1^{(j-1)}\ & y_2^{(j-1)}\ & \cdots & y_{k-1}^{(j-1)}\\
    y_1^{(j+1)}\ & y_2^{(j+1)}\ & \cdots & y_{k-1}^{(j+1)}\\
    \vdots &&&\\
    y_1^{(k)}\ & y_2^{(k)}\ & \cdots & y_{k-1}^{(k)}
    \end{array}
    \right|,\quad j=1,\ldots,k.
    \end{equation}
By \eqref{Eq:RepresentationOfCoefficients} and H\"older's
inequality,
    \begin{equation}\label{933}
    \begin{split}
    \|B_j^\star\|_{A^\frac{1}{k-j}_{\widehat\om}}^\frac{1}{k-j}
    &\lesssim\left\|\frac{\left(\sqrt[k]{W_k}\right)^{(k-j)}}{\sqrt[k]{W_k}}\right\|_{L^\frac{1}{k-j}_{\widehat\om}}^\frac{1}{k-j}
    +\left\|\frac{W_{j}}{W_k}\right\|_{L^\frac{1}{k-j}_{\widehat\om}}^\frac{1}{k-j}\\
    &\quad+\sum_{i=1}^{k-j-1}\left\|\frac{W_{k-i}}{W_k}\right\|_{L^\frac{1}{i}_{\widehat\om}}^\frac{1}{k-j}
    \left\|\frac{\left(\sqrt[k]{W_k}\right)^{(k-i-j)}}{\sqrt[k]{W_k}}\right\|_{L^\frac{1}{k-i-j}_{\widehat\om}}^\frac{1}{k-j}.
    \end{split}
    \end{equation}
The function $\sqrt[k]{W_k}$ is a well defined meromorphic
function in $\D$~\cite{HR2011,Kim}. Therefore
$\left\|\frac{\left(\sqrt[k]{W_k}\right)^{(k-i-j)}}{\sqrt[k]{W_k}}\right\|_{L^\frac{1}{k-i-j}_{\widehat\om}}$
is finite for each $i=1,\ldots,k-j$ by
Lemma~\ref{Lem:log-deriv-estimate}, \eqref{determinant-plane} and
Lemma~\ref{Lem:differentialfield}.

Finally, we note that the norms involving $W_{k-i}$ in \eqref{933}
are finite by Theorem~\ref{Thm:De-converse-growth}(i), because the
functions $1,y_1,\ldots,y_{k-1}$ satisfy \eqref{92} and are
linearly independent meromorphic solutions of
    \begin{equation}\label{meromorphic-equation-plane}\index{linear differential equation}
    y^{(k)}-\frac{W_{k-1}(z)}{W_k(z)}y^{(k-1)}+\cdots +
    (-1)^{k-1}\frac{W_1(z)}{W_k(z)}y'=0,
    \end{equation}
see \cite{HR2011,Kim}.

(ii) The proof is similar and is therefore omitted.
\end{proof}

A linear differential equation is called
\emph{Horowitz-oscillatory},\index{Horowitz-oscillatory} if
$\sum(1-|z_k|)^2<\infty$ for the zero sequence $\{z_k\}$ of each
non-trivial solution~\cite[Chapter~13]{H2011}.
Theorem~\ref{Thm:DE-Main} with $\om\equiv1$ shows that
\eqref{lde2} is Horowitz-oscillatory if and only if~$B^\star_j$
belongs to the classical weighted Bergman space $A^\frac1{k-j}_1$
for all $j=0,\ldots,k-2$.

The existing literature contains several concrete examples
illustrating the correspondence between the growth of
coefficients, the growth of solutions and the zero distribution of
solutions given in Theorem~\ref{Thm:DE-Main}(i). See, for example,
\cite{CHR2009} and the references therein. We next provide two
examples regarding to Theorem~\ref{Thm:DE-Main}(ii) which does not
give such a satisfactory correspondence in the case
$\om\in\widetilde{\I}$. To this end, consider first the locally
univalent analytic function
    $$
    p(z)=\frac{\log\frac1{1-z}}{1-z},\quad z\in\D.
    $$
The functions
    $$
    f_1(z)=\frac{1}{\sqrt{p'(z)}}\sin p(z),\quad f_2(z)=\frac{1}{\sqrt{p'(z)}}\cos p(z)
    $$
are linearly independent solutions of
    $$\index{linear differential equation}
    f''+B(z)f=0,\quad B(z)=(p'(z))^2+\frac12S_p(z),
    $$
where $S_p=\frac{p'''}{p'}-\frac32\left(\frac{p''}{p'}\right)^2$
is the \emph{Schwarzian derivative}\index{Schwarzian derivative}
of $p$. Therefore
    \begin{equation*}
    \begin{split}
    B(z)&=\frac{\left(\log\frac1{1-z}+1\right)^2}{(1-z)^4}
    -\frac12\frac{1}{\left(1+\log\frac1{1-z}\right)(1-z)^2}
    -\frac34\frac{1}{\left(1+\log\frac1{1-z}\right)^2(1-z)^2}
    \end{split}
    \end{equation*}
satisfies
    $$
    M_\frac12^\frac12(r,B)\asymp\frac{\log\frac1{1-r}}{1-r},\quad
    r\to1^-,
    $$
and hence $B\in A^\frac12_{\widehat{v}_\a}$ if and only if $\a>3$,
but $B\in A^\frac12_{\overline{v}_\a}$ if and only if $\a>2$. If
$\{z_k\}$ is the zero sequence of a solution $f$, then
Theorem~\ref{Thm:DE-Main}(ii) yields
    $$\index{$v_\a(r)$}
    \sum_kv_\a^\star(z_k)\asymp\sum_k\frac{1-|z_k|}{\left(\log\frac{2}{1-|z_k|}\right)^{\a-1}}<\infty
    $$
for all $\a>3$. The solution $f=C_1f_1+C_2f_2$, where
$|C_1|+|C_2|\ne0$, vanishes at the points $\{z_k\}$ that satisfy
    \begin{equation}\label{91}
    \frac{\log\frac1{1-z_k}}{1-z_k}=k\pi+\frac{1}2\arg\left(\frac{C_1-C_2i}{C_1+C_2i}\right),\quad k\in\mathbb{Z}.
    \end{equation}
Clearly the only possible accumulation point of zeros is $1$. If
$|1-z_k|<\delta$, then $w_k=\frac1{1-z_k}$ belongs to the right
half plane and $|w_k|>1/\delta$. Therefore the equation \eqref{91}
does not have solutions when $k$ is a large negative integer.
Moreover, \eqref{91} yields
    \begin{equation}\label{91j}\index{linear differential equation}
    \lim_{k\to\infty}|w_k|=+\infty
    \end{equation}
and
    $$
    |\arg w_k |=|\arg(\log w_k)|=\left|\arctan\left(\frac{\arg
    w_k}{\log|w_k|}\right)\right|\lesssim\left|\frac{\arg
    w_k}{\log|w_k|} \right|
    $$
for large $k$. So, if $\arg w_k\neq 0$, then
$|\log|w_k||\lesssim1$, which together with \eqref{91j} gives
$\arg(1-z_k)=-\arg w_k=0$, that is, $z_k\in (0,1)$ for $k>0$ large
enough. It follows that
    $$\index{$v_\a(r)$}
    \sum_kv_\a^\star(z_k)\asymp\sum_k\frac{1-|z_k|}{\left(\log\frac{2}{1-|z_k|}\right)^{\a-1}}\asymp\sum_{k\ge2}\frac{1}{k(\log
    k)^{\a-2}}<\infty
    $$
if and only if $\a>3$. We deduce that in this case the
implications (1)$\Rightarrow$(2)$\Rightarrow$(3) in
Theorem~\ref{Thm:DE-Main}(ii) are sharp, but (3)$\Rightarrow$(4)
is not. In view of Proposition~\ref{Prop:bergman-nevanlinna}, this
makes us questionize the sharpness of
Theorem~\ref{Thm:De-converse-growth}(ii). The answer, however, is
affirmative. Namely, the functions
    $$
    f_1(z)=\exp\left( \frac{\log \frac1{1-z}}{1-z} \right),\quad
    f_2(z)=\exp\left(-\frac{\log \frac1{1-z}}{1-z} \right)
    $$
are linearly independent solutions of
    \begin{equation}\label{2nd}\index{linear differential equation}
    f''+B_1(z)f'+B_0(z)f=0,
    \end{equation}
where
    \begin{equation*}
    \begin{split}
    B_1(z)=-\frac{2\log\frac{1}{1-z}+3}{\left(\log\frac1{1-z}+1\right)(1-z)}
    \end{split}
    \end{equation*}
and
    \begin{equation*}
    \begin{split}
    B_0(z)=-\frac{\left(\log\frac{1}{1-z}+1\right)^2}{(1-z)^4}
    \end{split}
    \end{equation*}
are analytic in $\D$. The coefficients satisfy both $B_1\in
A^1_{\overline{v}_\a}$ and $B_0\in A^\frac12_{\overline{v}_\a}$ if
and only if $\a>2$. Moreover, it was shown
in~\cite[Example~3]{CHR2009} that
    $$
    T(r,f_1)\asymp\log \frac 1{1-r}, \quad r\to1^-,
    $$
and since $f_2=1/f_1$, all solutions belong to $BN_{v_\a}$ for all
$\a>2$ by the first main theorem of Nevanlinna. This shows that
the statement in Theorem~\ref{Thm:De-converse-growth}(ii) is sharp
in this case, but the first implication in
Proposition~\ref{Prop:DE-implication} is not. Note also that if
$\{z_k\}$ is the zero sequence of a solution $f$, then
Proposition~\ref{Prop:bergman-nevanlinna} yields
    $$\index{$v_\a(r)$}
    \sum_kv_\a^\star(z_k)\asymp\sum_k\frac{1-|z_k|}{\left(\log\frac{2}{1-|z_k|}\right)^{\a-1}}<\infty
    $$
for all $\a>2$.

There is a clear and significant difference between the geometric
zero distribution of solutions in these two examples. Namely, in
the latter one the solution $f=C_1f_1+C_2f_2$, where $C_1\ne0\ne
C_2$, vanishes at the points $\{z_k\}$ that satisfy
    \begin{equation*}\index{linear differential equation}
    \frac{\log\frac1{1-z_k}}{1-z_k}=k\pi i+\frac{1}2\log\left(-\frac{C_2}{C_1}\right),\quad k\in\mathbb{Z}.
    \end{equation*}
It follows that $f$ has infinitely many zeros whose imaginary part
is strictly positive and the same is true for zeros with negative
imaginary part. This is in contrast to the first example in which
the zeros are real except possibly a finite number of them. Note
that the coefficients are essentially of the same growth in both
examples, but the leading term in the coefficient $B$ behaves
roughly speaking as $-B_0$, when $z$ is close to 1. It would be
desirable to establish a result similar to
Theorem~\ref{Thm:DE-Main}(i) for rapidly increasing weight $\om$,
and, in particular, to see how the geometric zero distribution of
solutions fits in the picture.

\chapter{Further Discussion}\label{sec:furtherdiscussion}

This chapter concerns topics that this monograph does not cover.
Here we will briefly discuss $q$-Carleson measures for $A^p_\om$
when $q<p$, generalized area operators acting on $A^p_\om$ as well
as questions related to differential equations and the zero
distribution of functions in $A^p_\om$. We include few open
problems that are particularly related to the weighted Bergman
spaces induced by rapidly increasing weights.

\section{Carleson measures}\index{Carleson measure}

The following  result can be proved by using arguments similar to
those in the proof of
Theorem~\ref{Thm-integration-operator-1}(iv).

\begin{theorem}\label{Thm-p>q-Carleson}
Let $0<q<p<\infty$, and let $\mu$ be a positive Borel measure
on~$\D$. If the function
    $$
    \Psi_\mu(u)=\int_{\Gamma(u)}\frac{d\mu(z)}{\om^\star(z)},\quad
    u\in\D\setminus\{0\},
    $$
belongs to $L^{\frac{p}{p-q}}_\om$, then $\mu$ is a $q$-Carleson
measure for $A^p_\om$.
\end{theorem}

Even though we do not know if the condition $\Psi_\mu\in
L^{\frac{p}{p-q}}_\om$ characterizes $q$-Carleson measures for
$A^p_\om$, we can show that this condition is necessary when
$A^p_\om$ is replaced by $L^p_\om$ and $\om\in\I\cup\R$. Namely,
if $I_d:\,L^p_\om\to L^q(\mu)$ is bounded and $g\in
L^{\frac{p}{q}}_\om$ is a positive function, then Fubini's
theorem, Lemma~\ref{le:cuadrado-tienda} and Corollary
\ref{co:maxbou} give
    \begin{equation*}\index{$M_\om(\vp)$}
    \begin{split}
    \int_\D g(u)\Psi_\mu (u)\om(u)\,dA(u)
    &=\int_\D\left(\frac{1}{\om^\star(z)}\int_{T(z)} g(u)\om(u)\,dA(u)\right)\,d\mu(z)\\
    &\lesssim\int_\D M_\om(g)(z)\,d\mu(z)
    =\int_\D \left(\left(M_\om(g)(z))\right)^{\frac{1}{q}}\right)^q\,d\mu(z)\\
    &\lesssim \int_\D \left(M_\om(g)(z))\right)^{\frac{p}{q}}\,\om(z)dA(z)
\lesssim\|g\|^{p/q}_{L^{\frac{p}{q}}_\om}.
    \end{split}
    \end{equation*}
Since $L^{\frac{p}{p-q}}_\om$ is the dual of
$L^{\frac{p}{q}}_\om$, we deduce $\Psi_\mu\in
L^{\frac{p}{p-q}}_\om$.

In view of these results it is natural to ask, if $q$-Carleson
measures for $A^p_\om$, with $\om\in\I\cup\R$, are characterized
by the condition $\Psi_\mu\in L^{\frac{p}{p-q}}_\om$?

\section{Generalized area operators}\index{generalized area operator}

Let $0<p<\infty$, $n\in\N$ and $f\in\H(\D)$, and let $\omega$ be a
radial weight. Then Theorem~\ref{ThmLittlewood-Paley} shows that
$f\in A^p_\om$ if and only if
    $$
    \int_\D\,\left(\int_{\Gamma(u)}|f^{(n)}(z)|^2
    \left(1-\left|\frac{z}{u}\right|\right)^{2n-2}\,dA(z)\right)^{\frac{p}2}\omega(u)\,dA(u)<\infty.
    $$
Therefore the operator
    $$\index{$N(f)$}\index{non-tangential maximal function}
    F_n(f)(u)=\left(\int_{\Gamma(u)}|f^{(n)}(z)|^2
    \left(1-\left|\frac{z}{u}\right|\right)^{2n-2}\,dA(z)\right)^{\frac{1}2},\quad
    u\in\D,
    $$
is bounded from $A^p_\om$ to $L^p_\om$ for each $n\in\N$. To study
the case $n=0$, let $\mu$ be a positive Borel measure on $\D$, and
consider the generalized area operator
    $$\index{$G^\om_\mu(f)$}
    G^\om_\mu(f)(\z)=\int_{\Gamma(\z)}|f(z)|\frac{d\mu(z)}{\om^\star(z)},\quad
    \z\in\D\setminus\{0\},
    $$
associated with a weight $\om\in\I\cup\R$. The method of proof of
Theorem~\ref{Thm-integration-operator-1} yields the following
result which is related to a study by Cohn on area
operators~\cite{Cohn}. In particular, the reasoning gives an
alternative way to establish the implication in
\cite[Theorem~1]{Cohn} in which the original proof relies on an
argument of John and Nirenberg in conjunction with the
Calderon-Zygmund decomposition.

\begin{theorem}\label{Cor:AreaOperator}
Let $0<p\le q<\infty$ and $\omega\in\I\cup\R$, and let $\mu$ be a
positive Borel measure on $\D$. Then $G^\om_\mu:A^p_\omega\to
L^q_\omega$ is bounded if and only if $\mu$ is a
$(1+p-\frac{p}{q})$-Carleson measure for $A^p_\omega$. Moreover,
$G^\om_\mu:A^p_\omega\to L^q_\omega$ is compact if and only if
$I_d:A^p_\omega\to L^p(\mu)$ is compact.
\end{theorem}

It would be interesting to find out the condition on $\mu$ that
characterizes the bounded area operators $G^\om_\mu$ from
$A^p_\omega$ to $L^q_\omega$, when $0<q<p<\infty$.

\section{Growth and oscillation of solutions}

In Chapter~\ref{difequ} we studied the interaction between the
growth of coefficients and the growth and the zero distribution of
solutions of linear differential equations. In particular, we
established a one-to-one correspondence between the growth of
coefficients, the growth of solutions and the oscillation of
solutions, when all solutions belong to the Bergman-Nevanlinna
class $\BN_\om$ and $\om$ is regular. In order to keep our
discussion simple here, let us settle to consider the second order
equation
    \begin{equation}\label{Eq:FurtherDE}
    f''+Bf=0, \quad B\in\H(\D).
    \end{equation}
Theorem~\ref{Thm:DE-Main}(i) states that, if $\om$ is regular,
then $B\in A^\frac12_{\widehat\om}$, if and only if, all
solutions~$f$ of \eqref{Eq:FurtherDE} belong to $\BN_\om$, if and
only if, the zero sequence $\{z_k\}$ of each non-trivial solution
satisfies $\sum_k\om^\star(z_k)<\infty$. Recall that
$\widehat{\omega}(r)=\int_r^1\om(s)\,ds$ and
    $$
    \omega^\star(z)=\int_{|z|}^1\omega(s)\log\frac{s}{|z|}s\,ds\asymp(1-|z|)\widehat{\omega}(|z|),\quad
    |z|\to1^-,
    $$
so the convergence of the sum $\sum_k\om^\star(z_k)$ is clearly a
weaker requirement than the Blaschke condition
$\sum_k(1-|z_k|)<\infty$. If $\om$ is rapidly increasing, then
Theorem~\ref{Thm:DE-Main}(ii) and the examples given at the end of
Section~\ref{Sec:SolutionsBergmanNevanlinna} show that the
situation is more delicate and it is not enough to measure the
distribution of zeros by a condition depending on their moduli
only. Needless to say that things get even more complex in general
if all solutions belong to the weighted Bergman space~$A^p_\om$
induced by $\om\in\I\cup\R$, see Chapter~\ref{sec:factorizacion}
and Section~\ref{Sec:SolutionsBergman}. The final end in the
oscillation is of course the growth restriction
$|B(z)|(1-|z|^2)^2\le1$, which guarantees that each non-trivial
solution of \eqref{Eq:FurtherDE} vanishes at most once in the unit
disc. This statement is equivalent to the well known theorem of
Nehari that provides a sufficient condition for the injectivity of
a locally univalent meromorphic function in~$\D$ in terms of the
size of its Schwarzian derivative. In view of the results in the
existing literature and those presented in Chapter~\ref{difequ},
it appears that the connection between the growth of the
coefficient $B$ and the growth and the oscillation of solutions
$f$ of \eqref{Eq:FurtherDE} is not well understood when Nehari's
condition fails, but Theorem~\ref{Thm:DE-Main}(i) is too rough.
This case seems to call for further research. In particular, it
would be desirable to establish a result similar to
Theorem~\ref{Thm:DE-Main}(i) for rapidly increasing weights, and
to see how the geometric zero distribution of solutions fits in
the picture. Going further, the equations \eqref{Eq:FurtherDE}
with solutions whose zeros constitute Blaschke sequences have
attracted attention during the last decade. For the current stage
of the theory of these equations as well as open problems we refer
to~\cite{H2011}.

\section{Zero distribution}

It is known that the density techniques employed in the study of
$A^p_\a$-zero sets and interpolating and
sampling\index{interpolating and sampling sequences} sequences are
related~\cite{S3,S2,S4}. When we were finishing the typesetting of
this monograph in July 2011, Jordi Pau kindly brought our
attention to the existence of the very recent paper by
Seip~\cite{Seip} on interpolation and sampling on the Hilbert
space $A^2_\om$, where $\om$ is assumed to be continuous such that
$\om(r)\lesssim\om(\frac{1+r}{2})$ for all $0\le r<1$. It is clear
that each weight $\om\in\widetilde{\I}\cup\R$ satisfies this
condition. Seip characterizes interpolating and sampling sequences
for $A^2_\om$ by using densities and points out that the
techniques used work also for $A^p_\om$ without any essential
changes of the arguments. It is natural to expect that these
techniques can be used to obtain information on the zero
distribution of functions in $A^p_\om$.

\bigskip

\subsection*{Acknowledgements}

We would like to thank the referee for a careful and meticulous
reading of the manuscript. In particular, the referee's comments
with regards to the question of when polynomials are dense in the
weighted Bergman space induced by a non-radial weight are highly
appreciated. The referee's suggestions and observations improved
the exposition of the monograph and undoubtedly made it easier to
read. We would also like to thank Alexandru Aleman for inspiring
discussions on the topic.

\backmatter
\bibliographystyle{amsalpha}

\printindex

\end{document}